\documentclass[12pt]{amsart}

\usepackage[T1]{fontenc}
\usepackage{amsmath,amssymb,amsfonts,amsthm}
\usepackage[foot]{amsaddr}
\usepackage{mathrsfs}
\usepackage{dsfont}
\usepackage{BOONDOX-calo}
\usepackage{bigints}
\usepackage{eurosym}

\usepackage[margin=0.8in]{geometry}
\usepackage{a4wide} %
\usepackage{changepage}
\usepackage{setspace}
\usepackage[english]{babel}
\usepackage{url}
\usepackage{xcolor}
\usepackage{ragged2e}
\usepackage{verbatim}
\usepackage{lipsum}
\usepackage{paralist}
\usepackage{appendix}

\usepackage{graphicx}
\usepackage{epsfig}
\usepackage{epstopdf}
\usepackage{subfigure}
\usepackage{float}
\usepackage{caption}
\usepackage{booktabs}
\usepackage{threeparttable}
\usepackage{multirow}
\usepackage{array}
\usepackage{adjustbox}
\usepackage{scalefnt}
\usepackage{rotating}
\usepackage{lscape}
\usepackage{pdflscape}
\usepackage{pgf}

\usepackage{algorithm}
\usepackage{algpseudocode}

\usepackage{natbib}
\usepackage[
  colorlinks=true,
  pdfstartview=FitV,
  linkcolor=blue,
  citecolor=blue,
  urlcolor=blue
]{hyperref}
\theoremstyle{definition}
\newtheorem{assumption}{Assumption}[section]
\newtheorem{example}{Example}[section]

\DeclareMathOperator{\var}{var}

\renewcommand{\P}{{\mathsf P}}

\newcommand{\by}{{\bf y}}

\newcommand{\cV}{{\mathcal  V}}

\newcommand{\cD}{{\mathcal  D}}

\newcommand{\bx}{{\bf x}}
\newcommand{\bX}{{\bf X}}
\newcommand{\bY}{{\bf Y}}

\newcommand{\E}{{\mathsf E}\hspace{0.1mm}}

\newcommand{\bR}{\mathbb R}

\allowdisplaybreaks

\allowdisplaybreaks
\DeclareFontFamily{OT1}{pzc}{}
\DeclareFontShape{OT1}{pzc}{m}{it}{<-> s * [1.10] pzcmi7t}{}
\DeclareMathAlphabet{\mathpzc}{OT1}{pzc}{m}{it}

\renewcommand{\P}{P}

\newtheoremstyle{exampstyle}
{4pt} 
{4pt} 
{\itshape} 
{} 
{\bfseries} 
{.} 
{.75em} 
{} 
\theoremstyle{exampstyle}

\numberwithin{table}{section}
\numberwithin{figure}{section}
\newtheorem{theorem}{Theorem}[section]
\newtheorem{lemma}{Lemma}[section]

\newcommand{\cq}{{\mathcal q}}

\def\beq{\begin{equation}}
\def\eeq{\end{equation}}
\def\bals{\begin{align*}}
\def\eals{\end{align*}}
\def\bal{\begin{align}}
\def\eal{\end{align}}

\numberwithin{equation}{section}
\numberwithin{theorem}{section}
\numberwithin{corollary}{section}
\allowdisplaybreaks
\begin{document}

\begin{adjustwidth}{-0.5in}{-0.5in}
\title[Changepoint detection with \textit{U}-statistics]{Sequential
monitoring for distributional changepoints using degenerate \textit{U}%
-statistics}
\author{B.\ Cooper Boniece$^{1*}$}
\email{cooper.boniece@drexel.edu}
\author{Lajos Horv\'ath$^2$}
\email{horvath@math.utah.edu}
\author{Lorenzo Trapani$^3$}
\email{lorenzo.trapani@unipv.it}
\address{$^{1}$Department of Mathematics, Drexel University, Philadelphia,
PA 19104 USA }
\address{$^2$Department of Mathematics, University of Utah, Salt Lake City,
UT 84112--0090 USA }
\address{$^3$Department of Economics and Management, University of Pavia,
Pavia, Italy; Department of Economics, Finance and Accounting; School of
Business and Economics, University of Leicester, Leicester, U.K.}
\address{$^*$Research supported in part under NSF grant DMS-2413558.}

\begin{abstract}

We investigate the online detection of changepoints in the distribution of a sequence of observations using a class of degenerate \textit{U}-statistic-type processes. We consider an ordinary (Chu--Stinchcombe--White-type) detector and a Page-type detector under open- and closed-ended monitoring, and introduce an expanding-baseline Page-type procedure that incorporates sufficiently old monitoring observations into the baseline sample. Under the null, we derive weak limits for all three procedures and justify a Monte Carlo approximation to their critical values. For the ordinary and Page-type detectors, we also establish consistency and limiting distributions for detection delays under both early and late changes. The theory requires only square summability of the eigenvalues associated with the degenerate kernel operator, rather than the stronger absolute-summability condition often imposed in related work. Simulations show competitive performance relative to recent mean-, covariance-, and empirical-CDF-based monitors, and an application to multivariate compressor-sensor data from a metro train illustrates the methodology.

\end{abstract}

\maketitle
\end{adjustwidth}

\doublespacing

\section{Introduction\label{intro}}

\noindent We consider the online detection of distributional changepoints in a sequence ${\mathbf X_i,i\geq1}$. Suppose that an initial historical sample ${\mathbf X_i,1\leq i\leq m}$ has common distribution $F$. As new observations arrive, we sequentially test
\begin{equation}
H_0:\mathbf X_{m+k}\sim F,\qquad k\geq1,
\label{null}
\end{equation}
against the alternative that the distribution changes at some monitoring time $k_\ast$.

Detecting structural instability is
arguably of great importance in all applied sciences. Examples include
economics and finance, where instability has direct implications for
forecasting and decision-making (see e.g. \citealp{smith2021break});
engineering, where the safety and serviceability of engineering structures
requires continuous monitoring (see e.g. \citealp{sun2020review}, and %
\citealp{malekloo2022machine}); and the analysis of biomedical time series
data (\citealp{fiecas2024generalised}). In many such applications, interest lies in changes in the entire distribution rather than specific moments such as the mean or the variance, whence the importance of testing for \textit{%
distributional} changes: \citet{fu2023multiple}, \textit{inter alia},
discuss several examples in economics and finance, including density
forecast and the detection of changes in the tail risk of financial
variables. Although the changepoint literature is quite extensive (see \citet{aue2024state} and \citet{chgreg}
for recent selective reviews), online detection of general distributional changes remains comparatively underexplored. Some recent contributions include papers on retrospective, offline detection
by \citet{inoue2001testing}, who uses the empirical distribution function,
and \citet{huvskova2006change} and \citet{boniece2025changepoint} who, 
\textit{inter alia}, use the empirical characteristic function; see also \citet{horvath2021monitoring}, for a recent online contribution.

\smallskip

\noindent Motivated by this gap, we develop an online framework based on a class of two-sample \textit{U}-statistics.  The use of \textit{U}%
-statistics-type processes in the context of changepoint detection goes back to at least \citet{csorgHo1989invariance}, and subsequently studied
in several contributions - examples include \citet{matteson:james:2014}, %
\citet{biau:bleakley:mason:2016} and \citet{dehling2022change} for
retrospective changepoint detection, and \citet{kirch2022sequential} for online
detection. In related work, \citet{boniece2025changepoint} study retrospective distributional
changepoint detection for functional data using a special case of the (generalised) energy distance (%
\citealp{szekely:rizzo:2005} and \citealp{szekely2013energy}). 
By developing online monitoring procedures for a broad class of degenerate two-sample \textit{U}-statistics, we can accommodate a wide range of distributional discrepancies, including rotation-invariant distances such as the energy distance that are naturally suited to multivariate data - as opposed to the use of statistics based on e.g.
Cramer's distance, such as the ones employed in \citet{inoue2001testing}. In our setting, at each monitoring time, the two-sample \textit{U}-statistics compare the historical sample with the monitoring observations available thus far. Under the null, the resulting processes fluctuate around zero, whereas after a distributional change they develop a systematic drift away from zero. In the retrospective setting, for a particular form of generalized energy distance, \citet{boniece2025changepoint} show that the corresponding process behaves, up to an asymptotically negligible remainder, like the square of a CUSUM process. We establish an analogous representation in the online setting for a broader class of degenerate two-sample \textit{U}-statistics.

\smallskip

\textit{Main contributions of this paper}

\smallskip

\noindent We make several contributions to the extant literature. First, we propose three distributional monitoring procedures based on degenerate two-sample \textit{U}-statistics: an ordinary (Chu--Stinchcombe--White-type) monitor, a Page-type monitor, and an expanding-baseline Page-type variant. The last incorporates sufficiently old monitoring observations into the baseline sample while restricting the comparison to a window of recently observed data, with the aim of reducing the dilution of later changes by unaffected monitoring observations. Second, we derive weak limits for all three procedures under the null in open- and closed-ended monitoring settings. The limits are expressed as suprema of weighted infinite sums of centered squared Wiener processes, with weights determined by the eigenvalues of the operator associated with the degenerate part of the kernel. This representation also yields a practical Monte Carlo method for obtaining critical values. For the ordinary and Page-type procedures, we establish consistency and derive detection-delay limits under both early and late changes.\\
\noindent A principal technical contribution is that the asymptotic theory requires only square summability of the kernel eigenvalues, rather than the stronger absolute-summability condition imposed in related work. Since square summability follows directly from the assumed second moment of the kernel, it can be verified without explicit knowledge of the kernel eigenvalues; by contrast, absolute summability is a stronger spectral condition that is generally difficult to check for an arbitrary kernel, since they depend on the (unknown) distribution $F$.  This allows the framework to accommodate a broad class of distance- and kernel-based discrepancies. 
 We also develop a retrospective procedure for assessing stability of the historical sample, discuss the construction of distribution-distinguishing kernels, and establish the validity of the proposed eigenvalue-based calibration. Simulations compare the procedures with recent mean-, covariance-, and empirical CDF monitors, and their use is further illuminated by two data illustrations.

\smallskip

\noindent This paper is organised as follows. Section \ref{fmwk} introduces the monitoring schemes, Sections \ref{main},  and \ref{s:refinements} develop the asymptotic theory, extensions, kernel choices, and calibration, and Section \ref{s:simulations} presents simulations and applications. Section \ref{conclusion} concludes, and all proofs and additional numerical results are provided in the Supplement.

\smallskip

\noindent NOTATION. Throughout, for positive sequences $a_{m},b_{m}$, we
write $a_{m}\sim b_{m}$ if $a_{m}/b_{m}\rightarrow 1$ as $m\rightarrow
\infty $. We denote $a_{m}\ll b_{m}$ to mean $a_{m}=o(b_{m})$ and similarly $%
a_{m}\gg b_{m}$ means $b_{m}=o(a_{m})$ as $m\rightarrow \infty $.
Convergence in distribution is denoted as $\overset{\mathcal{D}}{\rightarrow 
}$. We denote binomial coefficients as $\binom{p}{q}$.  We often write $a\vee b=\max \left\{ a,b\right\} $ and $a\wedge b=\min \left\{
a,b\right\} $.  Other relevant
notation is introduced later on in the paper.

\section{Assumptions and monitoring schemes\label{fmwk}}

\noindent Let $\mathbf{X}_{1}$, $\mathbf{X}_{2}$, $\ldots $ be a sequence of
random elements taking values in a separable metric space $(\mathcal{X},\rho
)$. We assume that there exists a historical training (or baseline) period $\left\{ 
\mathbf{X}_{i},1\leq i\leq m\right\} $ during which no change took place.
Letting $F$ denote the distribution of $\mathbf{X}_{1}$, we make the
following

\begin{assumption}
\label{a:historical_stability}It holds that $\mathbf{X}_{i}\sim F$ for all $%
i=1,\ldots ,m$.
\end{assumption}

\noindent Assumption \ref{a:historical_stability} is typical in this
literature, where it is also known as the \textit{noncontamination assumption%
} (\citealp{chu1996monitoring}). In the spirit of making assumptions that
are testable, as mentioned in the introduction, in Section \ref{sectrain} we
construct a test (based on the same approach as discussed herein for online
monitoring) to check retrospectively for no changepoint in the distribution
of $\left\{ \mathbf{X}_{i},1\leq i\leq m\right\} $. 

\smallskip

\noindent After the training period, incoming observations $\mathbf{X}_{m+k}$
are monitored, where $k\geq 1$ denotes the ``current'' monitoring time; we test for the null hypothesis
of distributional stability versus the alternative hypothesis that a change
occurs in the distribution at some point in time $k_{\ast }$:%
\begin{equation}
H_{A}:\mathbf{X}_{m+k}\sim 
\begin{cases}
F & k=1,2,\ldots ,k_{\ast }, \\ 
F_{\ast } & k=k_{\ast }+1,k_{\ast }+2,\ldots%
\end{cases}
\label{alternative}
\end{equation}%
where $k_{\ast }\geq 1$, and $F_{\ast }\neq F$ is an unspecified
distribution on $\mathcal{X}$. %

\smallskip

\noindent Throughout this work we use the following assumption.

\begin{assumption}
\label{e:assumption_ind}It holds that $\left\{ \mathbf{X}_{i},~i\geq1\right\} $ is an independent sequence.%
\end{assumption}

\noindent We now present the monitoring schemes, starting with a preview of
how they work. At each point during the monitoring horizon, $k$, we
construct a ``detector'' ${\mathcal{D}}%
_{m}(k)$, based on comparing the observations in the historical training
sample $\left\{ \mathbf{X}_{i},1\leq i\leq m\right\} $ against the
observations available in the monitoring sample up until $k$ $\left\{ 
\mathbf{X}_{m+i},1\leq i\leq k\right\} $. As mentioned in the introduction,
such a detector (heuristically) is constructed as a partial sum process of
quantities which, under the null of no break, have mean zero; consequently,
as $k$ increases, under the null ${\mathcal{D}}_{m}(k)$ should range within
a ``boundary (function)'' which evolves
with $k$, say $g_{m}(k)$. As soon as such boundary is crossed, the null is
rejected and a changepoint is marked; formally, $H_{0}$ is rejected as soon
as %
\begin{equation}
{\mathcal{D}}_{m}(k)>\mathcal{c}g_{m}(k),
\end{equation}%
where the constant $\mathcal{c}>0$ is a critical value chosen
in conjunction with the historical sample to control the asymptotic false alarm rate.\newline
We now introduce our detectors. Following \citet{matteson:james:2014}, %
\citet{biau:bleakley:mason:2016} and \citet{dehling2022change}, our
detectors ${\mathcal{D}}_{m}(k)$ are based on a subclass of \textit{degenerate} \textit{U}%
-statistics (see e.g. \citealp{vandervaart:1998}, for a general treatment).
Let $h:\mathcal{X}\times \mathcal{X}\rightarrow \mathbb{R}$ be any function
satisfying

\begin{assumption}
\label{a:assumption_on_h}It holds that $h(\mathbf{x},\mathbf{y})=h(\mathbf{y}%
,\mathbf{x})$; for \textit{i.i.d.}~elements $\mathbf{X},\mathbf{Y}\sim F$,
it holds that%
\begin{equation}
{\mathsf{E}}\hspace{0.1mm}h^{2}(\mathbf{X},\mathbf{Y})=\iint h^{2}(\mathbf{x}%
,\mathbf{y})dF(\mathbf{x})dF(\mathbf{y})<\infty .
\end{equation}
\end{assumption}

\noindent Assumption \ref{a:assumption_on_h} requires the second moment of $%
h(\mathbf{X},\mathbf{Y})$ to be finite. Heuristically, our statistics are
based on sums of $h\left( \mathbf{X}_{i},\mathbf{X}_{j}\right) $, and
therefore assuming that the second moment thereof is a natural requirement
to derive the asymptotics. As mentioned in the introduction, this part of
the assumption is testable: given a (user-chosen) kernel $h\left( \cdot
,\cdot \right) $, it can be checked whether its second moment is finite or
not based e.g. on the procedures discussed in \citet{trapani2016testing} and %
\citet{degiannakis2023superkurtosis}. Indeed, the assumption is
``constructive'': after determining how
many moments are admitted by the data, a $h\left( \cdot ,\cdot \right) $ can
be chosen, by the applied user, so as to satisfy the assumption.\newline
\noindent Note, importantly, that the assumption on the finiteness of the
second moment is for the kernel $h(\mathbf{X},\mathbf{Y})$, and \textit{not
for the data} $\mathbf{X}$: hence, $\mathbf{X}$ need not even admit any
finite polynomial moment \textit{per se}, as long as an appropriate kernel
is chosen.  

\smallskip

\noindent Given a kernel $h(\mathbf{x},\mathbf{y})$ satisfying Assumption %
\ref{a:assumption_on_h}, for each $m$ and $k\geq 2$, let 
\begin{align}
U_{m}(h;k)& =\frac{2}{km}\sum_{i=1}^{m}\sum_{j=m+1}^{m+k}h(\mathbf{X}_{i},%
\mathbf{X}_{j})-\binom{m}{2}^{-1}\sum_{1\leq i<j\leq m}h(\mathbf{X}_{i},%
\mathbf{X}_{j})  \label{e:def_Um(h;k)} \\
& -\binom{k}{2}^{-1}\sum_{m<i<j\leq m+k}h(\mathbf{X}_{i},\mathbf{X}_{j}). 
\notag
\end{align}%
Hence, given  $h$, we then define the ordinary detector (c.f.~\cite{chu1996monitoring})
\begin{equation}
\mathcal{D}_{m}^{(1)}(k)=m^{-1}k^{2}\left\vert U_{m}(h;k)\right\vert ,
\label{cusum}
\end{equation}%
and its Page-type counterpart (see \citealp{page1954continuous}; %
\citealp{fremdt2015page}; and \citealp{aue2024state}) 
\begin{equation}
\mathcal{D}_{m}^{(2)}(k)=m^{-1}\max_{0\leq r<k}(k-r)^{2}\left\vert
U_{m}(h;r,k)\right\vert ,  \label{page}
\end{equation}%
where for each $m,k\geq 2$, $0\leq r<k-1$, 
\begin{align}
U_{m}(h;r,k)& =\frac{2}{(k-r)m}\sum_{i=1}^{m}\sum_{j=m+r+1}^{m+k}h(\mathbf{X}%
_{i},\mathbf{X}_{j})  \label{e:def_page} \\
& -\binom{m}{2}^{-1}\sum_{1\leq i<j\leq m}h(\mathbf{X}_{i},\mathbf{X}_{j})-%
\binom{k-r}{2}^{-1}\sum_{m+r<i<j\leq m+k}h(\mathbf{X}_{i},\mathbf{X}_{j}). 
\notag
\end{align}%
\noindent We use the following family of \textit{weighted} boundary
functions 
\begin{equation}
g_{m}(k)=\left( \frac{k/m}{1+k/m}\right) ^{\beta }\left( 1+\frac{k}{m}%
\right) ^{2}=g\left( \frac{k}{m}\right) .  \label{e:boundaryfxn}
\end{equation}%
As is typical in this literature, the boundary functions defined in (\ref%
{e:boundaryfxn}) depend on a user-chosen weight $0\leq \beta <1$, which
determines the weights assigned to the fluctuations of $U_{m}(h;r,k)$: as $%
\beta $ increases, the weight also increases, and therefore higher
power/faster detection under the alternative may be expected. \citet{lajos04}%
, \citet{lajos07} and \citet{ghezzi2024fast} study online changepoint
detection based on the CUSUM process with various values of $\beta $; %
\citet{horvath2023real} and \citet{horvath2025sequential} study a weighted
version of the Page-CUSUM process.%

\noindent Although the specified kernel $h$ need not be degenerate,\footnote{Here, $h$ degenerate means  $\E h(\bx,\bX)\equiv 0$; i.e., its first-order Hoeffding projection is identically zero. } the
two-sample $U$-statistic \eqref{e:def_Um(h;k)} depends only on its degenerate part. Indeed,
if
\begin{equation}
\overline{h}(\mathbf{x},\mathbf{y})=h(\mathbf{x},\mathbf{y})-{\mathsf{E}}%
\hspace{0.1mm}h(\mathbf{X},\mathbf{y})-{\mathsf{E}}\hspace{0.1mm}h(\mathbf{x}%
,\mathbf{Y})+{\mathsf{E}}\hspace{0.1mm}h(\mathbf{X},\mathbf{Y}).
\label{e:degenerate_version_of_h}
\end{equation}%
where $\bX,\bY\stackrel{i.i.d.}{\sim}F$, then a simple algebraic cancellation shows
$$
U_m(h;r,k)=U_m(\overline h;r,k).
$$
Thus, while $h$ itself is arbitrary
subject to Assumption~\ref{a:assumption_on_h}, the effective kernel entering the asymptotic
theory is the degenerate kernel $\overline h$. Equivalently, $U_m(h;r,k)$ can be represented as a two-sample $U$-statistic whose first-order Hoeffding projections vanish under $H_0$, and is therefore a degenerate $U$-statistic.

\noindent For a chosen detector ${\mathcal{D}}_{m}(k)$, we consider two
separate types of monitoring schemes. First, an ``
open-ended'' (or indefinite monitoring horizon) scheme,
based on the stopping rule 
\begin{equation}
\mathcal{\tau }_{m}=\mathcal{\tau }_{m}(\mathcal{c})=%
\begin{cases}
\min \{k\geq 2:{\mathcal{D}}_{m}(k)>\mathcal{c}g_{m}(k)\} \\ 
\infty ,\quad \text{if ~}~{\mathcal{D}}_{m}(k)\leq \mathcal{c}g_{m}(k)~\text{%
~for all~}~k\geq 2.%
\end{cases}
\label{stop-infty}
\end{equation}%
The procedure goes on forever, until it rejects $H_{0}$ - corresponding to
having $\mathcal{\tau }_{m}<\infty $.

\noindent However, by definition, monitoring based on $\tau _{m}$ may 
\textit{never} terminate, which may not be suitable in some applications.
Thus, we also consider finite horizon (or ``closed'') monitoring schemes, which are based on the
stopping rule 
\begin{equation}
\mathcal{\tau }_{m,M}=\mathcal{\tau }_{m,M}(\mathcal{c})=%
\begin{cases}
\min \{2\leq k\leq M-1:{\mathcal{D}}_{m}(k)>\mathcal{c}g_{m}(k)\} \\ 
M,\quad \text{if ~}~{\mathcal{D}}_{m}(k)\leq \mathcal{c}g_{m}(k)~\text{~for
all~}~2\leq k<M,%
\end{cases}
\label{stop-M}
\end{equation}%
where $M\geq 2$ is a user-specified monitoring horizon.\footnote{%
Formally, monitoring based on $\tau _{m,M}$ rejects $H_{0}$ if $\mathcal{%
\tau }_{m,M}<M$}

\section{Main results\label{main}}

\noindent We report results under the null and under the alternatives for
the monitoring schemes based
on the detectors $\mathcal{D}_{m}^{(1)}(k)$\ and $\mathcal{D}_{m}^{(2)}(k)$;
a novel scheme is introduced in Section \ref{repurpose}. From here on, we
assume that Assumptions \ref{a:historical_stability}-\ref{a:assumption_on_h}
are in force, and thus we omit them from the statements of our results.

\noindent Let $\mathbf{X},\mathbf{Y}\overset{iid}{\sim }F$. For a given $h$ satisfying Assumption \ref{a:assumption_on_h}, let $\overline h$ be as in \eqref{e:degenerate_version_of_h}. 
To the function $h$, we associate the integral operator $A:\mathcal{L}%
^{2}(F)\rightarrow \mathcal{L}^{2}(F)$, defined by $Ag(\mathbf{x})={\mathsf{E%
}}\hspace{0.1mm}\overline h(\mathbf{x},\mathbf{Y})g(\mathbf{Y})=\int \overline{h}(%
\mathbf{x},\mathbf{y})g(\mathbf{y})dF(\mathbf{y})$. Under Assumption %
\eqref{a:assumption_on_h}, the spectral theorem (e.g. %
\citealp{riesz2012functional}) yields that there exists an orthonormal basis 
$\{\phi _{k}\}_{k\geq 1}$ of $\mathcal{L}^{2}(F)$ such that $A\phi _{\ell
}=\lambda _{\ell }\phi _{\ell }$, $\ell \geq 1$, where $\lambda _{\ell }\in 
\mathbb{R}$ for all $\ell \geq 1$,\footnote{%
With no loss of generality, we assume they are ordered as $\left\vert
\lambda _{1}\right\vert \geq \left\vert \lambda _{2}\right\vert \geq \ldots $%
} such that 
\begin{equation}
\sum_{\ell =1}\lambda _{\ell }^{2}<\infty .  \label{e:eigs_square_summable}
\end{equation}%
\noindent Let $\{W_{1}(u),u\geq 0\}$, $\{W_{2}(u),u\geq 0\}$, \ldots\ be
independent Wiener processes, and define 
\begin{equation}
\Gamma (u)=\sum_{\ell =1}^{\infty }\lambda _{\ell }\left( W_{\ell
}^{2}(u)-u\right) ;  \label{e:def_gamma}
\end{equation}%
the process defined in (\ref{e:def_gamma}) is typically found when studying
the limiting distribution of degenerate \textit{U}-statistics (e.g. %
\citealp{boos1980note}).

\subsection{Monitoring under $H_{0}$\label{monitor-H0}}

\noindent Based on the stopping rules defined in (\ref{stop-infty}) and (\ref%
{stop-M}), the case of no detection taking place corresponds to the events $%
\left\{ \tau _{m}=\infty \right\} $ and $\left\{ \tau _{m,M}=M \right\} 
$\ respectively. In this section, we study the probability of such events
under the null hypothesis $H_{0}$\ - and, therefore, the asymptotic
distribution of our statistics.\newline
We begin by presenting the limiting behaviour of the detector $\mathcal{D}_{m}^{(1)}(k)$\ defined in (\ref{cusum}). 

\begin{theorem}
\label{t:theorem_under_H0} Assume $H_{0}$ holds, and consider the detector $%
\mathcal{D}_{m}(k)=\mathcal{D}_{m}^{(1)}(k)$. Let $g_{m}$ be as in %
\eqref{e:boundaryfxn}. 
(i) As $m\rightarrow \infty $, 
\begin{equation}
\P \left\{ \tau _{m}<\infty \right\} \rightarrow \P \Big\{%
\sup_{0<u<1}u^{-\beta }|\Gamma (u)|>\mathcal{c}\Big\}.
\end{equation}%
(ii) Suppose $M=M_{m}\rightarrow \infty $ such that $M/m\rightarrow a_{0}\in
(0,\infty ]$, and let $u_{0}=a_{0}/(1+a_{0})$. Then
\begin{equation}
\P \{\tau _{m,M}<M \}\rightarrow \P \Big\{\sup_{0<u<u_{0}}u^{-\beta
}|\Gamma (u)|>\mathcal{c}\Big\}.
\end{equation}%
(iii) Suppose $M=M_{m}\rightarrow \infty $ such that $M/m\rightarrow 0$, and let
the boundary function $g_{m}$ be given by $g_{m}(k)=(M/m)\left( k/M\right)
^{\beta }$. Then
\begin{equation}
\P \{\tau _{m,M}<M \}\rightarrow \P \Big\{\sup_{0<u<1}u^{-\beta
}|\Gamma (u)|>\mathcal{c}\Big\}.
\end{equation}%
\end{theorem}

\noindent Theorem \ref{t:theorem_under_H0} contains the limiting
distribution of the test statistics in various cases. Part \textit{(i)} of
the theorem refers to an open-ended, indefinite-horizon monitoring scheme;
asymptotic control of the Type I error rate under the null is guaranteed by
choosing $\mathcal{c}=\mathcal{c}_{\alpha }$ such that $\P \Big\{%
\sup_{0<u<1}u^{-\beta }|\Gamma (u)|>\mathcal{c}_{\alpha }\Big\}=\alpha$.
Parts \textit{(ii)} and \textit{(iii)} provide analogous statements in the
finite-horizon monitoring setting. In particular, part \textit{(ii)}
corresponds to a ``long-horizon''
monitoring, in the sense the monitoring horizon $M$ is either comparable or
much larger than the length of the historical sample $m$. The limiting
distribution in both cases is given by the supremum of the weighted version
of $|\Gamma (u)|$; the only difference is in the interval over which the
supremum is taken. From a practical point of view, the relevant case is
always \textit{(ii)} - that is, critical values should be always computed
from the supremum taken over the interval $\left( 0,u_{0}\right) $, and case 
\textit{(i)} can be viewed as an always more conservative asymptotic
approximation.\\
\noindent Finally, part \textit{(iii)} corresponds to ``
short-horizon'' monitoring, where the length of the
monitoring horizon is effectively negligible compared to the length of the training period. In all cases, the critical values $\mathcal{c}%
_{\alpha }$ can be derived by simulations, based on the definition of $%
\Gamma (u)$ in (\ref{e:def_gamma}) - see Section \ref{implementation}.

\smallskip

\noindent We now study the limiting behavior of Page-type
monitoring scheme, based on $\mathcal{D}_{m}^{(2)}(k)$ defined in (\ref{page}%
). Define
the two parameter processes
\begin{align}
Z_\ell(u,v) &= W_{\ell }\left(
u\right) -\frac{1-u}{1-v}W_{\ell }\left( v\right)\\
G(u,v) &=\sum_{\ell =1}^{\infty }\lambda _{\ell }\bigg[Z^2_\ell(u,v)-\left(
u-v\left( \frac{1-u}{1-v}\right) \right) \left( 1-v\left( \frac{1-u}{1-v}
\right) \right) \bigg],  \label{e:G(u,v)}
\end{align}%
for each $0\leq u,v\leq 1$, with $\{W_{\ell }(u),u\geq 0\}$ as in %
\eqref{e:def_gamma}, and let 
\begin{equation}
\overline{\Gamma }(u)=\sup_{0<v< u}|G(u,v)|,\quad 0\leq v\leq 1.
\label{gamma-bar}
\end{equation}

\begin{theorem}
\label{t:page_under_H0} Assume $H_{0}$ holds, and consider the Page-type
detector $\mathcal{D}_{m}(k)=\mathcal{D}_{m}^{(2)}(k)$. Then the statements
of Theorem \ref{t:theorem_under_H0} hold with $\overline{\Gamma }$ replacing 
$\Gamma $. 
\end{theorem}

\subsection{Monitoring under the alternative\label{monitor-HA}}

\noindent Consider the following notation. Let $F_{\ast }=\theta G+(1-\theta
)F$, where $0<\theta <1$, and $G(\mathbf{x})$ is a distribution function
which, under the alternative, ``
contaminates'' $F$. Define\newline
\begin{align}
h_{1}(\mathbf{x})&=\int h(\mathbf{x},\mathbf{y})dF(\mathbf{y}),\quad h_{2}(%
\mathbf{x})=\int h(\mathbf{x},\mathbf{y})dF_{\ast }(\mathbf{y}),
\label{e:def_v(x)-2} \\
v(\mathbf{x})&=\int h(\mathbf{x},\mathbf{y})d(F(\mathbf{y})-G(\mathbf{y}%
))=\theta ^{-1}\left( h_{1}(\mathbf{x})-h_{2}(\mathbf{x})\right) ,
\label{e:def_v(x)} \\
\nu _{1}&=\int v(\mathbf{x})dF(\mathbf{x}),\quad \nu _{2}=\int v(\mathbf{x}%
)dF_{\ast }(\mathbf{x}).  \label{e:def_nu}
\end{align}

\begin{assumption}
\label{a:basic_HA} As $m\rightarrow \infty $, $m\theta ^{2}|\mathfrak{D}%
_{h}(F,G)|\rightarrow \infty$, where $\mathfrak{D}_{h}(F,G)$ is defined in %
\eqref{e:h_divergence}.
\end{assumption}
\noindent Assumption \ref{a:basic_HA} states that the change
can be ``small'', but not ``
too small'', in order for it to be detected. In particular,
whenever $|\mathfrak{D}_{h}(F,G)|\neq 0$, the ``degree of
contamination'' $\theta $ is required to be larger than $%
O\left( m^{-1/2}\right) $, but it can drift to zero, corresponding to the
case of a ``vanishing break''. By (\ref{e:def_nu}), 
$\theta ^{-1}(\nu _{1}-\nu _{2})$ $=$ $\int h(\mathbf{x},\mathbf{y}%
)d(F-G)^{2}(\mathbf{x},\mathbf{y})$ $=$ $\mathfrak{D}_{h}(F,G)$; hence
Assumption \ref{a:basic_HA} can be equivalently written as $m\theta |\nu
_{1}-\nu _{2}|\rightarrow \infty $, which is used in the proofs.

\noindent We also make the following regularity assumption regarding local alternatives:

\begin{assumption}
\label{a:2ndHa_cond} For $\mathbf{X}\sim F$ and $%
\mathbf{X}^{\ast }\sim F_{\ast }$,  $\sigma ^{2}=\var(v(\mathbf{X}))$,
and $\sigma _{\ast }^{2}=\var(v(\mathbf{X}^{\ast }))$, it holds that $\sigma |\mathfrak{D}
_{h}(F,G)|^{-1/2}\rightarrow \zeta$, and $\sigma _{\ast }|\mathfrak{D}%
_{h}(F,G)|^{-1/2}\rightarrow \zeta _{\ast }$, for some $\zeta,\zeta_*>0$
\end{assumption}

Note under Assumption %
\ref{a:basic_HA}, $\sigma _{\ast }$ and $\sigma $ may also drift to zero. 
\begin{theorem}
\label{alt:consistency}Under Assumptions \ref{a:basic_HA} and \ref{a:2ndHa_cond}, when either $%
\mathcal{D}_{m}(k)=\mathcal{D}_{m}^{(1)}(k)$ or $\mathcal{D}_{m}^{(2)}(k)$,
it holds that $\lim_{m\rightarrow \infty }P\left( \tau _{m}<\infty \right) $ 
$=$ $1$.
\end{theorem}

\noindent Whenever $(\mathcal{X},\rho )$ has strong negative type (see
Example \ref{ex:negtype} below), then under the choice $h(\mathbf{x},\mathbf{%
y})=\rho (\mathbf{x},\mathbf{y})$, Theorem \ref{alt:consistency} states that
our procedure is consistent against \textit{all} distributional change
alternatives, as long as $m\theta ^{2}\rightarrow \infty $.\\

\noindent The results derived thus far refer to power.  We now report several results, under $H_{A}$,
concerning  
\textit{detection delay} associated with our procedures $\kappa _{m}-k_{\ast
}$, where 
\begin{equation}
\kappa _{m}=%
\begin{cases}
\min \{k>k_{\ast }:{\mathcal{D}}_{m}(k)>\mathcal{c}g_{m}(k)\} \\ 
\infty ,\quad \text{if ~}~{\mathcal{D}}_{m}(k)\leq \mathcal{c}g_{m}(k)~\text{%
~for all~}~k>k_{\ast }.%
\end{cases}
\label{e:delaytime_def}
\end{equation}%
We focus on two distinct settings: an ``early
change'', when $k_{\ast }\leq C$ for some unknown fixed
constant $C>0$, and a ``late change'',
wherein $k_{\ast }=\lfloor c_{\ast }m\rfloor $ for some $c_{\ast }>0$.

\noindent The next two theorems provide the limiting distribution of $\kappa
_{m}-k_{\ast }$ when both $\mathcal{D}_{m}(k)=\mathcal{D}_{m}^{(1)}(k)$ and $%
\mathcal{D}_{m}^{(2)}(k)$. Let 
\begin{align}
& \rho =\frac{1-\beta }{2-\beta },\qquad w=\left( \frac{\mathcal{c}}{\theta
|\nu _{1}-\nu _{2}|}\right) ^{1/(2-\beta )},\qquad v_{m}=\frac{2\sigma
_{\ast }}{(2-\beta )|\nu _{1}-\nu _{2}|}(wm^{\rho })^{1/2},
\label{e:def_eta_etc} \\
& v_{m}^{\prime }=\frac{m^{1/2}}{\theta \left\vert \mathfrak{D}_{h}\left(
F,G\right) \right\vert ^{1/2}}.  \label{vmdash}
\end{align}

\begin{theorem}
\label{alt:nonpage} Assume $H_{A}$ holds. Let $\kappa _{m}$ be as in %
\eqref{e:delaytime_def} with $\mathcal{D}_{m}(k)=\mathcal{D}_{m}^{(1)}(k)$
and $g_{m}$ as in \eqref{e:boundaryfxn}. If $k_{\ast }\leq C$ with some $C>0$%
, and Assumptions \ref{a:basic_HA} and \ref{a:2ndHa_cond} hold, then%
\begin{equation}
\frac{\kappa _{m}-k_{\ast }-wm^{\rho }}{v_{m}}\overset{\mathcal{D}}{%
\rightarrow }\mathcal{N}(0,1).  \label{e:normalimit}
\end{equation}%
If $k_{\ast }=c_{\ast }m$ for some $c_{\ast }>0$, and Assumptions \ref%
{a:basic_HA} and \ref{a:2ndHa_cond} hold, then 
\begin{equation}
\frac{\kappa _{m}-k_{\ast }}{v_{m}^{\prime }}\overset{\mathcal{D}}{%
\rightarrow }\mathcal{H}_{c_{\ast }}(\mathcal{c}),  \label{e:limit_cH(c)}
\end{equation}%
where $\mathcal{H}_{c_{\ast }}$ is defined in \eqref{e:def_cH(u)} in the
Supplement.%
\end{theorem}

\begin{theorem}
\label{alt:page}Assume $H_{A}$ holds. Let $\kappa _{m}$ be as in %
\eqref{e:delaytime_def} based on the detector $\mathcal{D}_{m}(k)=\mathcal{D}%
_{m}^{(2)}(k)$, with $g_{m}$ as in \eqref{e:boundaryfxn}. If $k_{\ast }\leq
C $ with some $C>0$, and Assumptions \ref{a:basic_HA} and \ref{a:2ndHa_cond}
hold, and further 
\begin{equation}
\sigma _{\ast }\theta (wm^{\rho })^{{3/2}-\beta }\rightarrow \infty ,
\label{e:alt_nondegen2}
\end{equation}%
where $w,\rho $ are given in \eqref{e:def_eta_etc}, then the limit %
\eqref{e:normalimit} holds.\\
If $k_{\ast }=c_{\ast }m$ for some $c_{\ast }>0$, and Assumption \ref%
{a:basic_HA} and \ref{a:2ndHa_cond} hold, then 
\begin{equation}
\frac{\kappa _{m}-k_{\ast }}{v_{m}^{\prime }}\overset{\mathcal{D}}{%
\rightarrow }\widetilde{\mathcal{H}}_{c_{\ast }}(\mathcal{c}),
\label{e:limit_cH(c)-tilde}
\end{equation}%
where $\widetilde{\mathcal{H}}_{c_{\ast }}$ is defined in %
\eqref{e:def_cH(u)2} in the Supplement, and $v_{m}^{\prime }$ is as
in (\ref{vmdash}).
\end{theorem}

\noindent Theorems \ref{alt:nonpage} and \ref{alt:page} describe the delay
time under both monitoring schemes (ordinary and Page-type, respectively). The
theorems state that - in the early change regime where $k_{\ast }$
occurs a finite number of periods after the start of the monitoring,  roughly $wm^{\rho }$ observations after the change-point are needed before detection.
Since $\rho $ approaches $0$ as $\beta $ approaches $1$, choosing values of $%
\beta $ close to $1$ can shorten detection times; this is also observed in %
\citet{aue2004delay}.\newline
Considering the late change regime, as mentioned in the theorems, the
(lengthy) definition of the limit variables $\mathcal{H}_{c_{\ast }}$ and $%
\widetilde{\mathcal{H}}_{c_{\ast }}$ is relegated to equations (\ref%
{e:def_cH(u)}) and (\ref{e:def_cH(u)2}) in the Supplement, for ease of
exposition. We remark, however, that both $\mathcal{H}_{c_{\ast }}(\mathcal{c%
})$ and $\widetilde{\mathcal{H}}_{c_{\ast }}(\mathcal{c})$ are non-Gaussian,
strictly positive, and for the same fixed $\mathcal{c}$, the variable $%
\widetilde{\mathcal{H}}_{c_{\ast }}(\mathcal{c})$ can be seen to be
stochastically smaller than $\mathcal{H}_{c_{\ast }}(\mathcal{c})$,
reflecting a well-documented advantage of shorter delay times in Page-type
detection procedures under late changes (c.f. \citealp{fremdt2015page}).
Seeing as both $\mathcal{H}_{c_{\ast }}$ and $\widetilde{\mathcal{H}}%
_{c_{\ast }}$ are well-defined random variables, the theorems entail that,
in the late change regime, the number of observations needed in order to
detect a change is proportional to $v_{m}^{\prime }$. When $0<\theta <1$
(i.e., when the size of the break is fixed), this entails that the detection
delay is proportional to $m^{1/2}$; seeing as the breakdate $k_{\ast }$ is
proportional to $m$, this means that detection is relatively quick. On the
other hand, when $\theta \rightarrow 0$ (corresponding to a break of
vanishing size), this inflates $v_{m}^{\prime }$ and, therefore, the
detection delay.\newline
Finally, as can be expected, in all cases small values of $|\mathfrak{D}%
_{h}(F,G)|$ yield larger delay times.

\section{Complements and extensions\label{s:refinements}}

\subsection{An expanding baseline variant\label{repurpose}}

There are many possible variants of the basic two-sample monitoring setups, depending on how the historical sample and the incoming monitoring observations are used. To illustrate this flexibility, we consider one additional detector, denoted $\mathcal D_m^{(3)}$, based on expanding the baseline sample during the monitoring period. Note that the detector $\mathcal D_m^{(1)}$ and Page-type detector $\mathcal D_m^{(2)}$ keep the historical sample fixed. The variant below modifies the Page detector by allowing the baseline sample to expand: observations that are sufficiently far in the past are incorporated into the baseline sample, while the Page maximization is restricted to only a window comprised of the most recent data.

The motivation is that, over a long monitoring horizon, the earliest monitoring observations that have not led to rejection may provide useful additional information about the pre-change distribution. Hence, incorporating such observations can strengthen the baseline sample and prevent older monitoring data from diluting the effect of a later change.  However, this type of procedure  naturally involves a tradeoff: if a weak change is not detected quickly, some post-change observations may eventually be absorbed into the baseline sample; hence it may gain or lose power compared to the fixed-baseline-type monitoring. Other possibilities include considering all possible adjacent splits available at each time $k$, among others; see, e.g. \cite{gosmann2021new}, and \cite{aue2024state}.  

Fix $c_0>0$
and put
$$
c_m=\lfloor c_0 m\rfloor,\qquad b_k=(k-c_m)_+,\qquad
n_k=m+b_k,
$$
where $(x)_+=\max\{x,0\}$. 
Thus, at time $k$, the first $n_k$ observations are used as baseline data and
the Page maximization is carried out over the most recent $k\wedge c_m$
monitoring observations.  With the convention that $U_n(h;r,j)$ is defined as
in \eqref{e:def_page}, but with historical sample size $n$, define
\begin{equation}
\mathcal D_m^{(3)}(k)
=m^{-1}\max_{b_k\leq r\leq k-2}
(k-r)^2\left|U_{n_k}(h;r-b_k,k-b_k)\right|.\label{e:def_Dm3}
\end{equation}
The corresponding boundary is
\begin{equation}
g_m^{(3)}(k)= g\left(\frac{k}{n_k}\right)\left(1+\frac{b_k}{m}\right)^\gamma,
\label{e:def_gm3}
\end{equation}
where $\gamma>1/2$.  Since the baseline is expanded, the original boundary $g(k/m)$ is replaced with $g(k/n_k)$; the additional factor $\left(1+{b_k}/{m}\right)^\gamma$ is included to ensure that after expanding the baseline, the ratio $\mathcal D_m^{(3)}(k)/ g_m^{(3)}(k)$ remains stochastically bounded over indefinite monitoring horizons; in simulations, we set $\gamma=0.51$. 
Notice that, if $k\leq c_m$, then $b_k=0$, $n_k=m$, so that
$\mathcal D_m^{(3)}(k)=\mathcal D_m^{(2)}(k)$ for all such $k$.

We now describe the limit distribution under $H_0$. For
$0\leq y\leq v<u\leq1$, define
\begin{equation}\label{e:Z_y(u,v)}
Z_{\ell,y}(u,v) = 
W_\ell(u)-W_\ell(y)-\frac{1-u}{1-v}\big(W_\ell(v)-W_\ell(y)\big)
\end{equation}
and
\begin{align}
\label{e:def_Gy}
G_y(u,v)&=
\sum_{\ell=1}^{\infty}\lambda_\ell
\bigg[Z^2_{\ell,y}(u,v)-\bigg((u-v)+\left(\frac{u-v}{1-v}\right)^2(v-y)\bigg)\bigg]. \notag
\end{align}
The case $y=0$ reduces to \eqref{e:G(u,v)}, i.e.
$G_0(u,v)=G(u,v)$.
For the fixed $c_0>0$ above, set
\begin{equation}
u_{c_0}=\frac{c_0}{1+c_0},\qquad y(u)
=\begin{cases}
0&0\leq u\leq u_{c_0}\\
\displaystyle\frac{(1+c_0)u-c_0}{1-c_0+c_0 u},
&u_{c_0}<u<1.
\end{cases}\label{e:def_yu}
\end{equation}
Then define
\begin{equation}
\overline\Gamma^{(3)}(u)=\sup_{y(u)\leq v<u}
\left|G_{y(u)}(u,v)\right|.
\label{e:def_Gamma3}
\end{equation}
\begin{theorem}
\label{t:theorem_D3_under_H0}
Assume $H_0$ holds, and consider the detector
$\mathcal D_m(k)=\mathcal D_m^{(3)}(k)$. Let
$$
d^{(3)}(u) =u^\beta \big(1-y(u)\big)^{\beta-\gamma}  \big(1-uy(u)\big)^{2-\beta}. $$
Then, (i) as $m\to\infty$,
\begin{equation}
\sup_{k\geq 2}
\frac{\mathcal D_m^{(3)}(k)}{g_m^{(3)}(k)}
\Rightarrow
\sup_{0<u<1}
\frac{\overline\Gamma^{(3)}(u)}
{d^{(3)}(u)}
\label{e:limit_Dm3}.
\end{equation}
(ii) Suppose that $M=M_m\to\infty$ and
$M/m\to a_0\in(0,\infty]$, and set $u_0={a_0}/(1+a_0)$.
Then
\begin{equation}
\max_{2\leq k<M}
\frac{\mathcal D_m^{(3)}(k)}
{g_m^{(3)}(k)}
\Rightarrow
\sup_{0<u<u_0}
\frac{\overline\Gamma^{(3)}(u)}
{d^{(3)}(u)}.
\label{e:limit_Dm3_closed}
\end{equation}
\end{theorem}

\subsection{Testing for the stability of the training sample\label{sectrain}}

\noindent Assumption \ref{a:historical_stability} requires that the training/baseline
sample $\mathbf{X}_{1},$ ..., $\mathbf{X}_{m}$ is stable - that is, it
undergoes no breaks. As mentioned above, this is a typical, and testable,
assumption. We now (briefly) discuss a \textit{U}-statistic based approach
to test retrospectively for the null hypothesis of no distributional changes
in the training sample. We use the sequence 
\begin{align*}
{\mathfrak{R}}(k)=&\frac{2}{k(m-k)}\sum_{i=1}^{k}\sum_{j=k+1}^{m}h(\mathbf{X}%
_{i},\mathbf{X}_{j}) -\binom{k}{2}^{-1}\sum_{1\leq i<j\leq k}h(\mathbf{X}%
_{i},\mathbf{X}_{j}) \\
& -\binom{m-k}{2}^{-1}\sum_{k+1\leq i<j\leq m}h(\mathbf{X}_{i},\mathbf{X}%
_{j}),\;\;\;\;
\end{align*}%
for $2\leq k\leq m-2$, and define the corresponding process 
\begin{equation*}
\mathfrak{r}_{m}(t)=\left\{ 
\begin{array}{ll}
0,\;\;\;t\not\in \lbrack 2/m,1-2/m]\vspace{0.3cm} &  \\ 
mt^{2}(1-t)^{2}\mathfrak{R}(\lfloor mt\rfloor ),\;\;\;2/m\leq t\leq 1-2/m. & 
\end{array}%
\right.
\end{equation*}%
As is typical in this literature (\citealp{chgreg}), we consider a \textit{%
weighted} version of $\mathfrak{r}_{m}(t)$, in order to enhance the power of
our test in the presence of changes occurring close to the beginning/end of
the sample; we propose the following family of weight functions 
\begin{equation}
\mathfrak{q}(t)=(t(1-t))^{\zeta },\text{ for some }\zeta <1.  \label{weight}
\end{equation}%

A ``natural'' choice to detect the presence
of a possible change is to use the sup-norm of the weighted version of $%
\mathfrak{r}_{m}(t)$, viz. $\sup_{0<t<1}\left\vert \mathfrak{r}%
_{m}(t)\right\vert /\mathfrak{q}(t)$.

\begin{theorem}
\label{thewi} If Assumptions \ref{a:historical_stability}--\ref%
{a:assumption_on_h} hold, then %
\begin{equation*}
\sup_{0<t<1}\frac{\left\vert \mathfrak{r}_{m}(t)\right\vert }{\mathfrak{q}(t)%
}\overset{\mathcal{D}}{\rightarrow }\sup_{0\leq t\leq 1}\frac{1}{\mathfrak{q}%
(t)}\left\vert \sum_{\ell =1}^{\infty }\lambda _{\ell }\left( B_{\ell
}^{2}(t)-t(1-t)\right) \right\vert ,
\end{equation*}%
where $\{B_{\ell }(t),0\leq t\leq 1\},$ $\ell =1,2,\ldots $ are independent
Brownian bridges.
\end{theorem}

\noindent Theorem \ref{thewi} contains the limit of the maximally selected
weighted version of $\mathfrak{r}_{m}(t)$. 

Several further results such as power versus the alternative, and a
consistent estimator of the break date, could be readily derived by
extending the theory in \citet{chgreg}. The same result - for the case $%
\mathfrak{q}(t)=1$ - was proven by \citet{biau:bleakley:mason:2016}, under
the more restrictive condition $\sum_{\ell =1}^{\infty }|\lambda _{\ell
}|<\infty $. Hence, similarly to the other results above, Theorem \ref{thewi}
improves on the current literature by requiring the milder condition $%
\sum_{\ell =1}^{\infty }\lambda _{\ell }^{2}<\infty $.

\subsection{Examples of kernel functions\label{kernels}}

\noindent We discuss some examples of possible kernel functions $h\left(
\cdot ,\cdot \right) $, and a methodology (plus an example) to construct
``distribution-distinguishing'' kernels $%
h\left( \cdot ,\cdot \right) $ - that is, functions $h\left( \cdot ,\cdot
\right) $ which can discriminate \textit{any} change in distribution.

\begin{example}
\label{energy} Suppose $\mathcal{X}=\mathbb{R}^{d}$, and let $\eta \in (0,2)$%
. The kernel $h(\mathbf{x},\mathbf{y})=\Vert \mathbf{x}-\mathbf{y}\Vert
^{\eta }$ is connected with the energy distance between two independent
vectors $\mathbf{X},\mathbf{Y}\in \mathbb{R}^{n}$, defined as $\mathscr %
E_{\eta }(\mathbf{X},\mathbf{Y})$ $=$ $2{\mathsf{E}}\hspace{0.1mm}\left\Vert 
\mathbf{X}-\mathbf{Y}\right\Vert ^{\eta }$ $-$ ${\mathsf{E}}\hspace{0.1mm}%
\left\Vert \mathbf{X}-\mathbf{X}^{\prime }\right\Vert ^{\eta }$ $-$ ${%
\mathsf{E}}\hspace{0.1mm}\left\Vert \mathbf{Y}-\mathbf{Y}^{\prime
}\right\Vert ^{\eta }$, where $\mathbf{X}^{\prime },\mathbf{Y}^{\prime }$
are independent copies of $\mathbf{X}$ and $\mathbf{Y}$ respectively. %
\citet{szekely:rizzo:2005} show that $\mathscr E_{\eta }(\mathbf{X},\mathbf{Y%
})\geq 0$, with equality if and only if $\mathbf{X}\overset{{\mathcal{D}}}{=}%
\mathbf{Y}$. As also argued in \citet{biau:bleakley:mason:2016} and %
\citet{boniece2025changepoint}, $U_{m}(h;k)$ in \eqref{e:def_Um(h;k)} is the
empirical counterpart to $\mathscr E_{\eta }$, evaluating the distance
between the distribution of the training sample and that of the monitored
sequence up to time $k$. When trying to detect changepoint in possibly
multivariate time series, the energy distance is particularly advantageous
due to its rotational invariance (\citealp{szekely2013energy}).\footnote{%
As mentioned in the introduction, statistics based on other distances, such
as Cram\'{e}r's distance or the Cram\'{e}r--von Mises--Smirnov distance do
not share this property.} In the case of using $h(\mathbf{x},\mathbf{y}%
)=\Vert \mathbf{x}-\mathbf{y}\Vert ^{\eta }$, it is immediate to see that
Assumption \ref{a:assumption_on_h} holds as long as ${\mathsf{E}}\left\Vert 
\mathbf{X}\right\Vert ^{2\eta }<\infty $. In turn, this suggests that $\eta $
can e.g. be chosen \textit{a posteriori} by the applied user after checking
how many moments the data admit. %
\end{example}

\begin{example}
\label{grothendieck}\citet{chen2025minimax} propose the so-called \textit{%
Grothendieck divergence}, defined as $\mathscr G_{\eta }(\mathbf{X},\mathbf{Y%
})=2{\mathsf{E}}\hspace{0.1mm}\psi \left( \mathbf{X},\mathbf{Y}\right) -{%
\mathsf{E}}\hspace{0.1mm}\psi \left( \mathbf{X},\mathbf{X}^{\prime }\right) -%
{\mathsf{E}}\hspace{0.1mm}\psi \left( \mathbf{Y,Y}^{\prime }\right) $, where 
\begin{equation*}
\psi \left( \mathbf{x},\mathbf{y}\right) =\arccos \left[ \frac{%
1+\left\langle \mathbf{x},\mathbf{y}\right\rangle }{\sqrt{\left(
1+\left\langle \mathbf{x},\mathbf{x}\right\rangle \right) \left(
1+\left\langle \mathbf{y},\mathbf{y}\right\rangle \right) }}\right] ,
\end{equation*}%
satisfying Assumption \ref{a:assumption_on_h}, with no moment requirements
on $\mathbf{X}$ or $\mathbf{Y}$. By Proposition 1 in \citet{chen2025minimax}%
, the Grothendieck divergence is distribution distinguishing - that is, it is
nonzero if and only if the distributions of $\mathbf{X}$ and $\mathbf{Y}$
differ.
\end{example}

\begin{example}
\label{ex:negtype} Consider a separable metric space $\left( \mathcal{X}%
,\rho \right) $ with finite first moment. Then, $\left( \mathcal{X},\rho
\right) $ is said to have \textit{negative type} (\citealp{lyons:2013}), if
it holds that 
\begin{equation}
\mathfrak{D}_{\rho }(G_{1},G_{2})=\int \rho (\mathbf{x},\mathbf{y}%
)d(G_{1}-G_{2})^{2}(\mathbf{x},\mathbf{y})\leq 0.  \label{e:h_divergence}
\end{equation}%
The space $(\mathcal{X},\rho )$ is said to have \textit{strong negative type}
if \eqref{e:h_divergence} is satisfied with the additional property that
equality holds if and only if $G_{1}=G_{2}$. Hence, taking $h(\mathbf{x},%
\mathbf{y})=\rho (\mathbf{x},\mathbf{y})$ when $(\mathcal{X},\rho )$ has
strong negative type yields an \textit{omnibus} test for changes in the
distribution. Examples of spaces with strong negative type include $\mathbb{R%
}^{d}$ (the energy distance in Example \ref{energy} is a special case of %
\eqref{e:h_divergence}), or more generally all separable Hilbert spaces.
Notably, from \citet{lyons:2013}, if $(\mathcal{X},\rho )$ has negative
type, then for any $0<r<1$, $(\mathcal{X},\rho ^{r})$ has \textit{strong}
negative type. In particular, from \citet{meckes:2013}, if $1\leq p\leq 2$
and $\mathcal{X}=\mathcal{L}^{p}[0,1]$ is the space of real-valued $p$%
-integrable functions and $\rho $ its usual metric, then $(\mathcal{X},\rho
^{r})$ has strong negative type for any $0<r<1$. 
\noindent In the case of using the kernel $h(\mathbf{x},\mathbf{y})=\rho (%
\mathbf{x},\mathbf{y})$, it is immediate to see that Assumption \ref%
{a:assumption_on_h} holds as long as ${\mathsf{E}}\left[ \rho ^{2}(\mathbf{x}%
,\mathbf{y})\right] <\infty $. Then, similarly to Example \ref{energy}, the
definition of $\rho (\mathbf{x},\mathbf{y})$ is ``
constructive'', in that either it can be chosen based on
how many moments the data admit (as long as (\ref{e:h_divergence}) holds);
or, given a metric $\rho (\mathbf{x},\mathbf{y})$ and a dataset, it can be
tested whether Assumption \ref{a:assumption_on_h} holds by testing whether ${%
\mathsf{E}}\left[ \rho ^{2}(\mathbf{x},\mathbf{y})\right] <\infty $.
\end{example}

\begin{example}
\label{arlot} \citet{arlot2019kernel} study multiple changepoint detection
(retrospectively) based on positive semidefinite\footnote{%
That is, for each tuple $\left\{ x_{1},...,x_{n}\right\} $, the matrix $%
\left\{ K\left( x_{i},x_{j}\right) \right\} _{1\leq i,j\leq n}$ is
positive semidefinite.} kernels, providing several
examples of possible kernel functions suitable to various data types (e.g.
vector-valued data, multinomial data, text or graph-valued data; see their
Section 3.2); their paper also contains a comprehensive set of references on
the literature on kernel functions. Of particular interest is the family of 
\textit{characteristic kernels\footnote{Note the meaning of ``kernel" in the MMD literature typically assumes positive semidefiniteness; hence characteristic kernels are assumed positive semi-definite.}} (\citealp{fukumizu2008kernel}; %
\citealp{sriperumbudur2010hilbert}; \citealp{sriperumbudur2011universality}%
), which embed probability distributions injectively into a reproducing kernel Hilbert space,\footnote{%
For a positive semidefinite kernel $K$ with associated reproducing kernel Hilbert space $\mathcal H_K$, the kernel mean embedding of the distribution of $\mathbf X_i$ is
$\mu_i=\mathsf EK(\mathbf X_i,\cdot)\in\mathcal H_K,$
provided the expectation exists. 
} and are therefore ``distribution-distinguishing,'' as a change in the distribution always gives a change in the associated mean embedding. %
 A possible example of a characteristic kernel (see %
\citealp{fukumizu2003kernel}) is the Gaussian kernel $h(\mathbf{x},\mathbf{y}%
)=\exp \left( -\left\Vert \mathbf{x}-\mathbf{y}\right\Vert _{2}^{2}/\left(
2a^{2}\right) \right) $, where $a>0$ is a bandwidth parameter. By Corollary
16 in \citet{sejdinovic2013equivalence}, there is a
correspondence between characteristic kernels and (semi)metrics of the
strong negative type (up to a suitable shift equivalence and under mild moment assumptions). 
\end{example}

\smallskip

\noindent Examples \ref{ex:negtype} and \ref{arlot} suggest that it is
possible to choose $h(\mathbf{x},\mathbf{y})$ so as to be ``distribution-distinguishing'' - essentially, producing kernels
by means of kernels. Indeed, consider the user-chosen function $K\left( 
\mathbf{x},\mathbf{y}\right) :\mathcal{X}\times \mathcal{X}\rightarrow \mathbb{R}$,
such that $K\left( \mathbf{x},\mathbf{y}\right) $ is symmetric, positive
semi-definite, and the map $\mathbf{%
x\mapsto }K\left( \mathbf{\cdot },\mathbf{x}\right) $ is injective.\footnote{In the MMD literature, such kernels are called non-degenerate; though this is different from non-degeneracy in the typical $U$-statistic sense of a non-vanishing first-order Hoeffding projection.}
Given such a kernel, define the semimetric $\delta \left( \mathbf{x},\mathbf{%
y}\right) $ $=$ $K\left( \mathbf{x},\mathbf{x}\right) +K\left( \mathbf{y},%
\mathbf{y}\right) -2K\left( \mathbf{x},\mathbf{y}\right) $. %
\citet{sejdinovic2013equivalence} show that $\delta \left( \mathbf{x},%
\mathbf{y}\right) $ is a \textit{semimetric} of negative type on $\mathcal{X}
$. In turn, by Proposition 3 in \citet{sejdinovic2013equivalence}, this
entails that there are a Hilbert space $\mathcal{H}$ and an injective map $%
\varphi \left( \cdot \right) $ such that $\delta \left( \mathbf{x},\mathbf{y}%
\right) =\left\Vert \varphi \left( \mathbf{x}\right) -\varphi \left( \mathbf{%
y}\right) \right\Vert _{\mathcal{H}}^{2}$; therefore, $\delta ^{1/2}\left( 
\mathbf{x},\mathbf{y}\right) $ is a metric of negative type on $\mathcal{X}$%
. Then, based on Remark 3.19 in \citet{lyons:2013}, $\delta ^{s}\left( 
\mathbf{x},\mathbf{y}\right) $ is - for any $s\in \left( 0,1/2\right) $ - a
metric of \textit{strong} negative type. Thus, revisiting Example \ref%
{ex:negtype}, given a kernel $K\left( \mathbf{x},\mathbf{y}%
\right)$ under which $\bx \mapsto K(\cdot,\bx)$ is injective, the family of functions $h\left( \mathbf{x},\mathbf{y}\right) $ $%
= $ $\left[ K\left( \mathbf{x},\mathbf{x}\right) +K\left( \mathbf{y},\mathbf{%
y}\right) -2K\left( \mathbf{x},\mathbf{y}\right) \right] ^{s/2}$ defines a
family of ``distribution-distinguishing''
kernels for any $s\in \left( 0,1/2\right) $. Indeed, in the following
theorem we extend Remark 3.19 in \citet{lyons:2013}, showing that even $%
\delta ^{1/2}\left( \mathbf{x},\mathbf{y}\right) $ is distribution-distinguishing.

\begin{theorem}
\label{metric-neg-type}Let $\mathcal{X}$ be a separable, complete metric
space, and $K\left( \mathbf{x},\mathbf{y}\right) $ be a continuous kernel with $\mathbf x\mapsto K(\cdot,\mathbf x)$ injective, and let $\delta \left( \mathbf{x},\mathbf{%
y}\right) $ $=$ $K\left( \mathbf{x},\mathbf{x}\right) +K\left( \mathbf{y},%
\mathbf{y}\right) -2K\left( \mathbf{x},\mathbf{y}\right) $. Then $\delta ^{1/2}\left( \mathbf{x},\mathbf{y}%
\right) $ is a metric of strong negative type.
\end{theorem}

\noindent  Theorem~\ref{metric-neg-type} thus extends the standard construction from powers $\delta^s$ with $0<s<1/2$ to the endpoint $s=1/2$. To the best of our knowledge, it is new. According to the theorem, the
kernel $h\left( \mathbf{x},\mathbf{y}\right) $ $=$ $\left[ K\left( \mathbf{x}%
,\mathbf{x}\right) +K\left( \mathbf{y},\mathbf{y}\right) -2K\left( \mathbf{x}%
,\mathbf{y}\right) \right] ^{1/2}$, is distribution-distinguishing, and therefore,
considering Example \ref{arlot}, an \textit{omnibus} test for distributional
change can be based on it. 

\noindent It is easily seen that when $K(\bx,\by)$ is strictly positive definite, the map $\mathbf x \mapsto K(\cdot,\mathbf x)$ is injective.  Thus, in order to construct a
distribution distinguishing kernel $h\left( \mathbf{x},\mathbf{y}\right) $, it
suffices to follow the procedure above starting from a strictly positive definite
kernel.\footnote{%
Other sufficient conditions can be found in \citet{sriperumbudur2010hilbert}
and \citet{sriperumbudur2011universality}.} A leading example is based on
the Gaussian kernel, discussed in the next example.

\begin{example}
\label{deriv-gaussian}Consider the Gaussian kernel $K_{g}\left( \mathbf{x},%
\mathbf{y}\right) =\exp \left( -\left\Vert \mathbf{x}-\mathbf{y}\right\Vert
^{2}/\left( 2a^{2}\right) \right) $ for some $a>0$; note $\mathbf x\mapsto K_g(\cdot,\bx)$ is injective (see e.g. \citet{arlot2019kernel}). Then, by the above, it is easy to
see that $\delta ^{1/2}\left( \mathbf{x},\mathbf{y}\right) $ $=$ $\left[
K_{g}\left( \mathbf{x},\mathbf{x}\right) +K_{g}\left( \mathbf{y},\mathbf{y}%
\right) -2K_{g}\left( \mathbf{x},\mathbf{y}\right) \right] ^{1/2}$, is a
metric of negative type; further, by Theorem \ref{metric-neg-type}, it is
also a metric of \textit{strong} negative type.
\end{example}

\subsection{On implementation\label{implementation}}

\noindent The limiting processes of our monitoring schemes under $H_{0}$ all
depend on the (infinite sequence of) eigenvalues $\lambda _{i}$ of the
operator $A$ defined above, which necessitates some approximation when
obtaining critical values. A possible approach is based on estimating the
eigenvalues $\lambda _{i}$ from the historical sample via the $m\times m$
matrix $A_{m}$, where 
\begin{equation}
\left\{ A_{m}\right\} _{i,j}=\frac{1}{m}\Bigg(h(\mathbf{X}_{i},\mathbf{X}%
_{j})-h_{1,i}-h_{1,j}+{\binom{m}{2}}^{-1}\sum_{1\leq i^{\prime }<j^{\prime
}\leq m}h(\mathbf{X}_{i^{\prime }},\mathbf{X}_{j^{\prime }})\Bigg),
\label{e:def_Am}
\end{equation}%
with $h_{1,i}=\sum_{\ell =1}^{m}h(\mathbf{X}_{i},\mathbf{X}_{\ell })\mathbf{1%
}_{\{\ell \neq i\}}/\left( m-1\right) $. Let $\big|\widehat{\lambda }%
_{1,m}\big| \geq \big| \widehat{\lambda }_{2,m}\big| \geq
\ldots \geq \big|\widehat{\lambda }_{m,m}\big| $ denote the
eigenvalues of the matrix $A_{m}$, define the sigma-field $\mathcal{F}%
=\sigma \left\{ \mathbf{X}_{\ell },\ell \geq 1\right\} $, and let $%
\{W_{1}(u),u\geq 0\}$, $\{W_{2}(u),u\geq 0\},\ldots $ be independent Wiener
processes, independent of $\mathcal{F}$. The approximations to the limiting
processes $\Gamma (u)$, $\overline{\Gamma }(u)$ and $\Gamma (u,,b_{w},c_{w})$
under $H_{0}$ are constructed as follows%
\begin{eqnarray}
\widehat{\Gamma }_{m}(u) &=&\sum_{\ell =1}^{m}\widehat{\lambda }_{\ell
,m}\left( W_{\ell }^{2}(u)-u\right) ,  \label{gamma-1mc} \\
\widehat{\overline{\Gamma }}_{m}(u) &=&\sup_{0<v<u}|\widehat{G}%
_{0,m}\left( u,v\right) |,  \label{gamma-2mc} \\
\widehat{\overline\Gamma}^{(3)}_m(u)&=&\sup_{y(u)\leq v<u}|\widehat{G}%
_{y(u),m}\left( u,v\right) |,  \label{gamma-3mc} 
\end{eqnarray}%
where, in \eqref{gamma-2mc} and \eqref{gamma-3mc},
\begin{align*}
\widehat G_{y,m}(u,v)
&=
\sum_{\ell=1}^{m}\widehat \lambda_{\ell,m}
\bigg[ Z_{\ell,y}^2(u,v)  - \bigg((u-v)+\left(\frac{u-v}{1-v}\right)^2(v-y) \bigg)\bigg],%
\end{align*}
with $Z_{\ell,y}(u,v)$ as in \eqref{e:Z_y(u,v)} and $y(u)$  as in \eqref{e:def_yu}.

\noindent This method is proposed in \citet{biau:bleakley:mason:2016};
hereafter, we formalise it, showing that the approximations (\ref{gamma-1mc}%
)-(\ref{gamma-3mc}) converge (a.s. conditionally on the data) to the
limiting processes. Let ``$\Rightarrow _{\mathcal{F}}$%
'' denote the almost sure conditional weak convergence under $\P (\cdot
|\mathcal{F})$.

\begin{theorem}
\label{p:eig_approx} As $m\rightarrow \infty $, it holds that, for all $%
0<u_{0}\leq 1$ and $0\leq \beta <1$,%
\begin{equation}
\begin{aligned} \sup_{0<u<u_0}u^{-\beta}|\widehat \Gamma_m(u)|
&\Rightarrow_{\mathcal F} \sup_{0<u<u_0}u^{-\beta}|\Gamma(u)|,\\
\sup_{0<u<u_0}u^{-\beta}\,\widehat {\overline \Gamma}_m(u)
&\Rightarrow_{\mathcal F} \sup_{0<u<u_0}u^{-\beta}\,\overline \Gamma(u)\\
\sup_{0<u<u_0}\frac{\widehat{\overline{\Gamma}}^{(3)}_{m}(u)}{d^{(3)}(u)}
&\Rightarrow_{\mathcal F} \sup_{0<u<u_0}\frac{\overline{\Gamma}^{(3)}(u)}{d^{(3)}(u)}.
\end{aligned}
\label{e:conditional_weak_convergence}
\end{equation}%
\end{theorem}

\noindent The theorem requires that the number
of eigenvalues employed grows with $m$; in (\ref{gamma-1mc})-(\ref{gamma-3mc}%
) all the eigenvalues of $A_{m}$ are used, but employing only a fraction
(e.g., $m/2$) still yields the same result.

\section{Simulations and applications\label{s:simulations}}

\subsection{Simulation study\label{simulation}}

\noindent We report a set of Monte Carlo simulations to investigate the
empirical rejection frequencies and the detection delays under alternatives
of our procedures. We report only a set of simulations based on the case $%
\mathcal{X}=\mathbb{R}^{5}$.\footnote{%
Further simulations, which essentially confirm the results in this section,
are available upon request.} We use the following kernels: $h^{(1)}(\mathbf{x%
},\mathbf{y})$ $=$ $\Vert \mathbf{x}-\mathbf{y}\Vert _{1}^{1/2}$; $h^{(2)}(%
\mathbf{x},\mathbf{y})$ $=$ $\Vert \mathbf{x}-\mathbf{y}\Vert _{2}$; and $%
h^{(3)}(\mathbf{x},\mathbf{y})$ $=$ $\left[ 1-\exp (-\Vert \mathbf{x}-%
\mathbf{y}\Vert _{2}^{2}/(2a^{2}))\right] ^{1/2}$. The kernel $h^{(2)}$
corresponds to the usual energy distance; $h^{(3)}$ is based directly on
Example \ref{deriv-gaussian}, with $a$ set equal to the sample median of $%
\{\Vert \mathbf{X}_{i}-\mathbf{X}_{j}\Vert _{2},1\leq i,j\leq m\}$. Under $H_0$, we consider historical samples of length $m\in \left\{
50,100,200\right\} $, and we report results for each of the detectors ${%
\mathcal{D}}_{m}^{(i)}$, $i=1,2,3$, where for the boundary function %
\eqref{e:boundaryfxn} we set $\beta \in \{0,0.5,0.9\}$.  For $\mathcal D^{(3)}_m$, we set $\gamma=0.51$ and $c_0=1$ throughout.\footnote{%
For reference, recall that: ${\mathcal{D}}_{m}^{(1)}$ is the
``ordinary'' detection scheme defined in (%
\ref{cusum}); ${\mathcal{D}}_{m}^{(2)}$ is the ``Page-type'' scheme defined in (\ref{page}); and ${\mathcal{D%
}}_{m}^{(3)}$ is the expanding baseline
scheme introduced in (\ref{repurpose}).}

\noindent We begin by examining the performance of our procedures under $%
H_{0}$; in all cases, we generate the observations as $\mathbf{X}_{i}\sim
i.i.d.\mathcal{N}\left( 0,\mathbf{I}_{5}\right) $, and we set the monitoring
horizon $M=10m$. Empirical rejection frequencies are reported in Table \ref%
{tab:NullSize}.\footnote{%
Note that, for each empirical rejection frequency, the $95\%$ confidence
interval is $\left[ 0.04,0.06\right] $.}

\begin{table}[t!]
\caption{Empirical rejection probabilities under $H_0$, nominal level $0.05$}
\label{tab:NullSize}\captionsetup{font=scriptsize} \centering
{\tiny
\begin{tabular}{cccccccccccccc}
\hline\hline
&  &  &  &  &  &  &  &  &  &  &  &  &  \\
&  & Kernel & \multicolumn{3}{c}{$h^{\left(1\right)}$} &  & \multicolumn{3}{c}{$h^{\left(2\right)}$} &  & \multicolumn{3}{c}{$h^{\left(3\right)}$} \\
Scheme & $\beta$ &  & $m=50$ & $m=100$ & $m=200$ &  & $m=50$ & $m=100$ & $m=200$ &  & $m=50$ & $m=100$ & $m=200$ \\
&  &  &  &  &  &  &  &  &  &  &  &  &  \\
\hline
&  &  &  &  &  &  &  &  &  &  &  &  &  \\
${\mathcal{D}}_{m}^{(1)}$ & $0$ &  & $0.060$ & $0.048$ & $0.043$ &  & $0.050$ & $0.061$ & $0.059$ &  & $0.052$ & $0.062$ & $0.053$ \\
 & $0.5$ &  & $0.058$ & $0.062$ & $0.060$ &  & $0.060$ & $0.053$ & $0.062$ &  & $0.066$ & $0.051$ & $0.052$ \\
 & $0.9$ &  & $0.051$ & $0.044$ & $0.044$ &  & $0.055$ & $0.051$ & $0.061$ &  & $0.050$ & $0.051$ & $0.053$ \\
&  &  &  &  &  &  &  &  &  &  &  &  &  \\
${\mathcal{D}}_{m}^{(2)}$ & $0$ &  & $0.059$ & $0.049$ & $0.044$ &  & $0.051$ & $0.059$ & $0.056$ &  & $0.053$ & $0.059$ & $0.053$ \\
 & $0.5$ &  & $0.054$ & $0.061$ & $0.059$ &  & $0.059$ & $0.052$ & $0.057$ &  & $0.066$ & $0.050$ & $0.050$ \\
 & $0.9$ &  & $0.043$ & $0.037$ & $0.044$ &  & $0.050$ & $0.050$ & $0.051$ &  & $0.044$ & $0.045$ & $0.051$ \\
&  &  &  &  &  &  &  &  &  &  &  &  &  \\
${\mathcal{D}}_{m}^{(3)}$ & $0$ &  & $0.058$ & $0.051$ & $0.056$ &  & $0.048$ & $0.050$ & $0.048$ &  & $0.051$ & $0.053$ & $0.055$ \\
 & $0.5$ &  & $0.048$ & $0.070$ & $0.051$ &  & $0.053$ & $0.039$ & $0.052$ &  & $0.061$ & $0.050$ & $0.055$ \\
 & $0.9$ &  & $0.058$ & $0.046$ & $0.043$ &  & $0.040$ & $0.051$ & $0.051$ &  & $0.050$ & $0.046$ & $0.050$ \\
&  &  &  &  &  &  &  &  &  &  &  &  &  \\
\hline\hline
\end{tabular}
}
\end{table}

\noindent Broadly speaking, size control is ensured in all cases as $m$
increases. This can be read in conjunction with the online monitoring
literature, where often detection schemes are found to be conservative (we
refer e.g. to the simulations in \citealp{lajos07}, and the comments
therein). When using kernels $h^{(1)}$ and $h^{(2)}$, no oversizement is
observed whenever $m>50$, and our procedures have a (mild) tendency to
over-reject only in very few cases when $m=50$. Conversely, kernel $h^{(3)}$
seems to occasionally over-reject, unless $m \geq 100$; note, however, that
partnering $h^{(3)}$ with $\beta =0.9$ results in no oversizement even for $%
m $ as little as $50$. Hence, the results in the table offer several
guidelines to the applied user as far as the choice of the kernel and of the
weight $\beta $ are concerned.

\smallskip

\noindent We now turn to examining the power of our procedure. We consider
three main alternative hypotheses, where - in all cases - $\mathbf{X}_{i}%
\overset{iid}{\sim }\mathcal{N}(\mathbf{0},\mathbf{I}_{d})$ for $1\leq i\leq
k_{\ast }$ and subsequently changes into:%
\begin{eqnarray}
H_{A,1} &:&\mathbf{X}_{k_{\ast }+1}\overset{iid}{\sim }\mathcal{N}(%
\boldsymbol{\mu },\mathbf{I}_{d}),  \label{location} \\
H_{A,2} &:&\mathbf{X}_{k_{\ast }+1}\overset{iid}{\sim }\mathcal{N}(\mathbf{0}%
,\Sigma ),  \label{scale} \\
H_{A,3} &:&\mathbf{X}_{k_{\ast }+1}=\left( X_{k_{\ast }+1,1},\ldots
X_{k_{\ast }+1,d}\right) ^{\top }\text{ with }X_{k_{\ast }+1,i}\overset{iid}{%
\sim }t_{\nu }/\sqrt{\var(t_{\nu })}.  \label{tail}
\end{eqnarray}%
Equation (\ref{location}) corresponds to a location change; (\ref{scale}) to
a scale change with no change in location; and, finally, (\ref{tail}) is a
tail alternative, where the distribution of the data changes into a
Student's t with $\nu $ degrees of freedom. In all three cases, we consider
both the case of ``strong'' changes and
``moderate'' ones, depending on the size of the
change - ``strong'' changes correspond to $%
\boldsymbol{\mu =}\left( 0.3,...,0.3\right) ^{\intercal }$ in (\ref{location}%
), $\left\{ \Sigma \right\} _{i,j}=\exp \left( -\left\vert i-j\right\vert
/10\right) $ in (\ref{scale}), and $\nu =2.5$\ in (\ref{tail});
``moderate'' changes correspond to $\boldsymbol{%
\mu =}\left( 0.25,...,0.25\right) ^{\intercal }$ in (\ref{location}), $%
\left\{ \Sigma \right\} _{i,j}=\exp \left(- \left\vert i-j\right\vert
/5\right) $ in (\ref{scale}), and $\nu =3$\ in (\ref{tail}). All the powers
reported hereafter are size-adjusted - that is, each procedure has been
calibrated so as to ensure that the empirical rejection frequencies under
the null match the nominal level (set to $0.05$). For all alternative scenarios, we set $m=200$, and the horizon as $M=5m$. 

\noindent In a first set of experiments reported in Tables \ref%
{tab:PowerDelayStrongByAlt} and \ref{tab:PowerDelayWeakByAlt}, we consider the empirical
rejection frequencies and the delays for a randomised choice of $k_{\ast }$,%
\footnote{%
The value of $k_{\ast }$, at each iteration, has been picked from the set $\left\{
10,50,200\right\} $ with equal probability.} in the presence of a strong
change; for succinctness, we report results only for the choice $\beta = 0.5$ in \eqref{e:boundaryfxn}.
As the table shows, the power is satisfactory in all cases; detection based
on the scheme proposed in Section \ref{repurpose}, ${\mathcal{D}}_{m}^{(3)}$%
, seems to offer shorter delays, improving on both ${\mathcal{D}}_{m}^{(1)}$
and ${\mathcal{D}}_{m}^{(2)}$. Interestingly, this seems to be the case for
both strong and moderate changes, across all alternative hypotheses $%
H_{A,1}-H_{A,3}$, and for each choice of kernel $h\left( \cdot ,\cdot
\right) $.

\begin{table}[b!]
\caption{Empirical power and median delay - strong changes ($\beta=0.5$)}
\label{tab:PowerDelayStrongByAlt}\captionsetup{font=scriptsize} \centering
{\tiny
\begin{tabular}{ccccccccccccccc}
\hline\hline
&  &  &  &  &  &  &  &  &  &  &  &  &  &  \\
&  & Alternative & \multicolumn{1}{|c}{}  & \multicolumn{3}{c}{$H_{A,1}$} &  & \multicolumn{3}{c}{$H_{A,2}$} &  & \multicolumn{3}{c}{$H_{A,3}$} \\
&  &  &  &  &  &  &  &  &  &  &  &  &  &  \\
Detector & $\beta$ & Kernel & \multicolumn{1}{|c}{} & $h^{(1)}$ & $h^{(2)}$ & $h^{(3)}$ &  & $h^{(1)}$ & $h^{(2)}$ & $h^{(3)}$ &  & $h^{(1)}$ & $h^{(2)}$ & $h^{(3)}$ \\
&  &  & \multicolumn{1}{|c}{} &  &  &  &  &  &  &  &  &  &  &  \\
\hline
&  &  & \multicolumn{1}{|c}{} &  &  &  &  &  &  &  &  &  &  &  \\
${\mathcal{D}}_{m}^{(1)}$ & $0.5$ & Power & \multicolumn{1}{|c}{} & $0.991$ & $0.997$ & $0.996$ &  & $0.992$ & $0.994$ & $0.993$ &  & $0.994$ & $0.993$ & $0.994$ \\
 &  & Med. Delay & \multicolumn{1}{|c}{} & $87$ & $87$ & $89$ &  & $91$ & $75$ & $51$ &  & $58$ & $124$ & $77$ \\
&  &  & \multicolumn{1}{|c}{} &  &  &  &  &  &  &  &  &  &  &  \\
${\mathcal{D}}_{m}^{(2)}$ & $0.5$ & Power & \multicolumn{1}{|c}{} & $0.991$ & $0.996$ & $0.996$ &  & $0.993$ & $0.995$ & $0.993$ &  & $0.992$ & $0.993$ & $0.993$ \\
 &  & Med. Delay & \multicolumn{1}{|c}{} & $77$ & $78$ & $81$ &  & $86$ & $72$ & $48$ &  & $53$ & $117$ & $72$ \\
&  &  & \multicolumn{1}{|c}{} &  &  &  &  &  &  &  &  &  &  &  \\
${\mathcal{D}}_{m}^{(3)}$ & $0.5$ & Power & \multicolumn{1}{|c}{} & $0.986$ & $0.988$ & $0.990$ &  & $0.987$ & $0.988$ & $0.987$ &  & $0.987$ & $0.985$ & $0.988$ \\
 &  & Med. Delay & \multicolumn{1}{|c}{} & $70$ & $65$ & $71$ &  & $80$ & $62$ & $44$ &  & $49$ & $97$ & $64$ \\
&  &  & \multicolumn{1}{|c}{} &  &  &  &  &  &  &  &  &  &  &  \\
\hline\hline
\end{tabular}
}
\end{table}

\begin{table}[b!]
\caption{Empirical power and median delay - moderate changes ($\beta=0.5$)}
\label{tab:PowerDelayWeakByAlt}\captionsetup{font=scriptsize} \centering
{\tiny
\begin{tabular}{ccccccccccccccc}
\hline\hline
&  &  &  &  &  &  &  &  &  &  &  &  &  &  \\
&  & Alternative & \multicolumn{1}{|c}{}  & \multicolumn{3}{c}{$H_{A,1}$} &  & \multicolumn{3}{c}{$H_{A,2}$} &  & \multicolumn{3}{c}{$H_{A,3}$} \\
&  &  &  &  &  &  &  &  &  &  &  &  &  &  \\
Detector & $\beta$ & Kernel & \multicolumn{1}{|c}{} & $h^{(1)}$ & $h^{(2)}$ & $h^{(3)}$ &  & $h^{(1)}$ & $h^{(2)}$ & $h^{(3)}$ &  & $h^{(1)}$ & $h^{(2)}$ & $h^{(3)}$ \\
&  &  & \multicolumn{1}{|c}{} &  &  &  &  &  &  &  &  &  &  &  \\
\hline
&  &  & \multicolumn{1}{|c}{} &  &  &  &  &  &  &  &  &  &  &  \\
${\mathcal{D}}_{m}^{(1)}$ & $0.5$ & Power & \multicolumn{1}{|c}{} & $0.988$ & $0.993$ & $0.989$ &  & $0.991$ & $0.995$ & $0.993$ &  & $0.996$ & $0.824$ & $0.979$ \\
 &  & Med. Delay & \multicolumn{1}{|c}{} & $117$ & $117$ & $121$ &  & $136$ & $105$ & $67$ &  & $114$ & $323$ & $182$ \\
&  &  & \multicolumn{1}{|c}{} &  &  &  &  &  &  &  &  &  &  &  \\
${\mathcal{D}}_{m}^{(2)}$ & $0.5$ & Power & \multicolumn{1}{|c}{} & $0.991$ & $0.994$ & $0.993$ &  & $0.991$ & $0.995$ & $0.993$ &  & $0.996$ & $0.844$ & $0.988$ \\
 &  & Med. Delay & \multicolumn{1}{|c}{} & $103$ & $105$ & $110$ &  & $129$ & $100$ & $64$ &  & $105$ & $321$ & $171$ \\
&  &  & \multicolumn{1}{|c}{} &  &  &  &  &  &  &  &  &  &  &  \\
${\mathcal{D}}_{m}^{(3)}$ & $0.5$ & Power & \multicolumn{1}{|c}{} & $0.979$ & $0.984$ & $0.982$ &  & $0.986$ & $0.989$ & $0.985$ &  & $0.991$ & $0.446$ & $0.887$ \\
 &  & Med. Delay & \multicolumn{1}{|c}{} & $91$ & $85$ & $95$ &  & $116$ & $84$ & $57$ &  & $95$ & $158$ & $133$ \\
&  &  & \multicolumn{1}{|c}{} &  &  &  &  &  &  &  &  &  &  &  \\
\hline\hline
\end{tabular}
}
\end{table}

\noindent In order to assess more precisely the impact of the changepoint
location, we now report results for the three cases of break location used
above, viz.: a ``very early'' break
corresponding to $k_{\ast }=10$; a medium break distance with $k_{\ast }=50$; and a
``late'' break with $k_{\ast }=200$. We
report the detection delays, under a randomised alternative,\footnote{%
At each iteration, the alternative has been picked from the set $\left\{
H_{A,1},H_{A,2},H_{A,3}\right\} $ with equal probability.} for the case of a
strong change (Table \ref{tab:DelayStrongRandomAlt}) and of a moderate change
(Table \ref{tab:DelayWeakRandomAlt}); in Section \ref{furtherMC} in the
Supplement, we report the power (see Table \ref{tab:PowerStrongRandomAlt} for
strong changes, and Table \ref{tab:PowerWeakRandomAlt} for moderate changes).
Considering the former set of results first, the performance of all
detectors ${\mathcal{D}}_{m}^{(i)}$ is comparable in the presence of an
early change. Results are broadly the same under a medium changepoint
location, $k_{\ast }=50$, although - when using ${\mathcal{D}}_{m}^{(3)}$ - in some instances
the power deteriorates when $\beta =0.9 $. As can be expected, all results on the detection delay worsen
when the change occurs late; this is more pronounced in the case of the
detector ${\mathcal{D}}_{m}^{(1)}$, which is~``dragged
down''~by previous observations, and naturally improves
when past observations are either discarded or ``recycled''; however, ${\mathcal{D}}_{m}^{(2)}$ also worsens. Note that, at these signal strengths, the detector ${\mathcal{D}}%
_{m}^{(3)}$ offers a broadly comparable power, and better detection
delays. Similar
results are found in the case of a moderate change; though $\mathcal D_m^{(3)}$ sometimes incurs a loss in power due to the restricted window length $c_m$.

\begin{table}[bhpt!]
\caption{Median detection delay - strong changes (randomised alternative $H_{A,i}$)}
\label{tab:DelayStrongRandomAlt}\captionsetup{font=scriptsize} \centering
{\tiny
\begin{tabular}{ccccccccccccccc}
\hline\hline
&  &  &  &  &  &  &  &  &  &  &  &  &  &  \\
&  &  & \multicolumn{1}{|c}{} & \multicolumn{3}{c}{$k_{\ast}=10$} &  & \multicolumn{3}{c}{$k_{\ast}=50$} &  & \multicolumn{3}{c}{$k_{\ast}=200$} \\
Detector & $\beta$ & & \multicolumn{1}{|c}{} & $h^{\left( 1\right) }$ & $h^{\left( 2\right) }$ & $h^{\left( 3\right) }$ &  & $h^{\left( 1\right) }$ & $h^{\left( 2\right) }$ & $h^{\left( 3\right) }$ &  & $h^{\left( 1\right) }$ & $h^{\left( 2\right) }$ & $h^{\left( 3\right) }$ \\
&  &  & \multicolumn{1}{|c}{} &  &  &  &  &  &  &  &  &  &  &  \\
\hline
&  &  & \multicolumn{1}{|c}{} &  &  &  &  &  &  &  &  &  &  &  \\
${\mathcal{D}}_{m}^{(1)}$ & $0$ &  & \multicolumn{1}{|c}{} & $78$ & $84$ & $70$ &  & $93$ & $100$ & $83$ &  & $150$ & $163$ & $132$ \\
 & $0.5$ &  & \multicolumn{1}{|c}{} & $53$ & $63$ & $46$ &  & $75$ & $87$ & $64$ &  & $145$ & $166$ & $122$ \\
 & $0.9$ &  & \multicolumn{1}{|c}{} & $46$ & $52$ & $35$ &  & $80$ & $86$ & $60$ &  & $176$ & $186$ & $136$ \\
&  &  & \multicolumn{1}{|c}{} &  &  &  &  &  &  &  &  &  &  &  \\
${\mathcal{D}}_{m}^{(2)}$ & $0$ &  & \multicolumn{1}{|c}{} & $75$ & $83$ & $68$ &  & $84$ & $94$ & $77$ &  & $130$ & $146$ & $116$ \\
 & $0.5$ &  & \multicolumn{1}{|c}{} & $50$ & $62$ & $44$ &  & $66$ & $80$ & $58$ &  & $124$ & $147$ & $107$ \\
 & $0.9$ &  & \multicolumn{1}{|c}{} & $41$ & $49$ & $33$ &  & $69$ & $78$ & $54$ &  & $155.5$ & $168$ & $121$ \\
&  &  & \multicolumn{1}{|c}{} &  &  &  &  &  &  &  &  &  &  &  \\
${\mathcal{D}}_{m}^{(3)}$ & $0$ &  & \multicolumn{1}{|c}{} & $62$ & $67$ & $53$ &  & $69$ & $76$ & $60$ &  & $95$ & $102$ & $83$ \\
 & $0.5$ &  & \multicolumn{1}{|c}{} & $46$ & $53$ & $40$ &  & $61$ & $69$ & $53$ &  & $96$ & $103$ & $83$ \\
 & $0.9$ &  & \multicolumn{1}{|c}{} & $40$ & $46$ & $31$ &  & $67$ & $74$ & $51$ &  & $112$ & $117$ & $92$ \\
&  &  & \multicolumn{1}{|c}{} &  &  &  &  &  &  &  &  &  &  &  \\
\hline\hline
&  &  &  &  &  &  &  &  &  &  &  &  &  &  \\
\end{tabular}
}
\end{table}

\begin{table}[hbpt!]
\caption{Median detection delay - moderate changes (randomised alternative $H_{A,i}$)}
\label{tab:DelayWeakRandomAlt}\captionsetup{font=scriptsize} \centering
{\tiny
\begin{tabular}{ccccccccccccccc}
\hline\hline
&  &  &  &  &  &  &  &  &  &  &  &  &  &  \\
&  &  & \multicolumn{1}{|c}{} & \multicolumn{3}{c}{$k_{\ast}=10$} &  & \multicolumn{3}{c}{$k_{\ast}=50$} &  & \multicolumn{3}{c}{$k_{\ast}=200$} \\
Detector & $\beta$ & & \multicolumn{1}{|c}{} & $h^{\left( 1\right) }$ & $h^{\left( 2\right) }$ & $h^{\left( 3\right) }$ &  & $h^{\left( 1\right) }$ & $h^{\left( 2\right) }$ & $h^{\left( 3\right) }$ &  & $h^{\left( 1\right) }$ & $h^{\left( 2\right) }$ & $h^{\left( 3\right) }$ \\
&  &  & \multicolumn{1}{|c}{} &  &  &  &  &  &  &  &  &  &  &  \\
\hline
&  &  & \multicolumn{1}{|c}{} &  &  &  &  &  &  &  &  &  &  &  \\
${\mathcal{D}}_{m}^{(1)}$ & $0$ &  & \multicolumn{1}{|c}{} & $118$ & $116$ & $106$ &  & $141$ & $140$ & $125$ &  & $231$ & $221$ & $193$ \\
 & $0.5$ &  & \multicolumn{1}{|c}{} & $88$ & $92$ & $72$ &  & $118$ & $125$ & $98$ &  & $225$ & $223$ & $180$ \\
 & $0.9$ &  & \multicolumn{1}{|c}{} & $84$ & $80$ & $58$ &  & $134$ & $123$ & $96$ &  & $296$ & $247$ & $204$ \\
&  &  & \multicolumn{1}{|c}{} &  &  &  &  &  &  &  &  &  &  &  \\
${\mathcal{D}}_{m}^{(2)}$ & $0$ &  & \multicolumn{1}{|c}{} & $114$ & $114$ & $103$ &  & $129$ & $130$ & $114$ &  & $200$ & $197$ & $170$ \\
 & $0.5$ &  & \multicolumn{1}{|c}{} & $84$ & $89$ & $69$ &  & $105$ & $113$ & $88$ &  & $197$ & $199$ & $161$ \\
 & $0.9$ &  & \multicolumn{1}{|c}{} & $77$ & $76$ & $54$ &  & $119$ & $110$ & $85$ &  & $266$ & $224$ & $190$ \\
&  &  & \multicolumn{1}{|c}{} &  &  &  &  &  &  &  &  &  &  &  \\
${\mathcal{D}}_{m}^{(3)}$ & $0$ &  & \multicolumn{1}{|c}{} & $92$ & $86$ & $76$ &  & $103$ & $97$ & $84$ &  & $134$ & $120$ & $109$ \\
 & $0.5$ &  & \multicolumn{1}{|c}{} & $76$ & $70$ & $60$ &  & $96$ & $90$ & $76$ &  & $135$ & $122$ & $111$ \\
 & $0.9$ &  & \multicolumn{1}{|c}{} & $73$ & $63$ & $49$ &  & $114$ & $96$ & $77$ &  & $156$ & $136$ & $118$ \\
&  &  & \multicolumn{1}{|c}{} &  &  &  &  &  &  &  &  &  &  &  \\
\hline\hline
&  &  &  &  &  &  &  &  &  &  &  &  &  &  \\
\end{tabular}
}
\end{table}

\noindent To summarize the findings above, the monitoring schemes ${\mathcal{%
D}}_{m}^{(1)}$ and ${\mathcal{D}}_{m}^{(2)}$ typically have high power even for smaller-magnitude signals; however, this occurs with a
possibly large delay. The expanding-baseline detector $\mathcal D_m^{(3)}$ often shortens the delay,  but its restricted window and possible absorption of post-change observations can reduce power for weak signals. The procedures should therefore be viewed as complementary, rather than uniformly better or worse.

\noindent For further illustration, we include a short comparison study with recent detection procedures.  In addition to the three proposed detectors, we consider two versions of the procedure of \cite{gosmann2021new}, applied respectively to the mean (denoted GKD-mean) and to the vectorized covariance matrix (denoted GKD-cov).  We also include an ECDF-based detector\footnote{The full multivariate version of this procedure is based on evaluating ECDFs over a collection of points in $\mathbb R^d$; however, with $d=5$, this becomes computationally burdensome.  We therefore use a ``marginal'' version, in which one-dimensional versions are computed and then aggregated across coordinates. } inspired by \cite{holmes2024multi}, which we denote by mECDF. All procedures are size-adjusted by simulation under the null, using the same monitoring horizon and nominal level $0.05$.

\noindent The experiment uses $m=200$, $M=1000$, and $d=5$.  For the proposed detectors we use the kernel $h^{(3)}$ and $\beta=0.5$; for the GKD-type detectors we use the corresponding weight parameter $\gamma_{GKD}=0.25$\footnote{Note $0<\gamma_{GKD}<1/2$; roughly, the correspondence between $\beta$ and $\gamma_{GKD}$ is $\gamma_{GKD}=\beta/2$.}.  We consider an early change, $k^\ast=10$, and a late change, $k^\ast=500$.  The indicated mean, covariance, and tail changes are the same as the study in Table \ref{tab:PowerDelayWeakByAlt}; we also include a milder mixed alternative combining all three effects, so that no single effect is dominant: after the change the observations have location shift from $\boldsymbol \mu =0$ to $\boldsymbol{%
\mu =}\left( 0.15,...,0.15\right)^\top$, covariance matrix with entries $\Sigma_{ij}=0.25^{|i-j|}$, and standardized $t_5$ coordinates.  Table~\ref{tab:ComparisonStudy} reports post-change rejection probabilities and median detection delays, where the delays are computed conditional on post-change detection. The largest power and shortest delays for each scenario (along each column) are marked in bold.

\begin{table}[tbph!]\small
\caption{Empirical power and median delay, comparison study}
\label{tab:ComparisonStudy}\captionsetup{font=scriptsize} \centering
{\tiny
\begin{tabular}{cccccccccccc}
\hline\hline
&  &  &  &  &  &  &  &  &  &  &  \\
& Alternative & \multicolumn{1}{|c}{} & \multicolumn{4}{c}{Early change, $k^{\ast}=10$} &  & \multicolumn{4}{c}{Late change, $k^{\ast}=500$} \\
&  &  &  &  &  &  &  &  &  &  &  \\
Detector &  & \multicolumn{1}{|c}{} & Mean & Cov. & Tail & Mixed &  & Mean & Cov. & Tail & Mixed \\
&  &  &  &  &  &  &  &  &  &  &  \\
\hline
&  &  &  &  &  &  &  &  &  &  &  \\
${\mathcal{D}}_{m}^{(1)}$ & Power & \multicolumn{1}{|c}{} & $\mathbf{1.000}$ & $\mathbf{1.000}$ & $\mathbf{1.000}$ & $0.990$ &  & $0.612$ & $0.958$ & $0.339$ & $0.294$ \\
 & Med. Delay & \multicolumn{1}{|c}{} & $80$ & $47$ & $125$ & $144$ &  & $311$ & $242$ & $366$ & $342$ \\
&  &  &  &  &  &  &  &  &  &  &  \\
${\mathcal{D}}_{m}^{(2)}$ & Power & \multicolumn{1}{|c}{} & $\mathbf{1.000}$ & $\mathbf{1.000}$ & $\mathbf{1.000}$ & $\mathbf{0.991}$ &  & $0.793$ & $0.957$ & $0.400$ & $0.344$ \\
 & Med. Delay & \multicolumn{1}{|c}{} & $\mathbf{75}$ & $47$ & $120$ & $136$ &  & $319$ & $223$ & $379$ & $347$ \\
&  &  &  &  &  &  &  &  &  &  &  \\
${\mathcal{D}}_{m}^{(3)}$ & Power & \multicolumn{1}{|c}{} & $0.996$ & $\mathbf{1.000}$ & $0.941$ & $0.862$ &  & $0.959$ & $0.956$ & $0.946$ & $\mathbf{0.908}$ \\
 & Med. Delay & \multicolumn{1}{|c}{} & $77$ & $47$ & $119$ & $\mathbf{125}$ &  & $\mathbf{115}$ & $\mathbf{87}$ & $152$ & $\mathbf{159}$ \\
&  &  &  &  &  &  &  &  &  &  &  \\
GKD-mean & Power & \multicolumn{1}{|c}{} & $\mathbf{1.000}$ & $0.179$ & $0.075$ & $0.909$ &  & $0.874$ & $0.027$ & $0.015$ & $0.289$ \\
 & Med. Delay & \multicolumn{1}{|c}{} & $85$ & $217.5$ & $332$ & $195$ &  & $294$ & $178.5$ & $197$ & $332$ \\
&  &  &  &  &  &  &  &  &  &  &  \\
GKD-cov & Power & \multicolumn{1}{|c}{} & $0.046$ & $\mathbf{1.000}$ & $0.707$ & $0.951$ &  & $0.038$ & $\mathbf{0.964}$ & $0.230$ & $0.319$ \\
 & Med. Delay & \multicolumn{1}{|c}{} & $314$ & $\mathbf{27}$ & $133$ & $187$ &  & $213$ & $105$ & $180$ & $316$ \\
&  &  &  &  &  &  &  &  &  &  &  \\
mECDF & Power & \multicolumn{1}{|c}{} & $0.846$ & $0.094$ & $\mathbf{1.000}$ & $0.592$ &  & $\mathbf{0.966}$ & $0.043$ & $\mathbf{0.968}$ & $0.902$ \\
 & Med. Delay & \multicolumn{1}{|c}{} & $210$ & $525$ & $\mathbf{111}$ & $276$ &  & $221$ & $305$ & $\mathbf{133}$ & $277$ \\
&  &  &  &  &  &  &  &  &  &  &  \\
\hline\hline
\end{tabular}
}
\end{table}

The results show the expected behavior relative to targeted approaches: the mean-based GKD detector is effective for mean changes, the covariance-based GKD detector is strongest for pure covariance changes, and the mECDF detector is competitive for tail changes.  The proposed detectors ${\mathcal D}_m^{(i)}$ are nevertheless competitive across all alternatives.  In particular, the expanding-baseline detector ${\mathcal D}_m^{(3)}$ gives the most substantial gains for late changes, where it combines high power with markedly shorter detection delays, although this can come at the cost of reduced power for weaker signals.

\noindent Finally, to illustrate the potential tradeoff for $\cD_m^{(3)}$ more directly, Table \ref{tab:PowerDelaySmallMeanStress} reports a focused version of $H_{A,1}$ with $d=5$ and
$
\boldsymbol{\mu}=(\delta,\ldots,\delta)^\top,
$
where $\delta=0.20$ and $\delta=0.15$ correspond respectively to the weak and very weak settings.  Here ${\mathcal D}_m^{(1)}$ and ${\mathcal D}_m^{(2)}$ retain higher power by accumulating signal over longer post-change stretches, whereas ${\mathcal D}_m^{(3)}$ has lower power for the weakest shift but substantially shorter delays upon detection, with a delay that is (roughly) capped by the inspection window length $c_m$.

\begin{table}[bhtp!]
\caption{Empirical power and median delay - small mean changes ($h^{(2)}$, $\beta=0.5$)}
\label{tab:PowerDelaySmallMeanStress}\captionsetup{font=scriptsize} \centering
{\tiny
\begin{tabular}{ccccc}
\hline\hline\noalign{\vskip 2pt}
Detector & $\beta$ &  & $\delta=0.20$ &  $\delta=0.15$ \\[2pt]
\hline\noalign{\vskip 3pt}
${\mathcal{D}}_{m}^{(1)}$ & $0.5$ & Power & $0.969$ & $0.809$ \\
 &  & Med. Delay & $310$ & $492$ \\[6pt]
${\mathcal{D}}_{m}^{(2)}$ & $0.5$ & Power & $0.976$ & $0.849$ \\
 &  & Med. Delay & $261$ & $432$ \\[6pt]
${\mathcal{D}}_{m}^{(3)}$ & $0.5$ & Power & $0.864$ & $0.488$ \\
 &  & Med. Delay & $145$ & $168$ \\[2pt]
\hline\hline
\end{tabular}
}
\end{table}

\subsection{Empirical illustration\label{empirics}}

We apply our methodology to the MetroPT-3 dataset\footnote{The data were obtained from the UC Irvine Machine Learning Repository (\href{https://doi.org/10.24432/C5VW3R}{doi:10.24432/C5VW3R})} \citep{davari2021predictive,veloso2022metropt} which contains multivariate sensor readings from the Air Production Unit (APU) compressor of a metro train. The raw data consist of measurements recorded at one-second frequency, including pressure, motor-current and temperature measurements, and are accompanied with annotations indicating periods of known failure events. We focus on the reported high-stress air-leak event on July 15, 2020, which occurred between 14:30 and 19:00. Since the raw data are recorded at high sampling frequency, we convert these signals into lower-frequency observations by aggregating them over non-overlapping one-hour windows. We use three features derived from MetroPT-3 sensor signals: the within-window standard deviations of the pressure signals TP3 and DV pressure, and the within-window standard deviation of motor current.  For each retained 1-hour window, the resulting observation in window $i$ is
$$
\bX_i=\left(
\operatorname{sd}(\mathrm{TP3})_i,
\operatorname{sd}(\mathrm{DV\ pressure})_i,
\operatorname{sd}(\mathrm{Motor\ current})_i
\right)^\top ,
$$
where, for example, $\operatorname{sd}(\mathrm{TP3})_i$ denotes the within-window standard deviation of TP3 in window $i$. We retain only valid, nonempty windows after preprocessing. The three window-level features are then standardized coordinatewise using the corresponding historical sample.

\noindent We restrict this empirical
illustration to weekday operating periods. Specifically, each historical sample consists of the previous Monday--Friday period, and monitoring is carried out over the following Monday--Friday period. Serial dependence and
homogeneity diagnostics for the historical samples are reported in the Supplement; for the periods considered
below, these diagnostics do not indicate strong departures from homogeneity or serial independence after preprocessing.

\noindent We consider three monitoring periods. The first is the week containing the dataset-reported July 15 air-leak event. Its historical sample consists of
all $m=85$ valid one-hour windows retained from July 6--10, 2020 (from 00:00 on July 6 to 23:59 on July 10), while monitoring is carried out over July 13--17, 2020. The reported air-leak interval begins at 14:30 on July 15. We also consider two non-event weekday monitoring periods, chosen away from the reported failure events. These monitoring periods are August 10--14, 2020 and August 24--28, 2020, with corresponding historical samples taken from August 3--7 and August 17--21, respectively.

\begin{table}[!hpbt]\small
\caption{Stopping times for the MetroPT-3 dataset using kernel $h^{(2)}$ and $\beta=0.5$.}
\label{tab:metropt3_weekday_energy_beta05}
\centering
\begin{tabular}{l@{\quad}c@{\quad}ccc@{\qquad}ccc}
\hline\hline
\noalign{\vskip 4pt}
Monitoring period & Logged event & $m$ & $M$ & $\mathcal D_m^{(1)}$ & $\mathcal D_m^{(2)}$ & $\mathcal D_m^{(3)}$ \\
\noalign{\vskip 4pt}\hline\noalign{\vskip 4pt}
Jul 13--17 & 15-Jul 14:30--19:00 & 85 & 103 & 15-Jul 06:00 & 15-Jul 07:00 & 15-Jul 06:00 \\
Aug 10--14 & $-$ & 85 & 113 & \texttt{none} & \texttt{none} & \texttt{none} \\
Aug 24--28 & $-$ & 108 & 113 & \texttt{none} & \texttt{none} & \texttt{none} \\
\noalign{\vskip 4pt}\hline\hline
\end{tabular}
\begin{tablenotes}
\scriptsize
\item Results for kernels $h^{(1)}$ and $h^{(3)}$ were qualitatively similar and are omitted for brevity.
\end{tablenotes}
\end{table}

\noindent Table~\ref{tab:metropt3_weekday_energy_beta05} reports the stopping times
using the kernel $h^{(2)}$ and with $\beta=0.5$; $\mathcal D^{(3)}_m$ uses $\gamma=0.51$ and $c_0=1/5$, roughly corresponding to a 1-day inspection window. For the event week, all three
detectors stop in the early morning of July 15, several hours before the
reported air-leak interval begins at 14:30. Thus, all three procedures
identify a potential distributional change well in advance of the labeled high-stress
air-leak event. In contrast, none of the detectors stops during either of
the two non-event monitoring periods.

\noindent For illustration, Figure~\ref{fig:metropt3_detector_paths} displays the
detector paths for the event week, in monitoring time. The shaded region
marks the dataset-reported leakage event interval, while the horizontal dashed
lines indicate the calibrated critical values. The figure shows that all
three detector paths cross their respective thresholds well before the
reported event interval begins, in agreement with the stopping times in
Table~\ref{tab:metropt3_weekday_energy_beta05}.

\begin{figure}[hpbt!]
\centering
\includegraphics[width=0.8\linewidth]{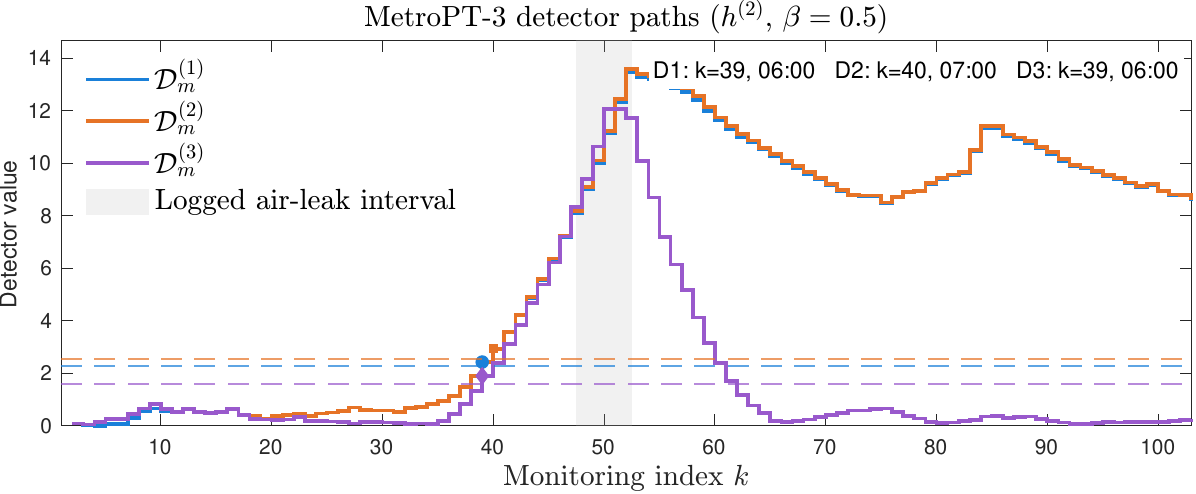}
\caption{Detector paths for the MetroPT-3 event week, using the energy
kernel $h^{(2)}$ and $\beta=0.5$. The shaded region indicates the
dataset-reported air-leak interval, and the dashed horizontal lines are the
corresponding critical values.}
\label{fig:metropt3_detector_paths}
\end{figure}

A second empirical illustration, based on infant ECG data, is reported in the Supplement.

\section{Discussion and conclusions\label{conclusion}}

We developed a flexible framework for sequentially detecting distributional changepoints using two-sample $U$-statistics. The proposed methodology includes ordinary  and Page-type detectors for both open- and closed-ended monitoring, together with an expanding-baseline variant designed to improve sensitivity to late changes. Under $H_0$, we derived the limiting distributions of all three procedures and established a consistent Monte Carlo calibration based on estimated kernel eigenvalues. For the ordinary and Page-type detectors, we further obtained consistency and detection-delay limits under both early and late changes. A key technical feature is that the theory requires only square summability of the eigenvalues of the degenerate kernel operator, rather than absolute summability.  We also propose a test
for the offline, retrospective detection of changepoints, which is useful
when testing for the maintained assumption that no changes have occurred during the historical sample. The simulations show broad sensitivity to a range of multivariate distributional changes. The ordinary and Page-type detectors are better able to accumulate weak signals, whereas the expanding-baseline detector can substantially reduce delays for later changes, at the possible cost of power against weaker alternatives. Comparisons with recent mean-, covariance-, and ECDF-based procedures, together with the empirical applications, illustrate both the flexibility and the limitations of the proposed approach.

\smallskip

\noindent Extending the theory beyond independence is a natural but technically demanding direction for future work. In particular, the kernel expansion used here is an $\mathcal L^2(F\times F)$ expansion, whereas under dependence the joint law of $(\bX_i,\bX_j)^\top$ is generally not $F\times F$. Thus, beyond establishing functional limit theory for the eigenfunction processes and spectral approximation of the operator $A$ under dependence, one must account for the expansion with the (lag-dependent) joint laws arising in the monitoring statistics.  The general metric-space formulation also suggests broader applications to functional-valued and other non-Euclidean observations. As in the Euclidean setting, however, kernel selection can be subtle, and may be especially consequential in richer data spaces because it determines which structural features of the observations drive the resulting discrepancy. These and related directions are currently under investigation.

\clearpage 

\appendix

\begin{center}
{\large\bfseries Supplementary Material}
\end{center}

\renewcommand*{\thesection}{\Alph{section}} \section{Additional simulation results\label{furtherMC}}

Tables \ref{tab:PowerStrongRandomAlt} and \ref{tab:PowerWeakRandomAlt} - complementing Tables \ref{tab:DelayStrongRandomAlt} and \ref{tab:DelayWeakRandomAlt} in the main paper, are reported hereafter. As can be seen, good power is maintained in the ``strong changes'' table throughout, either exactly or nearly equal to 1.  However, in the ``moderate changes'' panel, the detector $D_m^{(3)}$ may lose power at more distant changepoints on account of its limited inspection window.  

\begin{table}[hbpt!]
\caption{Empirical power - strong changes (randomised alternative $H_{A,i}$)}
\label{tab:PowerStrongRandomAlt}\captionsetup{font=scriptsize} \centering
{\tiny
\begin{tabular}{ccccccccccccccc}
\hline\hline
&  &  &  &  &  &  &  &  &  &  &  &  &  &  \\
&  &  & \multicolumn{1}{|c}{} & \multicolumn{3}{c}{$k_{\ast}=10$} &  & \multicolumn{3}{c}{$k_{\ast}=50$} &  & \multicolumn{3}{c}{$k_{\ast}=200$} \\
Detector & $\beta$ & & \multicolumn{1}{|c}{} & $h^{\left( 1\right) }$ & $h^{\left( 2\right) }$ & $h^{\left( 3\right) }$ &  & $h^{\left( 1\right) }$ & $h^{\left( 2\right) }$ & $h^{\left( 3\right) }$ &  & $h^{\left( 1\right) }$ & $h^{\left( 2\right) }$ & $h^{\left( 3\right) }$ \\
&  &  & \multicolumn{1}{|c}{} &  &  &  &  &  &  &  &  &  &  &  \\
\hline
&  &  & \multicolumn{1}{|c}{} &  &  &  &  &  &  &  &  &  &  &  \\
${\mathcal{D}}_{m}^{(1)}$ & $0$ &  & \multicolumn{1}{|c}{} & $1.000$ & $1.000$ & $1.000$ &  & $1.000$ & $1.000$ & $1.000$ &  & $0.992$ & $0.995$ & $0.996$ \\
 & $0.5$ &  & \multicolumn{1}{|c}{} & $1.000$ & $1.000$ & $1.000$ &  & $0.999$ & $1.000$ & $0.999$ &  & $0.978$ & $0.985$ & $0.983$ \\
 & $0.9$ &  & \multicolumn{1}{|c}{} & $0.996$ & $0.994$ & $0.994$ &  & $0.985$ & $0.985$ & $0.982$ &  & $0.972$ & $0.967$ & $0.960$ \\
&  &  & \multicolumn{1}{|c}{} &  &  &  &  &  &  &  &  &  &  &  \\
${\mathcal{D}}_{m}^{(2)}$ & $0$ &  & \multicolumn{1}{|c}{} & $1.000$ & $1.000$ & $1.000$ &  & $1.000$ & $1.000$ & $1.000$ &  & $0.991$ & $0.995$ & $0.996$ \\
 & $0.5$ &  & \multicolumn{1}{|c}{} & $1.000$ & $1.000$ & $1.000$ &  & $0.999$ & $1.000$ & $0.999$ &  & $0.977$ & $0.985$ & $0.983$ \\
 & $0.9$ &  & \multicolumn{1}{|c}{} & $0.996$ & $0.995$ & $0.995$ &  & $0.986$ & $0.985$ & $0.983$ &  & $0.970$ & $0.968$ & $0.963$ \\
&  &  & \multicolumn{1}{|c}{} &  &  &  &  &  &  &  &  &  &  &  \\
${\mathcal{D}}_{m}^{(3)}$ & $0$ &  & \multicolumn{1}{|c}{} & $1.000$ & $1.000$ & $1.000$ &  & $1.000$ & $1.000$ & $1.000$ &  & $0.972$ & $0.978$ & $0.975$ \\
 & $0.5$ &  & \multicolumn{1}{|c}{} & $1.000$ & $1.000$ & $1.000$ &  & $0.998$ & $0.998$ & $0.998$ &  & $0.962$ & $0.964$ & $0.968$ \\
 & $0.9$ &  & \multicolumn{1}{|c}{} & $0.995$ & $0.993$ & $0.994$ &  & $0.982$ & $0.977$ & $0.977$ &  & $0.963$ & $0.954$ & $0.950$ \\
&  &  & \multicolumn{1}{|c}{} &  &  &  &  &  &  &  &  &  &  &  \\
\hline\hline
&  &  &  &  &  &  &  &  &  &  &  &  &  &  \\
\end{tabular}
}
\end{table}

\vspace{3ex}

\begin{table}[hbpt!]
\caption{Empirical power - moderate changes (randomised alternative $H_{A,i}$,)}
\label{tab:PowerWeakRandomAlt}\captionsetup{font=scriptsize} \centering
{\tiny
\begin{tabular}{ccccccccccccccc}
\hline\hline
&  &  &  &  &  &  &  &  &  &  &  &  &  &  \\
&  &  & \multicolumn{1}{|c}{} & \multicolumn{3}{c}{$k_{\ast}=10$} &  & \multicolumn{3}{c}{$k_{\ast}=50$} &  & \multicolumn{3}{c}{$k_{\ast}=200$} \\
Detector & $\beta$ & & \multicolumn{1}{|c}{} & $h^{\left( 1\right) }$ & $h^{\left( 2\right) }$ & $h^{\left( 3\right) }$ &  & $h^{\left( 1\right) }$ & $h^{\left( 2\right) }$ & $h^{\left( 3\right) }$ &  & $h^{\left( 1\right) }$ & $h^{\left( 2\right) }$ & $h^{\left( 3\right) }$ \\
&  &  & \multicolumn{1}{|c}{} &  &  &  &  &  &  &  &  &  &  &  \\
\hline
&  &  & \multicolumn{1}{|c}{} &  &  &  &  &  &  &  &  &  &  &  \\
${\mathcal{D}}_{m}^{(1)}$ & $0$ &  & \multicolumn{1}{|c}{} & $1.000$ & $0.992$ & $1.000$ &  & $1.000$ & $0.985$ & $0.999$ &  & $0.993$ & $0.906$ & $0.982$ \\
 & $0.5$ &  & \multicolumn{1}{|c}{} & $1.000$ & $0.988$ & $1.000$ &  & $0.998$ & $0.967$ & $0.998$ &  & $0.977$ & $0.858$ & $0.963$ \\
 & $0.9$ &  & \multicolumn{1}{|c}{} & $0.996$ & $0.950$ & $0.991$ &  & $0.983$ & $0.910$ & $0.975$ &  & $0.955$ & $0.762$ & $0.918$ \\
&  &  & \multicolumn{1}{|c}{} &  &  &  &  &  &  &  &  &  &  &  \\
${\mathcal{D}}_{m}^{(2)}$ & $0$ &  & \multicolumn{1}{|c}{} & $1.000$ & $0.991$ & $1.000$ &  & $1.000$ & $0.987$ & $1.000$ &  & $0.993$ & $0.934$ & $0.995$ \\
 & $0.5$ &  & \multicolumn{1}{|c}{} & $1.000$ & $0.985$ & $1.000$ &  & $0.998$ & $0.972$ & $0.998$ &  & $0.980$ & $0.876$ & $0.976$ \\
 & $0.9$ &  & \multicolumn{1}{|c}{} & $0.997$ & $0.945$ & $0.994$ &  & $0.985$ & $0.906$ & $0.977$ &  & $0.964$ & $0.768$ & $0.943$ \\
&  &  & \multicolumn{1}{|c}{} &  &  &  &  &  &  &  &  &  &  &  \\
${\mathcal{D}}_{m}^{(3)}$ & $0$ &  & \multicolumn{1}{|c}{} & $0.999$ & $0.874$ & $0.981$ &  & $0.998$ & $0.853$ & $0.977$ &  & $0.973$ & $0.728$ & $0.922$ \\
 & $0.5$ &  & \multicolumn{1}{|c}{} & $0.997$ & $0.864$ & $0.978$ &  & $0.994$ & $0.838$ & $0.972$ &  & $0.964$ & $0.716$ & $0.904$ \\
 & $0.9$ &  & \multicolumn{1}{|c}{} & $0.988$ & $0.792$ & $0.957$ &  & $0.970$ & $0.756$ & $0.927$ &  & $0.913$ & $0.657$ & $0.846$ \\
&  &  & \multicolumn{1}{|c}{} &  &  &  &  &  &  &  &  &  &  &  \\
\hline\hline
&  &  &  &  &  &  &  &  &  &  &  &  &  &  \\
\end{tabular}
}
\end{table}

\clearpage

\section{Additional data illustration}

We also apply our methodology to the detection of changes in the heart
rate (ECG) recording of an infant. We use the BabyECG dataset in %
\citet{nason2000wavelet}: a series of $2,\!048$ observations recorded in beats
per minute, sampled overnight every $16$ seconds from 21:17:59 to 06:27:18,
from a $66$-day-old infant.\footnote{%
The data are available as part of the R package \textsf{wavethresh}, and
they were originally recorded by Prof. Peter Fleming, Dr Andrew Sawczenko
and Jeanine Young of the Institute of Child Health, Royal Hospital for Sick
Children, Bristol.}  The accompanying data also contain annotations of the infant's sleep state.

We focus on transition out of an annotated period of quiet sleep. To exclude the initial transition into this state, we discard its first $30$ observations, corresponding to $8$ minutes, and use the following 15 minutes as the historical sample,  corresponding to $m=56$ observations,  from 00:19:03 to 00:33:43. We use a horizon of $M=225$, which corresponds to 1 hour.  As a pre-whitening phase, we select among $\operatorname{AR}(p)$ models with $p\in\{0,1,2\}$, using BIC as a criterion; an $\operatorname{AR}(1)$ model is selected. The monitoring procedures are applied to the resulting one-step residuals, with the fitted model held fixed during monitoring. Monitoring begins at 00:33:59. The homogeneity and serial-dependence diagnostics reported in Section~\ref{s:datacomplements} provide no significant evidence against the maintained assumptions for the historical residuals.
For $\mathcal D_m^{(3)}$, we take $c_m=38$, corresponding to approximately $10.1$ minutes.

\begin{table}[hptb!]\small
\caption{Stopping times for the BabyECG data. The annotated exit from quiet sleep occurs 10.4 minutes after monitoring begins. Reported values are stopping times in minutes after monitoring begins.}
\label{tab:babyecg_Q04_allkernels_allbeta}
\centering
\begin{tabular}{c@{\quad}c@{\qquad}ccc}
\hline\hline
\noalign{\vskip 4pt}
Kernel & $\beta$ & $\mathcal D_m^{(1)}$ & $\mathcal D_m^{(2)}$ & $\mathcal D_m^{(3)}$ \\
\noalign{\vskip 4pt}\hline\noalign{\vskip 4pt}
$h^{(1)}$ & 0.0 & 15.5 & 14.9 & 10.9 \\
 & 0.5 & 15.2 & 12.8 & 11.2 \\
 & 0.9 & 15.5 & 12.8 & 11.2 \\
\noalign{\vskip 2pt}
$h^{(2)}$ & 0.0 & 15.7 & 15.2 & 11.2 \\
 & 0.5 & 15.5 & 15.2 & 11.2 \\
 & 0.9 & 15.7 & 15.2 & 11.7 \\
\noalign{\vskip 2pt}
$h^{(3)}$ & 0.0 & 15.7 & 14.9 & 10.9 \\
 & 0.5 & 15.5 & 12.8 & 10.9 \\
 & 0.9 & 15.7 & 14.9 & 11.2 \\
\noalign{\vskip 4pt}\hline\hline
\end{tabular}
\end{table}

During this monitoring period, there was an annotated sleep-state change at 10.4 minutes after the onset of monitoring. Table~\ref{tab:babyecg_Q04_allkernels_allbeta} reports stopping times for the three kernels and the considered values of $\beta$. The results are similar across kernels and boundary parameters: $\mathcal D_m^{(3)}$ signals closest to the annotated exit from quiet sleep, followed by $\mathcal D_m^{(2)}$ and then $\mathcal D_m^{(1)}$. Figure~\ref{fig:babyecg_detector_paths} shows the corresponding detector paths for kernel $h^{(2)}$ and $\beta=0.5$, together with the recorded sleep-state annotation.

\begin{figure}[hptb!]
\centering
\includegraphics[width=0.7\linewidth]{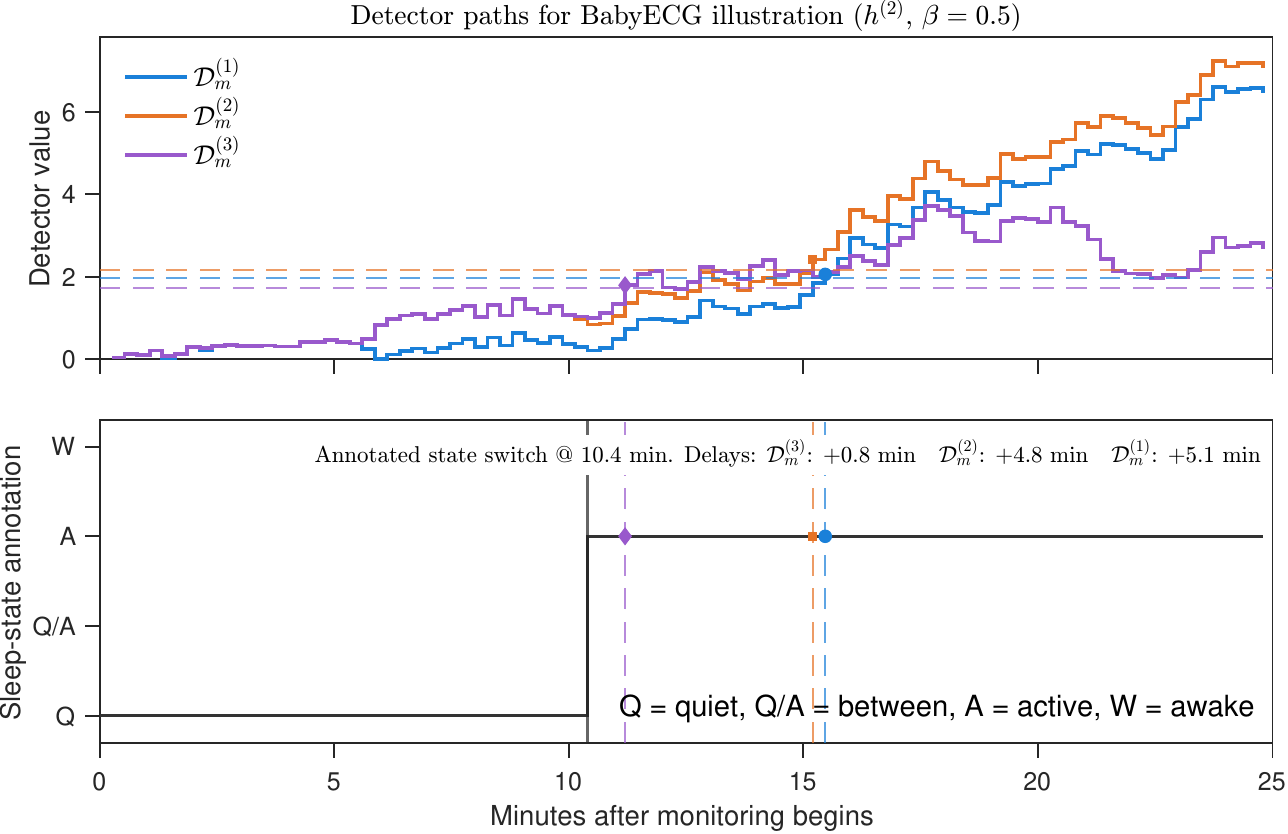}
\caption{Detector paths for the BabyECG data using the kernel $h^{(2)}$, with $\beta=0.5$. The lower panel shows the recorded sleep states: quiet (Q), intermediate (Q/A), active (A), and awake (W); the present segment includes transition from state $Q$ to state A.  Dashed horizontal lines denote the calibrated critical values, and markers indicate the first crossings.}
\label{fig:babyecg_detector_paths}
\end{figure}

\section{Complements to data examples}\label{s:datacomplements}

In this section, we report diagnostic results and additional information for the empirical examples. The diagnostics are not intended to provide a comprehensive validation of the maintained assumptions; rather, they serve as illustrative checks that the selected historical samples do not exhibit severe departures from serial independence or distributional stability.

Homogeneity is assessed using the statistic from Theorem~\ref{thewi}, with critical values obtained by Monte Carlo using the empirical eigenvalues of the kernel matrix $A_m$. Because our monitoring procedures are themselves distance-based, we assess serial dependence using distance-covariance methods, which are closely aligned with the framework of the paper (in particular, when using kernel $h^{(2)}$) and are sensitive to general, rather than only linear, forms of dependence.

Specifically, let  $h\geq 0$, let $(\bX_0',\bX_h')$ be an independent copy of $(\bX_0,\bX_h)$, and let $\bX_h''$ be an additional independent copy of $\bX_h$.  Define (see  \cite[equation~(2.12)]{davis:matsui:mikosch:wan:2018})
\begin{align*}
T^X(h)= \E\left\|\bX_0-\bX_0'\right\|\left\|\bX_h-\bX_h'\right\| &+\E\left\|\bX_0-\bX_0'\right\|
\E\left\|\bX_h-\bX_h'\right\|\\
&\quad -2\E\left[\left\|\bX_0-\bX_0'\right\| \left\|\bX_h-\mathbf X_h''\right\|\right].\end{align*} and
$$
R^X(h)
=\frac{T^X(h)}{T^X(0)}.
$$
Then, $0\leq R^X(h)\leq 1$,  and $R^X(h)=0$ if and only if $\mathbf X_0$ and $\mathbf X_h$ are independent. Let $T_{m}^X(h)$ and $R_{m}^X(h)$ denote the empirical versions
defined in Section~3.2 of \cite{davis:matsui:mikosch:wan:2018}, computed from $\mathbf X_1,\ldots,\mathbf X_m$, and set
\begin{equation}\label{e:MnH}
Q_{m,H}=m(m+2)\sum_{h=1}^{H}\frac{R_{m}^X(h)}{m-h},\quad 
M_{m,H}
=\max_{1\leq h\leq H}R_{m}^X(h).
\end{equation}
The statistic $Q_{m,H}$ is a Box-Ljung-type statistic based on the distance correlation inspired by \cite{fokianos:pitsillou:2017}, and $M_{m,H}$ records the largest dependence measure among the first $H$ lags.  By Theorem~3.1 and Corollary~3.7 of
\cite{davis:matsui:mikosch:wan:2018}, for each fixed $h$,
$$
R_{m}^X(h) \to R^X(h),\quad \text{a.s.}.
$$
Consequently, for fixed $H$, both statistics in \eqref{e:MnH} consistently summarize pairwise dependence over the first $H$ lags. In particular, under stationarity,
$\max_{1\leq h\leq H}R^X(h)=0$
if and only if $\bX_1,\ldots,\bX_{H+1}$ are pairwise independent.

Under the null hypothesis that
$\mathbf X_1,\bX_2,\ldots$ are independent and identically
distributed, the observations are exchangeable, and a permutation test is valid.  We therefore calibrate
$M_{m,H}$ and $Q_{m,H}$ by randomly permuting the historical sample and recomputing all lagged statistics.  For permutations $\pi_1,\ldots,\pi_B$, the permutation-based 
$p$-value is computed as
$$
\widehat p_Q =\frac{
1+\sum_{b=1}^B
\mathbf 1\left\{Q_{m,H}^{\pi_b}\geq Q_{m,H}\right\}}{B+1},\qquad \widehat p_M =\frac{
1+\sum_{b=1}^B\mathbf 1\left\{ M_{m,H}^{\pi_b}\geq M_{m,H}
\right\}}{B+1}.
$$

\bigskip

Finally, in order to check how many moments the data admit, we have used the tests
developed in \citet{trapani2016testing} and %
\citet{degiannakis2023superkurtosis}, which we summarize here for
completeness. The test is applied separately to each coordinate of the training sample. For a given coordinate, the null hypothesis that the moment of order $k$
of a random variable $X$\ does not exist, viz.%
\[
H_{0}:E\left\vert X\right\vert ^{k}=\infty ,
\]%
is implemented by constructing the statistic%
\[
\mu _{k}=\frac{m^{-1}\sum_{i=1}^{m}\left\vert X_{i}\right\vert ^{k}}{\left(
m^{-1}\sum_{i=1}^{m}\left\vert X_{i}\right\vert ^{2}\right) ^{k/2}},
\]%
computed from $\left\{ X_{i},1\leq i\leq m\right\}$, in the training sample and
subsequently%
\[
\psi _{k}=\exp \left( \mu _{k}\right) -1.
\]%
The statistic $\psi _{k}$\ is then randomised according to the following
algorithm:
\begin{description}
\item[Step 1] Generate an artificial sample $\left\{ \xi _{n}^{\left(
k\right) },1\leq n\leq N\right\} $, \textit{i.i.d. }across $n$ and
independently across $k$, with $\xi _{1}^{\left( k\right) }\sim N\left(
0,1\right) $, and define $\left\{ \psi _{k}^{1/2}\times \xi _{n}^{\left(
k\right) }\right\} _{n=1}^{N}$.

\item[Step 2] For $u\in \left\{ \pm \sqrt{2}\right\} $, generate $\zeta
_{n,N}^{\left( k\right) }\left( u\right) =I\left( \psi _{k}^{1/2}\times \xi
_{n}^{\left( k\right) }\leq u\right) $, $1\leq n\leq N$.

\item[Step 3] For each $u$, define 
\[
\vartheta _{m,N}^{\left( k\right) }\left( u\right) =\frac{2}{N^{1/2}}%
\sum_{n=1}^{N}\left[ \zeta _{n,N}^{\left( k\right) }\left( u\right) -\frac{1%
}{2}\right] ,
\]%
and then the test statistic%
\[
\Theta _{m,N}^{\left( k\right) }=\frac{1}{2}\left[ \left( \vartheta
_{m,N}^{\left( k\right) }\left( \sqrt{2}\right) \right) ^{2}+\left(
\vartheta _{m,N}^{\left( k\right) }\left( -\sqrt{2}\right) \right) ^{2}%
\right] .
\]
\end{description}

\citet{trapani2016testing} shows that, as $\min \left( m,N\right)
\rightarrow \infty $ with $N=O\left( m\right) $%
\[
\begin{array}{lll}
\Theta _{m,N}^{\left( k\right) }\overset{D^{* }}{\rightarrow } & \chi
_{1}^{2} & \text{ \ \ under }H_{0}, \\ 
N^{-1}\Theta _{m,N}^{\left( k\right) }\overset{P^{* }}{\rightarrow } & 
c_{0}>0 & \text{ \ \ under }H_{A},%
\end{array}%
\]%
where $P^{* }$ denotes the conditional probability with respect to the
sample, and \textquotedblleft $\overset{D^{* }}{\rightarrow }$%
\textquotedblright\ and \textquotedblleft $\overset{P^{* }}{\rightarrow }$%
\textquotedblright\ denote conditional convergence in distribution and in
probability according to $P^{* }$. In order to wash out dependence on the
randomness, we run the test for $1\leq b\leq B$ iterations,
each time defining a test statistic $_{\left( b\right) }\Theta
_{m,N}^{\left( k\right) }$, and computing the \textit{randomised confidence
function}%
\[
Q_{m,N,B}\left( \alpha \right) =\frac{1}{B}\sum_{b=1}^{B}I\left[ _{\left(
b\right) }\Theta _{m,N}^{\left( k\right) }\leq c_{\alpha }\right] ,
\]%
where $c_{\alpha }$ is defined as $P\{\chi _{1}^{2}\geq c_{\alpha }\}=\alpha 
$, for a given nominal level $\alpha \in \left( 0,1\right) $. Hence, the
decision rule in favour of $H_{0}$ is%
\begin{equation}
Q_{m,N,B}\left( \alpha \right) \geq \left( 1-\alpha \right) -\frac{\sqrt{%
\alpha \left( 1-\alpha \right) }}{f\left( B\right) },  \label{decision}
\end{equation}%
{\normalsize where the function $f\left( B\right) $ is user-defined such
that 
\begin{equation}
\underset{B\rightarrow \infty }{\liminf }\frac{B^{1/2}}{f\left( B\right) }%
\geq c_{\alpha }.  \label{lil2}
\end{equation}%
Following the indications in }\citet{trapani2016testing} and %
\citet{degiannakis2023superkurtosis}, we have used $N=B=m$, and $f\left(
B\right) =B^{1/4}$.

\bigskip

We now provide complementary results to our empirical illustrations.
\subsection{MetroPT-3 data}
For the MetroPT-3 empirical study, we report diagnostics for the historical samples used in the weekday monitoring periods. Table~\ref{tab:metropt3_weekday_serial_adcf_QM} reports serial-dependence diagnostics based on \eqref{e:MnH}, using $H=5$ and $H=10$ lags. The results do not indicate pronounced departures from serial independence. 

\begin{table}[h!]\small
\caption{Serial-dependence diagnostics for the MetroPT-3 historical samples. The statistics $Q_{m,H}$ and $M_{m,H}$ are defined in \eqref{e:MnH}; the reported $p$-values are obtained by permutation.}
\label{tab:metropt3_weekday_serial_adcf_QM}
\centering
\begin{tabular}{c@{\quad}c@{\quad}cc@{\quad}cc}
\hline\hline
\noalign{\vskip 4pt}
Historical sample & $H$ & $Q_{m,H}$ & $M_{m,H}$ & $\widehat p_Q$ & $\widehat p_M$ \\
\noalign{\vskip 4pt}\hline\noalign{\vskip 4pt}
Jul 6--10 & 5 & 27.412 & 0.112 & 0.597 & 0.075 \\
 & 10 & 56.885 & 0.112 & 0.759 & 0.192 \\
\noalign{\vskip 2pt}
Aug 3--7 & 5 & 23.622 & 0.087 & 0.259 & 0.226 \\
 & 10 & 45.471 & 0.087 & 0.475 & 0.454 \\
\noalign{\vskip 2pt}
Aug 17--21 & 5 & 25.287 & 0.085 & 0.496 & 0.094 \\
 & 10 & 46.578 & 0.085 & 0.927 & 0.212 \\
\noalign{\vskip 4pt}\hline\hline
\end{tabular}
\end{table}

Table~\ref{tab:metropt3_weekday_homogeneity_cp} reports the corresponding homogeneity diagnostics. These do not indicate evidence of a distributional change within the selected historical samples.

\begin{table}[hptb!]\small
\caption{Historical-sample homogeneity diagnostics for the MetroPT-3 empirical illustration. Homogeneity is assessed using the test from Theorem~\ref{thewi}, calibrated with empirical eigenvalues estimated from the historical sample. The diagnostic uses the energy kernel $h^{(2)}$, with $\beta=0.5$, matching the analogous detector configurations in Table~\ref{tab:metropt3_weekday_energy_beta05}.}
\label{tab:metropt3_weekday_homogeneity_cp}
\centering
\begin{tabular}{c@{\quad}cccc}
\hline\hline
\noalign{\vskip 4pt}
Historical sample & $m$ & Statistic & Critical value & $p$-value \\
\noalign{\vskip 4pt}\hline\noalign{\vskip 4pt}
Jul 6--10  & 85 & 0.247 & 0.906 & 0.769 \\
Aug 3--7  & 85 & 0.357 & 0.925 & 0.380 \\
Aug 17--21 & 108 & 0.314 & 0.948 & 0.474 \\
\noalign{\vskip 4pt}\hline\hline
\end{tabular}
\begin{tablenotes}
\scriptsize
\item The  critical values are for the nominal level $0.05$. Results for kernels $h^{(1)}$ and $h^{(3)}$ were qualitatively similar and are omitted for brevity.
\end{tablenotes}
\end{table}

Descriptive statistics and tests for moment existence are in Tables \ref{tab:descriptive1}-\ref{tab:descriptive3} below, and refer, by way of illustration, to the historical sample ranging from July 6th till July 10th.\footnote{Similar results are observed for the other two historical samples, and are available upon request.} The test for moment existence rejects the null of moment nonexistence at least up to order 4 in all cases, indicating that enough moments exist such that Assumption \ref{a:assumption_on_h} is satisfied in all specified kernel choices.

\begin{table}[h!]
\caption{Descriptive statistics and moment existence - TP3}
\label{tab:descriptive1}\captionsetup{font=scriptsize} \centering
{\tiny \centering
}
\par
{\tiny 

\begin{tabular}{llllllll}
\hline\hline
&  &  &  &  &  &  &  \\ 
\multicolumn{8}{c}{TP3} \\ 
\multicolumn{1}{c}{} & \multicolumn{1}{c}{} & \multicolumn{1}{c}{} & 
\multicolumn{1}{c}{} & \multicolumn{1}{|c}{} & \multicolumn{1}{c}{} & 
\multicolumn{1}{c}{} & \multicolumn{1}{c}{} \\ 
\multicolumn{3}{c}{Descriptive statistics} & \multicolumn{1}{c}{} & 
\multicolumn{1}{|c}{} & \multicolumn{3}{c}{Tests for moment existence} \\ 
&  &  &  & \multicolumn{1}{|l}{} &  &  &  \\ 
Mean &  & \multicolumn{1}{c}{$0.687$} &  & \multicolumn{1}{|l}{} & 
Degiannakis et al. (2023) &  & \multicolumn{1}{c}{$\underset{\left[ \text{%
reject }H_{0}\right] }{H_{0}:E\left\vert X\right\vert ^{4}=\infty }$} \\ 
St. Dev. &  & \multicolumn{1}{c}{$0.404$} &  & \multicolumn{1}{|l}{} &  &  & 
\multicolumn{1}{c}{} \\ 
Skewness &  & \multicolumn{1}{c}{$3.333$} &  & \multicolumn{1}{|l}{} &  &  & 
\multicolumn{1}{c}{$\underset{\left[ \text{NOT reject }H_{0}\right] }{%
H_{0}:E\left\vert X\right\vert ^{5}=\infty }$} \\ 
Kurtosis &  & \multicolumn{1}{c}{$14.527$} &  & \multicolumn{1}{|l}{} &  & 
& \multicolumn{1}{c}{} \\ 
&  &  &  & \multicolumn{1}{|l}{} &  &  & \multicolumn{1}{c}{$\underset{\left[
\text{NOT reject }H_{0}\right] }{H_{0}:E\left\vert X\right\vert ^{6}=\infty }
$} \\ 
&  &  &  & \multicolumn{1}{|l}{} &  &  & \multicolumn{1}{c}{} \\ 
&  &  &  & \multicolumn{1}{|l}{} & Jarque-Bera &  & \multicolumn{1}{c}{$%
\underset{\left[ \text{p-value=0.000}\right] }{H_{0}:\text{Gaussian data}}$}
\\ 
&  &  &  &  &  &  &  \\ \hline\hline
\end{tabular}

} 
\par
{\tiny 
\begin{tablenotes}
      \tiny
            \item The table contains the outcomes for the test by \citet{degiannakis2023superkurtosis} described above for the null that the moments of the series of the (standard deviations of) of TP3 of order $4$, $5$ and $6$ are non-existent. 
            
\end{tablenotes}
}
\end{table}

\begin{table}[h!]
\caption{Descriptive statistics and moment existence - DV}
\label{tab:descriptive2}\captionsetup{font=scriptsize} \centering
{\tiny \centering
}
\par
{\tiny 

\begin{tabular}{llllllll}
\hline\hline
&  &  &  &  &  &  &  \\ 
\multicolumn{8}{c}{DV} \\ 
\multicolumn{1}{c}{} & \multicolumn{1}{c}{} & \multicolumn{1}{c}{} & 
\multicolumn{1}{c}{} & \multicolumn{1}{|c}{} & \multicolumn{1}{c}{} & 
\multicolumn{1}{c}{} & \multicolumn{1}{c}{} \\ 
\multicolumn{3}{c}{Descriptive statistics} & \multicolumn{1}{c}{} & 
\multicolumn{1}{|c}{} & \multicolumn{3}{c}{Tests for moment existence} \\ 
&  &  &  & \multicolumn{1}{|l}{} &  &  &  \\ 
Mean &  & \multicolumn{1}{c}{$0.121$} &  & \multicolumn{1}{|l}{} & 
Degiannakis et al. (2023) &  & \multicolumn{1}{c}{$\underset{\left[ \text{%
reject }H_{0}\right] }{H_{0}:E\left\vert X\right\vert ^{4}=\infty }$} \\ 
St. Dev. &  & \multicolumn{1}{c}{$0.152$} &  & \multicolumn{1}{|l}{} &  &  & 
\multicolumn{1}{c}{} \\ 
Skewness &  & \multicolumn{1}{c}{$1.029$} &  & \multicolumn{1}{|l}{} &  &  & 
\multicolumn{1}{c}{$\underset{\left[ \text{reject }H_{0}\right] }{%
H_{0}:E\left\vert X\right\vert ^{5}=\infty }$} \\ 
Kurtosis &  & \multicolumn{1}{c}{$2.882$} &  & \multicolumn{1}{|l}{} &  &  & 
\multicolumn{1}{c}{} \\ 
&  &  &  & \multicolumn{1}{|l}{} &  &  & \multicolumn{1}{c}{$\underset{\left[
\text{NOT reject }H_{0}\right] }{H_{0}:E\left\vert X\right\vert ^{6}=\infty }
$} \\ 
&  &  &  & \multicolumn{1}{|l}{} &  &  & \multicolumn{1}{c}{} \\ 
&  &  &  & \multicolumn{1}{|l}{} & Jarque-Bera &  & \multicolumn{1}{c}{$%
\underset{\left[ \text{p-value=0.000}\right] }{H_{0}:\text{Gaussian data}}$}
\\ 
&  &  &  &  &  &  &  \\ \hline\hline
\end{tabular}

} 
\par
{\tiny 
\begin{tablenotes}
      \tiny
            \item The table contains the outcomes for the test by \citet{degiannakis2023superkurtosis} described above for the null that the moments of the series of the (standard deviations of) of DV of order $4$, $5$ and $6$ are non-existent. 
            
\end{tablenotes}
}
\end{table}

\begin{table}[h!]
\caption{Descriptive statistics and moment existence - Motor}
\label{tab:descriptive3}\captionsetup{font=scriptsize} \centering
{\tiny \centering
}
\par
{\tiny 

\begin{tabular}{llllllll}
\hline\hline
&  &  &  &  &  &  &  \\ 
\multicolumn{8}{c}{Motor} \\ 
\multicolumn{1}{c}{} & \multicolumn{1}{c}{} & \multicolumn{1}{c}{} & 
\multicolumn{1}{c}{} & \multicolumn{1}{|c}{} & \multicolumn{1}{c}{} & 
\multicolumn{1}{c}{} & \multicolumn{1}{c}{} \\ 
\multicolumn{3}{c}{Descriptive statistics} & \multicolumn{1}{c}{} & 
\multicolumn{1}{|c}{} & \multicolumn{3}{c}{Tests for moment existence} \\ 
&  &  &  & \multicolumn{1}{|l}{} &  &  &  \\ 
Mean &  & \multicolumn{1}{c}{$2.145$} &  & \multicolumn{1}{|l}{} & 
Degiannakis et al. (2023) &  & \multicolumn{1}{c}{$\underset{\left[ \text{%
reject }H_{0}\right] }{H_{0}:E\left\vert X\right\vert ^{4}=\infty }$} \\ 
St. Dev. &  & \multicolumn{1}{c}{$0.335$} &  & \multicolumn{1}{|l}{} &  &  & 
\multicolumn{1}{c}{} \\ 
Skewness &  & \multicolumn{1}{c}{$-4.615$} &  & \multicolumn{1}{|l}{} &  & 
& \multicolumn{1}{c}{$\underset{\left[ \text{reject }H_{0}\right] }{%
H_{0}:E\left\vert X\right\vert ^{8}=\infty }$} \\ 
Kurtosis &  & \multicolumn{1}{c}{$25.558$} &  & \multicolumn{1}{|l}{} &  & 
& \multicolumn{1}{c}{} \\ 
&  &  &  & \multicolumn{1}{|l}{} &  &  & \multicolumn{1}{c}{$\underset{\left[
\text{reject }H_{0}\right] }{H_{0}:E\left\vert X\right\vert ^{32}=\infty }$}
\\ 
&  &  &  & \multicolumn{1}{|l}{} &  &  & \multicolumn{1}{c}{} \\ 
&  &  &  & \multicolumn{1}{|l}{} & Jarque-Bera &  & \multicolumn{1}{c}{$%
\underset{\left[ \text{p-value=0.000}\right] }{H_{0}:\text{Gaussian data}}$}
\\ 
&  &  &  &  &  &  &  \\ \hline\hline
\end{tabular}

} 
\par
{\tiny 
\begin{tablenotes}
      \tiny
            \item The table contains the outcomes for the test by \citet{degiannakis2023superkurtosis} described above for the null that the moments of the series of the (standard deviations of) of Motor of order $4$, $8$ and $32$ are non-existent. 
            
\end{tablenotes}
}
\end{table}

\clearpage

\subsection{BabyECG data}

In Table \ref{tab:babyecg_Q04_serial}, we report serial dependence diagnostics based on \eqref{e:MnH} for the BabyECG dataset using $H=5$ and $H=10$. However, these diagnostics are here applied to residuals, and permutation calibration would thus incorrectly treat the residuals as exactly exchangeable. We therefore use a bootstrap procedure on the residuals: after selecting the AR order specification by BIC on the historical sample, we hold the selected AR($p$) order fixed $(p=1)$; in each bootstrap replication, the fitted residuals are resampled as innovations, a pseudo-series is generated from the fitted AR model, and an AR(1) model is re-estimated before recomputing the diagnostics. The results do not indicate pronounced departures from serial independence over the first ten lags. 

\begin{table}[hpbt!]\small
\caption{Serial-dependence diagnostics for the BabyECG historical residuals. The statistics $Q_{m,H}$ and $M_{m,H}$ are defined in \eqref{e:MnH}; their $p$-values are obtained from a bootstrap procedure applied to the residuals.}
\label{tab:babyecg_Q04_serial}
\centering
\begin{tabular}{c@{\quad}c@{\quad}cc@{\quad}cc}
\hline\hline
\noalign{\vskip 4pt}
Historical sample & $H$ & $Q_{m,H}$ & $M_{m,H}$ & $p_Q$ & $p_M$ \\
\noalign{\vskip 4pt}\hline\noalign{\vskip 4pt}
00:19:03 -- 00:33:43 & 5 & 16.586 & 0.098 & 0.576 & 0.278 \\
 & 10 & 44.662 & 0.146 & 0.199 & 0.111 \\
\noalign{\vskip 4pt}\hline\hline
\end{tabular}
\end{table}

In Table \ref{tab:babyecg_Q04_homogeneity}, we test for the non-contamination condition of Assumption \ref%
{a:historical_stability} based on the methodology developed in Section \ref%
{sectrain} in the main paper, using the same bootstrap procedure described above.

\begin{table}[hpbt!]\small
\caption{Homogeneity diagnostic for the BabyECG historical residuals. The critical value and $p$-value are obtained from the same recursive residual bootstrap used for the serial-dependence diagnostics.}
\label{tab:babyecg_Q04_homogeneity}
\centering
\begin{tabular}{c@{\quad}ccc}
\hline\hline
\noalign{\vskip 4pt}
Historical sample &  Statistic & Critical value & $p$-value \\
\noalign{\vskip 4pt}\hline\noalign{\vskip 4pt}
00:19:03--00:33:43 & 0.428 & 1.617 & 0.634 \\
\noalign{\vskip 4pt}\hline\hline
\end{tabular}
\end{table}

Descriptive statistics and tests for moment existence are in Table \ref{tab:descriptive} below. The tests reject the null of moment nonexistence at orders 8, 16, and 32, providing support that the data admit enough moments such that Assumption \ref{a:assumption_on_h} is satisfied in all specified kernel choices.

\begin{table}[h!]
\caption{Descriptive statistics and moment existence}
\label{tab:descriptive}\captionsetup{font=scriptsize} \centering
{\tiny \centering
}
\par
{\tiny

\begin{tabular}{llllllll}
\hline\hline
&  &  &  &  &  &  &  \\ 
\multicolumn{3}{c}{\textbf{Descriptive statistics}} &  & \multicolumn{1}{|l}{
} & \multicolumn{3}{c}{\textbf{Tests for moment existence}} \\ 
&  &  &  & \multicolumn{1}{|l}{} &  &  &  \\ 
Mean &  & \multicolumn{1}{c}{$4.793$} &  & \multicolumn{1}{|l}{} & 
Degiannakis et al. (2023) &  & \multicolumn{1}{c}{$\underset{\left[ \text{%
reject }H_{0}\right] }{H_{0}:E\left\vert X\right\vert ^{8}=\infty }$} \\ 
St. Dev. &  & \multicolumn{1}{c}{$0.063$} &  & \multicolumn{1}{|l}{} &  &  & 
\multicolumn{1}{c}{} \\ 
Skewness &  & \multicolumn{1}{c}{$0.098$} &  & \multicolumn{1}{|l}{} &  &  & 
\multicolumn{1}{c}{$\underset{\left[ \text{reject }H_{0}\right] }{%
H_{0}:E\left\vert X\right\vert ^{16}=\infty }$} \\ 
Kurtosis &  & \multicolumn{1}{c}{$3.864$} &  & \multicolumn{1}{|l}{} &  &  & 
\multicolumn{1}{c}{} \\ 
&  &  &  & \multicolumn{1}{|l}{} &  &  & \multicolumn{1}{c}{$\underset{\left[
\text{reject }H_{0}\right] }{H_{0}:E\left\vert X\right\vert ^{32}=\infty }$}
\\ 
&  &  &  & \multicolumn{1}{|l}{} &  &  & \multicolumn{1}{c}{} \\ 
&  &  &  & \multicolumn{1}{|l}{} & Jarque-Bera &  & \multicolumn{1}{c}{$%
\underset{\left[ \text{p-value}=0.041\right] }{H_{0}:\text{Gaussian data}}$}
\\ 
&  &  &  &  &  &  &  \\ \hline\hline
\end{tabular}

} 
\par
{\tiny 
\begin{tablenotes}
      \tiny
            \item The table contains the outcomes for the test by \citet{degiannakis2023superkurtosis} described above for the null that the moments of order $8$, $16$ and $32$ are non-existent.
            
\end{tablenotes}
}
\end{table}

\clearpage

\section{Preliminary lemmas\label{lemmas}}

We begin by collecting a series of lemmas which will be used to prove the
main results under $H_{0}$ (Lemmas \ref{l:term-by-term_maximal_bounds}-\ref%
{l:trunctail_of_limit}), under $H_{A}$ (Lemmas \ref{l:replace_with_q1_q2}-%
\ref{e:q3_statiticterms_only_lemma}), and the ones reported in Section \ref%
{s:refinements} (Lemmas \ref{prwe2}, \ref{l:remainder_neglig_R_L}, and \ref{prwe4}). Throughout this section,
Assumptions \ref{a:historical_stability}, \ref{e:assumption_ind}, and \ref%
{a:assumption_on_h} are in force, and hence we omit them from statements.
Prior to reporting the lemmas, we spell out some notation and several facts
which will be used throughout this section and the next one.

In all proofs, $C>0$ denotes a generic, finite constant independent of $m$
whose value may change line-to-line. For any interval $I\subseteq [0,\infty )$, we write $\mathbf{C}(I)$ to denote the space of continuous
real-valued functions on $I$ with the uniform topology, and $\mathbf{D}(I)$
the space of c\`{a}dl\`{a}g functions endowed with the Skorokhod
topology, and $\mathbf{C}^{r}(I)$ and $\mathbf{D}^{r}(I)$ for their $\mathbb{%
R}^{r}$-valued counterparts, with $r\geq 2$. We use $\Rightarrow$ to denote
weak convergence. When convenient for any $a,b\in \mathbb{R}$ we write $%
a\vee b=\max \{a,b\}$ and $a\wedge b=\min \{a,b\}$. Throughout, $\mathcal{F}%
=(\mathcal{F}_k)_{k\geq 1}$ denotes the natural filtration generated by the
sequence $\{\mathbf{X}_k,~k\geq 1\}$, i.e., $\mathcal{F}_k=\sigma(\mathbf{X}%
_1,\ldots,\mathbf{X}_{k})$.

\medskip

We first note the following important fact: it can be readily checked that for any function $f:\mathcal{X}\times 
\mathcal{X}\rightarrow \mathbb{R}$ of the form 
\begin{equation}\label{e:f_0_sum_form}
f(\mathbf{x},\mathbf{y})=f_{0}(\mathbf{x})+f_{0}(\mathbf{y})+c,
\end{equation}%
with some function $f_{0}:\mathcal{X}\rightarrow \mathbb{R}$, then for $%
U_{m}(~\cdot ~;r,k)$ as in \eqref{e:def_page}, 
\begin{equation}
U_{m}(f;r,k)=0,\quad m,k\geq 2\,\quad 0\leq r<k-1.  \label{e:UT(f;k)=0}
\end{equation}%
Hence, with 
$$
\overline h(\bx,\by)= h(\bx,\by)-\mathbb E h(\bx,\bY)-\mathbb E h(\bX,\by)+\mathbb E h(\bX,\bY),
\qquad \bX,\bY\stackrel{i.i.d.}{\sim}F,
$$
we see $\overline h(\bx,\by)-h(\bx,\by)$ is of the form \eqref{e:f_0_sum_form}, and hence
$$
U_m(h;r,k)=U_m(\overline h;r,k).
$$
Thus, in the proofs we may work with the degenerate kernel $\overline h$. We also note that, under
Assumption \eqref{a:assumption_on_h}, we may write 
\begin{equation}
\overline h(\mathbf{x},\mathbf{y})=\sum_{\ell =1}^{\infty }\lambda _{\ell }\phi _{\ell
}(\mathbf{x})\phi _{\ell }(\mathbf{y}),  \label{e:h(xy)=L2_expansion}
\end{equation}%
where the equality holds in the $\mathcal{L}^{2}(F\times F)$ sense, and for $%
\mathbf{X}\sim F$, 
\begin{equation}
{\mathsf{E}}\hspace{0.1mm}\phi _{\ell }(\mathbf{X})\phi _{\ell ^{\prime }}(%
\mathbf{X})=%
\begin{cases}
1, & \text{if }\ell =\ell ^{\prime }, \\ 
0, & \text{if }\ell \neq \ell ^{\prime }.%
\end{cases}
\label{e:eigvector_orthog}
\end{equation}%
Moreover, since $\overline h$ is degenerate, ${\mathsf{E}}\hspace{0.1mm}\overline h(%
\mathbf{X},\mathbf{y})=0$ $F$-a.e., i.e., the operator $A$ has $\phi (%
\mathbf{x})\equiv 1$ as eigenvector (with corresponding eigenvalue $0$), so
by orthogonality, for all $\ell $ such that $\lambda _{\ell }\neq 0$, we
have 
\begin{equation}
{\mathsf{E}}\hspace{0.1mm}\phi _{\ell }(\mathbf{X})=0.
\label{e:eig_zero_mean}
\end{equation}%
Define, for each integer $m,k\geq 1$, %
\begin{equation}
\begin{gathered} S_{\ell}(m) = \sum_{i=1}^{m}\phi_\ell({\bf X}_i),\qquad
S_\ell(k,m)=\sum_{j=m+1}^{m+k}\phi_\ell({\bf X}_j).\end{gathered}
\label{e:def_Sell(m)}
\end{equation}%
Define 
\begin{equation}
f_{\ell }(\mathbf{x},\mathbf{y})=\phi _{\ell }(\mathbf{x})\phi _{\ell }(%
\mathbf{y}),  \label{e:f_ell}
\end{equation}%
so that
\begin{equation}
U_{m}(\overline h;r,k)=\sum_{\ell =1}^{\infty }\lambda _{\ell }U_{m}(f_{\ell };r,k),
\label{u_m_h,k}
\end{equation}%
and define truncated version 
\begin{equation}
U_{m,L}(\overline h;r,k)=\sum_{\ell =1}^{L}\lambda _{\ell }U_{m}(f_{\ell };r,k).
\label{u_m_h,k_truncated}
\end{equation}%
A straightforward calculation shows that, letting $w=k-r$, we have 
\begin{align}
& m^{-1}w^{2}U_{m}(f_{\ell };k-w,k)  \notag \\
& =-m^{-1}\left( S_{\ell }(k,m)-S_{\ell }(k-w,m)-\frac{w}{m}S_{\ell
}(m)\right) ^{2}+\frac{w(m+w)}{m^{2}}+R_{\ell }(k,w,m),  \label{e:Um(f,k)}
\end{align}%
where 
\begin{align}
& R_{\ell }(k,w,m)  \notag \\
& =-\frac{w^{2}S_{\ell }^{2}(m)}{m^{3}(m-1)}+\frac{w^{2}}{m^{2}(m-1)}%
\sum_{i=1}^{m}\left( \phi _{\ell }^{2}(\mathbf{X}_{i})-1\right) +\frac{w^{2}%
}{m^{2}(m-1)}  \notag \\
& -\frac{[S_{\ell }(k,m)-S_{\ell }(k-w,m)]^{2}}{m(w-1)}+\frac{w}{m(w-1)}%
\sum_{j=m+(k-w)+1}^{m+k}\left( \phi _{\ell }^{2}(\mathbf{X}_{j})-1\right) +%
\frac{w}{m(w-1)}.  \label{e:def_R_ell(k,m)}
\end{align}

Lastly, to simplify some expressions, for any kernel $f(%
\mathbf{x},\mathbf{y})$ we set 
\begin{equation*}
U_{m}(f;r,k)=U_{m}(f;r\wedge (k-2),k\vee 2),\quad \text{if}~k\leq 2~\text{or}%
~r\geq k-1.
\end{equation*}%

We now present a lemma which is used throughout the appendix to bound quantities related to the ``cross term'' in $U_m(h;r,k)$ under both $H_0$ and $H_A$.

\begin{lemma}\label{l:basic_cross_lemma}
Let $\bY_1,\ldots, \bY_N$ be i.i.d. observations with distribution $F_Y$ and independent of $\bX_1,\ldots,\bX_M$. Let $K(\bx,\by)$ be a kernel, not necessarily symmetric, with $\E K^2(\bX_1,\bY_1)<\infty$, and with $\E K(\bX_1,\by)= 0$ $F_Y(d\by)$-a.e. and $\E K(\bx,\bY_1)=0$ $F(d\bx)$ a.e.  Then, with
$$
D(u,v)=\sum_{i=1}^{u}\sum_{j=1}^{v}K(\mathbf X_i,\mathbf Y_j),\qquad
1\leq u\leq M,\quad 1\leq v\leq N .
$$
it holds that
\begin{equation}\label{e:rectangular_cross_bound}
\E\left[\max_{1\leq u\leq M} \max_{1\leq v\leq N}|D(u,v)|^2 \right]
\leq
C MN \E K^2(\mathbf X_1,\bY_1).
\end{equation}

\end{lemma}
\begin{proof}
Since $\E[ K(\bX_i,\bY_j)| \bY_j]=0$ a.s. by assumption, conditioning on $\mathbf Y_1,\ldots,\mathbf Y_N$, the vector
$(D(u,1),\ldots,D(u,N))^\top$, $1\leq u\leq M$, is a martingale in $u$. Hence, by
Doob's inequality applied to the submartingale $\{\max_{1\leq v\leq N}|D(u,v)|^2,u=1,2\ldots\}$, 
$$
\E\left[\max_{1\leq u\leq M} \max_{1\leq v\leq N}|D(u,v)|^2\bigg| \mathbf Y_1,\ldots,\mathbf Y_N\right]
\leq
4\E\left[\max_{1\leq v\leq N}|D(M,v)|^2\bigg|\mathbf Y_1,\ldots,\mathbf Y_N\right].
$$
Taking expectations and applying Doob's inequality again, now to the
martingale $D(M,v)$ in $v$, gives
\begin{equation}\label{e:MfixMaxBound}
\E\left[\max_{1\leq v\leq N}|D(M,v)|^2\right] \leq C \E|D(M,N)|^2.
\end{equation}
and hence
\begin{equation}\label{e:doubleMaxBound}
\E\left[\max_{1\leq u\leq M}\max_{1\leq v\leq N}|D(u,v)|^2 \right] \leq C \E|D(M,N)|^2.
\end{equation}
Finally, since $\E K(\bX,\by)\equiv 0$ and  by independence,  $$\E K(\bX_i,\bY_j)K(\bX_{i'},\bY_{j})=\E[\E K(\bX_i,\bY_j)K(\bX_{i'},\bY_{j})| \bY_j] =0,\qquad i\neq i',$$
 and similarly, since $\E K(\bx,\bY)\equiv 0$ gives $\E K(\bX_i,\bY_j)K(\bX_{i},\bY_{j'})=0$, we find
$$
\E|D(M,N)|^2
=
MN \E K^2(\mathbf X_1,\mathbf Y_1),
$$
which, combining with \eqref{e:doubleMaxBound} proves \eqref{e:rectangular_cross_bound}. %
\end{proof}

\bigskip

\subsection{Lemmas under $H_0$}
We are now in a position to present our lemmas under $H_{0}$. Lemmas \ref{l:term-by-term_maximal_bounds} and \ref{l:remainder_neglig_0}, below, are used to provide uniform control over the difference between the process $U_m(\overline h;r,k)$ and its finite-expansion counterpart $U_{m,L}(\overline h;r,k)$ defined in \eqref{u_m_h,k_truncated}.

\begin{lemma}\label{l:term-by-term_maximal_bounds}
Let $K$ be a symmetric kernel with
$\E K^2(\mathbf X,\mathbf X')<\infty$, where $\mathbf X'$ is an independent
copy of $\mathbf X$, and with $\E K(\bX,\bx)=0$ $F$-a.s.. Then, for all integers $N,M,m\geq1$,
\begin{align}
\E
\max_{1\leq n\leq N}
\left(\sum_{1\leq i<j\leq n}K(\mathbf X_i,\mathbf X_j) \right)^2
&\leq
C N^2 \E K^2(\mathbf X,\mathbf X'),
\label{e:local_max_prefix}\\
\E \max_{1\leq q\leq N}\left(\sum_{i=1}^{m}\sum_{j=m+1}^{m+q}K(\mathbf X_i,\mathbf X_j)\right)^2
&\leq C m N \E K^2(\mathbf X,\mathbf X'),
\label{e:local_max_cross_mFix}\\
\E \max_{1\leq n<q\leq N}\left(\sum_{i=1}^{n}\sum_{j=n+1}^{q}K(\mathbf X_i,\mathbf X_j)\right)^2
&\leq C N^2 \E K^2(\mathbf X,\mathbf X'),
\label{e:local_max_cross}\\
\E \max_{\substack{0\leq a<b\leq N\\ b-a\leq M}}
\left(\sum_{a<i<j\leq b}K(\mathbf X_i,\mathbf X_j)\right)^2 &\leq C NM \E K^2(\mathbf X,\mathbf X').
\label{e:local_max_interval}
\end{align}
\end{lemma}

\begin{proof}
We first bound \eqref{e:local_max_prefix}. For each integer $q\geq 2$, write 
\begin{equation*}
Y_{q}=2\sum_{i=1}^{q-1}K(\bX_i,\bX_q)
\end{equation*}%
so that 
\begin{equation*}
\sum_{1\leq i\neq j\leq n} K(\bX_i,\bX_j) =\sum_{q=1}^{n}Y_{q}.
\end{equation*}%
Since $\E K(\bX,\bX')=0$,  ${\mathsf{E}}\hspace{0.1mm}Y_{q}=0$, and clearly 
$Y_{q}$ is $\mathcal{F}_{q}$-measurable, with ${\mathsf{E}}\hspace{0.1mm}%
(Y_{q}|\mathcal{F}_{q-1})=0$, implying $\sum_{q=1}^{k}Y_{q}$ is an $(%
\mathcal{F}_{k})_{k\geq 1}$-martingale. Moreover, using $\E K(\bX_i,\bX_q)K(\bX_{i'},\bX_q)=\E[\E K(\bX_i,\bX_q)K(\bX_{i'},\bX_q)|\bX_q]]=0$ when $i\neq i'$ and $i\neq q$, we have
\begin{align*}
{\mathsf{E}}\hspace{0.1mm}Y_{q}^{2}& =4\sum_{i,i^{\prime
}=1}^{q-1}\E K(\bX_i,\bX_q)K(\bX_{i'},\bX_q)\\
& =4(q-1)\E K^2(\bX,\bX')
\end{align*}%
Hence, Doob's maximal inequality gives 
\begin{equation*}
{\mathsf{E}}\hspace{0.1mm}\max_{1\leq n\leq N}\left(
\sum_{q=1}^{n}Y_{q}\right) ^{2}\leq 4\sum_{q=1}^{N}{\mathsf{E}}\hspace{0.1mm}%
Y_{q}^{2}=16\sum_{q=1}^{N}(q-1)\E K^2(\bX,\bX')\leq CN^{2}\E K^2(\bX,\bX').
\end{equation*}%
The bound  \eqref{e:local_max_cross_mFix} follows from Lemma \ref{l:basic_cross_lemma}. 
We next prove \eqref{e:local_max_cross}. Set
$$
A_N=
\max_{1\leq n<q\leq N}
\left|
\sum_{i=1}^{n}\sum_{j=n+1}^{q}K(\mathbf X_i,\mathbf X_j)
\right|.
$$
 Suppose first that $N=2^{j}$ for some $j\geq 1$. Let
$$
A_N^{(1)}
=\max_{1\leq a<b\leq N/2}
\left|\sum_{i=1}^{a}\sum_{j=a+1}^{b}K(\mathbf X_i,\mathbf X_j)
\right|
$$
and
$$
A_N^{(2)}
=\max_{ N/2<a<b\leq N}\left|\sum_{i=N/2+1}^{a}\sum_{j=a+1}^{b}K(\mathbf X_i,\mathbf X_j)\right|.
$$
$$
B_N
=
\max_{1\leq u,v\leq N/2}\left|\sum_{i=1}^{u}\sum_{j=N/2+1}^{N/2+v}K(\mathbf X_i,\mathbf X_j)
\right|.
$$
Now, if $b\leq N/2$, clearly
$$
\left|
\sum_{i=1}^{a}\sum_{j=a+1}^{b}K(\mathbf X_i,\mathbf X_j)
\right|
\leq A_N^{(1)}.
$$
Likewise, if $a\geq N/2$, then
\begin{align*}
\left|\sum_{i=1}^{a}\sum_{j=a+1}^{b}K(\mathbf X_i,\mathbf X_j)\right|
&\leq\left|\sum_{i=1}^{N/2}\sum_{j=a+1}^{b}K(\mathbf X_i,\mathbf X_j)\right|
+\left|\sum_{i=N/2+1}^{a}\sum_{j=a+1}^{b}K(\mathbf X_i,\mathbf X_j)\right|\\
&\leq 2B_N+A_N^{(2)}.
\end{align*}
Finally, if $a<N/2<b$, then
\begin{align*}
\left|\sum_{i=1}^{a}\sum_{j=a+1}^{b}K(\mathbf X_i,\mathbf X_j)\right|
&\leq
\left|\sum_{i=1}^{a}\sum_{j=a+1}^{N/2}K(\mathbf X_i,\mathbf X_j)\right| +\left|\sum_{i=1}^{a}\sum_{j=N/2+1}^{b}K(\mathbf X_i,\mathbf X_j).\right|
\end{align*}
The first term in the above sum is bounded by $A_N^{(1)}$, and the second is bounded by
$B_N$. Therefore, we find
$$
A_N
\leq
\max\{A_N^{(1)}, A_N^{(2)}\}+2B_N .
$$
By \eqref{e:rectangular_cross_bound},
$$
\E B_N^2\leq C N^2 \E K^2(\bX,\bX').
$$
Hence,
\begin{align*}
(\E A_N^2)^{1/2}\leq
\left(\E(A_N^{(1)})^2+\E(A_N^{(2)})^2\right)^{1/2}
+C N\left( \E K^2(\bX,\bX').\right)^{1/2}.
\end{align*}
By stationarity, $A_N^{(1)}\stackrel d = A_N^{(2)}\stackrel d = A_{N/2}$.  Hence,
\begin{equation}\label{e:A_N_L2_bound}
(\E A_N^2)^{1/2}
\leq
\sqrt{2}\big(\E A_{N/2}^2\big)^{1/2}
+
C N\left( \E K^2(\bX,\bX')\right)^{1/2}.
\end{equation}
Recalling $N=2^j$, set
$$
D_j
=\frac{\big(\E A_{2^j}^2\big)^{1/2}}
{2^j\left( \E K^2(\bX,\bX') \right)^{1/2}},
$$
The inequality \eqref{e:A_N_L2_bound} can be rewritten as
$$
D_{j+1}\leq 2^{-1/2}D_j+C .
$$
Iterating this inequality yields
\begin{align*}
D_{j+1}
&\leq
2^{-(j+1)/2}D_0 +C\sum_{q=0}^{j}2^{-q/2}
\leq C',
\end{align*}
Thus
$$
(\E A_{2^j}^2)^{1/2} \leq C2^j\left( \E K^2(\bX,\bX')\right)^{1/2},
$$
or equivalently,
$
\E A_{N}^2\leq C N^2 \E K^2(\bX,\bX').
$
For  a non-dyadic $N$, take $j$ such that $N\leq 2^j<2N$ and use
$A_N\leq A_{2^j}$. This proves \eqref{e:local_max_cross}.

It remains to prove \eqref{e:local_max_interval}. Since $b-a\leq M$ and also $b-a\leq b \leq N$, we may assume $M\leq N$. Partition ${1,\ldots,N}$ into consecutive blocks
$$
B_s={p_s+1,\ldots,q_s},\qquad s=1,\ldots,J,
$$
where $p_s=(s-1)M$, $q_s=\min\{sM,N\}$. Then $J\leq C N/M$ and $|B_s|\leq M$. Now, write
\begin{align*}
W_s=
\max_{p_s\leq a<b\leq q_s}
\left|\sum_{a<i<j\leq b}K(\mathbf X_i,\mathbf X_j)\right| & \stackrel d = \max_{0\leq u<v\leq q_s-p_s}
\left| \sum_{u<i<j\leq v}K(\mathbf X_i,\mathbf X_j)\right|\\
&  \leq \max_{0\leq u<v\leq M}
\left|\sum_{u<i<j\leq v}K(\mathbf X_i,\mathbf X_j)
\right| 
\end{align*}
However,  
\begin{align*}
\left|\sum_{u<i<j\leq v}K(\mathbf X_i,\mathbf X_j)\right|\leq  \left|\sum_{1\leq i<j\leq v}K(\mathbf X_i,\mathbf X_j)\right|+
\left|\sum_{1\leq i<j\leq u}K(\mathbf X_i,\mathbf X_j)\right|+
\left|\sum_{i=1}^{u}\sum_{j=u+1}^{v}K(\mathbf X_i,\mathbf X_j)\right|.
\end{align*}
Therefore, by \eqref{e:local_max_prefix} and \eqref{e:local_max_cross},
$$
\E W_s^2
\leq\E  \max_{0\leq u<v\leq M} \left|\sum_{u<i<j\leq v}K(\mathbf X_i,\mathbf X_j)\right|^2 \leq C M^2 \E K^2(\mathbf X,\mathbf X').
$$
Arguing similarly,  with 
$$
V_s = \max_{p_s\leq a<q_s<b\leq q_{s+1}}
\left|
\sum_{i=a+1}^{q_s}\sum_{j=q_s+1}^{b}
K(\mathbf X_i,\mathbf X_j)
\right|,
$$
we find
$$
\E V_s^2 \leq C M^2 \E K^2(\bX,\bX').
$$
Now let ${a+1,\ldots,b}$ be any interval with $b-a\leq M$.  By construction, either $p_s\leq a<b\leq q_s$ for some $1\leq s \leq J$ or $p_s\leq a<q_s<b$ for some $1 \leq s \leq J-1$ , so (setting $V_J=0$),%
\begin{align*}
\left|\sum_{a<i<j\leq b}K(\mathbf X_i,\mathbf X_j)\right|^2&\leq  W^2_s{\bf 1}_{\{p_s\leq a<b\leq q_s\}} +(W_s+W_{s+1}+V_s)^2 {\bf 1}_{\{p_s+1\leq a<q_s<b\}}\\
&  \leq C\sum_{s=1}^J (W^2_s + V^2_s).
\end{align*}
Hence,
\begin{align*}
\E \max_{\substack{0\leq a<b\leq N\\ b-a\leq M}}
\left|\sum_{a<i<j\leq b}K(\mathbf X_i,\mathbf X_j)\right|^2 &\leq C\sum_{s=1}^J\E(W^2_s + V^2_s)\\ &\leq C J M^2\E K^2(\mathbf X,\mathbf X')\leq C NM\E K^2(\mathbf X,\mathbf X').
\end{align*}
This proves \eqref{e:local_max_interval}.
\end{proof}

\begin{lemma}
\label{l:remainder_neglig_0} Let $x>0$. Under $H_0$, for any integer $L\geq
0 $,%
\begin{equation}  \label{e:tailbound_UmL}
\P \left\{\sup_{0\leq r < k <\infty} \frac{(k-r)^2 m^{-1} |U_m(\overline h;r,k)-
U_{m,L}(\overline h;r,k)|}{g_m(k)}>x\right\}\leq C x^{-2}\sum_{\ell=L+1}^\infty
\lambda_\ell^2,
\end{equation}
where $U_{m,L}$ is defined as in \eqref{e:def_UL}. Moreover, 
\begin{equation}  \label{e:uniform_limit_in_delta}
\limsup_{m\to\infty}\P \left\{\max_{0\leq r <k\leq m\delta} \frac{|(k-r)^2
m^{-1}U_m(\overline h;r,k)|}{g_m(k)}>x\right\}=O(\delta^{1-\beta}), \quad \delta \to 0.
\end{equation}
\end{lemma}

\begin{proof}
Note, to begin with, that, with $K_L(\bx,\by)=\sum_{\ell=L+1}^\infty \lambda_\ell\phi_\ell(\bx)\phi_\ell(\by)$,
\begin{align*}
\left\vert U_{m}(\overline h;r,k)-U_{m,L}(\overline h;r,k)\right\vert &\leq \frac{2}{\left( k-r\right) m}\left\vert
\sum_{i=1}^{m}\sum_{j=m+r+1}^{m+k}K_L(\mathbf{X}_{i},\mathbf{X}_{j})\right\vert \\
&+\binom{m}{2}^{-1}\left\vert \sum_{1\leq i<j\leq m}K_L(\mathbf{X}_{i},\mathbf{X}_{j}) \right\vert\\
& +\binom{k-r}{2}^{-1}\left|\sum_{m+r<i<j\leq m+k}K_L(\bX_i,\bX_j)\right| .
\end{align*}%
Let now $0<\delta \leq 1$. Since $g_{m}(k)\geq C(k/m)^{\beta }$ for all $%
1\leq k\leq m\delta $, any integer $L\geq 0$ we have 
\begin{align}
& \P \left\{ \max_{0<r<k\leq m\delta }\frac{(k-r)^{2}}{mg_{m}(k)}\frac{1}{%
(k-r)m}\left\vert \sum_{i=1}^{m}\sum_{j=m+r+1}^{m+k}K_L(\mathbf{X}_{i},\mathbf{X}_{j})\right\vert >x\right\}  \notag \\
& \quad \leq \P \left\{ \max_{0\leq r<k\leq m\delta }\frac{k^{1-\beta }}{%
m^{2-\beta }}\left\vert \sum_{i=1}^{m}\sum_{j=m+r+1}^{m+k}K_L(\mathbf{X}_{i},\mathbf{X}_{j})\right\vert >Cx\right\}  \notag \\
& \quad \leq \P \left\{ \max_{1\leq q\leq \lceil \log_2 (m\delta )\rceil
}\max_{2^{q-1}\leq k<2^{q}}\max_{0\leq r<k}\frac{k^{1-\beta }}{m^{2-\beta }}%
\left\vert \sum_{i=1}^{m}\sum_{j=m+r+1}^{m+k}K_L(\mathbf{X}_{i},\mathbf{X}_{j})\right\vert >Cx\right\}  \notag \\
& \quad \leq \sum_{q=1}^{\lceil \log_2 (m\delta )\rceil }\P \left\{ \ \frac{%
2^{q(1-\beta )}}{m^{2-\beta }}\max_{2^{q-1}\leq k<2^{q}}\max_{0\leq
r<k}\left\vert \sum_{i=1}^{m}\sum_{j=m+r+1}^{m+k}K_L(\mathbf{X}_{i},\mathbf{X}_{j})\right\vert >Cx\right\}  \label{e:firstterm_bound_unif_limits_UmL0}
\end{align}%
Using the bound 
\begin{align}
& \max_{0\leq r<k}\left\vert \sum_{i=1}^{m}\sum_{j=m+r+1}^{m+k}K_L(\mathbf{X}_{i},\mathbf{X}_{j})\right\vert  \notag \\
& \quad \leq \left\vert \sum_{i=1}^{m}\sum_{j=m+1}^{m+k}K_L(\mathbf{X}_{i},\mathbf{X}_{j})\right\vert +\max_{1\leq r<k}\left\vert
\sum_{i=1}^{m}\sum_{j=m+1}^{m+r}K_L(\mathbf{X}_{i},\mathbf{X}_{j})\right\vert ,
\label{e:bound_avoiding_w}
\end{align}%
which holds for each fixed $k$, and that
$$
\E K^2_L(\mathbf{X}_{i},\mathbf{X}_{j}) = \sum_{\ell=L+1}^\infty \lambda_\ell^2,
$$ the expression \eqref{e:firstterm_bound_unif_limits_UmL0}
is bounded by 
\begin{align}
& \sum_{q=1}^{\lceil \log_2(m\delta )\rceil }\P \left\{ \ \frac{2^{q(1-\beta
)}}{m^{2-\beta }}\max_{2^{q-1}\leq k<2^{q}}\left\vert
\sum_{i=1}^{m}\sum_{j=m+1}^{m+k}K_L(\mathbf{X}_{i},\mathbf{X}_{j})\right\vert
>Cx/2\right\}  \notag \\
& \qquad +\sum_{q=1}^{\lceil \log_2(m\delta )\rceil }\P \left\{ \ \frac{%
2^{q(1-\beta )}}{m^{2-\beta }}\max_{1\leq r<2^{q}}\left\vert
\sum_{i=1}^{m}\sum_{j=m+1}^{m+r}K_L(\mathbf{X}_{i},\mathbf{X}_{j})\right\vert
>Cx/2\right\}  \notag \\
& \quad \leq \frac{C}{x^{2}m^{4-2\beta }}\sum_{q=1}^{\lceil \log_2 (m\delta
)\rceil }m2^{2q(1-\beta )}2^{q}\sum_{\ell =L+1}^{\infty }\lambda _{\ell
}^{2}\leq \frac{C\delta ^{3-2\beta }}{x^{2}}\sum_{\ell =L+1}^{\infty
}\lambda _{\ell }^{2}.  \label{e:firstterm_bound_unif_limits_UmL}
\end{align}%
For any $T\geq 1$, we have $g_{m}(k)\geq C(k/m)^{2}$ for all $k\geq Tm$, and
applying \eqref{e:bound_avoiding_w} again we obtain 
\begin{align}
& \P \left\{ \sup_{k\geq mT}\max_{0\leq r<k}\frac{k}{m^{2}g_{m}(k)}%
\left\vert \sum_{i=1}^{m}\sum_{j=m+r+1}^{m+k}K_L(\mathbf{X}_{i},\mathbf{X}_{j})\right\vert >x\right\}  \notag \\
& \quad \leq \P \left\{ \sup_{k\geq mT}\max_{0\leq r<k}\frac{1}{k}\left\vert
\sum_{i=1}^{m}\sum_{j=m+r+1}^{m+k}K_L(\mathbf{X}_{i},\mathbf{X}_{j})\right\vert
>Cx\right\}  \notag \\
& \quad \leq \sum_{q=\lfloor \log_2 (mT)\rfloor }^{\infty }\P \left\{ \
\max_{2^{q-1}\leq k<2^{q}}\frac{1}{2^{q-1}}\left\vert
\sum_{i=1}^{m}\sum_{j=m+1}^{m+k}K_L(\mathbf{X}_{i},\mathbf{X}_{j})\right\vert
>Cx/2\right\}  \notag \\
& \qquad +\sum_{q=\lfloor \log_2 (mT)\rfloor }^{\infty }\P \left\{ \
\max_{1\leq r<2^{q}}\frac{1}{2^{q-1}}\left\vert
\sum_{i=1}^{m}\sum_{j=m+1}^{m+r}K_L(\mathbf{X}_{i},\mathbf{X}_{j})\right\vert
>Cx/2\right\}  \notag \\
& \quad \leq \frac{C}{x^{2}}\sum_{\ell =L+1}^{\infty }\lambda _{\ell
}^{2}\sum_{q=\lfloor \log_2 (mT)\rfloor }^{\infty }m2^{-q}\leq \frac{C}{Tx^{2}}%
\sum_{\ell =L+1}^{\infty }\lambda _{\ell }^{2}.
\label{e:firstterm_bound_unif_limits_UmL_largeint}
\end{align}%
In particular, if we take $\delta =T=1$, we obtain 
\begin{equation*}
\P \left\{ \sup_{2\leq r< k<\infty }\frac{(k-r)^{2}}{mg_{m}(k)}\left\vert 
\frac{1}{(k-r)m}\sum_{i=1}^{m}\sum_{j=m+1}^{m+k}K_L(\mathbf{X}_{i},\mathbf{X}_{j})\right\vert >x\right\} \leq \frac{C}{x^{2}}\sum_{\ell =L+1}^{\infty
}\lambda _{\ell }^{2}.
\end{equation*}%
Analogous arguments leading to \eqref{e:firstterm_bound_unif_limits_UmL} and %
\eqref{e:firstterm_bound_unif_limits_UmL_largeint} give 
\begin{align*}
\limsup_{m\rightarrow \infty }\P \left\{ \max_{2\leq k\leq \delta m}\frac{%
k^{2}}{mg_{m}(k)}\left\vert \frac{1}{m^{2}}\sum_{1\leq i<j\leq m}K_L(\mathbf{X}_{i},\mathbf{X}_{j})\right\vert >x\right\} & =O(\delta ^{2-\beta }), \\
\limsup_{m\rightarrow \infty }\P \left\{ {\max_{2\leq k\leq \delta
m}\max_{0\leq r<k}}\frac{(k-r)^{2}}{mg_{m}(k)}\left\vert \frac{1}{(k-r)^{2}}%
\sum_{m+r<i<j\leq m+k}K_L(\mathbf{X}_{i},\mathbf{X}_{j})\right\vert >x\right\} &
=O(\delta ^{1-\beta }),
\end{align*}%
as $\delta \rightarrow 0$, yielding \eqref{e:uniform_limit_in_delta}.  Similarly, 
\begin{align}\notag
&\P \left\{ \sup_{k\geq mT}\max_{0\leq r\leq k-2}%
\frac{(k-r)^{2}}{mg_{m}(k)}\left\vert \frac{1}{(k-r)^{2}}\sum_{m+r<i<j\leq
m+k}K_L(\mathbf{X}_{i},\mathbf{X}_{j})\right\vert >x\right\}\\
& \leq\frac{C}{Tx^2}\sum_{\ell=L+1}^\infty \lambda_\ell^2,\label{e:tailbound_monitoredterm}
\end{align}
and
\begin{align*}
\P \left\{ \sup_{k\geq mT}\frac{k^{2}}{%
 mg_{m}(k)}\left\vert \frac{1}{m^{2}}\sum_{1\leq i<j\leq m}K_L(\mathbf{X}_{i},\mathbf{X}_{j})\right\vert >x\right\} & \leq \frac{C}{x^2}\sum_{\ell=L+1}^\infty \lambda_\ell^2,
\end{align*}
which put together yields \eqref{e:tailbound_UmL}.
\end{proof}
\begin{lemma}\label{l:tail_variable}
Let
\begin{equation}\label{e:am(rk)}
a_m(r,k)=\frac{((k-r)/m)^2}{g_m(k)}
\end{equation}
and
$$
\mathcal H_{m} =-\frac{2}{m-1} \sum _{1\leq i<j\leq m} \overline h(\bX_i,\bX_j).
$$
then, for every $x>0$
\begin{equation}  \label{e:uniform_limit_in_T}
\lim_{T\to\infty}\limsup_{m\to\infty}
\P\left\{ \sup_{k\geq mT}\max_{0\leq r\leq k-2}\left|\frac{m^{-1}(k-r)^2U_{m}(\overline h;r,k)}{g_m(k)}
- a_m(r,k)\mathcal H_{m} \right|>x\right\}=0.
\end{equation}
\end{lemma}
\begin{proof}
Let $w=k-r$. Define
\begin{equation}\label{e:def_hL}
\mathcal H_{m,L} =-\frac{2}{m-1} \sum _{1\leq i<j\leq m} \overline h_L(\bX_i,\bX_j),\qquad \overline h_L(\bx,\by) = \sum_{\ell=1}^L \lambda_\ell\phi_\ell(\bx)\phi_\ell(\by).
\end{equation}
We claim that, for every fixed $L\geq 1$,
\begin{equation}\label{e:uniform_limit_in_T_finiteL}
\lim_{T\to\infty}\limsup_{m\to\infty}
\P\left\{\sup_{k\geq mT}\max_{0\leq r\leq k-2}\left|\frac{m^{-1}w^2U_{m,L}(\overline h;r,k)}{g_m(k)}-a_m(r,k)\mathcal H_{m,L}\right|>x\right\}=0.
\end{equation}
Indeed, by the same bounds used in Lemma~\ref{l:remainder_neglig_0},
$$
\lim_{T\to\infty}\limsup_{m\to\infty}
\P\left\{\sup_{k\geq mT}\max_{0\leq r\leq k-2}
\frac{m^{-1}w^2}{g_m(k)}
\left|\frac{2}{wm}
\sum_{i=1}^m\sum_{j=m+r+1}^{m+k} \overline h_L(\bX_i,\bX_j)\right|>x
\right\}=0
$$
and
$$
\lim_{T\to\infty}\limsup_{m\to\infty}
\P\left\{\sup_{k\geq mT}\max_{0\leq r\leq k-2}
\frac{m^{-1}w^2}{g_m(k)}
\left|\binom{w}{2}^{-1}\sum_{m+r<i<j\leq m+k} \overline h_L(\bX_i,\bX_j)
\right|>x\right\}=0.
$$
Combining the above bounds proves
\eqref{e:uniform_limit_in_T_finiteL}. Now, by Lemma~\ref{l:remainder_neglig_0},
$$
\lim_{L\to\infty}\sup_{m\geq1}
\P\left\{
\sup_{k\geq2}\max_{0\leq r\leq k-2}
\left|
\frac{m^{-1}(k-r)^2\big(U_m(\overline h;r,k)-U_{m,L}(\overline h;r,k)\big)}{g_m(k)}
\right|>x
\right\}=0.
$$
Moreover,
since
$
\E|\mathcal H_m-\mathcal H_{m,L}|^2
\leq
C\sum_{\ell=L+1}^{\infty}\lambda_\ell^2,
$ we readily find
$$
\lim_{L\to\infty}\sup_{m\geq1}\P\left\{|\mathcal H_m-\mathcal H_{m,L}|>x\right\}=0,
$$
proving
\eqref{e:uniform_limit_in_T}.
\end{proof}

The next lemma shows that the $U_{m,L}(\ell,k)$ can be approximated by a weighted sum of squared CUSUM-type statistics, based on the eigenfunctions of $\overline h$.
\begin{lemma}
\label{l:remainder_neglig_1} Under $H_0$, for any fixed $0<\delta<T<\infty$ and $L\geq1$, 
\begin{equation*}
\max_{\delta m\leq k\leq mT} \max_{2\leq w\leq k}\sum_{\ell=1}^L\left|\frac{ \lambda_\ell
R_\ell(k,w,m)}{g_m(k)}\right|=o_\P (1),
\end{equation*}
with $R_\ell(k,w,m)$ as defined in \eqref{e:def_R_ell(k,m)}.
\end{lemma}

\begin{proof} 
Fix any $1\leq \ell\leq L$; we proceed to analyze each term in $
R_\ell(k,w,m) $ separately. Since $g_m(k)=g(k/m)$, it is easily seen that
\begin{equation}\label{e:g_lower_delta_T}
\inf_{\delta m\leq k\leq mT}g_m(k)\geq c_{\delta,T}>0,
\end{equation}
and hence
\begin{equation*}
\max_{\delta m\leq k\leq mT}\max_{2\leq w\leq k}\frac{1}{g_m(k)}\frac{w^2}{m^2(m-1)}
=O(m^{-1}),
\end{equation*}
and
\begin{equation*}
\max_{\delta m\leq k\leq mT}\max_{2\leq w\leq k}\frac{1}{g_m(k)}\frac{w}{m(w-1)}
=O(m^{-1}).
\end{equation*}
Similarly, from \eqref{e:eigvector_orthog}, we have
\begin{align*}
&\max_{\delta m\leq k\leq mT}\max_{2\leq w\leq k}
\frac{1}{g_m(k)}\frac{w^2}{m^2(m-1)}
\left|\sum_{i=1}^m\bigl(\phi_\ell^2(\mathbf X_i)-1\bigr)\right|\\
&\qquad\leq
\frac{C}{m-1}
\left|\sum_{i=1}^m\bigl(\phi_\ell^2(\mathbf X_i)-1\bigr)\right|
=o_\P(1),
\end{align*}
and from \eqref{e:eig_zero_mean},
\begin{align*}
&\max_{\delta m\leq k\leq mT}\max_{2\leq w\leq k}
\frac{1}{g_m(k)}
\frac{w^2|S_\ell(m)|^2}{m^3(m-1)}
\leq
\frac{C}{m(m-1)}|S_\ell(m)|^2
=O_\P(m^{-1}).
\end{align*}

In view of \eqref{e:g_lower_delta_T}, it remains to establish the
corresponding  bounds for the remaining terms in \eqref{e:def_R_ell(k,m)} without the weight $g_m(k)$, uniformly over the
larger set $2\leq w\leq k\leq Tm$. We have
\begin{align*}
&\P \left\{\max_{2\leq w\leq k \leq Tm} \left| \frac{w}{m(w-1)}%
\sum_{j=m+(k-w)+1}^{m+k} \left(\phi_\ell^2(\mathbf{X}_j)-1\right)\right|
>x\right\} \\
& \leq \P \left\{\max_{2\leq k \leq Tm} \left| \frac1m\sum_{j=m+1}^{m+k}
\left(\phi_\ell^2(\mathbf{X}_j)-1\right)\right| >x/4\right\} \\
&\qquad\qquad + \P \left\{\max_{2\leq w\leq k \leq Tm} \left| \frac{1}{m}%
\sum_{j=m+1}^{m+(k-w)} \left(\phi_\ell^2(\mathbf{X}_j)-1\right)\right|
>x/4\right\} \\
\ & \leq Cm^{-1}{\mathsf{E}}\hspace{0.1mm} \left|\sum_{j=1}^{\lceil
Tm\rceil} \left(\phi_\ell^2(\mathbf{X}_j)-1\right)\right| =o(1),
\end{align*}
where the last inequality follows from Doob's maximal inequality and the $o(1)$ statement follows from the LLN in $\mathcal L^1$. Finally, we
will show 
\begin{align}  \label{e:|Sk-Sw|^2=o(1)}
&\P \left\{\max_{2\leq w\leq k \leq Tm} \frac{\left|
S_{\ell}(k,m)-S_{\ell}(k-w,m)\right|^2}{m(w-1)} >x\right\} =o(1),
\end{align}
which will complete the statement. Fix $0<\eta<1/3$. Then 
\begin{align*}
& \P \left\{\max_{2\leq w\leq k \leq Tm} \frac{\left|
S_{\ell}(k,m)-S_{\ell}(k-w,m)\right|^2}{m(w-1)} {\mathbf{1}}%
_{\{w>m^\eta\}}>x\right\} \\
& \leq \P \left\{\max_{2\leq w\leq k \leq Tm} \frac{\left|
S_{\ell}(k,m)|^2+|S_{\ell}(k-w,m)\right|^2}{m^{1+\eta}} >x/4\right\} \\
& \leq \P \left\{\max_{2\leq k \leq Tm} \frac{\left| S_{\ell}(k,m)\right|^2}{%
m^{1+\eta}} >x/8\right\} \\
& \leq C \frac{{\mathsf{E}}\hspace{0.1mm} \left| S_{\ell}(\lfloor mT\rfloor ,m)\right|^2}{x
m^{1+\eta}} \\
& \leq C m^{-\eta}.
\end{align*}
Next, we have 
\begin{align}
& \P \left\{\max_{2\leq w\leq k \leq Tm} \frac{\left|
S_{\ell}(k,m)-S_{\ell}(k-w,m)\right|^2}{m(w-1)} {\mathbf{1}}_{\{w\leq
m^\eta\}}>x\right\}  \notag \\
& \leq \P \left\{\max_{2\leq w\leq k \leq m^\eta } \frac{\left|
S_{\ell}(k,m)-S_{\ell}(k-w,m)\right|^2}{m} >x\right\} + \P %
\left\{m^{-1}\max_{m^\eta<k \leq Tm} Y_{k,m} >x\right\},
\label{nextwehave}
\end{align}
with 
\begin{align*}
Y_{k,m}&= \max_{2\leq w \leq m^{\eta} } \left| \frac{1}{\sqrt {w-1}}%
\sum_{j=m+(k-w)+1}^{m+k}\phi_\ell(\mathbf{X}_j) \right |^2 \\
& =\max_{2\leq w \leq m^{\eta} }\frac1{w-1}\left|
S_{\ell}(k,m)-S_{\ell}(k-w,m)\right|^2.
\end{align*}
The first term in \eqref{nextwehave} is negligible.  Indeed, when $2\leq w\leq k \leq m^{\eta}$,
\begin{align*}
\max_{2\leq w\leq k \leq m^\eta }\left|
S_{\ell}(k,m)-S_{\ell}(k-w,m)\right|^2 \leq m^{2\eta}\max_{1\leq j\leq \lceil m^\eta \rceil }\phi^2_\ell(\bX_{m+j}).
\end{align*}
Hence,
$$
\max_{2\leq w\leq k \leq m^\eta } \frac{\left|
S_{\ell}(k,m)-S_{\ell}(k-w,m)\right|^2}{m} \leq m^{3\eta-1} \frac{\max_{1\leq j\leq \lceil m^\eta \rceil }\phi^2_\ell(\bX_{m+j})}{m^\eta} = o_\P(1).
$$
For the second term in \eqref{nextwehave},  set
\begin{equation}\label{e:mathcalMn}
\mathcal M_n =
\max_{1\leq a\leq b\leq n} \left|\sum_{j=a}^b\phi_\ell(\mathbf X_{m+j})\right|,\qquad A_n = \mathcal M_n^2/n.
\end{equation}
A standard truncation argument shows
\begin{equation}\label{e:An_ui}
\lim_{K\to\infty}\sup_{n\geq 1} \E  A_n {\mathbf 1}_{\{A_n>K\}} = 0.
\end{equation}
Also, for each $q\geq 1$, partition $\{1,\ldots, mT\}$ into blocks of length $2^q$:
$$\qquad B_{q,r}= \{1+(r-1)2^q,\ldots,r2^q\},\qquad r=1,\ldots, N_q,\qquad N_q \leq C m/2^q.
$$
Note when $2^q\leq w<2^{q+1}$, the interval $\{k-w+1,\ldots,k\}$ is always contained in a union of three adjacent blocks $B_{q}(r)=B_{q,r}\cup B_{q,r+1}\cup B_{q,r+2}$ for some $r$. 
Hence,
\begin{align}\notag
Y_{k,m}
\leq C\max_{1\leq q\leq\lceil\log_2(m^\eta)\rceil}\max_{1\leq r\leq N_q}
\frac1{2^q}\max_{a,b\in B_q(r)}\left|\sum_{j=a}^b\phi_\ell(\mathbf X_{m+j})
\right|^2.
\end{align}
Thus, with $n_q=3\cdot 2^q$,
\begin{align*}
\P %
\left\{m^{-1}\max_{m^\eta<k \leq Tm} Y_{k,m} >x\right\} & \leq \sum_{q=1}^{\lceil\log_2( m^\eta)\rceil }\sum_{r=1}^{N_q}\P %
\left\{\frac1{2^q}\max_{a,b \in B_q(r)}\left|\sum_{j=a}^{b}\phi_\ell(\mathbf{X}_j)\right|^2 >mx\right\}\\
& \leq \sum_{q=1}^{\lceil\log_2( m^\eta)\rceil}\sum_{r=1}^{N_q}\P %
\left\{\frac1{n_q} \mathcal M_{n_q}^2 >mx/3\right\}\\
& \leq \sum_{q=1}^{\lceil\log_2( m^\eta)\rceil}\frac{1}{mx}\sum_{r=1}^{N_q}\E A_{n_q} {\mathbf 1}_{\{A_{n_q} >mx/3\}}\\
& \leq \frac{C}{xm}\sup_{n\geq1}
\E\left(
A_n\mathbf 1_{\{A_n>mx/3\}}\right)\sum_{q=1}^{\lceil\log_2( m^\eta)\rceil}  N_q \\
&\leq \frac{o(1)}{x} \sum_{q=1}^\infty \frac{1}{2^q} = o(1),
\end{align*}
where the $o(1)$  follows from \eqref{e:An_ui}. 
This shows the second term in  \eqref{nextwehave} vanishes, completing the proof.
\end{proof}

\bigskip

We now report two approximation lemmas that are central to the main proofs.  The first shows the weighted truncated processes $U_{m,L}$ can be approximated by limits driven by a linear combination of squares of Gaussian processes.

\begin{lemma}
\label{l:approxunderH0} Fix $L\geq 1$, and set 
\begin{equation}
U_{m,L}(r,k)=\sum_{\ell =1}^{L}\lambda _{\ell }U_{m}(f_{\ell };r,k),
\label{e:def_UL}
\end{equation}%
where $f_{\ell }$ is given in \eqref{e:f_ell}. Let 
\begin{equation}
\mathbb{U}_{m,L}(s,t)=m^{-1}\big((\lfloor mt\rfloor -\lfloor ms\rfloor )\vee
2\big)^{2}U_{m,L}\big(\lfloor ms\rfloor ,\lfloor mt\rfloor \big),\quad 0\leq
s\leq t.  \label{uml}
\end{equation}%
Also, for every $s,t\geq 0$, set%
\begin{equation}
\mathbb{V}_{L}(s,t)=-\sum_{\ell =1}^{L}\lambda _{\ell }\left[ \left(
W_{2,\ell }(t)-W_{2,\ell }(s)-(t-s)W_{1,\ell }(1)\right) ^{2}-(t-s)(1+t-s)%
\right] ,  \label{vlst}
\end{equation}%
where $\{W_{1,1}(t),t\geq 0\},$ $\{W_{2,1}(t),t\geq 0\},$ $%
\{W_{1,2}(t),t\geq 0\},$ $\{W_{2,2}(t),t\geq 0\},$ $\ldots $ are independent
Wiener processes. Then, we may define a sequence $\left\{ \mathbb{V}%
_{m,L},m\geq 1\right\} $ of processes $\mathbb{V}_{m,L}=\{\mathbb{V}%
_{m,L}(s,t),~s,t\geq 0\}$ such that for each $m$, $\mathbb{V}_{m,L}\overset{{%
\mathcal{D}}}{=}\mathbb{V}_{L}$, and for any $0<\delta <T$, 
\begin{equation}
\sup_{s,t\in I_{\delta ,T}}\left\vert \frac{\mathbb{V}_{m,L}(s,t)}{g(t)}-%
\frac{\mathbb{U}_{m,L}(s,t)}{g_{m}(\lfloor mt\rfloor )}\right\vert =o_{\P %
}(1),  \label{e:VmL(st)}
\end{equation}%
with $$I_{\delta ,T}=\{(s,t):\delta \leq t\leq T,~0\leq s\leq t\}.$$
\end{lemma}

\begin{proof}
For $0\leq s \leq t$, write 
\begin{align}
&\mathbb{U}_{m,L}^{\circ}(s,t)  \notag \\
&=-\sum_{\ell=1}^L\lambda_\ell\Bigg(\frac1m\left (S_{\ell}(\lfloor
mt\rfloor,m) - S_\ell(\lfloor ms\rfloor,m) - \frac{\lfloor mt\rfloor-\lfloor
ms\rfloor}mS_{\ell}(m)\right)^2  \notag \\
& \qquad\qquad\qquad\qquad\qquad\qquad\qquad\qquad- \frac{(\lfloor
mt\rfloor-\lfloor ms\rfloor)(\lfloor mt\rfloor-\lfloor ms\rfloor+m)}{m^2}%
\Bigg).  \label{e:U_mL_circ}
\end{align}
The Dudley-Wichura-Skorokhod Theorem (see e.g. \citealp{shorack:wellner:1986}%
, p. 47) entails that, for each $m$, one can construct independent Wiener
processes $\{W_{1,1,m}(t),t\geq0\},\{W_{2,1,m}(t),t\geq0\},\ldots,
\{W_{1,L,m}(t),t\geq0\},\{W_{2,L,m}(t),t\geq0\}$ such that 
\begin{equation}\label{e:skoro_coupling}
|m^{-1/2} S_\ell(m) - W_{1,\ell,m}(1)|+ \sup_{0\leq t \leq T} | m^{-1/2}
S_\ell(\lfloor mt\rfloor,m) - W_{2,\ell,m}(t)| =o_P(1),
\end{equation}
for all $1 \leq \ell
\leq L$. 
Hence, for all $1\leq \ell \leq L$, 
\begin{align*}
&\sup_{0\leq s \leq t \leq T} \left| m^{-1/2} [S_\ell(\lfloor mt\rfloor,m)-
S_\ell(\lfloor ms\rfloor,m)] - [W_{2,\ell,m}(t)-W_{2,\ell,m}(s)]\right|=o_P(1).
\end{align*}
and 
\begin{align*}
&\sup_{0\leq s \leq t \leq T} \left|\frac{\lfloor mt\rfloor-\lfloor ms\rfloor%
}{m} m^{-1/2}S_{\ell}(m)- (t-s)W_{1,\ell,m}(1)\right|=o_P(1).
\end{align*}
Thus, if we define, for all $0\leq s \leq t$, 
\begin{align}
&\mathbb{V }_{m,L}(s,t)  \notag \\
&=-\sum_{\ell=1}^L \lambda_\ell\left[\left(W_{2,\ell,m}(t)-W_{2,%
\ell,m}(s)-(t-s)W_{1,\ell,m}(1)\right)^2 - (t-s)(1+t-s)\right]
\end{align}
then 
\begin{equation}  \label{e:VmL(st)0}
\sup_{0\leq s \leq t \leq T} |\mathbb{V }_{m,L}(s,t)-\mathbb{U}%
_{m,L}^\circ(s,t)|=o_\P (1).
\end{equation}
Since $\sup_{\delta\leq t\leq T}|g_m\left(\lfloor mt\rfloor\right)-g(t)|\to
0 $ and $\inf_{\delta\leq t\leq T}|g(t)|>0$, from \eqref{e:VmL(st)0} we
obtain 
\begin{equation}  \label{e:VmL(st)1}
\sup_{s,t\in I_{\delta,T}}\left| \frac{\mathbb{V}_{m,L}(s,t)}{g(t)}-\frac{%
\mathbb{U}^\circ_{m,L}(s,t)}{g_m(\lfloor mt \rfloor)}\right|=o_\P (1).
\end{equation}
Lastly, Lemma \ref{l:remainder_neglig_1} yields 
\begin{equation*}
\sup_{s,t\in I_{\delta,T}} \frac{|\mathbb{U}_{m,L}(s,t)-\mathbb{U}%
^\circ_{m,L}(s,t)|}{g_m(\lfloor mt \rfloor)}\leq \max_{\delta m < k \leq
mT}\max_{2\leq w \leq k}\left|\displaystyle\sum_{\ell=1}^L\frac{ \lambda_\ell R_\ell(k,w,m)}{%
g_m(k)}\right|=o_\P (1),
\end{equation*}
which combined with \eqref{e:VmL(st)1} gives \eqref{e:VmL(st)}. 
\end{proof}

The next lemma shows the weak limit of the (weighted) $U_{m,L}$ can be itself approximated when $L$ is large.
\begin{lemma}
\label{l:trunctail_of_limit} For each $r,s,t\geq 0$, let 
\begin{equation}
\mathbb{V}(s,t)=-\sum_{\ell =1}^{\infty }\lambda _{\ell }\left[ \left(
W_{2,\ell }(t)-W_{2,\ell }(s)-(t-s)W_{1,\ell }(1)\right) ^{2}-(t-s)(1+t-s)%
\right] ,  \label{vst}
\end{equation}%
where $\{W_{1,1}(t),t\geq 0\},\{W_{2,1}(t),t\geq 0\},\{W_{1,2}(t),t\geq
0\},\{W_{2,2}(t),t\geq 0\},\ldots $ are independent Wiener processes and the
sums in (\ref{vst}) are understood as limits in $\mathcal{L}^{2}(\P )$.
Also, set 
\begin{equation*}
{\mathcal{V}}(s,t)=\frac{{\mathbb{V}}(s\wedge t,t)}{g(t)},\quad {\mathcal{V}}%
_{L}(s,t)=\frac{{\mathbb{V}}_{L}(s\wedge t,t)}{g(t)}\qquad s\geq 0,~t>0,
\end{equation*}%
with $\mathbb{V}_{L}$ as in (\ref{vlst}), and set $%
{\mathcal{V}}\left( s,0\right) ={\mathcal{V}}_{L}(s,0)=0$ for all $s\geq 0$.
Then, $\{{\mathcal{V}}(s,t)~s,t\geq 0\}$ admits a continuous version and ${%
\mathcal{V}}_{L}\Rightarrow {\mathcal{V}}$ in $\mathbf{C}([0,T]\times
\lbrack 0,T ])$ as $L\rightarrow \infty $, for each $T>0$. Moreover, with
$$
\mathcal H
=\sum_{\ell=1}^{\infty}
\lambda_\ell\left(1-W_{1,\ell}^{2}(1)\right),
$$
we have
\begin{equation}\label{e:V_tail_limit}
\sup_{0\leq s\leq t}
\left|
\mathcal V(s,t)-\frac{(t-s)^2}{g(t)}\mathcal H
\right|
\to0
\qquad\text{a.s.},\qquad t\to\infty.
\end{equation}
\end{lemma}

\begin{proof}
Set ${\mathcal{V}}%
_{0}\equiv 0$. For any $t>0$, $0\leq s\leq t$, and $L\geq 0$, 
\begin{align*}
& {\mathsf{E}}\hspace{0.1mm}\left\vert {\mathcal{V}}_{L}(s,t)-{\mathcal{V}}%
(s,t)\right\vert ^{2} \\
& ={\mathsf{E}}\hspace{0.1mm}\left\vert \frac{1}{g(t)}\sum_{\ell ={L+1}%
}^{\infty }\lambda _{\ell }\left[ \left( W_{2,\ell }(t)-W_{2,\ell
}(s)-(t-s)W_{1,\ell }(1)\right) ^{2}-(t-s)(1+t-s)\right] \right\vert ^{2} \\
& =\frac{2}{g^{2}(t)}((t-s)(1+t-s))^{2}\sum_{\ell ={L+1}}^{\infty }\lambda
_{\ell }^{2}.
\end{align*}%
Now, with 
$$
\mathcal H_L
=\sum_{\ell=1}^{L}\lambda_\ell
\left(1-W_{1,\ell}^{2}(1)\right).
$$
we have
$$
\E\left|\mathcal H-\mathcal H_{L}\right|^2
\leq
C\sum_{\ell=L+1}^{\infty}\lambda_\ell^2.
$$
This implies for each $n\geq 1$ and any $s_{1},t_{1},\ldots ,s_{n},t_{n}\geq
0$, $\left( {\mathcal{V}}_{L}(s_{1},t_{1}),\ldots {\mathcal{V}}%
_{L}(s_{n},t_{n}),\mathcal H_L\right) \Rightarrow \left( {\mathcal{V}}(s_{1},t_{1}),%
\ldots {\mathcal{V}}(s_{n},t_{n}),\mathcal H\right) $ as $L\rightarrow \infty $. We now show continuity of $\mathcal V$ and the convergence
\begin{equation}\label{e:VL,HL,joint}
(\mathcal V_L,\mathcal H_L)\implies (\mathcal V,\mathcal H),
\end{equation}
in $\mathbf C([0,T]^2)\times \bR$ for each $T>0$. 
Write 
\begin{align*}
Y_{\ell }(s,t)& =t^{-\beta }\left( W_{2,\ell}(t)-W_{2,\ell}(s\wedge t)-(t-(s\wedge
t))W_{1,\ell}(1)\right) ^{2} \\
& =t^{-\beta }\left( Z_{\ell }(t)-Z_{\ell }(t\wedge s)\right) ^{2},
\end{align*}%
where 
\begin{equation*}
Z_{\ell }(t)=W_{2,\ell }(t)-tW_{1,\ell }(1).
\end{equation*}%
With $m(s,t)=t^{-\beta }(t-(t\wedge s))(1+t-(t\wedge s))$, we have 
\begin{align}
{\mathcal{V}}_{L}(s,t)& =-(1+t)^{\beta -2}\sum_{\ell =1}^{L}\lambda _{\ell }%
\left[ Y_{\ell }(s,t)-m(s,t)\right]  \notag \\
& =:-(1+t)^{\beta -2}\sum_{\ell =1}^{L}\lambda _{\ell }\widetilde{Y}_{\ell}(s,t).  \label{e:Ytilde}
\end{align}%
Further, note for $0\leq s_{i}\leq t_{i}\leq T$ $i=1,2$ and $r\geq 1$,
Rosenthal's inequality yields 
\begin{align}
{\mathsf{E}}\hspace{0.1mm}& |(1+t_{1})^{2-\beta
}\cV_{L}(s_{1},t_{1})-(1+t_{2})^{2-\beta }\mathcal V_{L}(s_{2},t_{2})|^{2r}  \notag \\
\leq & \,C_{r}\!\left[ \sum_{\ell =1}^{L}|\lambda _{\ell }|^{2r}{\mathsf{E}}%
\hspace{0.1mm}\left\vert \widetilde{Y}_{\ell }(s_{1},t_{1})-\widetilde{Y}%
_{\ell }(s_{2},t_{2})\right\vert ^{2r}+\left( \sum_{\ell =1}^{L}\lambda
_{\ell }^{2}{\mathsf{E}}\hspace{0.1mm}|\widetilde{Y}_{\ell }(s_{1},t_{1})-%
\widetilde{Y}_{\ell }(s_{2},t_{2})|^{2}\right) ^{r}\right] .
\label{e:rosenthal_for_ZL}
\end{align}%
Now, 
\begin{align}
|Y_{\ell }(s_{1},t_{1})-Y_{\ell }(s_{2},t_{2})|^{2r}& \leq C\Big(%
|t_{1}^{-\beta }Z_{\ell }^{2}(t_{1})-t_{2}^{-\beta }Z_{\ell
}^{2}(t_{2})|^{2r}+|t_{1}^{-\beta }Z_{\ell }^{2}(s_{1})-t_{2}^{-\beta
}Z_{\ell }^{2}(s_{2})|^{2r}  \notag \\
& \qquad \qquad +|t_{1}^{-\beta }Z_{\ell }(t_{1})Z_{\ell
}(s_{1})-t_{2}^{-\beta }Z_{\ell }(t_{2})Z_{\ell }(s_{2})|^{2r}\Big).
\label{e:Y_preholderbound}
\end{align}%
We proceed to bound the expectation of each term in %
\eqref{e:Y_preholderbound}. Suppose for the moment that for any $T>0$, $%
0\leq s_{i}\leq t_{i}\leq T$, $i=1,2$, 
\begin{equation}
{\mathsf{E}}\hspace{0.1mm}\left( t_{1}^{-\beta /2}Z_{1}(s_{1})-t_{2}^{-\beta
/2}Z_{1}(s_{2})\right) ^{2}\leq C_{T}(|t_{1}-t_{2}|+|s_{1}-s_{2}|)^{a},
\label{eq:suppose}
\end{equation}%
for some $0<a<1-\beta $. Then, for any $r>0$, Gaussianity of $Z_{\ell }$
gives%
\begin{equation*}
{\mathsf{E}}\hspace{0.1mm}\left\vert t_{1}^{-\beta /2}Z_{\ell
}(s_{1})-t_{2}^{-\beta /2}Z_{\ell }(s_{2})\right\vert ^{2r}\leq
C_{r,T}(|t_{1}-t_{2}|+|s_{1}-s_{2}|)^{ar}.
\end{equation*}%
from which we obtain 
\begin{align*}
& {\mathsf{E}}\hspace{0.1mm}|t_{1}^{-\beta }Z_{\ell
}^{2}(s_{1})-t_{2}^{-\beta }Z_{\ell }^{2}(s_{2})|^{2r} \\
& \leq C\left( {\mathsf{E}}\hspace{0.1mm}|t_{1}^{-\beta /2}Z_{\ell
}(s_{1})-t_{2}^{-\beta /2}Z_{\ell }(s_{2})|^{4r}\right) ^{1/2}\left( {%
\mathsf{E}}\hspace{0.1mm}|t_{1}^{-\beta /2}Z_{\ell }(s_{1})|^{4r}+{\mathsf{E}%
}\hspace{0.1mm}|t_{2}^{-\beta /2}Z_{\ell }(s_{2})|^{4r}\right) ^{1/2} \\
& \leq C(|t_{1}-t_{2}|+|s_{1}-s_{2}|)^{2ar}.
\end{align*}%
Similarly, 
\begin{align*}
& {\mathsf{E}}\hspace{0.1mm}|t_{1}^{-\beta }Z_{\ell }(s_{1})Z_{\ell
}(t_{1})-t_{2}^{-\beta }Z_{\ell }(t_{2})Z_{\ell }(s_{2})|^{2r} \\
& \leq C\Big(\left( {\mathsf{E}}\hspace{0.1mm}|t_{1}^{-\beta /2}Z_{\ell
}(s_{1})-t_{2}^{-\beta /2}Z_{\ell }(s_{2})|^{4r}\right) ^{1/2}\left( {%
\mathsf{E}}\hspace{0.1mm}|t_{1}^{-\beta /2}Z_{\ell }(t_{1})|^{4r}\right)
^{1/2} \\
& \qquad +\left( {\mathsf{E}}\hspace{0.1mm}|t_{1}^{-\beta /2}Z_{\ell
}(t_{1})-t_{2}^{-\beta /2}Z_{\ell }(t_{2})|^{4r}\right) ^{1/2}\left( {%
\mathsf{E}}\hspace{0.1mm}|t_{2}^{-\beta /2}Z_{\ell }(s_{2})|^{4r}\right)
^{1/2}\Big) \\
& \leq C(|t_{1}-t_{2}|+|s_{1}-s_{2}|)^{2ar}.
\end{align*}%
Moreover, since $m(s,t)=m(s\wedge t,t)$ it is easily seen $m(s,t)$ is
locally $a$-H\"{o}lder continuous for any $0<a<1-\beta $. Hence, with $%
\widetilde{Y}_{\ell }$ as in \eqref{e:Ytilde}, 
\begin{equation*}
{\mathsf{E}}\hspace{0.1mm}|\widetilde{Y}_{\ell }(s_{1},t_{1})-\widetilde{Y}%
_{\ell }(s_{2},t_{2})|^{2r}\leq C(|t_{1}-t_{2}|+|s_{1}-s_{2}|)^{2ar}.
\end{equation*}%
From \eqref{e:rosenthal_for_ZL}, since $\sum_{\ell \geq 1}\lambda _{\ell
}^{2}<\infty $ we deduce, %
\begin{equation*}
{\mathsf{E}}\hspace{0.1mm}|(1+t_{1})^{2-\beta
}\mathcal V_{L}(s_{1},t_{1})-(1+t_{2})^{2-\beta }\mathcal V_{L}(s_{2},t_{2})|^{2r}\leq
C(|t_{1}-t_{2}|+|s_{1}-s_{2}|)^{2ar},
\end{equation*}%
Taking $r$ sufficiently large and applying Corollary 14.9 in %
\citet{kallenberg:2002} yields a continuous version of $\{(1+t)^{2-\beta }{%
\mathcal{V}}(s,t),~s,t\geq 0\}$ and tightness of the sequence $%
\{(1+t)^{2-\beta }\mathcal V_{L}(s,t),~s,t\geq 0\}$ in $\mathbf{C}([0,T]^2)$ for each $T>0$, and since $(1+t)^{2-\beta}$is continuous on $[0,T]$, we have tightness of $\mathcal V_L$ 
in $\mathbf C([0,T]^2)$; from which tightness of $(\mathcal V_L,\mathcal H_L)$ in $\mathbf C([0,T]^2)\times \bR$ follows;  hence we deduce \eqref{e:VL,HL,joint}.

Using the same Wiener processes $\{W_{1,\ell,m},W_{2,\ell,m},~\ell\geq 1,m\geq1\}$ in  \eqref{e:skoro_coupling} as in the proof of Lemma~\ref{l:approxunderH0}, for every fixed $L$ and $0<\delta<T$,
\begin{equation}\label{e:joint_fixedL_converg}
\sup_{(s,t)\in I_{\delta,T}}
\left|
\frac{\mathbb U_{m,L}(s,t)}{g_m(\lfloor mt\rfloor)}
-\frac{\mathbb V_{m,L}(s,t)}{g(t)}
\right|+\left|\mathcal H_{m,L}-\sum_{\ell=1}^L\lambda_\ell
\left(1-W_{1,\ell,m}^2(1)\right)\right|=o_\P(1),
\end{equation}
with $I_{\delta,T}$ as in \eqref{e:VmL(st)}.  With $\mathbb U_m$ defined the same as $\mathbb U_{m,L}$ in \eqref{uml}, but with $U_m(r,k)$ replacing $U_{m,L}(r,k)$,  Lemma~\ref{l:remainder_neglig_0} together with Chebyshev's inequality give, 
\begin{align*}
\lim_{L\to\infty}\limsup_{m\to\infty}\P\Bigg(
&\sup_{(s,t)\in I_{\delta,T}}
\left|
\frac{\mathbb U_m(s,t)-\mathbb U_{m,L}(s,t)}
{g_m(\lfloor mt\rfloor)}
\right|+
|\mathcal H_m-\mathcal H_{m,L}|>x
\Bigg)=0,
\end{align*}
for every $x>0$, where we used
$
\sup_m\E|\mathcal H_m-\mathcal H_{m,L}|^2
\leq C\sum_{\ell>L}\lambda_\ell^2.$
Combining this with \eqref{e:joint_fixedL_converg} and 
\eqref{e:VL,HL,joint}, we deduce, for every 
$1<T<R<\infty$, with $a_m(r,k)$ as in \eqref{e:am(rk)},
\begin{align*}
&\sup_{mT\leq k\leq mR}\max_{0\leq r\leq k-2}
\left|
\frac{m^{-1}(k-r)^2U_m(\overline h;r,k)}{g_m(k)}
-a_m(r,k)\mathcal H_m
\right|\\
&\qquad\Rightarrow
\sup_{T\leq t\leq R}\sup_{0\leq s\leq t}
\left|
\mathcal V(s,t)
-\frac{(t-s)^2}{g(t)}\mathcal H
\right|.
\end{align*}
Then, for every $x>0$,
\begin{align*}
&\P\left(
\sup_{T\leq t\leq R}\sup_{0\leq s\leq t}
\left|\mathcal V(s,t)-\frac{(t-s)^2}{g(t)}\mathcal H\right|>x
\right)\\
&\qquad\leq
\liminf_{m\to\infty}\P\left(
\sup_{mT\leq k\leq mR}\max_{0\leq r\leq k-2}
\left|\frac{m^{-1}(k-r)^2U_m(\overline h;r,k)}{g_m(k)}
-a_m(r,k)\mathcal H_m\right|>x\right)\\
&\qquad\leq
\liminf_{m\to\infty}
\P\left(
\sup_{k\geq mT}\max_{0\leq r\leq k-2}
\left|
\frac{m^{-1}(k-r)^2U_m(\overline h;r,k)}{g_m(k)}
-a_m(r,k)\mathcal H_m
\right|>x
\right).
\end{align*}
Letting $R\to\infty$, we obtain
\begin{align*}
&\P\left(
\sup_{t\geq T}\sup_{0\leq s\leq t}
\left|\mathcal V(s,t)-\frac{(t-s)^2}{g(t)}\mathcal H\right|>x
\right)\\
&\qquad\leq \liminf_{m\to\infty}
\P\left(\sup_{k\geq mT}\max_{0\leq r\leq k-2}\left|
\frac{m^{-1}(k-r)^2U_m(\overline h;r,k)}{g_m(k)}
-a_m(r,k)\mathcal H_m
\right|>x\right).
\end{align*}
Lemma~\ref{l:tail_variable} now implies that the above tends  to
zero as $T\to\infty$. Since $\sup_{t\geq T}(\cdot)$  is decreasing in $T$,
it follows that
$$
\sup_{t\geq T}\sup_{0\leq s\leq t}
\left|
\mathcal V(s,t)-\frac{(t-s)^2}{g(t)}\mathcal H
\right|
\to 0
\qquad\text{a.s.},\quad T\to\infty.
$$
We now conclude the proof by showing \eqref{eq:suppose}. 
Note for any $0<s_i\leq t_i$, $i=1,2$, 
\begin{align*}
&{\mathsf{E}}\hspace{0.1mm}\left(t_1^{-\beta/2}Z_\ell(s_1)-t_2^{-\beta/2}Z_%
\ell(s_2)\right)^2 \\
&\leq C\left({\mathsf{E}}\hspace{0.1mm}\left(t_1^{-\beta/2}W_{2,%
\ell}(s_1)-t_2^{-\beta/2}W_{2,\ell}(s_2)\right)^2 +
\left(t_1^{-\beta/2}s_1-t_2^{-\beta/2}s_2\right)^2\right).
\end{align*}
Without loss of generality suppose $s_1\geq s_2$. We have 
\begin{align*}
{\mathsf{E}}\hspace{0.1mm}\left(t_1^{-\beta/2}W_{2,\ell}(s_1)-t_2^{-%
\beta/2}W_{2,\ell}(s_2)\right)^2&=
t_1^{-\beta}s_1+t_2^{-\beta}s_2-2(t_1t_2)^{-\beta/2} s_2 \\
& = y^{-\beta/a}x^{1/a} +t_2^{-\beta}s_2-2y^{-\beta/(2a)}t_2^{-\beta/2}s_2 \\
&=f(x,y),
\end{align*}
where $x=s_1^a$, $y=t_1^a$. Note $x\leq y$. Since $0<a<1-\beta$, the mean
value theorem applied to $f(x,y)$ at $x_0=s_2^{a},y_0=t_2^a$ gives an $%
x_*,y_*$ with $y_*\geq x_*\geq s_2^a$ and 
\begin{align}
|f(x,y)|
&\leq C\Big(
y_*^{-\beta/a}x_*^{1/a-1}|x-x_0|  + \left(y_*^{-\beta/a-1}x_*^{1/a} +y_*^{-\beta/(2a)-1}t_2^{-\beta/2}s_2\right)|y-y_0|\Big) \notag\\
&\leq C \Big( y_*^{(1-\beta-a)/a}|x-x_0|+\left( y_*^{(1-\beta-a)/a} +s_2^{1-\beta-a} \right)|y-y_0|\Big) \notag\\
&\leq C\big(|x-x_0|+|y-y_0|\big) \leq C\bigl(|s_1-s_2|^a+|t_1-t_2|^a\bigr).
\label{e:holdercontbound}
\end{align}
Similarly if $s_2=0$, we have ${\mathsf{E}}\hspace{0.1mm}|t_1^{-%
\beta/2}W_{2,\ell}(s_1)-t_2^{-\beta/2}W_{2,\ell}(s_2)|^2=t_1^{-\beta}s_1\leq
s_1^{1-\beta}\leq s_1^a$, and thus \eqref{e:holdercontbound} holds for all $%
0\leq s_i\leq t_i \leq T$. Analogous arguments for \eqref{e:holdercontbound}
show $\left(t_1^{-\beta/2}s_1-t_2^{-\beta/2}s_2\right)^2\leq C(|t_1-t_2|^a +
|s_1-s_2|^a)$, which gives \eqref{eq:suppose}.

\end{proof}

The following lemmas collect additional estimates needed for the expanding-baseline detector $\mathcal D_m^{(3)}$. With $c_0>0$ we recall
$$
c_m=\lfloor {c}_0 m\rfloor,\qquad
b_k=(k-c_m)_+,\qquad
n_k=m+b_k.
$$
For $b_k\leq r\leq k-2$, write, for notational convenience,
$$
U_m^{(3)}(h;r,k)=U_{n_k}(h;r-b_k,k-b_k).
$$
We also recall
$$
g_m^{(3)}(k)
=
g\left(\frac{k}{n_k}\right)\left(1+\frac{b_k}{m}\right)^\gamma,
\qquad \gamma>1/2.
$$

\begin{lemma}\label{l:D3_finite_horizon_bounds}
Let $0<T<\infty$ and $x>0$. Under $H_0$, the following hold.

\begin{enumerate}
\item[(i)] For each fixed $L\geq1$ and $0<\delta<T$,
\begin{equation}\label{e:D3_remainder_bound}
\sup_{\delta m\leq k\leq mT}
\max_{b_k\leq r\leq k-2}\sum_{\ell=1}^{L}
\left|\frac{\lambda_\ell R_{\ell,m}^{(3)}(r,k)}{g_m^{(3)}(k)}\right|=o_P(1),
\end{equation}
where, with $w=k-r$,
$$
R_{\ell,m}^{(3)}(r,k)=\frac{n_k}{m}R_\ell(k-b_k,w,n_k),
$$
and $R_\ell$ is defined in \eqref{e:def_R_ell(k,m)}.

\item[(ii)] For each fixed $0<\delta<T$,
\begin{align}
&\lim_{L\to\infty}\limsup_{m\to\infty}
P\bigg\{\sup_{\delta m\leq k\leq mT}\max_{b_k\leq r\leq k-2}\left|\frac{m^{-1}(k-r)^2\{U_m^{(3)}(\overline h;r,k)-U_{m,L}^{(3)}(\overline h;r,k)\}}{g_m^{(3)}(k)}\right|>x\bigg\}=0.
\label{e:D3_eigentail_finite_horizon}
\end{align}

\item[(iii)] Finally,
\begin{equation}\label{e:D3_small_t_bound}
\lim_{\delta\to0}\limsup_{m\to\infty}
P\left\{\max_{2\leq k\leq m\delta}\max_{b_k\leq r\leq k-2}
\left|\frac{m^{-1}(k-r)^2U_m^{(3)}(\overline h;r,k)}{g_m^{(3)}(k)}\right|>x\right\}=0.
\end{equation}
\end{enumerate}
\end{lemma}

\begin{proof}
Write $w=k-r$. The identity \eqref{e:Um(f,k)} applied with historical sample size $n_k$, gives
\begin{align}
& m^{-1}w^2
U_{n_k}\bigl(f_\ell;r-b_k,k-b_k\bigr)  \notag\\
&\quad =
-m^{-1}
\left[
S_\ell(k,m)-S_\ell(r,m)
-\frac{w}{n_k}S_\ell(n_k)
\right]^2
+
\frac{w(n_k+w)}{m n_k}
+
R_{\ell,m}^{(3)}(r,k),
\label{e:D3_fell_expansion}
\end{align}
where
$$
R_{\ell,m}^{(3)}(r,k)
=
\frac{n_k}{m}R_\ell(k-b_k,w,n_k).
$$
On $\delta m\leq k\leq mT$,
$$
m\leq n_k\leq m(1+T),
\qquad
2\leq w\leq k-b_k\leq k\wedge c_m\leq C_Tm,
$$
and
$$
0<c_{\delta,T}
\leq
\inf_{\delta m\leq k\leq mT}g_m^{(3)}(k)
\leq
\sup_{\delta m\leq k\leq mT}g_m^{(3)}(k)
\leq C_{\delta,T}<\infty.
$$
Because $m\leq n_k\leq (1+T)m$, $2\leq w\leq k-b_k\leq C_Tm$ and
$g_m^{(3)}(k)$ is bounded away from zero uniformly over
$\delta m\leq k\leq mT$, it suffices to show, for each fixed $\ell$,
$$
\max_{m\leq n\leq (1+T)m}
\max_{2\leq w\leq q\leq C_Tm}
\left|R_\ell(q,w,n)\right|
=o_{\mathsf P}(1).$$
This follows by repeating the term-by-term argument in the proof of
Lemma~\ref{l:remainder_neglig_1}. Since
$n_k/m=O(1)$, this proves \eqref{e:D3_remainder_bound}.

We next prove \eqref{e:D3_eigentail_finite_horizon}. Write
$
K_L(\mathbf x,\mathbf y)
=\sum_{\ell=L+1}^{\infty} \lambda_\ell\phi_\ell(\mathbf x)\phi_\ell(\mathbf y).$
Then with $\overline h_L$ as in \eqref{e:def_hL}, we have 
$
\overline h-\overline h_L=K_L$
in $\mathcal L^2(F\times F)$. By the definition of $U_m^{(3)}$,
\begin{align}
&\left|
U_m^{(3)}(\overline h;r,k)-U_{m,L}^{(3)}(\overline h;r,k)
\right| \notag\\
&\quad\leq
\frac{2}{n_kw}
\left|\sum_{i=1}^{n_k}\sum_{j=m+r+1}^{m+k} K_L(\mathbf X_i,\mathbf X_j)
\right| +
\binom{n_k}{2}^{-1}
\left|\sum_{1\leq i<j\leq n_k} K_L(\mathbf X_i,\mathbf X_j)\right|
\notag\\
&\qquad
+
\binom{w}{2}^{-1}
\left|\sum_{m+r<i<j\leq m+k}K_L(\mathbf X_i,\mathbf X_j)
\right|,
\label{e:D3_tail_three_terms}
\end{align}
where $w=k-r$. It suffices to control the three terms in
\eqref{e:D3_tail_three_terms} after multiplication by $m^{-1}w^2$.  Note the following consequences of Lemma~\ref{l:term-by-term_maximal_bounds}: for
each fixed $T<\infty$,
\begin{align}
&\E\left[\sup_{\delta m\leq k\leq mT} \max_{b_k\leq r\leq k-2} \
\left| \sum_{i=1}^{n_k}\sum_{j=m+r+1}^{m+k} K_L(\mathbf X_i,\mathbf X_j) \right|^2\right] 
\leq C_Tm^2\sum_{\ell=L+1}^{\infty}\lambda_\ell^2,
\label{e:D3_cross_aux}
\end{align}
(using that $\sum_{i=1}^{n_k}\sum_{j=m+r+1}^{m+k}(\cdots) = \sum_{i=1}^{n_k}\sum_{j=n_k+1}^{m+k}(\cdots) -\sum_{i=1}^{n_k}\sum_{j=n_k+1}^{m+r}(\cdots)$)
and
\begin{align}
&\E\left[
\max_{1\leq n\leq m(1+T)}
\left|\sum_{1\leq i<j\leq n} K_L(\mathbf X_i,\mathbf X_j)\right|^2\right]
\leq
C_Tm^2\sum_{\ell=L+1}^{\infty}\lambda_\ell^2,
\label{e:D3_hist_aux}
\end{align}
while
\begin{align}
&\E\left[
\max_{0\leq r<k\leq mT}
\left|\sum_{m+r<i<j\leq m+k}K_L(\mathbf X_i,\mathbf X_j)\right|^2\right] \leq C_Tm^2\sum_{\ell=L+1}^{\infty}\lambda_\ell^2.
\label{e:D3_right_aux}
\end{align}
Now consider the first term in \eqref{e:D3_tail_three_terms}. Since
$n_k\geq m$ and $w\leq C_Tm$,
$$
\frac{m^{-1}w^2}{n_kw}
\leq
\frac{C_T}{m}.
$$
Hence, by \eqref{e:D3_cross_aux} and Markov's inequality,
\begin{align}
&\limsup_{m\to\infty}
P\bigg\{
\sup_{\delta m\leq k\leq mT}
\max_{b_k\leq r\leq k-2}
\frac{m^{-1}w^2}{n_kw}
\left|\sum_{i=1}^{n_k}\sum_{j=m+r+1}^{m+k}K_L(\mathbf X_i,\mathbf X_j)
\right|>x\bigg\}\notag\\
&\qquad\leq
C x^{-2}\sum_{\ell=L+1}^{\infty}\lambda_\ell^2.
\label{e:D3_cross_tail_bound}
\end{align}

For the second term in  \eqref{e:D3_tail_three_terms}, since $n_k\geq m$ and $w\leq C_Tm$, note
$$
m^{-1}w^2\binom{n_k}{2}^{-1}\leq\frac{C_T}{m}.
$$
Thus \eqref{e:D3_hist_aux} gives
\begin{align}
&\limsup_{m\to\infty}
P\bigg\{
\sup_{\delta m\leq k\leq mT}
\max_{b_k\leq r\leq k-2}
m^{-1}w^2\binom{n_k}{2}^{-1}
\left|
\sum_{1\leq i<j\leq n_k}
K_L(\mathbf X_i,\mathbf X_j)
\right|>x
\bigg\}
\notag\\
&\qquad\leq
C x^{-2}
\sum_{\ell=L+1}^{\infty}\lambda_\ell^2.
\label{e:D3_hist_tail_bound}
\end{align}
For the last term in  \eqref{e:D3_tail_three_terms}, note  $m^{-1}w^2\binom{w}{2}^{-1} \leq C/m.$
Therefore \eqref{e:D3_right_aux} yields
\begin{align}
&\limsup_{m\to\infty}
P\bigg\{
\sup_{\delta m\leq k\leq mT}
\max_{b_k\leq r\leq k-2}
m^{-1}w^2\binom{w}{2}^{-1}
\left|
\sum_{m+r<i<j\leq m+k}
K_L(\mathbf X_i,\mathbf X_j)
\right|>x
\bigg\}
\notag\\
&\qquad\leq
C x^{-2}
\sum_{\ell=L+1}^{\infty}\lambda_\ell^2.
\label{e:D3_right_tail_bound}
\end{align}
Combining \eqref{e:D3_cross_tail_bound}, \eqref{e:D3_hist_tail_bound}, and
\eqref{e:D3_right_tail_bound} with \eqref{e:D3_tail_three_terms} gives
\begin{align}
&\limsup_{m\to\infty}
P\bigg\{
\sup_{\delta m\leq k\leq mT}
\max_{b_k\leq r\leq k-2}
\left|
\frac{m^{-1}w^2
\{U_m^{(3)}(\overline h;r,k)-U_{m,L}^{(3)}(\overline h;r,k)\}}
{g_m^{(3)}(k)}
\right|>x
\bigg\}
\notag\\
&\qquad\leq
C x^{-2}
\sum_{\ell=L+1}^{\infty}\lambda_\ell^2.
\label{e:D3_eigentail_bound_prelim}
\end{align}
Letting $L\to\infty$ proves \eqref{e:D3_eigentail_finite_horizon}. For (iii), take any $\delta<c_0$. Then for all $2\leq k\leq m\delta$,
we have $b_k=0$, $n_k=m$, and hence
$$
U_m^{(3)}(\overline h;r,k)=U_m(\overline h;r,k),
\qquad
g_m^{(3)}(k)=g_m(k).
$$
Thus \eqref{e:D3_small_t_bound} follows directly from
\eqref{e:uniform_limit_in_delta}.  
\end{proof}

\begin{lemma}\label{l:D3_approxunderH0}
Fix $L\geq1$, $0<\delta<T<\infty$, and define
$$
b(t)=(t-c_0)_+.
$$
For $b(t)\leq s\leq t$, set
\begin{align}
\mathbb V_L^{(3)}(s,t)&=
-\sum_{\ell=1}^{L}\lambda_\ell
\Bigg[
\left\{
W_{2,\ell}(t)-W_{2,\ell}(s)
-\frac{t-s}{1+b(t)}
\left(W_{1,\ell}(1)+W_{2,\ell}(b(t))\right)
\right\}^2
\notag\\
&\qquad\qquad\qquad\qquad
-\left\{(t-s)+\frac{(t-s)^2}{1+b(t)}\right\}
\Bigg],
\label{e:def_VL_D3}
\end{align}
where the Wiener processes are as in Lemma~\ref{l:approxunderH0}. Then we may
define a sequence $\{\mathbb V_{m,L}^{(3)},m\geq1\}$ such that
$\mathbb V_{m,L}^{(3)}\overset{\mathcal D}{=}\mathbb V_L^{(3)}$ and
\begin{align}\notag
&\sup_{\substack{\delta\leq t\leq T\\ b(t)\leq s\leq t}}
\Bigg|
\frac{\mathbb V_{m,L}^{(3)}(s,t)} {g\big(t/(1+b(t))\big)\big(1+b(t)\big)^\gamma}\\
&\qquad\qquad\quad-
\frac{ (\lfloor mt\rfloor-\lfloor ms\rfloor)^2
U_{n_{\lfloor mt\rfloor},L}
\bigl(\overline h;\lfloor ms\rfloor-b_{\lfloor mt\rfloor},
\lfloor mt\rfloor-b_{\lfloor mt\rfloor}\bigr)} {m g_m^{(3)}(\lfloor mt\rfloor)} \Bigg|
=o_P(1).
\label{e:D3_fixedL_approx}
\end{align}
\end{lemma}

\begin{proof}
The proof is nearly the same as the proof of Lemma~\ref{l:approxunderH0}, using
\eqref{e:D3_fell_expansion} in place of \eqref{e:Um(f,k)}; we provide details where there are differences. Note  uniformly
over $\delta\leq t\leq T$ and $b(t)\leq s\leq t$,
$$
\frac{n_{\lfloor mt\rfloor}}{m}\to 1+b(t),
\qquad
\frac{\lfloor mt\rfloor-\lfloor ms\rfloor}{n_{\lfloor mt\rfloor}}
\to
\frac{t-s}{1+b(t)},
$$
and
$$
m^{-1/2}S_\ell(n_{\lfloor mt\rfloor})
=
m^{-1/2}S_\ell(m)
+
m^{-1/2}S_\ell(b_{\lfloor mt\rfloor},m)
\Rightarrow
W_{1,\ell}(1)+W_{2,\ell}(b(t)).
$$
Also,
$$
m^{-1/2}\{S_\ell(\lfloor mt\rfloor,m)-S_\ell(\lfloor ms\rfloor,m)\}
\Rightarrow
W_{2,\ell}(t)-W_{2,\ell}(s).
$$
Thus the leading term in \eqref{e:D3_fell_expansion} is approximated by
\eqref{e:def_VL_D3}. The centering term converges uniformly to
$$
(t-s)+\frac{(t-s)^2}{1+b(t)}.
$$
Finally, Lemma~\ref{l:D3_finite_horizon_bounds}(i) controls the remainder,
and
$$
g_m^{(3)}(\lfloor mt\rfloor)
\to
g\big(t/[1+b(t)]\big)[1+b(t)]^\gamma
$$
uniformly over $\delta\leq t\leq T$. This proves \eqref{e:D3_fixedL_approx}.
\end{proof}

\begin{lemma}\label{l:D3_infinite_horizon_bound}
Under $H_0$, for every $x>0$,
\begin{equation}\label{e:D3_large_T_bound}
\lim_{T\to\infty}\limsup_{m\to\infty}
P\left\{
\sup_{k\geq mT}
\max_{b_k\leq r\leq k-2}
\left|
\frac{m^{-1}(k-r)^2U_m^{(3)}(\overline h;r,k)}
{g_m^{(3)}(k)}
\right|>x
\right\}=0.
\end{equation}
\end{lemma}

\begin{proof}
Put
$$
a_k=\frac{n_k}{m}=1+\frac{b_k}{m},
\qquad
w=k-r.
$$
For
$q=0,1,2,\ldots$, set
$$
R_q=2^qT/2,\qquad 
\mathcal J_{m,q}(T)
=
\left\{
(r,k):
R_q\leq a_k<R_{q+1},\quad b_k\leq r\leq k-2
\right\}.
$$
The intervals $[R_q,R_{q+1})$, $q\geq0$ partition 
$[T/2,\infty)$. Since $c_m/m\to c_0$, for $T$  large enough and $m$
sufficiently large, $k\geq mT$ implies $k>c_m$ and
$$
a_k
=
1+\frac{k-c_m}{m}
\geq
1+T-\frac{c_m}{m}
\geq
\frac{T}{2}.
$$
Thus every pair $(r,k)$ with $k\geq mT$ and $b_k\leq r\leq k-2$ belongs to
exactly one of the sets $\mathcal J_{m,q}(T)$.  On
$\mathcal J_{m,q}(T)$,
$$
n_k\asymp R_qm,
\qquad
2\leq w\leq k-b_k\leq c_m\leq Cm,\qquad a_k^\gamma\geq R_q^\gamma.
$$
Moreover, since $k/n_k\to 1$ uniformly over $a_k\geq T/2$ as
$T\to\infty$, there is a constant $c>0$ such that, for all large $T$,
$$
\inf_{q\geq0}\inf_{(r,k)\in\mathcal J_{m,q}(T)}
g\left(\frac{k}{n_k}\right)\geq c.
$$

We claim that, for every $q\geq0$,
\begin{equation}\label{e:D3_shell_bound}
\limsup_{m\to\infty}P\left\{\max_{(r,k)\in\mathcal J_{m,q}(T)}
\left| \frac{m^{-1}w^2U_m^{(3)}(\overline h;r,k)}
{g_m^{(3)}(k)}\right|>x\right\}
\leq
\frac{C}{x^2}R_q^{1-2\gamma}.
\end{equation}
Once this is proved, the result follows immediately, because
\begin{align*}
&\limsup_{m\to\infty}
P\left\{
\sup_{k\geq mT}
\max_{b_k\leq r\leq k-2}
\left|
\frac{m^{-1}w^2U_m^{(3)}(\overline h;r,k)}
{g_m^{(3)}(k)}
\right|>x
\right\} \leq
\frac{C}{x^2}\sum_{q=0}^{\infty}R_q^{1-2\gamma}\leq \frac{C}{x^2}T^{1-2\gamma}.
\end{align*}
Since $\gamma>1/2$, the last expression tends to zero as $T\to\infty$. It remains to prove \eqref{e:D3_shell_bound}. Write
$$
U_m^{(3)}(\overline h;r,k) = A_{m,1}(r,k)-A_{m,2}(k)-A_{m,3}(r,k),
$$
where
\begin{align*}
A_{m,1}(r,k)
&=\frac{2}{n_kw} \sum_{i=1}^{n_k}\sum_{j=m+r+1}^{m+k} \overline h(\mathbf X_i,\mathbf X_j),\\
A_{m,2}(k)
&=\binom{n_k}{2}^{-1}\sum_{1\leq i<j\leq n_k} \overline h(\mathbf X_i,\mathbf X_j),\\
A_{m,3}(r,k)
&=\binom{w}{2}^{-1}\sum_{m+r<i<j\leq m+k}\overline h(\mathbf X_i,\mathbf X_j).
\end{align*}
Using Lemma \ref{l:term-by-term_maximal_bounds}, 
$$
\E\left[ \max_{1\leq n<s\leq C R_qm} \left| \sum_{i=1}^{n}\sum_{j=n+1}^{s}
\overline h(\mathbf X_i,\mathbf X_j)\right|^2 \right] \leq C R_q^2m^2.
$$
Since
$ {m^{-1}w^2 }/{n_kw} \leq C/R_qm,$ we find
\begin{equation}\label{e:D3_shell_cross_moment}
\E\left[
\sup_{(r,k)\in\mathcal J_{m,q}(T)}
|m^{-1}w^2A_{m,1}(r,k)|^2
\right]
\leq C.
\end{equation}
Similarly,
\begin{equation}\label{e:D3_shell_hist_moment}
\E\left[
\sup_{(r,k)\in\mathcal J_{m,q}(T)}
|m^{-1}w^2A_{m,2}(k)|^2
\right]
\leq C.
\end{equation}
For $A_{m,3}(k,r)$, using again Lemma \ref{l:term-by-term_maximal_bounds}, we  have
$$
\E\left[ \sup_{(r,k)\in\mathcal J_{m,q}(T)}
\left| \sum_{m+r<i<j\leq m+k} \overline h(\mathbf X_i,\mathbf X_j) \right|^2 \right]
\leq
C R_qm^2.
$$
Hence
\begin{equation}\label{e:D3_shell_right_moment}
\E\left[ \sup_{(r,k)\in\mathcal J_{m,q}(T)} |m^{-1}w^2A_{m,3}(r,k)|^2\right]
\leq C R_q.
\end{equation}
Combining \eqref{e:D3_shell_cross_moment},
\eqref{e:D3_shell_hist_moment}, and \eqref{e:D3_shell_right_moment}, and
using the lower bound on $g(k/n_k)$, gives
\begin{align*}
&\limsup_{m\to\infty}
P\left\{ \sup_{(r,k)\in\mathcal J_{m,q}(T)}
\left| \frac{m^{-1}w^2U_m^{(3)}(\overline h;r,k)}
{g_m^{(3)}(k)} \right|>x
\right\} 
\leq
\frac{C}{x^2}R_q^{1-2\gamma}.
\end{align*}
This proves \eqref{e:D3_shell_bound}, and hence
\eqref{e:D3_large_T_bound}.
\end{proof}

\bigskip
\subsection{Lemmas under $H_A$}
\noindent The next few lemmas are used under $H_A$. We first set up some notation. Let%
\begin{equation*}
\mu _{1}=\iint h(\mathbf{x},\mathbf{y})dF(\mathbf{x})dF(\mathbf{y}),\quad
\mu _{2}=\iint h(\mathbf{x},\mathbf{y})dF_{* }(\mathbf{x})dF_{* }(%
\mathbf{y}),
\end{equation*}%
\begin{equation*}
\mu _{12}=\iint h(\mathbf{x},\mathbf{y})dF(\mathbf{x})dF_{* }(\mathbf{y}).
\end{equation*}%
\begin{equation*}
h_{1}(\mathbf{x})=\int h(\mathbf{x},\mathbf{y})dF(\mathbf{y}),\quad h_{2}(%
\mathbf{x})=\int h(\mathbf{x},\mathbf{y})dF_{* }(\mathbf{y})
\end{equation*}%
Also, with $\nu _{1},\nu _{2}$ as in \eqref{e:def_nu}, we note 
\begin{equation*}
\nu _{1}=\theta ^{-1}(\mu _{1}-\mu _{12}),\quad \nu _{2}=\theta ^{-1}(\mu
_{12}-\mu _{2}),\quad \nu _{1}-\nu _{2}=\theta ^{-1}\left( \mu _{1}+\mu
_{2}-2\mu _{12}\right) .
\end{equation*}%
Whenever convenient we write $\mathbf{X}_{i}^{* }$ in place of $\mathbf{X}%
_{i}$ for $i>m+k_{* }$. We also set 
\begin{equation}
z_{i}=v(\mathbf{X}_{i})-\nu _{1},\qquad z_{i}^{* }=v(\mathbf{X}_{i}^{*
})-\nu _{2}.  \label{e:def_zi}
\end{equation}

Below, we set any sum $\sum_{j=a}^b(\ldots)=0$ whenever $b<a$. We proceed to
decompose the summations appearing in \eqref{e:def_page} for $k\geq k_*+1$
into drift, degenerate, and nondegenerate terms. For any $k\geq k_*+1$, $%
0\leq r < k_*$, write 
\begin{align*}
&\sum_{i=1}^m\sum_{j=m+r+1}^{m+k}h(\mathbf{X}_i,\mathbf{X}_j)-m (k_*-r)\mu_1
- m(k-k_*)\mu_{12} \\
& \quad=R_{m,1}(r,k)+(k_*-r)\sum_{i=1}^m\left[h_1(\mathbf{X}_i)-\mu_1\right]%
+m\sum_{i=m+r+1}^{m+k_*}\left[h_1(\mathbf{X}_i)-\mu_1\right] \\
&\qquad\qquad +(k-k_*)\sum_{i=1}^m\left[h_2(\mathbf{X}_i) -\mu_{12}\right] +
m\sum_{j=m+k_*+1}^{m+k}\left[h_1(\mathbf{X}_j^*)-\mu_{12}\right] \\
& \quad=R_{m,1}(r,k)+(k-r)\sum_{i=1}^m\left[h_1(\mathbf{X}_i)-\mu_1\right]%
+m\sum_{i=m+r+1}^{m+k_*}\left[h_1(\mathbf{X}_i)-\mu_1\right] \\
&\qquad\qquad -\theta (k-k_*)\sum_{i=1}^mz_i + m\sum_{j=m+k_*+1}^{m+k}\left[%
h_1(\mathbf{X}_j^*)-\mu_{12}\right], \\
&= R_{m,1}(r,k)+T_{m,1}(r,k),
\end{align*}
with 
\begin{align}
R_{m,1}(r,k) &= \sum_{i=1}^m\sum_{j=m+r+1}^{m+k_*}\left[h(\mathbf{X}_i,%
\mathbf{X}_j)-h_1(\mathbf{X}_i)-h_1(\mathbf{X}_j)+\mu_1\right]  \notag \\
& \quad +\sum_{i=1}^m\sum_{j=m+k_* +1}^{m+k}\left[h(\mathbf{X}_i,\mathbf{X}%
_j^*)-h_2(\mathbf{X}_i)-h_1(\mathbf{X}^*_j)+\mu_{12}\right]  \notag \\
&=R_{m,1,1}(r) + R_{m,1,2}(k_*,k).  \label{e:R_m11}
\end{align}
When $k_*<r< k$, 
\begin{align*}
&\sum_{i=1}^m\sum_{j=m+r+1}^{m+k}h(\mathbf{X}_i,\mathbf{X}_j) -
m(k-r)\mu_{12} \\
& \quad=R_{m,1}(r,k)+(k-r)\sum_{i=1}^m\left(\left[h_1(\mathbf{X}_i)
-\mu_{1}\right]-\theta z_i \right)+ m\sum_{j=m+r+1}^{m+k}\left[h_1(\mathbf{X}%
_j^*)-\mu_{12}\right] \\
& \quad= R_{m,1}(r,k) + T_{m,1}(r,k)
\end{align*}
with 
\begin{align}
R_{m,1}(r,k) & =\sum_{i=1}^m\sum_{j=m+r+1}^{m+k}\left[h(\mathbf{X}_i,\mathbf{%
X}_j^*)-h_2(\mathbf{X}_i)-h_1(\mathbf{X}^*_j)+\mu_{12}\right]  \notag \\
&= R_{m,1,2}(r,k),  \label{e:R_m12}
\end{align}
Similarly,%
\begin{align*}
\sum_{1\leq i<j \leq m}h(\mathbf{X}_i,\mathbf{X}_j) - {\binom{m}{2}}\mu_1 &
= R_{m,2} + (m-1) \sum_{i=1}^m \left[h_1(\mathbf{X}_i) -\mu_{1}\right], \\
& = R_{m,2} + T_{m,2},
\end{align*}
with 
\begin{equation}
R_{m,2}= \sum_{1\leq i<j \leq m}\left[h(\mathbf{X}_i,\mathbf{X}_j) - h_1(%
\mathbf{X}_i)- h_1(\mathbf{X}_j) + \mu_1\right].  \label{e:R_m2}
\end{equation}
For the third summation in \eqref{e:def_page}, when $0\leq r \leq k_*$, 
\begin{align*}
&\sum_{m+r< i< j \leq m+k}h(\mathbf{X}_i,\mathbf{X}_j)- \left[{\binom{k_*-r}{%
2}}\mu_1 + {\binom{k-k_*}{2}}\mu_{2} + (k_*-r)(k-k_*)\mu_{12}\right] \\
& \quad= \sum_{m+r< i<j \leq m+ k_*}\left[h(\mathbf{X}_i,\mathbf{X}_j) - \mu_1%
\right]+ \sum_{m+k_* < i<j \leq m+k}\left[h(\mathbf{X}^*_i,\mathbf{X}^*_j) -
\mu_2\right] \\
& \quad \qquad\qquad + \sum_{i= m+r+1}^{m+k_*}\sum_{j=m+k_*+1}^{m+k}\left[h(%
\mathbf{X}_i,\mathbf{X}^*_j) - \mu_{12}\right] \\
& ~~= R_{m,3}(r,k)+(k-r-1) \sum_{i=m+r+1}^{m+ k_*}\left[h_1(\mathbf{X}%
_i)-\mu_1\right]-\theta (k-k_*-1)\sum_{j=m+k_*+1}^{m+k}z_j^* \\
& \quad \qquad\qquad \qquad - \theta(k-k_*)\sum_{i= m+r+1}^{m+k_*}z_i +
(k-r-1) \sum_{j=m+k_*+1}^{m+k}\left[ h_1(\mathbf{X}^*_j)-\mu_{12}\right] \\
&= R_{m,3}(r,k) + T_{m,3}(r,k),
\end{align*}
with 
\begin{align}
R_{m,3}(r,k) &= \sum_{m+r < i<j \leq m+ k_*}\left[h(\mathbf{X}_i,\mathbf{X}%
_j) - h_1(\mathbf{X}_i)-h_1(\mathbf{X}_j)+ \mu_1\right]  \notag \\
& \quad +\sum_{m+k_* < i<j \leq m+k}\left[h(\mathbf{X}^*_i,\mathbf{X}^*_j) -
h_2(\mathbf{X}^*_i)-h_2(\mathbf{X}^*_j)+\mu_2\right]  \notag \\
& \quad + \sum_{i= m+r+1}^{m+k_*}\sum_{j=m+k_*+1}^{m+k}\left[h(\mathbf{X}_i,%
\mathbf{X}^*_j) - h_2(\mathbf{X}_i)-h_1(\mathbf{X}_j^*)+ \mu_{12}\right] 
\notag \\
&=:R_{m,3,1}(r)+ R_{m,3,2}(k_*,k) + R_{m,3,3}(r,k),  \label{e:R_m31}
\end{align}
and when $k_*< r <k$, 
\begin{align*}
\sum_{m+r< i< j \leq m+k}h(\mathbf{X}_i,\mathbf{X}_j)- {\binom{k-r}{2}}%
\mu_{2}&= R_{m,3}(r,k)+(k-r-1)\sum_{j=m+r+1}^{m+k}\left[h_2(\mathbf{X}^*_j) -
\mu_2\right] \\
& = R_{m,3}(r,k) + T_{m,3}(r,k),
\end{align*}
with 
\begin{align}
R_{m,3}(r,k) &= \sum_{m+r< i<j \leq m+k}\left[h(\mathbf{X}^*_i,\mathbf{X}%
^*_j) - h_2(\mathbf{X}^*_i)-h_2(\mathbf{X}^*_j)+\mu_2\right]  \notag \\
&= R_{m,3,2}(r,k).  \notag
\end{align}

This gives, for $k\geq k_*+1$, 
\begin{align}  \label{e:U_as_q}
(k-r)^2U_m(\overline h;r,k) & = q_{1}(r,k) + q_{2}(r,k) + q_{3}(r,k),
\end{align}
with 
\begin{align}
q_1(r,k) &= p_1(r,k)\mu_1 + p_{12}(r,k)\mu_{12} + p_2(r,k)\mu_2,
\label{e:def_q1_orig}
\end{align}
where 
\begin{align*}
p_1(r,k)&=%
\begin{cases}
\left(2(k-r)(k_*-r)-(k-r)^2 -\frac{(k-r)(k_*-r)(k_*-r-1)}{k-r-1}\right) & 
0\leq r \leq k_* \\ 
-(k-r)^2 & k_*<r < k%
\end{cases}
\\
p_{12}(r,k)&=%
\begin{cases}
2\left((k-r)(k-k_*) -\frac{ (k-r) (k_*-r)(k-k_*)}{k-r-1}\right) & 0\leq r
\leq k_* \\ 
2(k-r)^2 & k_*<r <k%
\end{cases}
\\
p_{2}(r,k)&=%
\begin{cases}
\frac{(k-r)(k-k_*)(k-k_*-1)}{k-r-1} & 0\leq r \leq k_* \\ 
-(k-r)^2 & k_*<r <k%
\end{cases}%
,
\end{align*}
and after some cancellation, %
\begin{align}
q_{2}(r,k)&= (k-r)^2\left[\frac{2T_{m,1}(r,k)}{(k-r)m} - \frac{T_{m,2}}{{%
\binom{m}{2}}} - \frac{T_{m,3}(r,k)}{{\binom{k-r }{2 }}} \right]  \notag \\
&=%
\begin{cases}
\label{e:def_q2_orig} \begin{aligned}[b] 2\theta(k-k_*)
&\bigg[-\frac{k-r}m\sum_{i=1}^m z_i+\frac{k-r}{k-r-1}\sum_{i=m+r+1}^{m+k_*}
z_i\\ & \qquad\qquad\quad+
\frac{k-r}{k-r-1}\left(\frac{k-k_*-1}{k-k_*}\right)%
\sum_{i=m+k_*+1}^{m+k}z_i^*\bigg]\end{aligned}, & 0\leq r\leq k_* \\ 
\displaystyle 2\theta(k-r)\left[-\frac{k-r}{m}\sum_{i=1}^m z_i +
\sum_{i=m+r+1}^{m+k} z_i^* \right] & k_*< r < k.%
\end{cases}%
\ 
\end{align}
Lastly, 
\begin{equation}  \label{e:def_q3_orig}
q_{3}(r,k) =(k-r)^2\left[\frac{2R_{m,1}(r,k)}{(k-r)m} - \frac{R_{m,2}}{{%
\binom{m}{2}}} - \frac{R_{m,3}(r,k)}{{\binom{k-r }{2 }}} \right].
\end{equation}

The next lemma provides an approximation of the drift term $q_1$ and nondegenerate term $q_2$ by asymptotically equivalent but simpler terms.
\begin{lemma}
\label{l:replace_with_q1_q2} Let $y_m> k_*$ be any sequence with $%
y_m\to\infty$, and for $0\leq r < k$, set

\begin{equation}  \label{e:cq1_cq2}
\begin{gathered} \mathcal q_1(r,k) = -(k-(k_*\vee r))^2\theta(\nu_1-\nu_2),
\\ \mathcal q_2(r,k) = 2\theta (k-(k_*\vee r)) \left[-\frac{k-
r}m\sum_{i=1}^m z_i + {\bf 1}_{\{r<k_*\}}\sum_{i=m+r+1}^{m+k_*} z_i +
\sum_{i=m+(k_*\vee r)+1}^{m+k}z_i^* \right]. \end{gathered}
\end{equation}

Then, for $q_1(r,k)$ and $q_2(r,k)$ as in \eqref{e:U_as_q},

\begin{align}
&\max_{k_*<k\leq y_m}\max_{0\leq r < k} \frac{|q_1(r,k)-\mathcal{q}_1(r,k)|}{%
mg_m(k)} \leq C\theta |\nu_1-\nu_2| \left(\frac{y_m-k_*}{k_*}\wedge 1\right)%
\big((k_*/m)^{1-\beta}\wedge (k_*/m)^{-1}\big),  \label{e:q1_approx_bound}
\end{align}
and for any $\delta>0$ 
\begin{align}
&\max_{k_*<k\leq y_m}\max_{0\leq r < k} \frac{|q_2(r,k)-\mathcal{q}_2(r,k)|}{%
mg_m(k)}  \notag \\
&\quad\leq C \theta k_*^{-1}\left((k_*/m)^{1-\beta}\wedge
(k_*/m)^{-1}\right)\left((1-k_*/y_m)O_\P (\sigma k_*^{1/2})+O_\P %
(\sigma_*(y_m-k_*)^{1/2+\delta})\right)  \label{e:q2_approx_bound}
\end{align}
\end{lemma}
\begin{proof}
The bounds are immediate when $k_*\leq r<k$, so we only consider $0\leq r
<k_*$. Note with $p_i(r,k)$ as in \eqref{e:def_q1_orig}, 
\begin{align*}
p_1(r,k) &= 2(k-r)(k_*-r)-(k-r)^2 -(k_*-r)^2 + \varepsilon_{1}(r,k) \\
& = -(k-k_*)^2+ \varepsilon_{1}(r,k),
\end{align*}
with 
\begin{equation*}
\varepsilon_1(r,k) = -\frac{(k_*-r)(k_*-k)}{k-r-1}.
\end{equation*}
Similarly, 
\begin{align*}
p_{12}(r,k) & =2(k-k_*)^2-2\varepsilon_1(r,k), \\
p_{2}(r,k) & = -(k-k_*)^2+ \varepsilon_{1}(r,k),
\end{align*}
Hence, %
\begin{align*}
\max_{k_*<k\leq y_m} \frac{\max_{0\leq r <k_*}|\varepsilon_1(r,k)|}{mg_m(k)}
&\leq C \frac{1}{m g_m(k_*)}\max_{k_*<k\leq y_m} \varepsilon_1(0,k) \\
& \leq C\left(\frac{y_m-k_*}{k_*}\wedge 1\right)\big((k_*/m)^{1-\beta}\wedge
(k_*/m)^{-1}\Big),
\end{align*}
which gives \eqref{e:q1_approx_bound}. Likewise, 
\begin{align*}
q_2(r,k) &= \mathcal{q}_2(r,k) +2\theta(k-k_*)\left[\frac{1}{k-r-1}%
\sum_{i=m+r+1}^{m+k_*} z_i - \frac{k_*-r}{(k-r-1)(k-k_*)}%
\sum_{i=m+k_*+1}^{m+k}z_i^*\right] \\
&= \mathcal{q}_2(r,k) + 2\theta \varepsilon_{2}(r,k),
\end{align*}
and 
\begin{align*}
\max_{k_*<k\leq y_m} \frac{\max_{0\leq r<k}|\varepsilon_2(r,k)|}{mg_m(k)}
&\leq \max_{k_*<k\leq y_m} \frac{C}{m g_m(k)}\left(\frac{k-k_*}{k}
\left|\sum_{i=m+1}^{m+k_*} z_i\right|+\frac{k_*}{k}\left|%
\sum_{i=m+k_*+1}^{m+k}z_i^*\right|\right) \\
& \leq \frac{C }{m g_m(k_*)} \left( (1-k_*/y_m)O_\P (\sigma k_*^{1/2}) +
O_\P \big(\sigma_* (y_m-k_*)^{1/2+\delta}\big)\right),
\end{align*}
which gives \eqref{e:q2_approx_bound}. 
\end{proof}

The next few lemmas concern bounds and approximations for $q_3$, under $H_A$.
\begin{lemma}
\label{l:size_of_q3} With $q_3(r,k)$ as in \eqref{e:def_q3_orig}, for any
sequence $y_m\geq k_*$ with $y_m\to\infty$, %
\begin{equation*}
\max_{k_*<k\leq y_m}\max_{0\leq r <k} \frac{|q_3(r,k)|}{mg_m(k)}= O_\P \left(%
\big((y_m/m)^{(1-\beta)}\big)\wedge 1 \right),
\end{equation*}
and 
\begin{equation*}
\max_{k \geq k_*}\max_{0\leq r <k} \frac{|q_3(r,k)|}{mg_m(k)}=O_\P (1).
\end{equation*}
\end{lemma}
\begin{proof}
Write 
\begin{align*}
\frac{q_3(r,k)}{mg_m(k)} &=\frac{(k-r)^2}{mg_m(k)}\left[\frac{2R_{m,1}(r,k)}{%
(k-r)m} - \binom{m}{2}^{-1} R_{m,2} - \binom{k-r }{2}^{-1} R_{m,3}(r,k) %
\right] \\
& = \widetilde A_{m,1}(r,k)-\widetilde A_{m,2}(r,k)-\widetilde A_{m,3}(r,k).
\end{align*}
It suffices to establish 
\begin{align}
\max_{k_*<k\leq y_m}\max_{0\leq r <k} |\widetilde A_{m,i}(r,k)|&= O_\P \left(\big(%
(y_m/m)^{(1-\beta)}\big)\wedge 1 \right),  \label{e:Ri_asymp_1} \\
\max_{k\geq k_*}\max_{0\leq r <k} |\widetilde A_{m,i}(r,k)|&= O_\P (1),
\label{e:Ri_asymp_2}
\end{align}
for $i=1,2,3.$ For brevity we consider only $i=3$ since $i=1,2$ are
essentially the same but simpler. Write 
\begin{equation}  \label{e:h11h12h22}
\begin{aligned} \overline h_{11}({\bf x},{\bf y}) &= h({\bf x},{\bf y}) -
h_1({\bf x}) - h_1({\bf y}) +\mu_1, \\ \overline h_{22}({\bf x},{\bf y}) &=
h({\bf x},{\bf y}) - h_2({\bf x}) - h_2({\bf y}) +\mu_2, \\ \overline
h_{12}({\bf x},{\bf y}) &= h({\bf x},{\bf y}) - h_2({\bf x}) - h_1({\bf y})
+\mu_{12}, \end{aligned}
\end{equation}
So that%
\begin{equation*}
R_{m,3}(r,k)= R_{m,3,1}(r)\mathbf{1}_{\{r \leq k_*\}} + R_{m,3,2}(r\vee
k_*,k)+ R_{m,3,3}(r,k)\mathbf{1}_{\{r \leq k_*\}},
\end{equation*}
with 
\begin{gather*}
R_{m,3,1}(r) = \sum_{m+r < i<j \leq m+ k_*}\overline h_{11}(\mathbf{X}_i,%
\mathbf{X}_j),\quad R_{m,3,2}(r,k) = \sum_{m+r < i<j \leq m+k}\overline
h_{22}(\mathbf{X}^*_i,\mathbf{X}^*_j), \\
R_{m,3,3}(r,k) = \sum_{i= m+r+1}^{m+k_*}\sum_{j=m+k_*+1}^{m+k}\overline
h_{12}(\mathbf{X}_i,\mathbf{X}^*_j).
\end{gather*}
For $R_{m,3,1}(r)$, Lemma \ref{l:term-by-term_maximal_bounds} gives

\begin{align*}
{\mathsf{E}}\hspace{0.1mm} \max_{ 0\leq r \leq k_*}\left(
R_{m,3,1}(r)\right)^2 \leq 4 {\mathsf{E}}\hspace{0.1mm} \left(R_{m,3,1}(0)%
\right)^2 \leq k_*^2 {\mathsf{E}}\hspace{0.1mm}\overline h_{11}^2(\mathbf{X},%
\mathbf{Y}).
\end{align*}

Now, since $g_m(k)\geq C(k/m)^\beta$, we have%
\begin{align}
\P \left\{\max_{k_*<k\leq y_m}\max_{0\leq r <k_*}\frac{1}{mg_m(k)}
|R_{m,3,1}(r)|>x\right\} &\leq \P \left\{\max_{0\leq r <k_*} \frac{k_*^{-\beta}%
}{m^{1-\beta}} |R_{m,3,1}(r)|>C x\right\}  \notag \\
& \leq C x^{-2}\frac{k_*^{2-2\beta}}{m^{2(1-\beta)}} = O((y_m/m)^{2(1-\beta)}).
\end{align}
On the other hand, since $g_m(k)\geq C((k/m)^\beta\mathbf{1}_{\{k\leq m\}} +
(k/m)^2\mathbf{1}_{\{k>m\}} )$, it follows that %
\begin{align}
&\P \left\{\sup_{k\geq k_*}\max_{0\leq r <k_*}\frac{1}{mg_m(k)}
|R_{m,3,1}(r)|>x\right\} \\
&\quad\leq \P \left\{\max_{0\leq r <k_*}\left( \frac{m}{k_*^2}\mathbf{1}%
_{\{k_*>m\}}+ \mathbf{1}_{\{k_*\leq m\}}\frac{m^{\beta-1}}{k_*^\beta}%
\right)|R_{m,3,1}(r)|>C x\right\}  \notag \\
&\quad \leq C x^{-2}.
\end{align}
Now, for $R_{m,3,2}(r,k)$, suppose first $k_*\leq y_m\leq C m$. Using Lemma %
\ref{l:uncorrelated_maximal}, we have%
\begin{align}
&\P \left\{\max_{k_*<k\leq y_m}\max_{0\leq r <k}\frac{1}{mg_m(k)}
|R_{m,3,2}(r\vee k_*,k)|>x\right\} \\
&\quad\leq \P \left\{\max_{k_*<k\leq y_m} \frac{k^{-\beta}}{m^{1-\beta}}
\max_{k_*\leq r <k}|R_{m,3,2}(r,k)|>C x\right\}  \notag \\
&\quad\leq \P \left\{\max_{\lfloor \log_2(k_*)\rfloor < q \leq \lceil \log_2
y_m\rceil} \max_{2^{q-1}\leq k< 2^q}\max_{k_*\leq r <k}\frac{2^{-\beta(q-1)}}{%
m^{1-\beta}} |R_{m,3,2}(r,k)|>C x\right\}  \notag \\
&\quad\leq Cx^{-2}\sum_{q=\lfloor \log_2 k_* \rfloor + 1}^{\lceil \log_2 y_m\rceil}%
\frac{2^{-2\beta q}}{m^{2(1-\beta)}} (2^{q}-k_*)^2  \notag \\
&\quad \leq C x^{-2}\frac{y_m^{2(1-\beta)}}{m^{2(1-\beta)}}
\label{eR_3_asymp_1}
\end{align}
On the other hand, if $y_m>Cm$, since $g_m(k)\geq C(k/m)^2$ for $k\geq m$, we
have 
\begin{align}
&\P \left\{\max_{m\vee k_*\leq k\leq y_m}\frac{1}{mg_m(k)} \max_{k_*\leq r
<k}|R_{m,3,2}(r,k)|>x\right\}  \notag \\
&\quad\leq \P \left\{\max_{m\vee k_*\leq k\leq y_m} m k^{-2}\max_{k_*\leq r
<k}|R_{m,3,2}(r,k)|>C x\right\}  \notag \\
&\quad\leq \P \left\{\max_{\lfloor \log_2(m\vee k_*)\rfloor < q \leq \lceil \log
y_m\rceil} \max_{2^{q-1}\leq k< 2^q} m2^{-2(q-1)}\max_{k_*\leq r <k}
|R_{m,3,2}(r,k)|>C x\right\}  \notag \\
&\quad\leq Cx^{-2}\sum_{q=\lfloor \log_2m \rfloor + 1}^{\lceil \log_2 y_m\rceil}
m^2 2^{-2q}  \notag \\
& \quad\leq C x^{-2},  \label{eR_3_asymp_2}
\end{align}
which, combined with \eqref{eR_3_asymp_1}, gives \eqref{e:Ri_asymp_1}. %
Likewise, analogous steps leading to \eqref{eR_3_asymp_2} show 
\begin{equation*}
\max_{k\geq k_*}\max_{0\leq r <k}\frac{1}{mg_m(k)} |R_{m,3,2}(r\vee
k_*,k)|=O_\P (1).
\end{equation*}
Repeating the above arguments \textit{mutatis mutandis} for $R_{m,3,3}(r,k)$
then gives the claim.
\end{proof}

\begin{lemma}
\label{l:uncorrelated_maximal} Let $\overline h_{12}$ and $\overline h_{22}$
be as in \eqref{e:h11h12h22}. Then for any $y\geq k_*+2$, 
\begin{align}
\E  \max_{k_*<k\leq y} \max_{k_*\leq r<k} \left|\sum_{i=
1}^{m}\sum_{j=m+r+1}^{m+k} \overline h_{12}(\mathbf{X}_i,\mathbf{X}%
_j^*)\right|^2 &\leq Cm(y-k_*),  \label{e:crossterm_maximal2}
\\
{\mathsf{E}}\hspace{0.1mm} \max_{k_*<k\leq y} \max_{k_*\leq r < k}
\left|\sum_{m+r< i<j \leq m+k} \overline h_{22}(\mathbf{X}_i^*,\mathbf{X}%
_j^*)\right|^2&\leq C(y-k_*)^2.  \label{e:quadterm_maximal3} \\
{\mathsf{E}}\hspace{0.1mm} \max_{0\leq r< k_*} \max_{k_*<k\leq y}
\left|\sum_{i= m+r+1}^{m+k_*}\sum_{j=m+k_*+1}^{m+k}\overline h_{12}(\mathbf{X%
}_i,\mathbf{X}_j^*)\right|^2&\leq Ck_*(y-k_*)
\label{e:crossterm_maximal}
\end{align}
\end{lemma}

\begin{proof}
The bounds \eqref{e:crossterm_maximal2} and  \eqref{e:crossterm_maximal} follow from Lemma \ref{l:basic_cross_lemma},  and after relabeling indices, \eqref{e:quadterm_maximal3} follows from Lemma \ref{l:term-by-term_maximal_bounds}.
\end{proof}

\begin{lemma}
\label{e:q3_statiticterms_only_lemma} Suppose $k_* = c_*m$. With $R_{m,1,i}$,%
$R_{m,2}$ and $R_{m,3,i}$ as in \eqref{e:R_m11}, \eqref{e:R_m2}, and %
\eqref{e:R_m31}, respectively, let 
\begin{align}  \label{e:cq3}
\mathcal{q}_3(r,k) =(k-r)^2 \left(\frac{2\mathbf{1}_{\{r \leq
k_*\}}R_{m,1,1}(r)}{(k-r)m} - \frac{2R_{m,2}}{m(m-1)} - \frac{2 \mathbf{1}%
_{\{r \leq k_*\}}R_{m,3,1}(r) }{(k-r)(k-r-1)}\right).
\end{align}
Then, for any $T>0$, and $0<\delta<1$, 
\begin{align*}
& \max_{k_*< k \leq k_* + Tm^{1-\delta}}\max_{0\leq r<k}\left|\frac{q_3(r,k)%
}{mg_m(k)} - \frac{\mathcal{q}_3(r,k)}{mg_m(k)} \right| \\
&=2\max_{k_*< k \leq k_* + Tm^{1-\delta}}\max_{0\leq r<k}\frac{(k-r)^2}{m
g_m(k)}\left|\frac{R_{m,1,2}(r\vee k_*,k)}{(k-r)m}- \frac{R_{m,3,2}(r \vee
k_*,k) + \mathbf{1}_{\{r \leq k_*\}}R_{m,3,3}(r,k)}{(k-r)(k-r-1)}\right| \\
& = o_\P (1).
\end{align*}
\end{lemma}

\begin{proof}
We treat each of the terms $R_{m,1,2}(r,k)$, $R_{m,3,2}(r,k)$, and $%
R_{m,3,3}(r,k)$ separately. Since $g_m(k)\geq C(k/m)^{2}$ for all $k> k_*$,
using Lemma \ref{l:uncorrelated_maximal} we get%
\begin{align*}
&{\mathsf{E}}\hspace{0.1mm} \max_{k_*< k \leq k_* +
Tm^{1-\delta}}\max_{0\leq r<k} \left|\frac{(k-r)}{m^2g_m(k)}%
R_{m,1,2}(k_*\vee r,k) \right|^2 \\
&\quad\leq C k_*^{-2}{\mathsf{E}}\hspace{0.1mm}\max_{k_*< k \leq k_* +
Tm^{1-\delta}} \max_{k_*\leq r<k}|R_{m,1,2}(r,k)|^2 \\
&\quad = C m^{-\delta}.
\end{align*}
Similarly, again using Lemma \ref{l:uncorrelated_maximal}, 
\begin{align*}
&{\mathsf{E}}\hspace{0.1mm} \max_{k_*<k \leq k_* + Tm^{1-\delta}}
\max_{0\leq r<k}\left|\frac{1}{m g_m(k)}R_{m,3,2}(r \vee k_*,k) \right|^2 \\
&\quad\leq C m^{2}k_*^{-4}{\mathsf{E}}\hspace{0.1mm} \max_{k_*\leq k \leq
k_* + Tm^{1-\delta}} \max_{k_*\leq r < k} \left|R_{m,3,2}(r,k) \right|^2 \\
&\quad \leq C m^{-2\delta}%
\end{align*}
Again applying Lemma \ref{l:uncorrelated_maximal} we obtain %
\begin{align*}
{\mathsf{E}}\hspace{0.1mm} \max_{k_*< k \leq k_* + Tm^{1-\delta}}\max_{0\leq
r<k}\left|\frac{1}{m g_m(k)}R_{m,3,3}(r,k) \right|^2 \leq C
m^{-\delta}%
\end{align*}
\end{proof}

\bigskip

\subsection{Lemmas for Section \ref{s:refinements}}
We conclude this section with a set of lemmas which will be used for the
proofs of the results in Section \ref{s:refinements}. Throughout, we use the notation
$$
K_L(\mathbf x,\mathbf y)
=\sum_{\ell=L+1}^{\infty} \lambda_\ell\phi_\ell(\mathbf x)\phi_\ell(\mathbf y).$$
Let 
\begin{align*}
{{\mathfrak{R}}}^{(>L)}(k)=\frac{2}{k(m-k)}\sum_{i=1}^k\sum_{j=k+1}^m&K_L(\bX_i,\bX_j)-%
\frac{1}{k(k-1)}\sum_{1\leq i\neq j\leq
k} K_L(\bX_i,\bX_j) \\
&-\frac{1}{(m-k)(m-k-1)}\sum_{k+1\leq i\neq j\leq
m}K_L(\bX_i,\bX_j).
\end{align*}

\begin{lemma}
\label{prwe2} If Assumptions \ref{a:historical_stability}--\ref%
{a:assumption_on_h} hold, then we have 
\begin{equation*}
P\left\{ \max_{2\leq k\leq m-2}\frac{1}{{\mathfrak{q}}(k/m)}\frac{%
k^{2}(m-k)^{2}}{m^{3}}\left\vert {{\mathfrak{R}}}^{(>L)}(k)\right\vert >x\right\}
\leq \frac{c}{x^{2}}\sum_{\ell =L+1}^{\infty }\lambda _{\ell }^{2},
\end{equation*}%
for all $x>0$ and $L\geq 0$.
\end{lemma}
\begin{proof}
We note 
\begin{align*}
\frac{k^2(m-k)^2}{m^3}\left|{\mathfrak{R}}^{(>L)}(k)\right|&\leq \frac{2k(m-k)}{m^3%
}\left|\sum_{i=1}^k\sum_{j=k+1}^mK_L(\bX_i,\bX_j)\right| \\
&+ \frac{2(m-k)^2}{m^3}\left|\sum_{1\leq i\neq j\leq
k}K_L(\bX_i,\bX_j)\right| + \frac{2k^2}{m^3}\left|\sum_{k+1\leq i\neq j\leq
m}K_L(\bX_i,\bX_j)\right|.
\end{align*}
When $2\leq k \leq m/2$, $\mathfrak{q}(k/m)\geq C (k/m)^\zeta$, and
$$
\frac{1}{{\mathfrak{q}}(k/m)}\frac{%
k(m-k)}{m^{3}}\leq C\left(\frac{m}{k}\right)^\zeta \frac{k(m-k)%
}{m^3} \leq C\frac{k^{1-\zeta}}{m^{2-\zeta}} \leq C m^{-1}.
$$
Lemma \ref{l:term-by-term_maximal_bounds} yields via Markov's inequality that 
\begin{align*}
P&\left\{ \max_{2\leq k \leq m/2} \left(\frac{m}{k}\right)^\zeta \frac{k(m-k)%
}{m^3}\left|
\sum_{i=1}^k\sum_{j=k+1}^mK_L(\mathbf{X}_{i},\mathbf{X}_{j}) \right|>x\right\} \\
&\leq P\left\{ \max_{2\leq k \leq m/2} \left|
\sum_{i=1}^k\sum_{j=k+1}^mK_L(\mathbf{X}_{i},\mathbf{X}_{j}) \right|>xm\right\} \\
&\leq \frac{c}{x^2}\sum_{\ell=L+1}^\infty\lambda_\ell^2.
\end{align*}
Similarly, since $\left(m/k\right)^\zeta k^2 %
m^{-3} \leq m^{-1} (k/m)^{2-\zeta}\leq m^{-1}$ on $2\leq k \leq m/2$, 
\begin{align*}
P&\left\{ \max_{2\leq k \leq m/2} \left(\frac{m}{k}\right)^\zeta \frac{k^2}{%
m^3}\left| \sum_{k+1\leq i\neq j\leq
m}K_L(\mathbf{X}_{i},\mathbf{X}_{j})\right|>x\right\} \\
&\leq \frac{c}{x^2}\sum_{\ell=L+1}^\infty\lambda_\ell^2.
\end{align*}
For the remaining term, using $\left(m/k\right)^\zeta (m-k)^2 %
m^{-3} \leq  k^{-\zeta}m^{\zeta-1}$ on $2\leq k \leq m/2,$
\begin{align*}
P&\left\{ \max_{2\leq k \leq m/2} \left(\frac{m}{k}\right)^\zeta \frac{%
(m-k)^2}{m^3}\left| \sum_{1\leq i\neq j\leq
k}K_L(\mathbf{X}_{i},\mathbf{X}_{j})\right|>x\right\} \\
&\leq P\left\{ \max_{2\leq k \leq m/2}  k^{-\zeta}\left| \sum_{1\leq i\neq
j\leq k }K_L(\mathbf{X}_{i},\mathbf{X}_{j}) \right|>xm^{1-\zeta}\right\} \\
&\leq \sum_{z=1}^{\lceil \log_2 (m/2)\rceil} P\left\{ \max_{2^{z-1}\leq k \leq 2^z} \left|
\sum_{1\leq i\neq j\leq k}K_L(\mathbf{X}_{i},\mathbf{X}_{j}) \right|>xm^{1-\zeta}2^{(z-1)\zeta}\right\} \\
&\leq \frac{c }{x^2}m^{2\zeta-2}\sum_{z=1}^{\lceil \log_2 (m/2)\rceil }2^{2z(1-\zeta)}\sum_{\ell=L+1}^\infty\lambda_\ell^2 \\
&\leq \frac{c}{x^2}\sum_{\ell=L+1}^\infty\lambda_\ell^2.
\end{align*}
By symmetry, the same argument works for the range $m/2<k\leq m-2$, which completes the proof.
\end{proof}

\bigskip
By Lemma \ref{prwe2}, it remains to study $\mathfrak R_L(k)=\mathfrak R(k)-\mathfrak R^{(>L)}(k)$.  With $\overline h_L$ as in \eqref{e:def_hL}, we have
$$
\mathfrak R_L(k)=U_k(\overline h_L;0,m-k),
$$
and, using the notation \eqref{e:def_Sell(m)}, by \eqref{e:Um(f,k)}, we find
\begin{align}
&\frac{(m-k)^2}{k}U_k(f_\ell;0,m-k)\\
&=
-\frac{1}{k}
\left(
S_\ell(m)-S_\ell(k)-\frac{m-k}{k}S_\ell(k)
\right)^2
+\frac{m(m-k)}{k^2}
+R_\ell(m-k,m-k,k) \notag \\
&=
-\frac{m^2}{k^3}
\left(
S_\ell(k)-\frac{k}{m}S_\ell(m)
\right)^2
+\frac{m(m-k)}{k^2}
+R_\ell(m-k,m-k,k),
\end{align}
with $R_\ell(\cdot,\cdot,\cdot)$ is defined in \eqref{e:def_R_ell(k,m)}.
Hence,
\begin{align}
\frac{k^2(m-k)^2}{m^3} \mathfrak R_L(k)
&=-\frac1m
\sum_{\ell=1}^L\lambda_\ell
\left[
\left(
S_\ell(k)-\frac{k}{m}S_\ell(m)
\right)^2
-\frac{k(m-k)}{m}
\right] \notag \\
&\quad
+\frac{k^3}{m^3}\sum_{\ell=1}^L\lambda_\ell R_\ell(m-k,m-k,k).
\end{align}

\begin{lemma}
\label{l:remainder_neglig_R_L}
Under $H_0$, for every fixed $L\geq 1$,
$$
\max_{2\leq k\leq m-2}
\frac{1}{\mathfrak q(k/m)}
\left(\frac{k}{m}\right)^3 \sum_{\ell=1}^L\left| \lambda_\ell R_\ell(m-k,m-k,k)\right| =o_\P(1), $$
where $R_\ell(k,w,m)$ is defined in \eqref{e:def_R_ell(k,m)}.
\end{lemma}
\begin{proof}
It is enough to prove the claim for a fixed $\ell$. Write
$w=m-k$. From \eqref{e:def_R_ell(k,m)}, we have
\begin{align}
&\left(\frac{k}{m}\right)^3 |R_\ell(w,w,k)|\label{e:R(k)_remainder_bound}\\
\notag&\quad\leq
\frac{w^2}{m^3(k-1)}S_\ell^2(k)
+\frac{k w^2}{m^3(k-1)}
\left|
\sum_{i=1}^k(\phi_\ell^2(\bX_i)-1)\right| +\frac{k w^2}{m^3(k-1)}\\
\notag&\quad+
\frac{k^2}{m^3(w-1)}\left(S_\ell(m)-S_\ell(k)\right)^2+ \frac{k^2w}{m^3(w-1)}\left|\sum_{i=k+1}^m(\phi_\ell^2(\bX_i)-1)\right|+\frac{k^2w}{m^3(w-1)} .
\end{align}
By symmetry in $k$ and $m-k$, we need only to establish the first three terms on the right-hand side of \eqref{e:R(k)_remainder_bound} are negligible uniformly in $2\leq k \leq m-2$ after
multiplication by $\mathfrak q(k/m)^{-1}$. For any fixed integer $a\geq 2$,
$$
\max_{2\leq k\leq a}
\frac{1}{\mathfrak q(k/m)}
\frac{k w^2}{m^3(k-1)}\left|\sum_{i=1}^k(\phi_\ell^2(\bX_i)-1)\right|
=O_\P(m^{\zeta-1})
=o_\P(1).
$$
On the other hand, for $a\leq k\leq m-2$,
$$
\frac{1}{\mathfrak q(k/m)}
\frac{k w^2}{m^3(k-1)}\left|\sum_{i=1}^k(\phi_\ell^2(\bX_i)-1)\right|
\leq
C\frac1k \left|\sum_{i=1}^k(\phi_\ell^2(\bX_i)-1)\right|.
$$
Hence, by the law of large numbers,
$$
\lim_{a\to\infty}\limsup_{m\to\infty}
\P\left\{\max_{a\leq k\leq m-2}
\frac{1}{\mathfrak q(k/m)}
\frac{k w^2}{m^3(k-1)}\left|\sum_{i=1}^k(\phi_\ell^2(\bX_i)-1)\right|>  x \right\}=0.
 $$
Hence the second term on the right-hand side of \eqref{e:R(k)_remainder_bound} is negligible.  We now turn to the first term. Write
$$
S_\ell^2(k) =k + \sum_{i=1}^k(\phi_\ell^2(\bX_i)-1) +\sum_{1\leq i\neq j\leq k}
\phi_\ell(\bX_i)\phi_\ell(\bX_j).
$$
After multiplying $\mathfrak q(k/m)^{-1}{w^2}m^{-3}(k-1)^{-1}$, the second term above is negligible by the preceding argument. For the third term, using Lemma \ref{l:term-by-term_maximal_bounds},
$$
\E \max_{1\leq k\leq n}\left|\sum_{1\leq i\neq j\leq k}
\phi_\ell(\bX_i)\phi_\ell(\bX_j)\right|^2\leq Cn^2.
$$
Therefore, for $k\leq m/2$, a dyadic argument gives, for every $x>0$,
\begin{align*}
\P\left\{\max_{2\leq k\leq m/2}
\frac{1}{\mathfrak q(k/m)}
\frac{w^2}{m^3(k-1)}\left|\sum_{1\leq i\neq j\leq k}
\phi_\ell(\bX_i)\phi_\ell(\bX_j)\right|>x
\right\}
&\leq
C \frac{m^{-2+2\zeta}}{x^2}
\sum_{r=1}^{\lceil \log_2(m)\rceil }2^{-2r\zeta}
=o(1).
\end{align*}
For $k>m/2$,
\begin{align*}
&\max_{m/2<k\leq m-2}
\frac{1}{\mathfrak q(k/m)}
\frac{w^2}{m^3(k-1)}\left|\sum_{1\leq i\neq j\leq k}
\phi_\ell(\bX_i)\phi_\ell(\bX_j)\right|\\
&\qquad\leq
\frac{C}{m^2}\max_{1\leq k\leq m}\left|\sum_{1\leq i\neq j\leq k}
\phi_\ell(\bX_i)\phi_\ell(\bX_j)\right|
=o_\P(1).
\end{align*}
Finally, %
$$
\max_{2\leq k\leq m-2}
\frac{1}{\mathfrak q(k/m)}
\frac{k w^2}{m^3(k-1)}
\leq
Cm^{\zeta-1}
=o(1).
$$
Combining the preceding bounds completes the proof.
\end{proof}

\begin{lemma}
\label{prwe4} If Assumptions \ref{a:historical_stability}--\ref%
{a:assumption_on_h} hold, then %
\begin{align*}
& \left\{ \frac{1}{{\mathfrak{q}}^{1/2}(t)}\frac{1}{m^{1/2}}\left( S_{\ell
}(\lfloor mt\rfloor )-\frac{\lfloor mt \rfloor }{m}S_{\ell }(m)\right),\quad 0\leq t\leq 1,~1\leq \ell \leq L\right\}
\\
& \hspace{2cm}\Rightarrow \left\{ \frac{1}{{\mathfrak{q}}^{1/2}(t)}B_{\ell
}(t),\quad 0\leq t\leq 1,~1\leq \ell \leq L\right\} ,\quad in\quad \mathbf{D}%
^{L}[0,1]
\end{align*}%
where $\{B_{1}(t),0\leq t\leq 1\},\ldots ,\{B_{L}(t),0\leq t\leq 1\}$ are
independent Brownian bridges.
\end{lemma}

\begin{proof}
The result is taken from Chapter 1 of \cite%
{chgreg}.
\end{proof}

\clearpage

\section{Main proofs\label{proofs}}

\begin{proof}[Proof of Theorem \protect\ref{t:theorem_under_H0}]
Recall ${\mathcal{D}}_{m}^{(1)}(k)=m^{-1}k^{2}|U_{m}(\overline h;k)|$, (\ref{u_m_h,k}%
), and (\ref{u_m_h,k_truncated}). From \eqref{e:Um(f,k)}, we have 
\begin{align}
& m^{-1}k^{2}U_{m,L}(\overline h;k)  \notag \\
\ & =\sum_{\ell =1}^{L}m^{-1}k^{2}\lambda _{\ell }U_{m}(f_{\ell };0,k) 
\notag \\
& =-\sum_{\ell =1}^{L}\lambda _{\ell }\left( \frac{1}{m}\left( S_{\ell
}(k,m)-\frac{k}{m}S_{\ell }(m)\right) ^{2}-\frac{k(k+m)}{m^{2}}\right)
+\sum_{\ell =1}^{L}\lambda _{\ell }R_{\ell }(k,k,m).
\label{e:UmL_w_remainder}
\end{align}%
For each real number $t\geq 2/m$, let 
\begin{equation}
\mathcal{U}_{m}(t)=\frac{m^{-1}\lfloor mt\rfloor ^{2}U_{m}(\overline h;\lfloor
mt\rfloor )}{g_{m}\left( \lfloor mt\rfloor \right) },~~\mathcal{U}_{m,L}(t)=%
\frac{m^{-1}\lfloor mt\rfloor ^{2}U_{m,L}(\overline h;\lfloor mt\rfloor )}{g_{m}\left(
\lfloor mt\rfloor \right) }=\frac{\mathbb{U}_{m,L}(0,t)}{g_{m}\left( \lfloor
mt\rfloor \right) },  \label{e:Um(t)}
\end{equation}%
where $\mathbb{U}_{m,L}$ is given in \eqref{uml}, and set $\mathcal{U}%
_{m}(t)=\mathcal{U}_{m}(2/m),$ $\mathcal{U}_{m,L}(t)=\mathcal{U}_{m,L}(2/m)$
for $0\leq t<2/m$. We have 
\begin{equation*}
\sup_{t\geq 0}|\mathcal{U}_{m}(t)|=\sup_{k\geq 2}\frac{{\mathcal{D}}_{m}(k)}{%
g_{m}(k)}. 
\end{equation*}%
With ${\mathcal{V}}_{m,L}(t)=\mathbb{V}_{m,L}(0,t)/g(t)$, where $\mathbb{V}%
_{m,L}$ is defined in Lemma \ref{l:approxunderH0}, applying Lemma \ref%
{l:approxunderH0} we have, for any $0<\delta <T$,

\begin{equation}
\sup_{\delta \leq t\leq T}|\mathcal{U}_{m,L}(t)-{\mathcal{V}}%
_{m,L}(t)|=\sup_{\delta \leq t\leq T}\left\vert \frac{\mathbb{V}_{m,L}(0,t)}{%
g(t)}-\frac{\mathbb{U}_{m,L}(0,t)}{g_{m}(\lfloor mt\rfloor )}\right\vert =o_{%
\P }(1).  \label{e:sko_int}
\end{equation}

On the other hand, setting 
\begin{equation}
{\mathcal{V}}(t)=-\frac{1}{g(t)}\sum_{\ell =1}^{\infty }\lambda _{\ell }%
\left[ \left( W_{2,\ell }(t)-tW_{1,\ell }(1)\right) ^{2}-t(1+t)\right]
,\quad t>0,  \label{e:def_V}
\end{equation}%
and ${\mathcal{V}}(0)=0$, Lemma \ref{l:trunctail_of_limit} implies that ${\mathcal{V}}$ is well-defined, and for any fixed $m$, as $L\rightarrow
\infty $, 
\begin{equation}
{\mathcal{V}}_{m,L}\Rightarrow {\mathcal{V}}\quad \text{in}\quad \mathbf{D}%
[\delta ,T].  \label{e:VL_to_V}
\end{equation}%

Additionally, from Lemma \ref{l:remainder_neglig_0}, we have %
\begin{equation}
\lim_{L\rightarrow \infty }\sup_{m\geq 1}\P \left\{ \sup_{t\geq 0}|\mathcal{U%
}_{m}(t)-\mathcal{U}_{m,L}(t)|>x\right\} =0,  \label{e:remainder_neglig_0}
\end{equation}%
which combined with \eqref{e:sko_int} and \eqref{e:VL_to_V} implies (see
e.g. Theorem 3.2 in \citealp{billingsley:1968}) 
\begin{equation}
\mathcal{U}_{m}\Rightarrow {\mathcal{V}}\quad \text{in}\quad \mathbf{D}%
[\delta ,T].  \label{e:cU_m_to_cV}
\end{equation}%
On the other hand, Lemma \ref{l:trunctail_of_limit} implies 
\begin{equation}
\sup_{0\leq t\leq \delta }|{\mathcal{V}}(t)|\rightarrow 0,\quad \text{a.s.}%
\quad \delta \rightarrow 0.  \label{e:V_a_to_V}
\end{equation}%
Further, by Lemma \ref{l:remainder_neglig_0}, expression %
\eqref{e:uniform_limit_in_delta}, %
\begin{equation}
\lim_{\delta \rightarrow 0}\limsup_{m\rightarrow \infty }\P \left\{
\sup_{0\leq t\leq \delta }|\mathcal{U}_{m}(t)|>x\right\} =0.
\label{e:uniform_limit_in_a}
\end{equation}%
Combining \eqref{e:V_a_to_V}, and \eqref{e:uniform_limit_in_a} gives 
\begin{equation}
\mathcal U_{m}\Rightarrow {\mathcal{V}}\quad \text{in}\quad D[0,T],
\label{e:weak_to_cV_D0T}
\end{equation}%
for any $T>0$. 

With $a_m$ and $\mathcal H_{m}$ as in Lemma~\ref{l:tail_variable},  first note that, 
$$
\sup_{k\geq mT}|a_m(0,k)-1| 
\to 0,\qquad T\to\infty.
$$
 Hence Lemma~\ref{l:tail_variable} yields, for every
$x>0$,
\begin{equation}\label{e:Um_to_Hm}
\lim_{T\to\infty}\limsup_{m\to\infty}\P\left\{
\sup_{t\geq T} \left|\mathcal U_{m}(t)-\mathcal H_{m} \right|>x \right\}=0.
\end{equation}
Moreover, from Lemma \ref{l:trunctail_of_limit}, 
\begin{equation}\label{e:def_H_variable}
\mathcal V(t)\to \mathcal H=\sum_{\ell=1}^{\infty}\lambda_\ell\left(1-W_{1\ell}^2(1)\right),
\qquad\text{a.s.,}\quad t\to\infty.
\end{equation}
Thus
\begin{equation}\label{e:V_to_H}
\sup_{t\geq T}|\mathcal V(t)-\mathcal H|\to0 \qquad\text{a.s.},\qquad T\to\infty.
\end{equation}
So, with 
$$
\qquad
\mathcal H_L
=
\sum_{\ell=1}^{L}\lambda_\ell
\left(1-W_{1,\ell}^2(1)\right),
$$
we have $\mathcal H_{m,L}\Rightarrow \mathcal H_L$ for each fixed $L$. Moreover,
$$
\E|\mathcal H_m-\mathcal H_{m,L}|^2 \leq C\sum_{\ell=L+1}^{\infty}\lambda_\ell^2,\qquad 
\E|\mathcal H-\mathcal H_L|^2 \leq C\sum_{\ell=L+1}^{\infty}\lambda_\ell^2.
$$
Hence $\mathcal H_m\Rightarrow \mathcal H$ as $m\to\infty$. Combining this with \eqref{e:weak_to_cV_D0T}, and repeating the same
$L$-truncation argument, we have, for every fixed $T>0$,
$$
(\mathcal U_m,\mathcal H_m)\Rightarrow(\mathcal V,\mathcal H)\quad\text{in }D[0,T]\times\mathbb R .
$$
By the Dudley-Wichura-Skorokhod theorem, we may assume
that, for a sequence $(\mathcal V^{(m)},\mathcal H^{(m)})\stackrel d = (\mathcal V,\mathcal H)$,  
$$
\sup_{0\leq t\leq T}|\mathcal U_m(t)-\mathcal V^{(m)}(t)|=o_\P(1),
\qquad
|\mathcal H_m-\mathcal H^{(m)}|=o_\P(1), 
$$
after possibly extending the probability space.

Moreover, since $\mathcal H=\lim_{t\to\infty}\mathcal V(t)$ a.s., we may take
$\mathcal H^{(m)}=\lim_{t\to\infty}\mathcal V^{(m)}(t)$. Hence, from the decomposition %
\begin{align*}
\left|\sup_{t\geq0}|\mathcal U_m(t)|-\sup_{t\geq0}|\mathcal V^{(m)}(t)|
\right|
&\leq\sup_{0\leq t\leq T}|\mathcal U_m(t)-\mathcal V^{(m)}(t)| \\
&\quad+\sup_{t\geq T}|\mathcal U_m(t)-\mathcal H_m| \\
&\quad+|\mathcal H_m-\mathcal H^{(m)}| \\
&\quad+\sup_{t\geq T}|\mathcal V^{(m)}(t)-\mathcal H^{(m)}|,
\end{align*}
we see that first and third terms converge to zero in probability for fixed $T$, and the second and fourth terms can be made arbitrarily small in probability by taking $T$ large enough.  Hence
\begin{equation}
\sup_{t\geq 0}|\mathcal{U}_{m}(t)|\Rightarrow \sup_{t\geq 0}|{\mathcal{V}}%
(t)|.  \label{e:weak_to_cV}
\end{equation}%
Now, checking covariance functions, one can easily verify that 
\begin{equation}
\left\{ \frac{W_{2,\ell }(t)-tW_{1,\ell }(1)}{1+t},~t\geq 0,~\ell \geq
1\right\} \overset{{\mathcal{D}}}{=}\left\{ W_{\ell }\left( \frac{t}{1+t}%
\right) ~t\geq 0,~\ell \geq 1\right\} ,  \label{e:checking_covariance}
\end{equation}%
where $\{W_{1}(t),t\geq 0\},\{W_{2}(t),t\geq 0\},\ldots $ are independent
Wiener processes. Thus, recalling \eqref{e:boundaryfxn}, we have 
\begin{align}
\sup_{t\geq 0}|{\mathcal{V}}(t)|& \overset{{\mathcal{D}}}{=}\sup_{t\geq
0}\left( \frac{t}{1+t}\right) ^{-\beta }\left\vert \sum_{\ell =1}^{\infty
}\lambda _{\ell }\left[ W_{\ell }^{2}\left( \frac{t}{1+t}\right) -\frac{t}{%
1+t}\right] \right\vert  \notag \\
& \overset{{\mathcal{D}}}{=}\sup_{0<u\leq 1}u^{-\beta }\left\vert \sum_{\ell
=1}^{\infty }\lambda _{\ell }\left[ W_{\ell }^{2}\left( u\right) -u\right]
\right\vert ,  \label{e:change_of_vari_to_Gamma(t)}
\end{align}%
yielding part \textit{(i)} of the theorem. Turning to part \textit{(ii)},
for simplicity write $M_{m}=M$. Since $M/m\rightarrow a_{0}$, and%
\begin{equation}
\sup_{2\leq k\leq M}\frac{{\mathcal{D}}_{m}(k)}{g_{m}(k)}=\sup_{0\leq t\leq
M/m}|\mathcal{U}_{m}(t)|,
\end{equation}%
the same arguments above yield $\sup_{0\leq t\leq M/m}|\mathcal{U}%
_{m}(t)|\Rightarrow \sup_{0\leq t\leq a_{0}}|{\mathcal{V}}(t)|$, and the
result follows from the change of variables in %
\eqref{e:change_of_vari_to_Gamma(t)}.

Turning now to part \textit{(iii)} of the theorem, for any $t\geq 2/M$,
define 
\begin{equation*}
\widetilde{\mathcal{U}}_{m}(t)=\frac{\lfloor Mt\rfloor ^{2}U_{m}(\overline h;\lfloor
Mt\rfloor )}{M(\lfloor Mt\rfloor /M)^{\beta }},\quad \widetilde{\mathcal{U}}%
_{m,L}(t)=\frac{\lfloor Mt\rfloor ^{2}U_{m,L}(\overline h;\lfloor Mt\rfloor )}{M(\lfloor
Mt\rfloor /M)^{\beta }}, 
\end{equation*}%
so that 
\begin{equation*}
\max_{2\leq k\leq M}\frac{{\mathcal{D}}_{m}^{(1)}(k)}{g_{m}(k)}%
=\sup_{2/M\leq t\leq 1}|\widetilde{\mathcal{U}}_{m}(t)|. 
\end{equation*}%
Also, for each $t>0$ let 
\begin{align*}
& \widetilde{\mathcal{U}}_{m,L}^{\circ }(t) \\
& =-\left( \frac{\lfloor Mt\rfloor }{M}\right) ^{-\beta }\sum_{\ell
=1}^{L}\lambda _{\ell }\left( \frac{1}{M}\left( S_{\ell }(\lfloor Mt\rfloor
,m)-\frac{\lfloor Mt\rfloor }{m}S_{\ell }(m)\right) ^{2}-\frac{\lfloor
Mt\rfloor (\lfloor Mt\rfloor +m)}{mM}\right) ,
\end{align*}%
and 
\begin{equation*}
\widetilde{{\mathcal{V}}}_{L}(t)=-t^{-\beta }\sum_{\ell =1}^{L}\lambda
_{\ell }\left[ W_{\ell }^{2}\left( t\right) -t\right] ,\quad \widetilde{{%
\mathcal{V}}}(t)=-t^{-\beta }\sum_{\ell =1}^{\infty}\lambda _{\ell }\left[
W_{\ell }^{2}\left( t\right) -t\right] . 
\end{equation*}%
Arguing as in the case of part \textit{(i)}, we need only establish the weak
convergence of 
\begin{equation}
\widetilde{\mathcal{U}}_{m,L}^{\circ }\Rightarrow \widetilde{{\mathcal{V}}}%
_{L},\quad \text{in}\quad \mathbf{D}[\delta ,1],  \label{e:needtoshow_iii}
\end{equation}%
for every fixed $L\geq 1$ and $0<\delta <1$. 
Observe that
\begin{align*}
&\widetilde{\mathcal U}_{m,L}^{\circ}(k/M)
+\left( \frac{k}{M}\right)^{-\beta}\sum_{\ell=1}^{L}\lambda_\ell\left(\frac{1}{M}S_\ell^2(k,m)-\frac{k}{M}\right) \\
&\quad = -\left( \frac{k}{M}\right)^{-\beta}\sum_{\ell=1}^{L}\lambda_\ell
\left[\frac{1}{M}\left(\left(S_{\ell }(k,m)-\frac{k}{m}S_{\ell }(m)\right)^2-S_{\ell }^2(k,m)\right)-\frac{k^2}{mM}\right].
\end{align*}

However, since $m^{-1}S_{\ell}(m)=O_{\P}(m^{-1/2})$, for every fixed
$0<\delta<1$ we have
\begin{align*}
& \max_{\delta M\leq k\leq M}
\left( \frac{k}{M}\right) ^{-\beta }
\left|
\sum_{\ell =1}^{L}\lambda _{\ell }
\left[
\frac{1}{M}
\left\{
\left(S_{\ell }(k,m)-\frac{k}{m}S_{\ell }(m)\right)^2
-
S_{\ell }^2(k,m)
\right\}
-\frac{k^2}{mM}
\right]
\right| \\
&\qquad =o_{\P}(1).
\end{align*}
Finally, we have
\begin{align*}
M^{-1/2}\left( S_{1}(\lfloor Mt\rfloor ,m),\ldots ,S_{L}(\lfloor Mt\rfloor
,m)\right)
&\overset{{\mathcal D}}{=}
M^{-1/2}\left( S_{1}(\lfloor Mt\rfloor ,0),\ldots ,S_{L}(\lfloor Mt\rfloor ,0)\right) \\
&\Rightarrow \left( W_{1}(t),\ldots ,W_{L}(t)\right)
\quad \text{in}\quad \mathbf{D}[0,1],
\end{align*}
and the continuous mapping theorem yields
$$
-\left( \frac{\lfloor Mt\rfloor }{M}\right) ^{-\beta }
\sum_{\ell=1}^{L}\lambda _{\ell }
\left(
\frac{1}{M}S_{\ell }^2(\lfloor Mt\rfloor,m)
-
\frac{\lfloor Mt\rfloor}{M}
\right)
\Rightarrow
\widetilde{{\mathcal V}}_{L}(t)
\quad \text{in}\quad \mathbf{D}[\delta ,1],
$$
giving \eqref{e:needtoshow_iii}. The remainder of the proof is the same as
in case (i).
\end{proof}

\begin{proof}[Proof of Theorem \protect\ref{t:page_under_H0}]
The proof is largely the same as Theorem \ref{t:theorem_under_H0}, though we
provide details where there are important differences. %
Let %
\begin{equation*}
\mathbb{U}_{m}(s,t)=m^{-1}\big((\lfloor mt\rfloor -\lfloor ms\rfloor )\vee 2%
\big)^{2}U_{m}\big(\overline h;\lfloor ms\rfloor ,\lfloor mt\rfloor \big),\quad 0\leq
s\leq t, 
\end{equation*}%
and let $\mathbb{U}_{m,L}(s,t)$ be as in \eqref{uml}. For any $0\leq s\leq t$%
, let 
\begin{equation}
\mathcal{U}_{m}(s,t)=\frac{\mathbb{U}_{m}(s,t)}{g_{m}(\lfloor mt\rfloor \vee
2)},\quad \mathcal{U}_{m,L}(s,t)=\frac{\mathbb{U}_{m,L}(s,t)}{g_{m}(\lfloor
mt\rfloor \vee 2)}, \notag
\end{equation}%
and for any real-valued function $\{u(s,t),s,t\geq 0\}$, write 
\begin{equation}
\Psi u(t)=\sup_{0\leq s\leq t}|u(s,t)|,\quad t>0.  \label{e:def_Psi}
\end{equation}%
With $\mathbb{V}_{m,L}$ as defined in Lemma \ref{l:approxunderH0}, set ${%
\mathcal{V}}_{m,L}(s,t)=\mathbb{V}_{m,L}(s,t)/g(t)$. Lemma \ref%
{l:approxunderH0} gives, for any $0<\delta <T$, 
\begin{equation*}
\sup_{\delta \leq t\leq T}|\Psi \mathcal{U}_{m,L}(t)-\Psi \mathcal{V}%
_{m,L}(t)|\leq \sup_{s,t\in I_{\delta ,T}}|\mathcal{U}_{m,L}(s,t)-{\mathcal{V%
}}_{m,L}(s,t)|=o_{\P }(1). 
\end{equation*}%
We again have from Lemma \ref{l:remainder_neglig_0} 
\begin{align}
& \limsup_{L\rightarrow \infty }\sup_{m\geq 1}\P \left\{ \sup_{\delta \leq
t<\infty }|\Psi \mathcal{U}_{m}(t)-\Psi \mathcal{U}_{m,L}(t)|>x\right\}
\notag\\
& \qquad \leq \limsup_{L\rightarrow \infty }\sup_{m\geq 1}\P \left\{
\sup_{0\leq s\leq t<\infty }|\mathcal{U}_{m}(s,t)-\mathcal{U}%
_{m,L}(s,t)|>x\right\} .
\end{align}%
With ${\mathcal{V}}(s,t)=\mathbb{V}(s,t)/g(t)$, Lemma \ref%
{l:trunctail_of_limit} shows ${\mathcal{V}}$ admits a version ${\mathcal{V}}%
\in \mathbf{C}([0,T]^2)$; 
hence for any fixed $m$, and any $T>0$, 
\begin{equation*}
\{\Psi {\mathcal{V}}_{m,L}(t),t\geq 0\}\Rightarrow \{\Psi {\mathcal{V}}%
(t),t\geq 0\}\qquad \text{in}\quad \mathbf{C}[0,T],\quad L\rightarrow \infty
, 
\end{equation*}%
which combined with \eqref{e:sko_int} and \eqref{e:VL_to_V} implies%
\begin{equation}
\{\Psi \mathcal{U}_{m}(t),~t\geq 0\}\Rightarrow \{\Psi {\mathcal{V}}%
(t),~t\geq 0\}\quad \text{in}\quad \mathbf{D}[\delta ,T].
\label{e:cU_m_to_cV_page}
\end{equation}%
On the other hand, Lemma \ref{l:trunctail_of_limit} implies 
\begin{equation}
\sup_{0\leq t\leq \delta }\Psi {\mathcal{V}}(t)\rightarrow 0,\quad \text{a.s.
}\quad \delta \rightarrow 0.  \label{e:V_a_to_V_page}
\end{equation}%
Further, by Lemma \ref{l:remainder_neglig_0}, expression %
\eqref{e:uniform_limit_in_delta}, %
\begin{equation}
\lim_{\delta \rightarrow 0}\limsup_{m\rightarrow \infty }\P \left\{
\sup_{0\leq t\leq \delta }\Psi \mathcal{U}_{m}(t)>x\right\} =0.
\label{e:uniform_limit_in_a_page}
\end{equation}%
Combining \eqref{e:cU_m_to_cV_page}, \eqref{e:V_a_to_V_page}, and %
\eqref{e:uniform_limit_in_a_page} gives, for every $T>0$ 
\begin{equation}
\{\Psi \mathcal{U}_{m}(t),~t\geq 0\}\Rightarrow \{\Psi {\mathcal{V}}%
(t),~t\geq 0\}\quad \text{in}\quad \mathbf{D}[0,T].  \label{e:PsiUm_to_PsiV}
\end{equation}%

Now, recall $\mathcal H_m$ from Lemma~\ref{l:tail_variable} and
$\mathcal H$ from the proof of Theorem~\ref{t:theorem_under_H0}. As before,
$\mathcal H_m\Rightarrow \mathcal H$, and the convergence in
\eqref{e:PsiUm_to_PsiV} holds jointly with $\mathcal H_m\Rightarrow\mathcal H$.

Moreover, Lemma~\ref{l:tail_variable} gives, for every $x>0$,
$$
\lim_{T\to\infty}\limsup_{m\to\infty}
\P\left\{\sup_{t\geq T}\left|\Psi\mathcal U_m(t)-|\mathcal H_m|\right|>x\right\}=0.
$$
Indeed, for $t=k/m$,
$$
\Psi\mathcal U_m(t)
=\max_{0\leq r\leq k-2}\left|\mathcal U_m(r/m,k/m)\right|,
$$
and
$$
\max_{0\leq r\leq k-2}a_m(r,k)=a_m(0,k),\qquad\sup_{k\geq mT}|a_m(0,k)-1|\to0,
\quad T\to\infty.
$$
Also, since $\Psi\mathcal V(t)\to |\mathcal H|$ a.s. as $t\to\infty$,
$$
\sup_{t\geq T}\left|\Psi\mathcal V(t)-|\mathcal H|\right|
\to0
\qquad\text{a.s.},\qquad T\to\infty.
$$
Then, appealing again to the Dudley-Wichura-Skorokhod theorem, for each fixed $T>0$ we may assume
that, for a sequence $(\Psi \mathcal V^{(m)},\mathcal H^{(m)})\stackrel d = (\Psi\mathcal V,\mathcal H)$, 
$$
\sup_{0\leq t\leq T}
|\Psi\mathcal U_m(t)-\Psi\mathcal V^{(m)}(t)|=o_\P(1),
\qquad
|\mathcal H_m-\mathcal H^{(m)}|=o_\P(1).$$
Therefore,

\begin{align*}
\left|
\sup_{t\geq0}\Psi\mathcal U_m(t)-\sup_{t\geq0}\Psi\mathcal V^{(m)}(t)
\right|&\leq
\sup_{0\leq t\leq T}
|\Psi\mathcal U_m(t)-\Psi\mathcal V^{(m)}(t)| \\
&\quad+
\sup_{t\geq T}
\left|\Psi\mathcal U_m(t)-|\mathcal H_m|\right| \\
&\quad+|\mathcal H_m-\mathcal H^{(m)}| \\
&\quad+
\sup_{t\geq T}
\left|\Psi\mathcal V^{(m)}(t)-|\mathcal H^{(m)}|
\right|.
\end{align*}
The first and third terms tend to zero in probability for any fixed $T$, while the second and
fourth terms are made arbitrarily small for $T$ large enough.  Hence
$$
\sup_{t\geq0}\Psi\mathcal U_m(t)\Rightarrow\sup_{t\geq0}\Psi\mathcal V(t).
$$

From \eqref{e:checking_covariance}, writing 
\begin{equation*}
H(s,t)=\sum_{\ell =1}^{\infty }\lambda _{\ell }\bigg[\left( W_{\ell }\left( 
\frac{t}{1+t}\right) -\frac{1+s}{1+t}W_{\ell }\left( \frac{s}{1+s}\right)
\right) ^{2}-\frac{(t-s)(1+t-s)}{(1+t)^{2}}\bigg], 
\end{equation*}%
we have 
\begin{align}
\sup_{t\geq 0}|\Psi {\mathcal{V}}(t)|& \overset{{\mathcal{D}}}{=}\sup_{0\leq
s\leq t<\infty }\left( \frac{t}{1+t}\right) ^{-\beta }\left\vert
H(s,t)\right\vert  \notag \\
& =\sup_{0<u\leq v<1}v^{-\beta }\left\vert H\left( \frac{u}{1-u},\frac{v}{1-v%
}\right) \right\vert  \notag \\
& =\sup_{0<u\leq v<1}v^{-\beta }\bigg|\sum_{\ell =1}^{\infty }\lambda _{\ell
}\bigg[\left( W_{\ell }\left( v\right) -\frac{1-v}{1-u}W_{\ell }\left(
u\right) \right) ^{2}  \notag \\
& \qquad \qquad \qquad \qquad \qquad \qquad -\left( v-u\left( \frac{1-v}{1-u}%
\right) \right) \left( 1-u\left( \frac{1-v}{1-u}\right) \right) \bigg].
\label{e:change_of_vari_to_Gamma(t)_page}
\end{align}%
The proof of parts \textit{(ii)} and \textit{(iii)} are similar to the
proofs of Theorem \ref{t:theorem_under_H0}\textit{(ii)}-\textit{(iii)} and
thus omitted.
\end{proof}

\begin{proof}[Proof of Theorem \protect\ref{alt:consistency}]
Set
$$d_m=|\mathfrak D_h(F,G)|.$$
By the definition of $\nu_1$ and $\nu_2$, $|\nu_1-\nu_2|=\theta d_m$, and  Assumptions \ref{a:basic_HA} and \ref{a:2ndHa_cond}  give
$$
m\theta^2 d_m\to\infty,\qquad 
\sigma+\sigma_*=O(d_m^{1/2}).
$$
Now, first suppose $k_*=O(m)$, and write
$$
y=y_m=2k_*\vee m.
$$
Then $y=O(m)$, $y>k_*$, and $y-k_*\asymp m$. By Lemma
\ref{l:replace_with_q1_q2}, applied with $r=0$,
\begin{equation}
\left|
\frac{y^{2}U_m(\overline h;0,y)}{m g_m(y)}-\frac{\mathcal q_1(0,y)+\mathcal q_2(0,y)+q_{3}(0,y)}{m g_m(y)}
\right|
=O_{\P}(1).
\label{e:application_of_replaceq1q2}
\end{equation}
Also, by the definition of $\mathcal q_1$,
\begin{equation}
\frac{|\mathcal q_1(0,y)|}{m g_m(y)}
\geq
C m\theta|\nu_1-\nu_2|=C m\theta^2d_m .
\label{q1_easybound}
\end{equation}
Next we have
\begin{align}
\frac{|\mathcal q_2(0,y)|}{m g_m(y)}
&\leq C\theta \left(
\left|\sum_{i=1}^{m}z_i\right|+\left|\sum_{i=m+1}^{m+k_*}z_i\right|+\left|\sum_{i=m+k_*+1}^{m+y}z_i^*\right|
\right) \notag\\
\notag&=O_{\P}\left(\theta m^{1/2}(\sigma+\sigma_*)\right)\\
&=O_{\P}(\theta m^{1/2}d_m^{1/2})
=o_{\P}(m\theta^2d_m),
\label{q2_easybound}
\end{align}
where the last equality follows from $m\theta^2d_m\to\infty$. Finally, Lemma
\ref{l:size_of_q3} gives
\begin{equation}
\frac{|q_{3}(0,y)|}{m g_m(y)}=O_{\P}(1)=o_{\P}(m\theta^2d_m).
\label{q3_easybound}
\end{equation}
Combining \eqref{e:application_of_replaceq1q2}--\eqref{q3_easybound}, we obtain
$$
\frac{y^2|U_m(\overline h;0,y)|}{m g_m(y)}
\geq
C m\theta^2d_m(1+o_{\P}(1))
\overset{\P}{\longrightarrow}\infty .
$$
Therefore
$$
\frac{\mathcal D_m^{(2)}(y)}{g_m(y)}\geq \frac{\mathcal D_m^{(1)}(y)}{g_m(y)}=\frac{y^2|U_m(\overline h;0,y)|}{m g_m(y)}
\overset{\P}{\longrightarrow}\infty,
$$
Thus $\P(\tau_m<\infty)\to1$ for both monitoring schemes when $k_*=O(m)$. Now suppose that $m=o(k_*)$, and set
$
y=2k_*.$
Then Lemma \ref{l:replace_with_q1_q2}, applied with $r=0$, again gives
\eqref{e:application_of_replaceq1q2}. Since $C(k_*/m)^2\leq g_m(2k_*)\leq C'(k_*/m)^2$, the
same calculation as above gives \eqref{q1_easybound}. Moreover, we again find
\begin{align}
\frac{|\mathcal q_2(0,2k_*)|}{m g_m(2k_*)}
&\leq
C\theta\left(
\left|\sum_{i=1}^{m}z_i\right|
+
\frac{m}{k_*}\left|\sum_{i=m+1}^{m+k_*}z_i\right|
+
\frac{m}{k_*}\left|\sum_{i=m+k_*+1}^{m+2k_*}z_i^*\right|
\right) \notag\\
&=
O_{\P}\left(\theta m^{1/2}(\sigma+\sigma_*)\right)=o_{\P}(m\theta^2d_m),
\label{q2_easybound_late}
\end{align} and Lemma \ref{l:size_of_q3} again gives \eqref{q3_easybound}. 
Consequently,
$\mathcal D_m^{(2)}(2k_*)/g_m(2k_*)\geq \mathcal D_m^{(1)}(2k_*)/g_m(2k_*)
\overset{\P}{\longrightarrow}\infty,$
Hence $\P(\tau_m<\infty)\to1$ also when $m=o(k_*)$.
\end{proof}

\begin{proof}[Proof of Theorem \protect\ref{alt:nonpage}]
We begin with part \textit{(i)}. We first proceed to find a sequence $%
y_{m}\rightarrow \infty $ for which $\P (\kappa _{m}\leq y_{m})$ has a
nontrivial limit. %
Set 
\begin{equation}
\mathcal{q}_{2,1}(r,k)=2\theta (k-(k_{* }\vee r))\sum_{i=m+k_{*
}+1}^{m+k}z_{i}^{* },  \label{e:cq_21}
\end{equation}%
If we choose $y_{m}$ in such a way that 
\begin{equation}
y_{m}\rightarrow \infty ,\qquad m^{-1}y_{m}\rightarrow 0
\label{e:ym_conditions}
\end{equation}%
then for $\mathcal{q}_{1}(r,k)$ as in \eqref{e:cq1_cq2}, %
\begin{equation}
a_{m}=\frac{|\mathcal{q}_{1}(0,y_{m})|}{mg_{m}(y_{m})}=\theta |\nu _{1}-\nu
_{2}|m\left( \frac{y_{m}/m}{1+y_{m}/m}\right) ^{2-\beta }\left(1-\frac{k_*}{y_m}\right)^2\sim \theta |\nu
_{1}-\nu _{2}|m^{\beta -1}y_{m}^{2-\beta }.  \label{a_mq}
\end{equation}
Under \eqref{e:ym_conditions}, we also have 
\begin{align}
\frac{\mathcal{q}_{2,1}(0,y_m)}{mg_{m}(y_m)}& =\mathcal{q}_{2,1}(0,y_m)m^{\beta
-1}y_{m}^{-\beta }(1+y_{m}/m)^{\beta -2}  \notag \\
& =2\theta m^{\beta -1}y_{m}^{1-\beta }\left( \sum_{i=m+k_{*
}+1}^{m+y_{m}}z_{i}^{* }\right) \frac{(1-k_{* }/y_{m})}{%
(1+y_{m}/m)^{2-\beta }}  \notag \\
& =b_{m}\left( \frac{1}{\sigma _{* }y_{m}^{1/2}}\sum_{i=m+k_{*
}+1}^{m+y_{m}}z_{i}^{* }\right) \frac{(1-k_{* }/y_{m})}{%
(1+y_{m}/m)^{2-\beta }},  \label{e:form_of_q2}
\end{align}%
with 
\begin{equation}
b_{m}=b_{m}(y_{m})=2\sigma _{* }\theta m^{\beta -1}y_{m}^{{3/2}-\beta }.
\label{e:bm_case1}
\end{equation}%
With $\rho =(1-\beta )/(2-\beta )$, we may pick $y_{m}$ satisfying %
\eqref{e:ym_conditions} as a solution to 
\begin{equation*}
y_{m}=k_{* }+w_{1}m^{\rho }(1+w_{2}b_{m}(y_{m})), 
\end{equation*}%
where $w_{1}^{2-\beta }=\mathcal{c}(\theta |\nu _{1}-\nu _{2}|)^{-1}$, and $%
w_{2}$ is a constant to be later specified such that 
\begin{equation}
a_{m}\rightarrow \mathcal{c},\quad b_{m}^{-1}(\mathcal{c}-a_{m})\rightarrow
-x.  \label{e:ym_conditions2}
\end{equation}%
Indeed, since $\theta |\nu _{1}-\nu _{2}|m\rightarrow \infty $ under Assumption \ref{a:basic_HA} and $(\mathcal{c}\theta |\nu _{1}-\nu
_{2}|)^{1/2}\sim C\theta \sigma _{* }$ under Assumption \ref{a:2ndHa_cond}, 
\begin{align*}
b_{m}(2w_{1}m^{\rho })& =C(\theta |\nu _{1}-\nu _{2}|)^{-(\rho
+1)/2}m^{-\rho /2}\sigma _{* }\theta \\
& =C(m\theta |\nu _{1}-\nu _{2}|)^{-\rho /2}\sigma _{* }\theta /(\theta
|\nu _{1}-\nu _{2}|)^{1/2}=o(1).
\end{align*}%
Thus, the function $\varphi (y)=y-k_{* }-w_{1}m^{\rho }(1+w_{2}b_{m}(y))$
is easily seen to have a root in the region $(k_*,k_*+2w_{1}m^{\rho })$ for all
large $m$ and any fixed $w_{2}$, which satisfies $y_{m}\sim w_{1}m^{\rho }$
and in particular satisfies \eqref{e:ym_conditions}. 
From \eqref{e:ym_conditions2} we obtain 
\begin{align}
\P \left\{ \kappa _{m}<y_{m}\right\} & =\P \left\{ \max_{k_{* }<k\leq
y_{m}}\frac{k^{2}|U_{m}(h;0,k)|}{mg_{m}(k)}>\mathcal{c}\right\}  \notag \\
& =\P \left\{ b_{m}^{-1}\left( \max_{k_{* }<k\leq y_{m}}\frac{%
k^{2}|U_{m}(h;0,k)|}{mg_{m}(k)}-a_{m}\right) >-x+o(1)\right\} .
\label{e:normal_prelimit}
\end{align}%
Recall the decomposition \eqref{e:U_as_q}. Applying Lemma \ref%
{l:replace_with_q1_q2}, we have 
\begin{align}
b_{m}^{-1}\max_{k_{* }<k\leq y_{m}}\frac{\max_{0\leq b<k}|q_{1}(0,k)-%
\mathcal{q}_{1}(0,k)|}{mg_{m}(k)}& \leq Cb_{m}^{-1}\theta |\nu _{1}-\nu
_{2}|(k_{* }/m)^{1-\beta }  \notag \\
& =C\sigma _{* }^{-1}|\nu _{1}-\nu _{2}|(k_{* }/y_{m})^{(1-\beta
)}y_{m}^{-1/2}  \notag \\
& =o(1),  \label{e:q1_swap}
\end{align}%
and for all small $\delta >0$, 
\begin{align}
b_{m}^{-1}\max_{k_{* }<k\leq y_{m}}\frac{\max_{0\leq b<k}|q_{2}(0,k)-%
\mathcal{q}_{2}(0,k)|}{mg_{m}(k)}& \leq Cb_{m}^{-1}\theta k_{*
}^{-1}(k_{* }/m)^{1-\beta }O_{\P }(y_{m}^{1/2+\delta })  \notag \\
& =O_{\P }\big(\theta k_{* }^{-1}(k_{* }/y_{m})^{1-\beta
}y_{m}^{\delta }\big)=o_{\P }(1).  \label{e:q2_swap}
\end{align}%
Also, 
\begin{align}
& b_{m}^{-1}\max_{k_{* }<k\leq y_{m}}\frac{\max_{0\leq b<k}|\mathcal{q}%
_{2,1}(0,k)-\mathcal{q}_{2}(0,k)|}{mg_{m}(k)} \\
& \qquad \leq Cb_{m}^{-1}\max_{k_{* }<k\leq y_{m}}\theta (k/m)^{1-\beta
}\left\vert -\frac{k}{m}\sum_{i=1}^{m}z_{i}+\sum_{i=m+1}^{m+k_{*
}}z_{i}\right\vert  \notag \\
& \leq C\theta b_{m}^{-1}y_{m}^{1-\beta }m^{\beta -1}\sigma \left( O_{\P %
}(y_{m}m^{-1/2})+O_{\P }(m^{1/2})\right)  \notag \\
& \leq Cy_{m}^{-1/2}\left( O_{\P }(y_{m}m^{-1/2})+O_{\P }(1)\right) =o_{\P %
}(1),  \notag
\end{align}%
and from Lemma \ref{l:size_of_q3}, 
\begin{equation}
b_{m}^{-1}\max_{k_{* }<k\leq y_{m}}\frac{|q_{3}(0,k)|}{mg_{m}(k)}\leq
Cb_{m}^{-1}(y_{m}/m)^{(1-\beta )}=O_{\P }\left( \frac{1}{%
\sigma _*\theta y^{1/2}}\right) =o_{\P }(1),  \label{e:q21_swap}
\end{equation}
where we used Assumption \ref{a:2ndHa_cond} to conclude $\sigma
_{* }\theta y^{1/2}\sim C (m \theta |\nu_1-\nu_2|)^{\rho/2}\rightarrow \infty $.
From the bounds \eqref{e:q1_swap}--\eqref{e:q21_swap}, in view of %
\eqref{e:normal_prelimit}, it suffices to show 
\begin{equation}
b_{m}^{-1}\left( \max_{k_{* }<k\leq y_{m}}\frac{|\mathcal{q}_{1}(0,k)+%
\mathcal{q}_{2,1}(0,k)|}{mg_{m}(k)}-a_{m}\right) \Rightarrow \mathcal{N}%
(0,1).  \label{e:q1+q21_to_normal}
\end{equation}%
So, note (c.f. \eqref{a_mq}) ${|\mathcal{q}_{1}(0,k)|}/{mg_{m}(k)}$ is
increasing in $k$. Hence, for any $0\leq \delta <1$, 
\begin{equation}
\max_{ k_*\leq k\leq y_m(1-\delta) }\frac{|\mathcal{q}_1(0,k)|}{mg_m(k)}%
=\theta|\nu_1-\nu_2|m^{\beta-1}\left[y_m(1-\delta)\right]^{2-\beta}%
\left(1+o(1)\right),  \label{a73-1}
\end{equation}

\begin{equation}
\min_{ (1-\delta)y_m\leq k\leq y_m }\frac{|\mathcal{q}_1(0,k)|}{mg_m(k)}%
=\theta|\nu_1-\nu_2|m^{\beta-1}\left[y_m(1-\delta)\right]^{2-\beta}%
\left(1+o(1)\right).  \label{a73-2}
\end{equation}
Also, from \eqref{e:form_of_q2}, for all $k_{* }/y_{m}<s<1$, 
\begin{equation}
b_{m}^{-1}\frac{\mathcal{q}_{2,1}(0,\lfloor y_{m}s\rfloor )}{mg_{m}(\lfloor
y_{m}s\rfloor )}=\left( \frac{\lfloor y_{m}s\rfloor }{y_{m}}\right)
^{1-\beta }\left( \frac{1}{y_{m}^{1/2}\sigma_*}\sum_{i=m+k_{* }+1}^{m+\lfloor
y_{m}s\rfloor }z_{i}^{* }\right) \frac{(1-k_{* }/\lfloor y_{m}s\rfloor
)}{(1+\lfloor y_{m}s\rfloor /m)^{2-\beta }},  \label{a74}
\end{equation}
the functional central limit theorem gives 
\begin{equation*}
b_{m}^{-1}\frac{\mathcal{q}_{2,1}\big(0,\lfloor y_{m}s\rfloor \vee (k_{*
}+1)\big)}{mg_{m}\big(\lfloor y_{m}s\rfloor \vee (k_{* }+1)\big)}%
\Rightarrow s^{1-\beta }W(s)\quad \text{in}\quad \mathbf{D}[0,1], 
\end{equation*}%
where $\{W(s),s\geq 0\}$ is a Wiener process. %
Now, from \eqref{a73-1} and \eqref{a73-2}, 
\begin{align}
& b_{m}^{-1}\left( \max_{k_{* }<k\leq y_{m}(1-\delta )}\frac{|\mathcal{q}%
_{1}(0,k)+\mathcal{q}_{2,1}(0,k)|}{mg_{m}(k)}-a_{m}\right)  \notag \\
& \leq O_{\P }(1)+(2\sigma _{* })^{-1}\left( [1-\delta ]^{2-\beta
}-1\right) |\nu _{1}-\nu _{2}|y_{m}^{1/2}(1+o(1))\overset{\P }{\rightarrow }%
-\infty .  \label{e:alt_to_minusinf}
\end{align}%
On the other hand, if we let 
\begin{equation*}
A_{m}=\left\{ \omega :b_{m}^{-1}\max_{(1-\delta )y\leq k\leq y}\frac{|%
\mathcal{q}_{2,1}(0,k)|}{mg_{m}(k)}<b_{m}^{-1}\min_{(1-\delta )y\leq k\leq y}%
\frac{|\mathcal{q}_{1}(0,k)|}{mg_{m}(k)}\right\} , 
\end{equation*}%
then \eqref{a73-1}, \eqref{a73-2} and \eqref{a74} give $\P %
(A_{m})\rightarrow 1$, and for each $\omega \in A_{m}$, 
\begin{equation}
\frac{|\mathcal{q}_{1}(0,k)+\mathcal{q}_{2,1}(0,k)|}{mg_{m}(k)}=\frac{|%
\mathcal{q}_{1}(0,k)|}{mg_{m}(k)}+\text{sgn}(\mathcal{q}_{1}(0,k))\frac{%
\mathcal{q}_{2,1}(0,k)}{mg_{m}(k)},\quad (1-\delta )y_{m}\leq k\leq y_{m}.
\label{e:on_the_set_Am}
\end{equation}%
Note $\text{sgn}(\mathcal{q}_{1}(r,k))=-\text{sgn}(\nu _{1}-\nu _{2})$. Thus,
if $(\nu _{1}-\nu _{2})<0$,%
\begin{align*}
& \P \left( A_{m}\cap \left\{ b_{m}^{-1}\left( \max_{y_{m}(1-\delta )\leq
k\leq y_{m}}\frac{|\mathcal{q}_{1}(0,k)+\mathcal{q}_{2,1}(k)|}{mg_{m}(k)}-%
\frac{|\mathcal{q}_{1}(y_{m})+\mathcal{q}_{2,1}(0,y_{m})|}{mg_{m}(y_{m})}%
\right) >x\right\} \right) \\
& =\P \left( A_{m}\cap \left\{ b_{m}^{-1}\left( \max_{y_{m}(1-\delta )\leq
k\leq y_{m}}\frac{\mathcal{q}_{1}(0,k)+\mathcal{q}_{2,1}(0,k)}{mg_{m}(k)}-%
\frac{\mathcal{q}_{1}(y_{m})+\mathcal{q}_{2,1}(0,y_{m})}{mg_{m}(y_{m})}%
\right) >x\right\} \right) \\
& \leq \P \left( b_{m}^{-1}\left( \max_{y_{m}(1-\delta )\leq k\leq y_{m}}%
\left[ \frac{\mathcal{q}_{1}(0,k)}{mg_{m}(k)}-\frac{\mathcal{q}_{1}(0,y_{m})%
}{mg_{m}(y_{m})}\right] +\max_{y_{m}(1-\delta )\leq k\leq y_{m}}\left[ \frac{%
\mathcal{q}_{2,1}(0,k)}{mg_{m}(k)}-\frac{\mathcal{q}_{2,1}(0,y_{m})}{%
mg_{m}(y_{m})}\right] \right) >x\right) \\
& =\P \left\{ b_{m}^{-1}\max_{y_{m}(1-\delta )\leq k\leq y_{m}}\left\vert 
\frac{\mathcal{q}_{2,1}(0,k)}{mg_{m}(k)}-\frac{\mathcal{q}_{2,1}(0,y_{m})}{%
mg_{m}(y_{m})}\right\vert >x/2\right\} \\
& \rightarrow \P \left\{ \sup_{(1-\delta )\leq s\leq 1}\left\vert s^{1-\beta
}W(s)-W(1)\right\vert >x/2\right\} ,
\end{align*}%
where on the third line we used increasingness of ${\mathcal{q}_{1}(0,k)}/{%
mg_{m}(k)} $ and on the last line we used \eqref{a74}. Analogous reasoning
holds in the case $\nu _{1}-\nu _{2}>0$. Thus, by continuity of $W$, 
\begin{equation}
\lim_{\delta \rightarrow 0}\limsup_{m\rightarrow \infty }\P \left\{
b_{m}^{-1}\left( \max_{y_{m}(1-\delta )\leq k\leq y_{m}}\frac{|\mathcal{q}%
_{1}(0,k)+\mathcal{q}_{2,1}(0,k)|}{mg_{m}(k)}-\frac{|\mathcal{q}_{1}(0,y)+%
\mathcal{q}_{2,1}(0,y)|}{mg_{m}(y)}\right) >x\right\} =0.
\label{e:bycontinuityofW}
\end{equation}%
Now, from \eqref{e:on_the_set_Am}, 
\begin{equation}
b_{m}^{-1}\left( \frac{|\mathcal{q}_{1}(0,y_{m})+\mathcal{q}_{2,1}(0,y_{m})|%
}{mg_{m}(y_{m})}-a_{m}\right) =-b_{m}^{-1}\text{sgn}(\nu _{1}-\nu _{2})\frac{%
\mathcal{q}_{2,1}(0,y_{m})}{mg_{m}(y_{m})}+o_{\P }(1)\Rightarrow W(1),
\label{e:q21_to_W(1)}
\end{equation}%
which, together with \eqref{e:alt_to_minusinf} and \eqref{e:bycontinuityofW}
yields the limit \eqref{e:q1+q21_to_normal}. %
From \eqref{e:normal_prelimit}, we then obtain 
\begin{equation}
\P \{\kappa _{m}<y_{m}\}\rightarrow 1-\Phi (-x)=\Phi (x),\quad m\rightarrow
\infty .  \label{e:tau_to_normal}
\end{equation}%
Now we choose $w_{2}$ so that \eqref{e:ym_conditions2} holds. Note $%
w_{1}^{2-\beta }=\mathcal{c}(\theta |\nu _{1}-\nu _{2}|)^{-1}$ clearly gives 
$a_{m}\rightarrow \mathcal{c}$. Moreover, 
\begin{align*}
\mathcal{c}-a_{m}& =\mathcal{c}-\theta |\nu _{1}-\nu _{2}|m^{\beta
-1}y_{m}^{2-\beta }(1+y_{m}/m)^{\beta -2}(1-k_{* }/y_{m})^{2} \\
& =\mathcal{c}-\mathcal{c}(1+w_{2}b_{m})^{2-\beta }(1+y_{m}/m)^{\beta
-2}(1-k_{* }/y_{m})^{\beta} \\
& =-\mathcal{c}(2-\beta )w_{2}b_{m}+o(b_{m}).
\end{align*}%
where we used that $b_{m}\gg y_{m}m^{-1}$ since (noting that $\theta \sigma
_{* }\sim C(\theta |\nu _{1}-\nu _{2}|)^{1/2}\sim Cw_{1}^{-(2-\beta )/2}$
under Assumption \ref{a:2ndHa_cond}), 
\begin{align*}
b_{m}\gg y_{m}m^{-1}& \iff \sigma _{* }\theta m^{-\beta }\left(
w_{1}m^{\rho }\right) ^{1/2-\beta }\rightarrow \infty \\
& \iff m^{\rho /2+\beta (1-\rho )}w_{1}^{-(1+\beta )/2}\rightarrow \infty ,
\\
& \iff \left( m\theta |\nu _{1}-\nu _{2}|\right) ^{(1+\beta )/(4-2\beta
)}\rightarrow \infty ,
\end{align*}%
which holds under Assumption \ref{a:basic_HA}, and also we used that $%
b_{m}\gg k_{* }y_{m}^{-1}$, which holds since 
\begin{align*}
b_{m}\gg k_{* }y_{m}^{-1}& \iff \sigma _{* }\theta m^{-(2-\beta )\rho
}\left( w_{1}m^{\rho }\right) ^{5/2-\beta }\rightarrow \infty \\
& \iff m^{\rho /2}w_{1}^{(3-\beta )/2}\rightarrow \infty ,
\end{align*}%
which always holds. So, choosing $w_{2}=(\mathcal{c}(2-\beta ))^{-1}x$, we
obtain the second statement in \eqref{e:ym_conditions2}, implying %
\eqref{e:tau_to_normal} holds for the sequence $y_{m}$. %
Now, since $y_{m}\sim w_{1}m^{\rho }$, we have 
\begin{align*}
y_{m}-k_{* }-w_{1}m^{\rho }& =w_{1}w_{2}m^{\rho }b_{m} \\
& \sim w_{1}w_{2}m^{\rho }m^{-\rho (2-\beta )}(w_{1}m^{\rho })^{{3/2}-\beta }
\\
& =(2\sigma _{* }\theta )w_{2}w_{1}^{2-\beta }(w_{1}m^{\rho })^{1/2} \\
& =\frac{2\sigma _{* }x}{(2-\beta )|\nu _{1}-\nu _{2}|}(w_{1}m^{\rho
})^{1/2},
\end{align*}%
from which we obtain 
\begin{equation*}
\P \{\kappa _{m}<y_{m}\}\sim \P \left\{ \frac{(2-\beta )|\nu _{1}-\nu _{2}|}{%
2\sigma _{* }}\frac{\kappa _{m}-k_*-w_{1}m^{\rho }}{(w_{1}m^{\rho })^{1/2}}%
<x\right\} \rightarrow \Phi (x). 
\end{equation*}%
We now turn to part \textit{(ii)}. First we set up some notation used in the
proof and define the limit variable appearing \eqref{e:limit_cH(c)}. Let $%
\{W_{1}(t),t\geq 0\},$ $\{W_{2}(t),t\geq 0\}$,\ldots be independent Wiener
processes, and let $\{V_{1}(t),t\geq 0\},$ $\{V_{2}(t),t\geq 0\},$ each be
Wiener processes with 
\begin{equation}
\begin{gathered} {\E}\hspace{0.1mm} V_1(t)V_2(t)=0, \quad {\mathsf
E}\hspace{0.1mm} V_1(t)W_{1,\ell}(t)=\eta_\ell t, \quad {\mathsf
E}\hspace{0.1mm} V_1(t)W_{2,\ell}(t)=0,\\ {\E}\hspace{0.1mm}
V_2(t)W_{1,\ell}(t)=0,\quad {\E}\hspace{0.1mm}
V_2(t)W_{2,\ell}(t)=\eta_\ell t, \end{gathered}
\label{e:define_limitprocess_underHA}
\end{equation}%
where, with $v(\mathbf{x})$ as in \eqref{e:def_v(x)}, and $\phi _{\ell }(%
\mathbf{x})$ as in \eqref{e:h(xy)=L2_expansion}, 
\begin{equation*}
\sigma ^{-1}{\mathsf{E}}\hspace{0.1mm}v(\mathbf{X}_{1})\phi _{\ell }(\mathbf{%
X}_{1})=\eta _{\ell }. 
\end{equation*}%
Also, let %
\begin{equation*}
Y(t,c_{* })=\frac{ t^{2}+2\zeta t\left( V_{2}(c_{* })-c_{*
}V_{1}(1)\right) +\mathbb{V}(0,c_{* })}{g(c_{* })},  
\end{equation*}%
with $\mathbb{V}(s,t)$ as in \eqref{vst}. Finally, we define 
\begin{equation}
\mathcal{H}_{c_{* }}(u)=\inf \left\{ x\geq 0:\sup_{0\leq t\leq
x}|Y(t,c_{* })|\geq u\right\} ,  \label{e:def_cH(u)}
\end{equation}%
i.e,. $\mathcal{H}_{c_{* }}(u)$ is the left-continuous inverse of $%
x\mapsto \sup_{0\leq t\leq x}|Y(t,c_{* })|$. We are now ready to proceed
with the proof.

For simplicity write $\Delta=\theta|\nu_1-\nu_2|=\theta^2|\mathfrak{D}%
_h(F,G)|$. We first show, for any $T>0$, (c.f. \eqref{e:U_as_q}) 
\begin{align}
&\max_{k_*<k\leq k_*+ (m/\Delta)^{1/2}T}\left | \frac{k^2U_m(\overline h;k)}{mg_m(k)}-%
\frac{\mathcal{q}_1(0,k) + {\mathcal{q}}_{2,2}(0,k)+ \mathcal{q}_3(0,k_*)}{%
mg_m(k)}\right| =o_\P (1),  \label{e:nonpage_alt_wefirstshow_longchange}
\end{align}
where $\mathcal{q}_1(r,k)$ is given in \eqref{e:cq1_cq2}, and 
\begin{align}
{\mathcal{q}}_{2,2}(r,k) &= 2\theta (k-(k_*\vee r)) \left[-\frac{k-r}%
m\sum_{i=1}^m z_i + \mathbf{1}_{\{r<k_*\}}\sum_{i=m+r+1}^{m+k_*} z_i\right] 
\notag \\
& =: 2\theta \sigma (k-(k_*\vee r)) \left[-\frac{k-r}m V_{1,m} + V_{2,m}(r) %
\right] ,  \label{e:def_q23} \\
\mathcal{q}_3(r,k) &=(k-r)^2 \left(\frac{2\mathbf{1}_{\{r \leq
k_*\}}R_{m,1,1}(r)}{(k-r)m} - \frac{2R_{m,2}}{m(m-1)} - \frac{2 \mathbf{1}%
_{\{r \leq k_*\}}R_{m,3,1}(r) }{(k-r)(k-r-1)}\right),\notag
\end{align}
with $R_{m,1,1}(r),R_{m,2},$ and $R_{m,3,1}$ as in \eqref{e:R_m11}, %
\eqref{e:R_m2} and \eqref{e:R_m31}.

Lemma \ref{l:replace_with_q1_q2} immediately gives 
\begin{equation}
\max_{k_{* }<k\leq k_{* }+(m/\Delta )^{1/2}T}\max_{0\leq r<k}\left( 
\frac{|\mathcal{q}_{1}(r,k)-q_{1}(r,k)|+|\mathcal{q}_{2}(r,k)-q_{2}(r,k)|}{%
mg_{m}(k)}\right) =o_{\P }(1).  \label{e:approx_q1}
\end{equation}%
With $\mathcal{q}_{2}(r,k)$ in \eqref{e:cq1_cq2}, we have 
\begin{equation}
\max_{k_{* }<k\leq k_{* }+(m/\Delta )^{1/2}T}\max_{0\leq r<k}\frac{|%
\mathcal{q}_{2}(r,k)-{\mathcal{q}}_{2,2}(r,k)|}{mg_{m}(k)}=o_{\P }(1).
\label{e:approx_q23}
\end{equation}%
Indeed, for any $T>0$, the law of the iterated logarithm gives 
\begin{align*}
& \max_{k_{* }<k\leq k_{* }+(m/\Delta )^{1/2}T}\max_{0\leq r<k}\frac{|%
\mathcal{q}_{2}(r,k)-{\mathcal{q}}_{2,2}(r,k)|}{mg_{m}(k)} \\
& \qquad \leq \max_{k_{* }\leq k\leq k_{* }+T(m/\Delta )^{1/2}}\frac{%
\theta (k-k_{* })}{mg_{m}(k)}\max_{k_{* }\leq r<k}\left\vert
\sum_{i=m+r+1}^{m+k}z_{i}^{* }\right\vert \\
& \quad \leq C\theta m^{-1/2}\max_{k_{* }\leq k\leq k_{* }+(m/\Delta
)^{1/2}T}\left\vert \sum_{i=m+k_{* }+1}^{m+k}z_{i}^{* }\right\vert \\
& \quad =O_{\P }\left( \theta \sigma _{* }m^{-1/2}\big((m/\Delta
)^{1/2}\log \log (m/\Delta )\big)^{1/2}\right) \\
& \quad =O_{\P }\left( \big((m/\Delta )^{-1/2}\log \log (m/\Delta )\big)%
^{1/2}\right) ,
\end{align*}%
where we used that $\sigma _{* }^{2}\sim C|\mathfrak{D}_{h}(F,G)|\sim
C|\Delta |/\theta ^{2}$ due to Assumption \ref{a:2ndHa_cond}, giving %
\eqref{e:approx_q23}. Applying Lemma \ref{e:q3_statiticterms_only_lemma}, 
\begin{equation}
\max_{k_{* }<k\leq k_{* }+(m/\Delta )^{1/2}T}\max_{0\leq r<k}\frac{%
|q_{3}(r,k)-{\mathcal{q}}_{3}(r,k)|}{mg_{m}(k)}=o_{\P }(1).
\label{e:approx_q3}
\end{equation}%
Next, we claim that, for any $\delta >0$ 
\begin{equation}
\max_{k_{* }<k\leq k_{* }+(m/\Delta )^{1/2}T}\max_{0\leq r\leq k_{*
}-2}\frac{|\mathcal{q}_{3}(r,k)-(k_{* }-r)^{2}U_{m}(\overline{h};r,k_{*
})|}{mg_{m}(k)}=o_{\P }(1).  \label{a88}
\end{equation}
First note \eqref{e:UT(f;k)=0} implies%
\begin{equation*}
\frac{\mathcal{q}_{3}(r,k_*)}{(k_{* }-r)^{2}}=\frac{2R_{m,1,1}(r)}{%
(k_{* }-r)m}-\frac{2R_{m,2}}{m(m-1)}-\frac{2R_{m,3,1}(r)}{(k_{*
}-r)(k_{* }-r-1)}=U_{m}(\overline{h};r,k_{* }). 
\end{equation*}%
With $R_{m,1,1}(r)$ as in \eqref{e:R_m11}, we have 
\begin{align*}
& \max_{k_{* }<k\leq k_{* }+(m/\Delta )^{1/2}T}\max_{0\leq r\leq
k_{* }-2}\left\vert \frac{(k_{* }-r)^{2}}{m}\frac{R_{m,1,1}(r)}{(k_{*
}-r)m}-\frac{(k-r)^{2}}{m}\frac{R_{m,1,1}(r)}{(k-r)m}\right\vert \\
& \quad =\max_{k_{* }<k\leq k_{* }+(m/\Delta )^{1/2}T}\frac{k-k_{* }%
}{m}\max_{0\leq r\leq k_{* }}\left\vert \frac{R_{m,1,1}(r)}{m}\right\vert
\\
& \quad =O_{\P }\left( (\Delta m)^{-1/2}\right) ,
\end{align*}%
where we used that $\max_{0\leq r\leq k_{* }}|R_{m,1,1}(r)|=O_{\P }(m)$
due to Lemma \ref{l:uncorrelated_maximal}. Similarly, with $R_{m,3,1}(r)$ as
in \eqref{e:R_m31}, using the mean value theorem applied to $%
f(k)=(k-r)/(k-r-1)$, 
\begin{align*}
& \max_{k_{* }<k\leq k_{* }+(m/\Delta )^{1/2}T}\max_{0\leq r\leq
k_{* }-2}\left\vert \frac{(k-r)^{2}}{m}\frac{R_{m,3,1}(r)}{(k-r)(k-r-1)}-\frac{(k_{* }-r)^{2}}{m}\frac{R_{m,3,1}(r)}{(k_{*
}-r)(k_{* }-r-1)}\right\vert \\
& \quad \leq 2\max_{k_{* }<k\leq k_{* }+(m/\Delta )^{1/2}T}\frac{%
C(k-k_{* })}{m}\max_{0\leq r\leq k_{* }-2}\frac{|R_{m,3,1}(r)|}{(k_{*
}-r-1)^{2}} \\
& \quad =O_{\P }((\Delta m)^{-1/2}),
\end{align*}%
since $\max_{0\leq r\leq k_{* }-1}|R_{m,3,1}(r)|/(k_{* }-r)^{2}=O_{\P %
}(1)$ again due to Lemma \ref{l:uncorrelated_maximal}. Lastly, 
\begin{align}
& \max_{k_{* }<k\leq k_{* }+(m/\Delta )^{1/2}T}\max_{0\leq r\leq
k_{* }}\frac{|(k-r)^{2}-(k_{* }-r)^{2}|}{m}\frac{|R_{m,2}|}{m(m-1)} 
\notag \\
& \leq C\max_{k_{* }<k\leq k_{* }+(m/\Delta )^{1/2}T}\frac{k-k_{* }%
}{m}\frac{k_{* }|R_{m,2}|}{m(m-1)}=O_{\P }((\Delta m)^{-1/2}).
\end{align}%
Since $g_{m}(k_{* })\geq C>0$, we therefore have \eqref{a88}, which
combined with \eqref{e:approx_q1}, \eqref{e:approx_q23}, and %
\eqref{e:approx_q3} gives %
\eqref{e:nonpage_alt_wefirstshow_longchange}. Noting that $\mathcal{q}%
_{3}(0,k_{* })=k_{* }^{2}U_{m}(\overline{h};k_{* }),$ we now show 
\begin{equation*}
\max_{k_{* }<k\leq k_{* }+(m/\Delta )^{1/2}T}\frac{|\mathcal{q}%
_{1}(0,k)+{\mathcal{q}}_{2,2}(0,k)+k_{* }^{2}U_{m}(\overline{h};k_{*
})|}{mg_{m}(k)}\Rightarrow \sup_{0\leq t\leq T}|Y(t,c_{* })|. 
\end{equation*}%
For each $0\leq t\leq T$, let 
\begin{equation}
Y_{m}(t)=\frac{\mathcal{q}_{1}(0,k_{* }+\lfloor (m/\Delta )^{1/2}t\rfloor
)+\mathcal{q}_{2,2}(0,k_{* }+\lfloor (m/\Delta )^{1/2}t\rfloor )+k_{*
}^{2}U_{m}(\overline{h};k_{* })}{m}  \label{e:def_Zm(t)}
\end{equation}%
and 
\begin{equation*}
Y_{m,L}(t)=\frac{\mathcal{q}_{1}(0,k_{* }+\lfloor (m/\Delta
)^{1/2}t\rfloor )+\mathcal{q}_{2,2}(0,k_{* }+\lfloor (m/\Delta
)^{1/2}t\rfloor )+k_{* }^{2}U_{m,L}(\overline{h};k_{* })}{m}, 
\end{equation*}%
where $U_{m,L}$ is given by \eqref{e:def_UL}. Clearly, 
\begin{equation*}
\sigma ^{-1}{\mathsf{E}}\hspace{0.1mm}z_{i}\phi _{\ell }(\mathbf{X}_{i})=\sigma ^{-1}{%
\mathsf{E}}\hspace{0.1mm}v(\mathbf{X}_{i})\phi _{\ell }(\mathbf{X}_{i})=\eta
_{\ell }. 
\end{equation*}%
Hence, we deduce the joint weak convergence 
\begin{align}
& m^{-1/2}\big(S_{1}(m),\ldots ,S_{L}(m),S_{1}(\lfloor mt\rfloor ,m),\ldots
,S_{L}(\lfloor mt\rfloor ,m),V_{1,m},V_{2,m}\big)  \notag
\label{e:jointweak} \\
& \qquad \Rightarrow \big(W_{1,1}(1),\ldots ,W_{1,L}(1),W_{2,1}(t),\ldots
,W_{2,L}(t),V_{1}(1),V_{2}(c_{* })\big),\quad \text{in}\quad \mathbf{D}%
[0,T].
\end{align}%
Lemma \ref{l:remainder_neglig_1} implies (c.f. \eqref{e:UmL_w_remainder}) 
\begin{equation*}
\left\vert \frac{k_{* }^{2}U_{m,L}(\overline{h};k_{* })}{m}-\sum_{\ell
=1}^{L}\lambda _{\ell }\left( \frac{1}{m}\left( S_{\ell }(\lfloor mc_*\rfloor
,m)-\frac{\lfloor mc_*\rfloor }{m}S_{\ell }(m)\right) ^{2}-\frac{\lfloor
mc_*\rfloor (\lfloor mc_*\rfloor +m)}{m^{2}}\right) \right\vert =o_{\P }(1). 
\end{equation*}%
Hence, we deduce that as $m\rightarrow \infty $, 
\begin{align}
& Y_{m,L}(t)  \notag \\
& =-\frac{\lfloor (m/\Delta )^{1/2}t\rfloor ^{2}\theta (\nu _{1}-\nu _{2})}{m}%
+\frac{2\theta \sigma \lfloor (m/\Delta )^{1/2}t\rfloor }{m^{1/2}}%
\left[ \frac{V_{2,m}(0)}{m^{1/2}}-\frac{c_{* }m+\lfloor (m/\Delta
)^{1/2}t\rfloor }{m}\frac{V_{1,m}}{m^{1/2}}\right]  \notag \\
& \qquad \qquad \qquad +\frac{k_{* }^{2}U_{m,L}(\overline{h},k_{* })}{m%
}  \notag \\
& \Rightarrow Y_{L}(t),\quad \text{in}~\mathbf{D}[0,T],  \label{e:ZmL_to_ZL}
\end{align}%
where (recalling $\sigma\theta /\Delta ^{1/2}=\sigma/%
\mathfrak{D}_{h}(F,G)^{1/2}$ and $\zeta$ in %
Assumption \ref{a:2ndHa_cond}),%
\begin{equation*}
Y_{L}(t)=t^{2}+2\zeta t\left( V_{2}(c_{* })-c_{*
}V_{1}(1)\right) -\sum_{\ell =1}^{L}\lambda _{\ell }\left[ \left( W_{2,\ell
}(c_{* })-c_{* }W_{1,\ell }(1)\right) ^{2}-c_{* }(1+c_{* })%
\right] . 
\end{equation*}%
Moreover, since $\sum_{\ell \geq 1}\lambda _{\ell }^{2}<\infty $, an
application of Cauchy-Schwarz gives%

\begin{equation}
\lim_{L\rightarrow \infty }\limsup_{m\rightarrow \infty }\P \left\{
\sup_{0\leq t\leq T}|Y_{m}(t)-Y_{m,L}(t)|>x\right\} =0.
\label{e:ZmL_closetoZm}
\end{equation}%
So, if we now let 
\begin{align}
Y(t)& =t^{2}+2\zeta t  \left( V_{2}(c_{* })-c_{*
}V_{1}(1)\right) -\sum_{\ell =1}^{\infty }\lambda _{\ell }\left[ \left(
W_{2,\ell }(c_{* })-c_{* }W_{1,\ell }(1)\right) ^{2}-c_{*
}(1+c_{* })\right]  \notag \\
& = t^{2}+2\zeta t  \left( V_{2}(c_{* })-c_{*
}V_{1}(1)\right) +\mathbb{V}(0,c_{* }),  \label{e:def_Z(t)_whengamma1}
\end{align}
it is easily seen that $\sup_{0\leq t\leq T}|Y_{L}(t)-Y(t)|=o_{\P }(1)$,
implying $Y_{L}\Rightarrow Y$ in $\mathbf{D}[0,T]$, which together with %
\eqref{e:ZmL_to_ZL} and \eqref{e:ZmL_closetoZm} gives 
\begin{equation*}
Y_{m}\Rightarrow Y~\text{in }\mathbf{D}[0,T]. 
\end{equation*}%
Then, the continuous mapping theorem gives 
\begin{equation}
\max_{k_{* }<k\leq k_{* }+(m/\Delta )^{1/2}T}\frac{{\mathcal{D}}_{m}(k)%
}{g_{m}(k)}=\sup_{0\leq t\leq T}\frac{|Y_{m}(t)|}{g(c_{* }+\lfloor
(m/\Delta )^{1/2}t\rfloor /m)}+o_{\P }(1)\Rightarrow \sup_{0\leq t\leq T}%
\frac{|Y(t)|}{g(c_{* })}.  \label{e:alt_gamma>1}
\end{equation}%
In other words, %
\begin{equation*}
\max_{k_{* }<k\leq k_{* }+(m/\Delta )^{1/2}T}\frac{{\mathcal{D}}_{m}(k)%
}{g_{m}(k)}\Rightarrow \sup_{0\leq t\leq T}|Y(t,c_{* })|. 
\end{equation*}%
Thus, %
\begin{align*}
\P \left\{ \frac{\kappa _{m}-k_{* }}{(m/\Delta )^{1/2}}<x\right\} & =\P %
\{\kappa _{m}<k_{* }+x(m/\Delta )^{1/2}\} \\
& =\P \left\{ \max_{k_{* }\leq k\leq k_{* }+x(m/\Delta )^{1/2}}\frac{%
\mathcal{D}_{m}^{(1)}(k)}{g_{m}(k)}>\mathcal{c}\right\} \\
& \rightarrow \P \left\{ \sup_{0\leq t\leq x}|Y(t,c_{* })|>\mathcal{c}%
\right\} \\
& =\P \left\{ \mathcal{H}_{c_{* }}(\mathcal{c})<x\right\} ,
\end{align*}%
as was to be shown.
\end{proof}

\begin{proof}[Proof of Theorem \protect\ref{alt:page}]
For any $y>k_*$, 
\begin{align}
\P \left\{\kappa_m< y\right\} & = \P \left\{\max_{k_*<k\leq y}\frac{{%
\mathcal{D}}_m^{(2)}(k)}{g_m(k)} > \mathcal{c}\right\} \\
& = \P \left\{\max_{k_*<k\leq y}\frac{\max_{0\leq r <k }w^2 | U_m(h;r,k)|}{%
mg_m(k)} > \mathcal{c}\right\}.
\end{align}
The argument for part \textit{(i)} is essentially the same as in the proof
of Theorem \ref{alt:nonpage}\textit{(i)}, so we highlight only the main
differences. With $y=y_m> k_*$ as in \eqref{e:ym_conditions}, from the
bounds \eqref{e:q1_swap}--\eqref{e:q21_swap}, it suffices to show 
\begin{equation}  \label{e:q1+q21_to_normal_2}
b^{-1}_m\left(\max_{k_*<k\leq y_m}\frac{\max_{0\leq r<k}|\mathcal{q}_1(r ,k)
+ \mathcal{q}_{2,1}(r,k)|}{mg_m(k)}-a_m \right)\Rightarrow \mathcal{N}(0,1),
\end{equation}
where $\mathcal{q}_1(r,k)$ is given in \eqref{e:cq1_cq2} and $\mathcal{q}%
_{2,1}(r,k)$ is in \eqref{e:cq_21}. Now, since $\max_{0\leq r<k}|\mathcal{q}%
_1(r,k)|= |\mathcal{q}_1(0,k)|$, from \eqref{e:q1+q21_to_normal}, we have 
\begin{align}
&b_m^{-1}\left(\max_{ k_*<k\leq y_m(1-\delta)}\frac{\max_{0\leq r<k}|%
\mathcal{q}_1(r,k) + \mathcal{q}_{2,1}(r,k)|}{mg_m(k)}-a_m\right) \overset{%
\P }{\to}-\infty. \notag
\end{align}

On the other hand, uniformly for $y_m(1-\delta)\leq k\leq y_m$, an
elementary maximization yields, with probability tending to one, 
$$
\max_{0\leq r<k}
\left|
\mathcal{q}_1(r,k)+\mathcal{q}_{2,1}(r,k)
\right|=\left|
\mathcal{q}_1(0,k)+\mathcal{q}_{2,1}(0,k)
\right|.$$
Hence,
\begin{align*}
\lim_{\delta\to 0}\limsup_{m\to\infty}\P \bigg\{b_m^{-1}\bigg(\max_{
y_m(1-\delta)\leq k \leq y_m}&\frac{\max_{0\leq r<k}|\mathcal{q}_1(r,k) + 
\mathcal{q}_{2,1}(r,k)|}{mg_m(k)} \\
\ & \qquad\qquad\qquad - \frac{|\mathcal{q}_1(0,y_m)+\mathcal{q}_{2,1}(0,y_m)|}{mg_m(y_m)}\bigg) >x\bigg\}=0. 
\end{align*}
From \eqref{e:q21_to_W(1)}, we obtain 
\begin{equation*}
\P \{\kappa_m<y_m\}\to 1-\Phi( -x)=\Phi(x),\quad m\to\infty,
\end{equation*}
and the rest of the proof is identical to that of Theorem \ref{alt:nonpage}%
\textit{(i)}.

Now we turn to part \textit{(ii)}. Recall $k_{* }=c_{* }m$. Write 
\begin{equation*}
\overline{\mathcal{q}}(r,k)=%
\begin{cases}
\mathcal{q}_{1}(r,k)+{\mathcal{q}}_{2,2}(r,k)+(k_{* }-r)^{2}U_{m}(%
\overline{h};r,k_{* }) & 0\leq r\leq k_{* }, \\ 
\mathcal{q}_{1}(r,k)+{\mathcal{q}}_{2,2}(r,k)+\mathcal{q}_{3}(r,k) & 
r>k_{* }.%
\end{cases}%
\end{equation*}%
Using the bounds above, and Lemma \ref{l:replace_with_q1_q2}, 
\begin{equation}
\max_{k_{* }<k\leq k_{* }+(m/\Delta )^{1/2}T}\max_{0\leq r\leq
k-2}\left\vert \frac{(k-r)^2U_{m}(h;r,k)}{mg_{m}(k)}-\frac{\overline{\mathcal{q}}(r,k)%
}{mg_{m}(k)}\right\vert =o_{\P }(1).
\label{e:page_alt_wefirstshow_longchange}
\end{equation}%
Define%
\begin{align*}
Y_{m,1}(s,t)& =m^{-1}\overline{\mathcal{q}}\left( \lfloor ms\rfloor ,k_{*
}+\lfloor (m/\Delta )^{1/2}t\rfloor \right) , & 0 \leq s\leq c_{*
}~0\leq t\leq T, \\
Y_{m,2}(s,t)& =m^{-1}\overline{\mathcal{q}}(k_{* }+\lfloor (m/\Delta
)^{1/2}(s\wedge t)\rfloor ),k_{* }+\lfloor (m/\Delta )^{1/2}t\rfloor ), & 
0 \leq s,t\leq T,
\end{align*}%
so that 
\begin{equation}
\begin{aligned} \sup_{0\leq s \leq c_*}|Y_{m,1}(s,t)| &= m^{-1}\max_{0 \leq r\leq
k_*}\left | \overline{\cq}\big(r,k_* + \lfloor (m/\Delta)^{1/2}
t\rfloor\big)\right|,\\ \sup_{0\leq s \leq t}|Y_{m,2}(s,t)| &= m^{-1}\max_{k_*<r<
k_*+\lfloor (m/\Delta)^{1/2} t\rfloor}\left | \overline{\cq}\big(r,k_* +
\lfloor (m/\Delta)^{1/2} t\rfloor\big)\right|. \end{aligned}
\label{e:max_split}
\end{equation}%
With $V_{1,m}$ and $V_{2,m}(r)$ as in \eqref{e:def_q23}, we have 
\begin{align}
Y_{m,1}(s,t)& =\frac{\lfloor (m/\Delta )^{1/2}t\rfloor ^{2}\theta (\nu
_{1}-\nu _{2})}{m}  \notag \\
& \qquad +\frac{2\theta \sigma\lfloor (m/\Delta )^{1/2}t\rfloor }{%
m^{1/2}}\left[ \frac{V_{2,m}(\lfloor ms\rfloor )}{m^{1/2}}-\frac{\lfloor
(m/\Delta )^{1/2}t\rfloor +k_{* }-\lfloor ms\rfloor }{m}\frac{V_{1,m}}{%
m^{1/2}}\right]  \notag \\
& \qquad +\frac{(k_{* }-\lfloor ms\rfloor )^{2}U_{m}(\overline{h},\lfloor
ms\rfloor ,k_{* })}{m},  \label{e:Ym1}
\end{align}%
and 
\begin{align}
Y_{m,2}(s,t)& =\frac{(\lfloor (m/\Delta )^{1/2}t\rfloor -\lfloor (m/\Delta
)^{1/2}s\rfloor )^{2}\theta (\nu _{1}-\nu _{2})}{m}  \notag \\
& \qquad -\frac{2\theta \sigma(\lfloor (m/\Delta )^{1/2}t\rfloor
-\lfloor (m/\Delta )^{1/2}s\rfloor )}{m^{1/2}}\left[ \frac{\lfloor (m/\Delta
)^{1/2}t\rfloor -\lfloor (m/\Delta )^{1/2}s\rfloor }{m}\frac{V_{1,m}}{m^{1/2}%
}\right]  \notag \\
& \qquad -\frac{(\lfloor (m/\Delta )^{1/2}t\rfloor -\lfloor (m/\Delta
)^{1/2}s\rfloor )^{2}}m\frac{2R_{m,2}}{m(m-1)}  \label{e:Ym2}
\end{align}%
Note the last term in \eqref{e:Ym2} is negligible, since $R_{m,2}=O_\P(m)$, and
\begin{equation*}
-\frac{(\lfloor (m/\Delta )^{1/2}t\rfloor -\lfloor (m/\Delta
)^{1/2}s\rfloor )^{2}}m\frac{2R_{m,2}}{m(m-1)}= O_\P(1/(m\Delta))=o_\P(1).
\end{equation*}%
Arguing
as in \eqref{e:jointweak}, we deduce the joint weak convergence 
\begin{equation*}
\left( m^{-1/2}V_{1,m},m^{-1/2}\left(V_{2,m}(0)-V_{2,m}(\lfloor ms\rfloor)\right)
,\frac{(k_{* }-\lfloor ms\rfloor
)^{2}}mU_{m}(\overline{h},\lfloor ms\rfloor ,k_{* })
\right) \Rightarrow \left( V_{1}(1),V_{2}(s),\mathbb{V}(s,c_{* })
\right) , 
\end{equation*}%
in $\mathbf{D}^{3}[0,c_{* }].$ Then, the Dudley-Wichura-Skorokhod Theorem
gives for each $m\geq 1$, Wiener processes $V_{1}^{(m)}$, $V_{2}^{(m)}$, a
process $\{\mathbb{V}^{(m)}(s,c_{* }),0\leq s\leq c_{* }\}\overset{{%
\mathcal{D}}}{=}\{\mathbb{V}(s,c_{* }),0\leq s\leq c_{* }\}$ such that $(V_{1}^{(m)}(1),V_{2}^{(m)}(s),\mathbb{V}%
^{(m)}(s,c_{* }))^{\top }\overset{{\mathcal{D}}}{=}%
(V_{1}(1),V_{2}(s),\mathbb{V}(s,c_{* }))^{\top }$ in $\mathbf{C}%
^{3}[0,c_{* }]$ satisfying 
\begin{equation*}
\begin{gathered} \sup_{0\leq s \leq c_*}\left|\mathbb
V^{(m)}(s,c_*)-\frac{(k_*-\lfloor m s\rfloor)^2}mU_{m}(\overline h, \lfloor
ms\rfloor, k_*) \right|=o_P(1),\quad \left|m^{-1/2}V_{m,1}
-V_1^{(m)}(1)\right|=o_P(1),\\\sup_{0\leq s \leq
c_*}\left|V_2^{(m)}(s)-\frac{V_{2,m}(0)-V_{2,m}(\lfloor
ms\rfloor)}{m^{1/2}}\right|=o_P(1). \end{gathered}
\end{equation*}%
This gives %
\begin{align*}
& \sup_{0\leq t\leq T}\sup_{0\leq s\leq c_* }\Big|Y_{m,1}(s,t) \\
& \qquad - t^{2}-2\zeta t  \left( V_{2}^{(m)}(c_{*
})-V_{2}^{(m)}(s)-(c_{* }-s)V_{1}^{(m)}(1)\right) -\mathbb{V}%
^{(m)}(s,c_{* })\Big|=o_{P}(1),
\end{align*}%
and 
\begin{equation*}
\sup_{0\leq s,t\leq T}|Y_{m,2}(s,t)- (t-s)^{2}|=o_{P}(1). 
\end{equation*}%
In particular, in view of \eqref{e:max_split}, and the convergence $%
\max_{k_{* }<k\leq k_{* }+(m/\Delta )^{1/2}T}|g_{m}(k)-
g(c_{* })|\to 0$, we obtain 
\begin{align*}
\max_{k_{* }<k\leq k_{* }+(m/\Delta )^{1/2}T}\max_{0\leq r\leq
k}\left\vert \frac{\overline{\mathcal{q}}(r,k)}{mg_{m}(k)}\right\vert &
\Rightarrow \frac{1}{g(c_{* })}\sup_{0\leq t\leq T}\max \left\{
\sup_{0\leq s\leq c_{* }}|Y_{1}(s,t)|,\sup_{0\leq s\leq
t}|Y_{2}(s,t)|\right\} \\
& =\sup_{0\leq t\leq T}Y(t,c_{* }),
\end{align*}%
where
\begin{align*}
Y_{1}(s,t)& = t^{2}+2\zeta t  \left( V_{2}(c_{*
})-V_{2}(s)-(c_{* }-s)V_{1}(1)\right) +\mathbb{V}(s,c_{* }), \\
Y_{2}(s,t)& = (t-s)^{2},\\
Y(t,c_*)& = \frac{1}{g(c_*)}\max \left\{
\sup_{0\leq s\leq c_{* }}|Y_{1}(s,t)|,\sup_{0\leq s\leq
t}|Y_{2}(s,t)|\right\}.
\end{align*}%
Then, with 
$$
Y_*(x,c_*)=\sup_{0\leq t \leq x}Y(t,c_*)
$$
 we may define 
\begin{equation}
\widetilde{\mathcal{H}}_{c_{* }}(u)=\inf \left\{ x\geq 0: Y_*(x,c_*)\geq u\right\} .  \label{e:def_cH(u)2}
\end{equation}%
Recalling \eqref{e:page_alt_wefirstshow_longchange}, we finally have 
\begin{align*}
\P \left\{ \frac{\kappa _{m}-k_{* }}{(m/\Delta )^{1/2}}<x\right\} & =\P %
\{\kappa _{m}<k_{* }+x(m/\Delta )^{1/2}\} \\
& =\P \left\{ \max_{k_{* }\leq k\leq k_{* }+x(m/\Delta )^{1/2}}\frac{%
\mathcal{D}_{m}^{(2)}(k)}{g_{m}(k)}>\mathcal{c}\right\} \\
& \rightarrow \P \left\{ \sup_{0\leq t\leq x}Y(t,c_{* })>\mathcal{c}%
\right\} \\
& =\P \left\{ \widetilde{\mathcal{H}}_{c_{* }}(\mathcal{c})<x\right\} ,
\end{align*}%
as was to be shown.

\end{proof}

\begin{proof}[Proof of Theorem \protect\ref{t:theorem_D3_under_H0}]
The proof is along the same lines as ~Theorem \ref{t:theorem_under_H0}.  Recall
$
b(t)=(t-c_0)_+.$
Define
$$
I_{\delta,T}^{(3)}
=
\{(s,t):\delta\leq t\leq T,\ b(t)\leq s\leq t\}.
$$
We first prove the convergence on $I_{\delta,T}^{(3)}$, where
$0<\delta<T<\infty$. Combining Lemma~\ref{l:D3_approxunderH0},
Lemma~\ref{l:D3_finite_horizon_bounds}(i), and the continuous mapping theorem,
we obtain, for every fixed $L\geq1$,
\begin{align}
\max_{\delta m\leq k\leq mT}
\max_{b_k\leq r\leq k-2}
\left| \frac{m^{-1}(k-r)^2U_{m,L}^{(3)}(\overline h;r,k)} {g_m^{(3)}(k)} \right| \Rightarrow
\sup_{(s,t)\in I_{\delta,T}^{(3)}}
\left|\frac{\mathbb V_L^{(3)}(s,t)} {g\big(t/(1+b(t))\big)(1+b(t))^{\gamma}}\right|,
\label{e:D3_fixed_L_compact_conv}
\end{align}
where $\mathbb V_L^{(3)}$ is defined in \eqref{e:def_VL_D3}.
Here the map is continuous because $\mathbb V_L^{(3)}$ has continuous sample
paths, the set $I_{\delta,T}^{(3)}$ is compact, and the denominator is
bounded away from zero on $I_{\delta,T}^{(3)}$.

Next, Lemma~\ref{l:D3_finite_horizon_bounds}(ii) gives, for every $x>0$,
\begin{align}
&\lim_{L\to\infty}\limsup_{m\to\infty}
P\bigg\{
\bigg|
\max_{\delta m\leq k\leq mT}
\max_{b_k\leq r\leq k-2}
\left|
\frac{m^{-1}(k-r)^2U_m^{(3)}(\overline h;r,k)}
{g_m^{(3)}(k)}
\right|
\notag\\
&\qquad\qquad\qquad\qquad
-
\max_{\delta m\leq k\leq mT}
\max_{b_k\leq r\leq k-2}
\left|
\frac{m^{-1}(k-r)^2U_{m,L}^{(3)}(\overline h;r,k)}
{g_m^{(3)}(k)}
\right|
\bigg|>x
\bigg\}=0.
\label{e:D3_sample_trunc_compact}
\end{align}

The process $\mathbb V^{(3)}(s,t)$, defined pointwise as the
$\mathcal L^2$ limit of $\mathbb V_L^{(3)}(s,t)$ as $L\to\infty$, is
well-defined. Moreover, the same increment bounds used in the proof of
Lemma~\ref{l:trunctail_of_limit} show that
${\mathbb V_L^{(3)}:L\geq1}$ is tight in
$\mathbf C(I_{\delta,T}^{(3)})$ and that $\mathbb V^{(3)}$ admits a
continuous version. Since the $\mathcal L^2$ convergence also implies
convergence of the finite-dimensional distributions, we obtain
\begin{equation}
\mathbb V_L^{(3)}\Rightarrow\mathbb V^{(3)}
\qquad\text{in }\mathbf C(I_{\delta,T}^{(3)}).
\label{e:D3_unweighted_limit_trunc_compact}
\end{equation}
Finally, since
$(s,t)\mapsto g\left(\frac{t}{1+b(t)}\right)(1+b(t))^\gamma$
is continuous and bounded away from zero on $I_{\delta,T}^{(3)}$, the
continuous mapping theorem yields
\begin{equation}
\frac{\mathbb V_L^{(3)}(s,t)}
{g\big(t/(1+b(t))\big)(1+b(t))^\gamma}
\Rightarrow
\frac{\mathbb V^{(3)}(s,t)}
{g\big(t/(1+b(t))\big)(1+b(t))^\gamma}
\qquad\text{in }\mathbf C(I_{\delta,T}^{(3)}).
\label{e:D3_limit_trunc_compact}
\end{equation}
By \eqref{e:D3_fixed_L_compact_conv}, \eqref{e:D3_sample_trunc_compact}, and
\eqref{e:D3_limit_trunc_compact}, we conclude
\begin{align}
&\max_{\delta m\leq k\leq mT}
\max_{b_k\leq r\leq k-2}
\left|
\frac{m^{-1}(k-r)^2U_m^{(3)}(\overline h;r,k)}
{g_m^{(3)}(k)}
\right| \Rightarrow
\sup_{(s,t)\in I_{\delta,T}^{(3)}}
\left|
\frac{\mathbb V^{(3)}(s,t)}
{g\big(t/(1+b(t))\big)(1+b(t))^{\gamma}}
\right|.
\label{e:D3_compact_conv}
\end{align}

We now consider the range $0<t<\delta$. Lemma~\ref{l:D3_finite_horizon_bounds}(iii)
gives, for every $x>0$,
\begin{align}
\lim_{\delta\downarrow0}\limsup_{m\to\infty}
P\left\{
\max_{2\leq k\leq m\delta}
\max_{b_k\leq r\leq k-2}
\left|
\frac{m^{-1}(k-r)^2U_m^{(3)}(\overline h;r,k)}
{g_m^{(3)}(k)}
\right|>x
\right\}=0.
\label{e:D3_lower_endpoint_sample}
\end{align}
For the limiting process, if $\delta<c_0$, then $b(t)=0$ on
$0\leq t\leq\delta$, and $\mathbb V^{(3)}(s,t)$ coincides with $\mathbb V(s,t)$, and we have
\begin{align}
\lim_{\delta\downarrow0}
P\left\{ \sup_{0<t\leq\delta}\sup_{0\leq s\leq t}
\left|\frac{\mathbb V^{(3)}(s,t)}{g(t)}\right|>x
\right\}=0.
\label{e:D3_lower_endpoint_limit}
\end{align}
We now turn to the range $t>T$.  Lemma~\ref{l:D3_infinite_horizon_bound}
gives, for every $x>0$,
\begin{align}
\lim_{T\to\infty}\limsup_{m\to\infty}
P\left\{
\sup_{k\geq mT}
\max_{b_k\leq r\leq k-2}
\left|
\frac{m^{-1}(k-r)^2U_m^{(3)}(\overline h;r,k)}
{g_m^{(3)}(k)}
\right|>x
\right\}=0.
\label{e:D3_upper_endpoint_sample}
\end{align}
We also need the corresponding statement for $\mathbb V^{(3)}$. Fix $T<S<\infty$.
By \eqref{e:D3_compact_conv},
\begin{align*}
&\max_{mT\leq k\leq mS}\max_{b_k\leq r\leq k-2}
\left| \frac{m^{-1}(k-r)^2U_m^{(3)}(\overline h;r,k)}
{g_m^{(3)}(k)}
\right| \\
&\qquad\Rightarrow
\sup_{T\leq t\leq S} \sup_{b(t)\leq s\leq t} \left| \frac{\mathbb V^{(3)}(s,t)}
{g\big(t/(1+b(t))\big)(1+b(t))^{\gamma}}\right|.
\end{align*}
Now,
\begin{align*}
&P\left\{
\sup_{T\leq t\leq S} \sup_{b(t)\leq s\leq t}
\left|\frac{\mathbb V^{(3)}(s,t)}
{g\big(t/(1+b(t))\big)(1+b(t))^{\gamma}}\right|>x
\right\} \\
&\qquad\leq
\liminf_{m\to\infty}
P\left\{\max_{mT\leq k\leq mS}\max_{b_k\leq r\leq k-2}\left|
\frac{m^{-1}(k-r)^2U_m^{(3)}(\overline h;r,k)}
{g_m^{(3)}(k)}\right|>x\right\} \\
&\qquad\leq
\limsup_{m\to\infty}
P\left\{
\sup_{k\geq mT}\max_{b_k\leq r\leq k-2}\left|\frac{m^{-1}(k-r)^2U_m^{(3)}(\overline h;r,k)}
{g_m^{(3)}(k)}\right|>x\right\}.
\end{align*}
Letting $S\to\infty$ and then $T\to\infty$, and using
\eqref{e:D3_upper_endpoint_sample}, gives
\begin{align}
\lim_{T\to\infty}
P\left\{\sup_{t\geq T}\sup_{b(t)\leq s\leq t}\left|\frac{\mathbb V^{(3)}(s,t)}
{g\big(t/(1+b(t))\big)(1+b(t))^{\gamma}}
\right|>x
\right\}=0.
\label{e:D3_upper_endpoint_limit}
\end{align}
Putting together \eqref{e:D3_compact_conv}--\eqref{e:D3_upper_endpoint_limit} gives
\begin{equation}\label{e:D3_penultimate}
\sup_{k\geq2}
\max_{b_k\leq r\leq k-2}
\left| \frac{m^{-1}(k-r)^2U_m^{(3)}(\overline h;r,k)}
{g_m^{(3)}(k)} \right| \Rightarrow
\sup_{t>0}\sup_{b(t)\leq s\leq t}
\left| \frac{\mathbb V^{(3)}(s,t)}
{g\big(t/(1+b(t))\big)(1+b(t))^{\gamma}}\right|
\end{equation}
It remains only to show the right-hand side in \eqref{e:D3_penultimate} is equal in law to the limit. Let
$$
u=\frac{t}{1+t},\qquad v=\frac{s}{1+s},\qquad
y=\frac{b(t)}{1+b(t)} 
$$
Moreover, since $b(t)\leq s\leq t$,  we have $y(u)\leq v\leq u$. Now, for each $\ell$, the Gaussian process 
$
\widetilde Z_\ell(t)=W_{2,\ell}(t)-tW_{1,\ell}(1).
$ has the same
distribution as
${(1+t)W_\ell(t/(1+t)),t\geq0}$. Hence \begin{align*}
&W_{2,\ell}(t)-W_{2,\ell}(s)
-\frac{t-s}{1+b(t)}
\big(W_{1,\ell}(1)+W_{2,\ell}(b(t))\big) \\
&\qquad =
\widetilde Z_\ell(t)-\widetilde Z_\ell(s)
-\frac{t-s}{1+b(t)}\widetilde Z_\ell(b(t)) \\
&\qquad\overset{\mathcal D}{=}
(1+t)
\left[W_\ell(u)-W_\ell(y) - \frac{1-u}{1-v}\big(W_\ell(v)-W_\ell(y)\big)
\right].
\end{align*}
Also,
$$
\frac{(t-s)+(t-s)^2/(1+b(t))}{
(1+t)^2}=
(u-v)+\left(\frac{u-v}{1-v}\right)^2(v-y).
$$
Therefore, %
$$
(1+t)^{-2}|\mathbb V^{(3)}(s,t)|
\overset{\mathcal D}{=}
\left|G_{y(u)}(u,v)\right|.
$$
Consequently, the right-hand side of \eqref{e:D3_penultimate} is equal in
law to
$$
\sup_{0<u<1}
\sup_{y(u)\leq v<u}
\frac{|G_{y(u)}(u,v)|}
{(1-u)^2
g\left(\dfrac{t}{1+b(t)}\right)
\big(1+b(t)\big)^\gamma}.
$$
Using the identities
$$
1+b(t)=\frac{1}{1-y(u)}, \quad 
\frac{t}{1+b(t)} =\frac{u\big(1-y(u)\big)}{1-u},
$$
we obtain, 
\begin{align*}
&(1-u)^2
g\left(\frac{t}{1+b(t)}\right)
\big(1+b(t)\big)^\gamma  \\
&\qquad =
(1-u)^2
\left(\frac{u\big(1-y(u)\big)}{1-u}\right)^{\beta}
\left(1+\frac{u\big(1-y(u)\big)}{1-u}\right)^{2-\beta}
\big(1-y(u)\big)^{-\gamma} \\
&\qquad =
u^\beta
\big(1-y(u)\big)^{\beta-\gamma}
\big(1-uy(u)\big)^{2-\beta} \\
&\qquad =
d^{(3)}(u).
\end{align*}
Thus,
$$
\sup_{t>0}\sup_{b(t)\leq s\leq t}
\left|
\frac{\mathbb V^{(3)}(s,t)}
{g\left(t/(1+b(t))\right)\big(1+b(t)\big)^\gamma}
\right|
\overset{\mathcal D}{=}
\sup_{0<u<1}
\frac{\overline\Gamma^{(3)}(u)}{d^{(3)}(u)},
$$
giving (i). The closed-ended result (ii) follows by the same argument, with the supremum restricted to the finite monitoring interval, and its proof is therefore omitted.
\end{proof}

\begin{proof}[Proof of Theorem \protect\ref{thewi}]
With $t_k=k/m$,
$$
\sup_{0<t<1}\frac{|\mathfrak r_m(t)|}{\mathfrak q(t)}
=
\max_{2\leq k\leq m-2}
\frac{1}{\mathfrak q(t_k)}
\frac{k^2(m-k)^2}{m^3}|\mathfrak R(k)|
+o_\P(1).
$$
Fix $L\geq1$. By Lemma \ref{prwe2}, for every $x>0$,
$$
\lim_{L\to\infty}\limsup_{m\to\infty}
\P\left\{\max_{2\leq k\leq m-2} \frac{1}{\mathfrak q(t_k)} \frac{k^2(m-k)^2}{m^3}|\mathfrak R(k)-\mathfrak R_L(k)|>x
\right\}=0.
$$
Lemma
\ref{l:remainder_neglig_R_L} gives
$$
\max_{2\leq k\leq m-2}
\frac{1}{\mathfrak q(t_k)}
\left|\frac{k^2(m-k)^2}{m^3}\mathfrak R_L(k)
+\frac1m\sum_{\ell=1}^L\lambda_\ell
\left[\left(S_\ell(k)-t_kS_\ell(m)\right)^2-\frac{k(m-k)}{m}\right]\right|=o_\P(1).
$$
By Lemma \ref{prwe4}, 
$$
\max_{2\leq k\leq m-2}
\frac{1}{\mathfrak q(t_k)}
\frac{k^2(m-k)^2}{m^3}
|\mathfrak R_L(k)|
\Rightarrow
\sup_{0<t<1}
\frac{1}{\mathfrak q(t)}
\left|
\sum_{\ell=1}^L\lambda_\ell
\left(B_\ell^2(t)-t(1-t)\right)
\right|.
$$
The  same tail bound as in Lemma \ref{prwe2} gives, for every $x>0$,
$$
\P\left\{
\sup_{0<t<1}
\frac{1}{\mathfrak q(t)}
\left|
\sum_{\ell=L+1}^\infty
\lambda_\ell\{B_\ell^2(t)-t(1-t)\}
\right|>x
\right\}
\leq
\frac{C}{x^2}\sum_{\ell=L+1}^\infty\lambda_\ell^2
\to0.
$$
as $L\to\infty$.  Combining the above statements yields the result.
\end{proof}

\begin{proof}[Proof of Theorem \protect\ref{metric-neg-type}]  We show that $\mathfrak D_{\delta^{1/2}}(\nu_1,\nu_2)=0$ if and only if $\nu_1=\nu_2$.  
Recall that, by the Moore-Aronszajn theorem, the positive (semi)definite
kernel $K$ yields a unique RKHS $\mathcal{H}_{K}$ of
real-valued functions on $\mathcal{X}$ with reproducing kernel $K$. Consider the map $x\mapsto \varphi(x):=K\left( \cdot
,x\right)$. By assumption, $\varphi$ is injective, and by the reproducing property, 
\begin{equation*}
K\left( x,y\right) =\left\langle \varphi \left( x\right) ,\varphi \left(
y\right) \right\rangle _{\mathcal{H}_{K}},
\end{equation*}%
and therefore%
\begin{equation*}
\delta ^{1/2}\left( x,y\right) =\left[ K\left( x,x\right) +K\left(
y,y\right) -2K\left( x,y\right) \right] ^{1/2}=\left\Vert \varphi \left(
x\right) -\varphi \left( y\right) \right\Vert _{\mathcal{H}_{K}}.
\end{equation*}%
Moreover, since $K$ is continuous, so is $%
\varphi \left( x\right) $ by Lemma 4.29 in \citet{christmann2008support}, and separability of $\mathcal X$ then gives that $\mathcal{H}_{K}$ is separable (Lemma 4.33 in %
\citealp{christmann2008support}).
Hence, by Theorem 3.16 in \citet{lyons:2013}, the space $\left( \mathcal{H}%
_{K},\rho \right) $ is of strong negative type, having defined $\rho \left(
x,y\right) =\left\Vert x-y\right\Vert _{\mathcal{H}_{K}}$. In other words,
if $\mathbb{P}_{1}$ and $\mathbb{P}_{2}$ are two Borel measures defined on $%
\mathcal{H}_{K}$, given $Z$, $Z^{\prime }\overset{i.i.d.}{\sim }\mathbb{P}%
_{1}$ and $W$, $W^{\prime }\overset{i.i.d.}{\sim }\mathbb{P}_{2}$, the
quantity%
\begin{equation*}
\mathfrak{D}_{\rho }\left( \mathbb{P}_{1},\mathbb{P}_{2}\right) =2{\mathsf{E}}%
\left\Vert W-Z\right\Vert _{\mathcal{H}_{K}}-{\mathsf{E}}\left\Vert
W-W^{\prime }\right\Vert _{\mathcal{H}_{K}}-{\mathsf{E}}\left\Vert
Z-Z^{\prime }\right\Vert _{\mathcal{H}_{K}},
\end{equation*}%
is zero if and only if $\mathbb{P}_{1}=\mathbb{P}_{2}$. Consider now any two
Borel probability measures $\nu _{1}$ and $\nu _{2}$ on $\mathcal{X}$, and
let $\mathbb{P}_{i}=\nu _{i}\circ \varphi ^{-1}$. Then, if $X$, $X^{\prime }%
\overset{i.i.d.}{\sim }\nu _{1}$ and $Y$, $Y^{\prime }\overset{i.i.d.}{\sim }%
\nu _{2}$, it holds that $\varphi \left( X\right) \sim \mathbb{P}_{1}$ and $%
\varphi \left( Y\right) \sim \mathbb{P}_{2}$, and 
\begin{eqnarray*}
\mathfrak{D}_{\delta^{1/2} }\left( \nu _{1},\nu _{2}\right) &=&2{\mathsf{E}}%
\left\Vert \varphi \left( X\right) -\varphi \left( Y\right) \right\Vert _{%
\mathcal{H}_{K}}-{\mathsf{E}}\left\Vert \varphi \left( X\right) -\varphi
\left( X^{\prime }\right) \right\Vert _{\mathcal{H}_{K}}-{\mathsf{E}}%
\left\Vert \varphi \left( Y\right) -\varphi \left( Y^{\prime }\right)
\right\Vert _{\mathcal{H}_{K}} \\
&=&\mathfrak{D}_{\rho }\left( \mathbb{P}_{1},\mathbb{P}_{2}\right) .
\end{eqnarray*}%
Hence, if $\mathfrak D_{\delta^{1/2}}(\nu_1,\nu_2)=0$, then $\mathbb P_1=\mathbb P_2$, i.e., for any Borel set $B\subseteq \mathcal{H}%
_{K}$, $\mathbb{P}_{1}\left( B\right) =\mathbb{P}_{2}\left( B\right) $.
Consider now a compact set $A\subseteq \mathcal X$; then, $\varphi
\left( A\right) $ also is compact - and therefore it is a Borel set in $%
\mathcal{H}_{K}$ - and therefore, by injectivity of $\varphi$, $\mathbb P_i(\varphi(A))= \nu_i(\varphi^{-1}(\varphi(A)))=\nu_i(A)$, so
\begin{equation*}
\nu _{1}\left( A\right) =\mathbb{P}_{1}\left( \varphi \left( A\right)
\right) =\mathbb{P}_{2}\left( \varphi \left( A\right) \right) =\nu
_{2}\left( A\right) .
\end{equation*}
Given that $\mathcal{X}$ is a complete and separable metric space, every
Borel measure is Radon (Theorem 7.1.7, \citealp{bogachev2007measure});
hence, for any Borel set $C\subseteq \mathcal{X}$ 
\begin{equation*}
\nu _{1}\left( C\right) =\sup \left\{ \nu _{1}\left( D\right) :D\subseteq C,D%
\text{ compact}\right\} =\sup \left\{ \nu _{2}\left( D\right) :D\subseteq C,D%
\text{ compact}\right\} =\nu _{2}\left( C\right),
\end{equation*}
i.e., $\mathfrak{D}_{\delta^{1/2} }\left( \nu _{1},\nu _{2}\right) =0$ implies $\nu_1=\nu_2$.  On the other hand, since $\nu_1=\nu_2$ immediately gives $\mathfrak{D}_{\delta^{1/2} }\left( \nu _{1},\nu _{2}\right) =0$, the proof is complete.
\end{proof}

\begin{proof}[Proof of Theorem \protect\ref{p:eig_approx}] For $\ell>m$, set $\widehat \lambda_{\ell,m}=0$.  
Fix a collection of independent Wiener processes $\{\{W_{i,\ell}(t),t\geq0%
\},\ell \geq 1,i=1,2\}$ independent of $\mathcal{F}=\sigma(X_1,X_2,\ldots)$.
Defining $\widetilde Y(s,t)$ as in \eqref{e:Ytilde} based on $%
\{W_{i,\ell},\ell\geq 1,i=1,2\},$ let 
\begin{align*}
\widehat{\mathcal{V}}_m(s,t)&= - (1+t)^{\beta-2}\sum_{\ell=1}^\infty 
\widehat{\lambda}_{\ell,m}\widetilde Y_\ell(s,t)\overset{d}{=}%
-(1+t)^{\beta-2}\sum_{\ell=1}^\infty \widehat{\lambda}_{\pi(\ell),m}\widetilde Y_\ell(s,t)=:\widehat{\mathcal{V}}_m(s,t;\pi)
\end{align*}
where $\pi:\{1,2,\ldots,\}\to\{1,2,\ldots,\}$ is any permutation. Similarly,
we may construct ${\mathcal{V}}(s,t)$ as in Lemma %
\ref{l:trunctail_of_limit} based on this same sequence of Wiener
processes, so that 
\begin{align*}
{\mathcal{V}}(s,t)-\widehat{\mathcal{V}}_m(s,t;\pi)&= -
(1+t)^{\beta-2}\sum_{\ell=1}^\infty (\lambda_\ell-\widehat{\lambda}%
_{\pi(\ell),m})\widetilde Y_\ell(s,t).
\end{align*} 
Pick a sequence of permutations $\pi_m$ such that, for each $m$,
$$\sum_{\ell=1}^\infty (\lambda_\ell-\widehat{\lambda}%
_{\pi_m(\ell),m})^2\leq \inf_\pi\sum_{\ell=1}^\infty (\lambda_\ell-\widehat{\lambda}%
_{\pi(\ell),m})^2+1/m.$$
Then,  
\begin{align*}
{\mathsf{E}}\hspace{0.1mm}[({\mathcal{V}}(s,t)-\widehat{\mathcal{V}}%
_m(s,t;\pi_m))^2|\mathcal{F}]& = (1+t)^{2(\beta-2)}{\mathsf{E}}\hspace{0.1mm}%
[\widetilde Y_1^2(s,t)|\mathcal{F}]\sum_{\ell=1}^\infty (\lambda_\ell-%
\widehat{\lambda}_{\pi_m(\ell),m})^2 \\
& = C(s,t) \sum_{\ell=1}^\infty (\lambda_\ell-\widehat{\lambda}%
_{\pi_m(\ell),m})^2\to 0,\quad \text{a.s.},\ 
\end{align*}
where we used that $\inf_\pi\sum_{\ell=1}^\infty (\lambda_\ell-\widehat{\lambda}%
_{\pi(\ell),m})^2\to 0$ a.s. as $m\to\infty$ as a consequence of \cite[Theorem 3.1%
]{koltchinskii:gine:2000}. In particular, this implies
for each $n\geq 1$ and any $s_1,t_1,\ldots,s_n,t_n\geq0$, 
\begin{equation*}
\left(\widehat {\mathcal{V}}_m(s_1,t_1),\ldots \widehat {\mathcal{V}}_m(s_n,t_n)\right)%
\Rightarrow_{\mathcal{F}} \left({\mathcal{V}} (s_1,t_1),\ldots {\mathcal{V}}%
(s_n,t_n)\right).
\end{equation*}
For tightness, recall from the proof of Lemma~\ref{l:trunctail_of_limit} that, for each $r>1$,
$$
\E\left|
\widetilde Y_\ell(s_1,t_1)-\widetilde Y_\ell(s_2,t_2)
\right|^{2r}
\leq
C\bigl(|t_1-t_2|+|s_1-s_2|\bigr)^{2ar}
$$
for some $0<a<1-\beta$, where $C$ may depend on $r$. Hence, Rosenthal's inequality gives
\begin{align}
&\E\Big[
\big|
(1+t_1)^{2-\beta}\widehat{\mathcal V}_m(s_1,t_1)-(1+t_2)^{2-\beta}\widehat{\mathcal V}_m(s_2,t_2)\big|^{2r}\Big|\mathcal F \Big] \notag\\
&\quad\leq C_r\Bigg[
\sum_{\ell=1}^{\infty}|\widehat\lambda_{\ell,m}|^{2r}
\E\left|\widetilde Y_\ell(s_1,t_1)-\widetilde Y_\ell(s_2,t_2)\right|^{2r}\notag\\
&\qquad\qquad+\notag\
\left(\sum_{\ell=1}^{\infty}\widehat\lambda_{\ell,m}^{2}
\E\left|\widetilde Y_\ell(s_1,t_1)-\widetilde Y_\ell(s_2,t_2)
\right|^{2}\right)^r\Bigg]\\
&\quad\leq
C\left(\sum_{\ell=1}^{\infty}\widehat\lambda_{\ell,m}^{2}\right)^r
\bigl(|t_1-t_2|+|s_1-s_2|\bigr)^{2ar}.
\label{e:eig_approx_tightness}
\end{align}
Moreover, $\inf_\pi\sum_{\ell=1}^\infty (\lambda_\ell-\widehat{\lambda}%
_{\pi(\ell),m})^2\to 0$ a.s. and square summability of $\{\lambda_\ell,\ell \geq 1\}$ gives  
$
\sup_{m\geq1}\sum_{\ell=1}^{\infty}\widehat\lambda_{\ell,m}^{2}
<\infty
$ \text{a.s.}. Hence, almost surely,
$$
\E\Big[\big|(1+t_1)^{2-\beta}\widehat{\mathcal V}_m(s_1,t_1)-(1+t_2)^{2-\beta}\widehat{\mathcal V}_m(s_2,t_2)\big|^{2r}\Big| \mathcal F\Big]
\leq
C(\omega)
\bigl(|t_1-t_2|+|s_1-s_2|\bigr)^{2ar}.
$$
By taking $r$ large enough, we obtain that for a.s. $%
\omega$, the law of $\{(1+t)^{2-\beta}\widehat{\mathcal{V}}_m(s,t),~~0\leq s,t\leq T\}$ in $C([0,T]^2)$ is tight under $\P (\cdot|\mathcal{F%
})(\omega)$, for any $T>0$ and hence the same is true of $\{\widehat{\mathcal{V}}%
_m(s,t),~~0\leq s,t\leq T\}$. We obtain
\begin{equation}\label{e:V-hat_to_V_in_C0T}
\{\widehat{\mathcal{V}}_m(s,t),~~0\leq s,t\leq T\} \Rightarrow_{\mathcal F} \{{\mathcal{V}}(s,t),~~0\leq s,t\leq T\} 
\end{equation}
in $\mathbf C([0,T]^2)$, for each $T>0$. Turning to the range $t\in[T,\infty)$, define
$$
\widehat{\mathcal H}_m(\pi_m)
=\sum_{\ell=1}^{\infty}\widehat\lambda_{\pi_m(\ell),m}
\left(1-W_{1,\ell}^2(1)\right).
$$
Conditioning on $\mathcal F$, and repeating the argument leading to
\eqref{e:V_tail_limit}, with $\lambda_\ell$ replaced by
$\widehat\lambda_{\pi_m(\ell),m}$, 
gives, for every $x>0$,
\begin{align}
\lim_{T\to\infty}\limsup_{m\to\infty}
\P\Bigg(&\sup_{t\geq T}\sup_{0\leq s\leq t}
\left|\widehat{\mathcal V}_m(s,t;\pi_m)-\frac{(t-s)^2}{g(t)}\widehat{\mathcal H}_m(\pi_m)\right|>x\Bigg
\rvert\,\mathcal F\Bigg)=0
\label{e:eig_empirical_tail}
\end{align}
almost surely. Moreover,
\begin{align*}
\E\left[\left|\widehat{\mathcal H}_m(\pi_m)-\mathcal H\right|^2\Big\rvert\mathcal F\right] &=2\sum_{\ell=1}^{\infty}\left(\widehat\lambda_{\pi_m(\ell),m}-\lambda_\ell\right)^2 \to0
\end{align*}
almost surely. Since
$$
0\leq\frac{(t-s)^2}{g(t)}\leq\frac{t^2}{g(t)}
=\left(\frac{t}{1+t}\right)^{2-\beta}
\leq1,
$$
we have
\begin{align*}
\sup_{t\geq T}\sup_{0\leq s\leq t}
\left|\widehat{\mathcal V}_m(s,t;\pi_m)-\mathcal V(s,t)
\right|&\leq
\sup_{t\geq T}\sup_{0\leq s\leq t}\left|\widehat{\mathcal V}_m(s,t;\pi_m)
-\frac{(t-s)^2}{g(t)}\widehat{\mathcal H}_m(\pi_m)\right|\\
&\qquad+
\left|\widehat{\mathcal H}_m(\pi_m)-\mathcal H\right|+\sup_{t\geq T}\sup_{0\leq s\leq t}\left|
\mathcal V(s,t)-\frac{(t-s)^2}{g(t)}\mathcal H
\right|.
\end{align*}
Hence, by \eqref{e:eig_empirical_tail} and
\eqref{e:V_tail_limit}, 
\begin{align*}
\lim_{T\to\infty}\limsup_{m\to\infty}
\P\left(
\sup_{t\geq T}\sup_{0\leq s\leq t}
\left|
\widehat{\mathcal V}_m(s,t;\pi_m)-\mathcal V(s,t)
\right|>x
\bigg\rvert\,\mathcal F
\right)
=0
\end{align*}
almost surely. Putting this together with \eqref{e:V-hat_to_V_in_C0T}, 
we conclude, for any $a_0\in (0,\infty]$,
$$
\sup_{0\leq t <a_0} |\widehat{\mathcal V}_m(0,t;\pi_m)| \Rightarrow_{\mathcal F} \sup_{0\leq t <a_0}|{\mathcal V}(0,t)|,
$$
$$\sup_{0\leq t <a_0}\sup_{0\leq s\leq t} |\widehat{\mathcal V}_m(s,t;\pi_m)| \Rightarrow_{\mathcal F}  \sup_{0\leq t <a_0}\sup_{0\leq s\leq t} |{\mathcal V}(s,t)|.
$$
 Using the change of variables $u=t/(1+t)$, with $u_0=a_0/(1+a_0)$, we then obtain
$$
\sup_{0\leq u \leq u_0}u^{-\beta}|\widehat\Gamma_m(u)| \Rightarrow_{\mathcal F} \sup_{0\leq u \leq u_0}u^{-\beta}|\Gamma(u)|,
\qquad \sup_{0\leq u \leq u_0}u^{-\beta}\widehat{\overline\Gamma}_m(u)\Rightarrow_{\mathcal F}\sup_{0\leq u \leq u_0}u^{-\beta}\overline\Gamma(u).$$
The third assertion follows from the same argument applied to
$\mathbb V^{(3)}$ and its empirical-eigenvalue analog ($\widehat{\mathbb V}^{(3)}_m$, say) except that the
normalized processes converge to zero as $t\to\infty$, so that no analog
of $\mathcal H$ is needed. Indeed, repeating conditionally on $\mathcal F$
the argument used to obtain \eqref{e:D3_upper_endpoint_limit}, with
$\lambda_\ell$ replaced by
$\widehat\lambda_{\pi_m(\ell),m}$ and using
$
\sup_{m\geq1}\sum_{\ell=1}^{\infty}
\widehat\lambda_{\ell,m}^{2}<\infty$ a.s. 
gives, for every $x>0$,
\begin{align*}
\lim_{T\to\infty}\limsup_{m\to\infty}
\P\Bigg(&\sup_{t\geq T}\sup_{b(t)\leq s\leq t}
\left|\frac{\widehat{\mathbb V}^{(3)}_m(s,t)}{g\big(t/(1+b(t))\big)\big(1+b(t))^\gamma}\right|>x\bigg\rvert\,\mathcal F\Bigg)
=0
\end{align*}
almost surely. Together with the corresponding compact convergence and
\eqref{e:D3_upper_endpoint_limit}, the change of variables used in the proof
of Theorem~\ref{t:theorem_D3_under_H0} yields the third assertion of
\eqref{e:conditional_weak_convergence}.

\end{proof}

\begin{adjustwidth}{-0pt}{-0pt}

{\footnotesize {\ 
\bibliographystyle{chicago}
\bibliography{BHTbiblio}
} }

\end{adjustwidth}

\end{document}